\documentclass[Numbered]{sn-jnl}

\usepackage{graphicx}%
\usepackage{multirow}%
\usepackage{amsmath,amssymb,amsfonts}%
\usepackage{amsthm}%
\usepackage{mathrsfs}
\usepackage[title]{appendix}%
\usepackage{xcolor}%
\usepackage{textcomp}%
\usepackage{manyfoot}%
\usepackage{booktabs}%
\usepackage{algorithm}%
\usepackage{algorithmicx}%
\usepackage{algpseudocode}%
\usepackage{listings}

\usepackage{mathtools}
\usepackage[normalem]{ulem}
\usepackage{epsfig}
\usepackage{enumitem,subcaption}
\setenumerate{label=\arabic*.,ref=\arabic*}
\usepackage{setspace}
\usepackage{comment}
\usepackage[all]{xy}
\usepackage{hyperref}
\usepackage{setspace}
\usepackage{mathrsfs}
\usepackage[title]{appendix}

\newtheorem{thm}{Theorem}[section]

\newtheorem{cor}[thm]{Corollary}

\newtheorem{lemma}[thm]{Lemma}
\newtheorem{prop}[thm]{Proposition}

\newtheorem{defn}[thm]{Definition}
\newtheorem{definition}[thm]{Definition}
\newtheorem{remark}[thm]{Remark}
\newtheorem{ex}[thm]{Example}
\newtheorem{nota}[thm]{Notation}

\newtheorem{conj}[thm]{Conjecture}
\newtheorem{prob}[thm]{Problem}

\newtheorem{ques}[thm]{Open Question}

\numberwithin{equation}{section}

\newcommand{\bs}{\backslash}
\newcommand{\stek}{\operatorname{Stek}}  
\newcommand{\spec}{\operatorname{Spec}} 
\newcommand{\restr}[1]{\lower0.4ex\hbox{$|$}\lower0.7ex
	\hbox{$\scriptstyle{#1}$}}

\DeclareMathOperator\arcosh{arccosh}

\newcommand{\bdry}{\Sigma}

\newcommand{\Cx}{\mathbf C} 
\newcommand{\N}{\mathbf N} 
\newcommand{\R}{\mathbf R} 
\newcommand{\Z}{\mathbf Z}
\newcommand{\D}{\mathbb D}

\newcommand{\om}{\Omega}
\newcommand{\DtN}{\mathcal{D}}

\newcommand{\Sp}{\mathbf S} 

\newcommand{\eps}{\varepsilon}
\newcommand{\lap}{\Delta} 
\newcommand{\pp}{(p)}
\newcommand{\Acal}{\mathcal{A}}
\newcommand{\rs}{\operatorname{rs}} 
\newcommand{\rssig}{\sigma^{\rs}}
\newcommand{\sig}{\Sigma}
\newcommand{\hsig}{\widehat{\sigma}^{\,\rs}}
\newcommand{\lapsp}{\lap_\sig^{\pp}}
\newcommand{\hla}{\widehat{\lambda}}
\newcommand{\rsDtN}{\DtN^{\rs}}
\newcommand{\DtNp}{\rsDtN_p}
\newcommand{\Hcal}{\mathcal{H}}
\newcommand{\kDtN}{\DtN^K}
\newcommand{\teta}{\widetilde{\eta}}
\newcommand{\hsigk}{\widehat{\sigma}^{K}}
\newcommand{\ksig}{\sigma^K}
\newcommand{\orb}{\mathcal{O}}
\newcommand{\pam}{\partial\om}
\newcommand{\B}{\mathbb{B}}
\newcommand{\sph}{\mathbb{S}}
\newcommand{\gap}{\operatorname{Gap}} 
\newcommand{\diam}{\operatorname{Diam}}
\newcommand{\inj}{\operatorname{inj}}
\newcommand{\mob}{\text{M\"ob}} 
\newcommand{\bla}{\overline{\lambda}}
\newcommand{\bsig}{\overline{\sigma}}
\newcommand{\mom}{\mathcal{M}_\om}

\makeatletter
\@namedef{subjclassname@2020}{\textup{2020} Mathematics Subject Classification}
\makeatother

\begin{document}
\title[Steklov Eigenvalue Problem]{Some recent developments on the Steklov Eigenvalue Problem}

\abstract{
    The Steklov eigenvalue problem, first introduced over 125 years ago, has seen a surge of interest in the past few decades.  This article is a tour of some of the recent developments linking the Steklov eigenvalues and eigenfunctions of compact Riemannian manifolds to the geometry of the manifolds.    Topics include isoperimetric-type upper and lower bounds on Steklov eigenvalues (first in the case of surfaces and then in higher dimensions), stability and instability of eigenvalues under deformations of the Riemannian metric, optimisation of eigenvalues and connections to free boundary minimal surfaces in balls, inverse problems and isospectrality, discretisation, and the geometry of eigenfunctions.   We begin with background material and motivating examples for readers that are new to the subject.   Throughout the tour, we frequently compare and contrast the behavior of the Steklov spectrum with that of the Laplace spectrum.   We include many open problems in this rapidly expanding area.
}

\author[1]{\fnm{Bruno}\sur{Colbois}}\email{bruno.colbois@unine.ch}

\author[2]{\fnm{Alexandre}\sur{Girouard}}\email{alexandre.girouard@mat.ulaval.ca}

\author[3]{\fnm{Carolyn}\sur{Gordon}}\email{csgordon@dartmouth.edu}

\author[4]{\fnm{David}\sur{Sher}}\email{dsher@depaul.edu}

\affil[1]{\orgdiv{Institut de Math\'ematiques},\orgname{Universit\'e de Neuch\^atel}, \street{Rue
  Emile-Argand 11}, \postcode{CH-2000}\city{Neuch\^atel}, \country{Switzerland}}
  
    \affil[2]{\orgdiv{D\'epartement de math\'ematiques et de statistique}, \street{Pavillon Alexandre-Vachon}, \orgname{Universit\'e Laval}, \city{Qu\'ebec QC}, \postcode{G1V 0A6}, \country{Canada}}
    
    \affil[3]{\orgdiv{Department of Mathematics}, \orgname{Dartmouth College}, \city{Hanover}, \state{NH}, \postcode{03755}, \country{USA}}

 \affil[4]{\orgdiv{Department of Mathematical Sciences}, \orgname{DePaul University}, \street{2320 N Kenmore Ave.}, \city{Chicago}, \state{IL}, \postcode{60614}, \country{USA}}


\maketitle

\setcounter{tocdepth}{1}
\tableofcontents

\subsection*{Acknowledgements}
The authors would like to thank the following people who have read preliminary versions of this paper and helped the authors improve the presentation in various ways: Iosif Polterovich, Alessandro Savo, Asma Hassanezhad, Katie Gittins, Chakradhar, Jean Lagac\'e, Mikhail Karpukhin, Antoine M\'etras, Nilima Nigam, and Romain Petrides. We thank David Webb for providing several drawings.  Bruno Colbois and Alexandre Girouard would also like to thank the participants in the Qu\'ebec-Neuch\^atel-Montr\'eal doctoral seminar, who have patiently attended lectures on many parts of this paper during the year 2021-2022, and Léonard Tschanz for his pictures.

\subsection*{Funding}
The first author acknowledges support of the Swiss National Science Foundation project ‘Geometric Spectral Theory’, grant number 200021-19689.

\section{Introduction}

In this paper, we give a survey of some recent developments concerning the Steklov eigenvalues and eigenfunctions of compact manifolds with boundary. The last few years have seen intense interest in these topics, with significant progress in many directions since the publication of the survey paper of Girouard and Polterovich~\cite{GiPo2017}. We focus on the time period since the earlier survey but try to include enough background to make the current survey somewhat self-contained. Even so, we are only able to touch on some of the topics, as a quick search on MathSciNet for titles containing \emph{Steklov eigenvalue} over this time period reveals over 100 papers. The selection of topics which are covered here is naturally influenced by the tastes and knowledge of the authors.

Let $\Omega$ be a smooth compact Riemannian manifold of dimension $d+1\geq 2$ with boundary $\bdry=\partial\Omega$. The Dirichlet-to-Neumann operator $\DtN:C^\infty(\bdry)\to C^\infty(\bdry)$ is defined by
$\DtN f=\partial_\nu{\hat{f}}$, where $\nu$ is the outward normal  along the boundary $\bdry$ and where the function $\hat{f}\in C^\infty(\Omega)$ is the unique harmonic extension of $f$ to the interior of $\Omega$. 
The eigenvalues of $\DtN$ are known as \emph{Steklov eigenvalues} of $\Omega$. They form an unbounded sequence
$0=\sigma_0\leq\sigma_1\leq\sigma_2\leq\cdots\to\infty$,
where as usual each eigenvalue is repeated according to its multiplicity. There exists a corresponding sequence of eigenfunctions $f_k\in C^\infty(\bdry)$ that forms an orthonormal basis of $L^2(\bdry)$. Their harmonic extensions $u_k=\hat{f_k}$ are solutions of the Steklov spectral problem given by
\[\begin{cases}
  \Delta u_k=0&\text{in }\Omega,\\
  \partial_{\nu}u_k=\sigma_k u_k&\text{on }\bdry.
\end{cases}\]
The functions $u_k$ are referred to as the \emph{Steklov eigenfunctions} of $\Omega$, while the sequence of eigenvalues $(\sigma_k)_{k\geq 0}$ is the \emph{Steklov spectrum} of the manifold $\Omega$.
We are interested in the rich interplay between the Steklov spectral data and various geometric features of the manifold $\Omega$. 
For basic motivation we refer to the introduction of the aforementioned paper~\cite{GiPo2017}, while for historical background and physical motivations readers are invited to look at the paper~\cite{KuKu2014} by Kuznetsov, Kulczycki, Kwaśnicki, Nazarov, Poborchi, Polterovich and Siudeja.

There are only a few manifolds for which the Steklov eigenvalues can be computed explicitly. 
In general, one must instead use methods other than direct computation to study Steklov eigenvalues. One critical tool is variational characterisations. The \emph{Rayleigh--Steklov quotient} of a function $u$ in the Sobolev space $H^1(\Omega)$ is given by 
\begin{equation*}
R(u)= \frac{\int_{\Omega}\vert \nabla u\vert^2 dV_{\Omega}}{\int_{\Sigma}u^2dV_{\Sigma}}.
\end{equation*}
Denote by $\mathcal E(k)$ the set of all $k$-dimensional subspaces of $H^1(\Omega)$, and let $\mathcal H(k)\subset \mathcal E(k)$ consist of those $k$-dimensional subspaces of $H^1(\Omega)$ that are orthogonal to constant functions on $\Sigma$.     The following equation gives two convenient formulations of the variational characterisation of the Steklov eigenvalues $\sigma_k(\om)$ for all $k\in \Z^+$:
\begin{equation*}
\sigma_k(\Omega)=  \min_{E\in \mathcal E(k+1)}\max_{0\not=u \in E}R(u)=\min_{V\in \mathcal H(k)}\max_{0\not=u \in V}R(u).
\end{equation*}
In particular,
\[\sigma_1(\Omega)=\min\{R(u)\,:\,u\in H^1(\Omega)\quad \mbox{and}\quad \int_{\Sigma}u\,dV_\Sigma=0\}.\]
Similar  variational characterisations for mixed Steklov--Neumann and Steklov--Dirichlet eigenvalues are also presented in Section~\ref{section:Prelim}. 

Observe that the numerator of the Rayleigh--Steklov quotient is the Dirichlet energy, which is a conformal invariant in dimension two, while the denominator depends only on the metric on the boundary $\bdry$.  Thus, in the case of surfaces, the Steklov spectrum is invariant under conformal changes in the metric away from the boundary.

\subsection{Overview of the survey}

In Section~\ref{section:Prelim}, after first introducing notational conventions that will be used throughout the paper, we provide basic background on the Steklov eigenvalue problem and examples.  Two examples presented in detail are metric balls in Euclidean space and the particularly useful example of cylinders $\Omega=[0,L]\times M$  over a closed Riemannian manifold $M$. 
Even these simple examples are sufficient to motivate many of the questions that will be studied in this survey. In particular, both examples illustrate the interplay between the Steklov eigenvalues of $\Omega$ and the Laplace eigenvalues of the tangential Laplacian on $\partial\Omega$. The behaviour of the Steklov spectrum of cylinders as $L$ tends to 0 or to $+\infty$ is also a meaningful preview of results to come. 

Each of the remaining sections is devoted to a particular topic or set of topics. We discuss each briefly here.

One of the oldest and most active lines of investigation regarding Steklov eigenvalues is the search for isoperimetric-type geometric inequalities. For simply-connected planar domains $\Omega\subset\R^2$ this goes back to Weinstock~\cite{We1954}. 

In the last few years the variational characterisations of Steklov eigenvalues were used together with a combination of tools, in particular from complex analysis, to obtain upper bounds for the perimeter-normalised Steklov eigenvalues $\sigma_k(\om)L(\bdry)$ of compact Riemannian surfaces $\om$ with boundary in terms of the genus and number of boundary components of $\om$. This is discussed in Section~\ref{section:surfaces}. One of the major developments that took place in the last few years is the full solution of the isoperimetric problem for Steklov eigenvalues of planar domains, without constraint on the number of boundary components. This was obtained by Girouard, Karpukhin and Lagacé in~\cite{GiKaLa2021} through the use of homogenisation theory by perforation, which provides examples saturating previous bounds by Kokarev~\cite{Ko2014}. This technique also reveals an interesting connection between area-normalised eigenvalues of the Laplace operator on a closed surface and perimeter-normalised Steklov eigenvalues for domains in that surface.

For manifolds $\Omega$ of arbitrary dimension $d+1$, sharp upper bounds and existence results are out of reach at the moment. One of the first difficulties that is encountered is that the Dirichlet energy is no longer conformally invariant in dimension larger than two. This is often circumvented by using H\"older's inequality to replace the Dirichlet energy  $\int_{\Omega}|\nabla u|^2\,dV$ by an expression containing $\int_{\Omega}|\nabla u|^{d+1}\,dV$, which is conformally invariant. However by doing so one loses precision. Section \ref{section:boundshigher} presents many upper bounds for the eigenvalues $\sigma_k$ in terms of various geometric features of the manifold.

An important recent innovation was the realisation by Karpukhin and M\'etras that apart from the perimeter and volume normalisation, there is a another normalisation that appears to be particularly well suited to the study of upper bounds for Steklov eigenvalues.  This normalisation will be introduced in Section~\ref {section:boundshigher} and will play an important role both there and later in the survey.

Thus far we have discussed bounds for the normalised Steklov eigenvalues.   Do there exist Riemannian metrics on a given underlying surface that realise these bounds?  This is the subject of Section~\ref{sec:existence}.  Fraser and Schoen~\cite{FrSc2016} first discovered and studied a deep connection between such extremal metrics on surfaces and free boundary minimal surfaces in Euclidean balls. \textbf{}This connection has not only influenced advances on the Steklov eigenvalue
problem, including many of the results in Section \ref{section:surfaces}, but also yields important applications to
the study of minimal surfaces. Fraser and Schoen went on to introduce innovative techniques to address the existence of Riemannian metrics maximising the first non-zero normalised Steklov eigenvalue.   Their ideas have led to striking advances in addressing existence of extremal metrics for normalised eigenvalues both for the Steklov problem and for the Laplace eigenvalue problem on closed Riemannian manifolds.

An interesting way to get bounds on Steklov eigenvalues is to discretise the manifold and compare its spectrum with the spectrum of a Steklov problem on a graph.  While obtaining a meaningful discretisation requires significant geometric constraints on the manifold, a number of interesting results have been obtained.  
Discretisation motivates the study of the spectrum of the Steklov problem on a finite graph in and of itself (see Section \ref{Section: discretisation}).

For $\om$ a compact $(d+1)$-dimensional Riemannian manifold with boundary $\bdry$, Section~\ref{stek.forms} addresses two ways of defining an analogue of the Dirichlet-to-Neumann operator on the space $\Acal^p(\bdry)$ of smooth $p$-forms on $\bdry$ for each $p=1,\dots, d$.  Both operators have discrete spectra, allowing one to define notions of the Steklov spectrum for $p$-forms and to study their properties.  

In Section~\ref{sec:inv probs pos} we consider the inverse spectral problem for the Steklov problem, with an emphasis on ``positive" results, that is, finding geometric information that can be recovered from the Steklov spectrum. Along the way, we explain recent developments in the theory of Steklov spectral asymptotics.

Next in Section~\ref{sec: isospectrality}, we address "negative" inverse spectral results. We discuss general techniques for constructing pairs or continuous families of compact Riemannian manifolds with boundary that have the same Steklov spectrum.  By comparing their geometry, we identify geometric invariants that are \emph{not} spectrally determined.   

Section~\ref{sec:eigenfunctions} is devoted to the geometry of Steklov eigenfunctions. We discuss the interior decay of Steklov eigenfunctions. We also present the best known bounds on the volumes of the nodal sets of both Steklov eigenfunctions and their restrictions to the boundary (the Dirichlet-to-Neumann eigenfunctions). Finally, we explain some recent results on nodal counts and the density of nodal sets.

In Appendix~\ref{Section:Radon}, we present some material on variational eigenvalues of Radon measures. One could start with this section, or simply refer to it when needed. Indeed, the setting presented here allows the unification of many well known eigenvalue problems. In particular, homogenisation procedures that relate isoperimetric problems for Steklov, Laplace and various other eigenvalue problems are natural in this setting.

There are numerous open problems scattered throughout the survey.  For the convenience of the reader, all of these problems are gathered together in Appendix~\ref{app:open ques} along with references to their locations in the survey.  As we hope this survey conveys, new techniques and results are being introduced into the study of the Steklov spectrum at a very rapid pace.  In particular, it is of course possible that some of the open problems may be resolved quickly.

Among the topics not covered in this survey is the significant progress in the development of numerical methods for computing Steklov eigenvalues and eigenfunctions. Interested readers are invited to begin with the paper~\cite{BrGa2020} by Bruno and Galkowski where these computations are performed with a view towards nodal geometry. This is based on methods that are developed by Akhmetgaliyev, Kao, and Osting in \cite{AkElOs2017}. Another relevant paper is~\cite{OuKaOs2021} by Oudet, Kao, and Osting, where numerical isoperimetric shape optimisation is used to compute free boundary minimal surfaces. For a survey of recent results and interesting discussions we also recommend \cite{MoZh2022} by Monk and Zhang as well as \cite{YoXiLi2019} by Liu, Xie, and Liu.

\section{Motivating examples and preliminary material}\label{section:Prelim}

The background and basic examples in section will serve as motivation for many of the questions that we will study. Readers who are well acquainted with the Steklov eigenvalue problem may want to skip ahead to Section~\ref{section:surfaces} after going over our notational conventions below and then come back to this section when needed.

\subsection*{Notational conventions}\label{nota.convention}

\begin{enumerate}
\item $(M,g)$ and $(N,g)$ will usually denote complete Riemannian manifolds (whether or not compact).  We will often suppress the name of the metric $g$.  
 
\item Throughout the paper, compact manifolds with boundary will systematically be denoted by $\Omega$ and the boundary of $\om$ will be denoted by $\Sigma$.   We also use $\Omega$ for bounded domains $\Omega\subset M$ with nonempty boundary.

\item The dimension of $\om$ will usually be denoted by $d+1$, so that $d$ is the dimension of its boundary $\Sigma$.

\item The volume of $(\om,g)$  will usually be written $|\om|_g$ or simply $|\om|$ if the metric $g$ is understood. We will continue to denote by $g$ the Riemannian metric induced by $g$ on $\Sigma=\pam$.  In particular, $|\Sigma|_g$ will denote the $d$-dimensional volume of $\Sigma$.  Sometimes we will also use notation such as $A$ and $L$ for the area of a surface and the length of a curve. 
\item We use the positive definite Laplacian defined by $\Delta_g=-\text{div}_g\circ\nabla_gf$.
\item The Riemannian volume form on $\om$ will usually be denoted $dV_{(\om,g)}$ (with the subscripts $\om$ and/or $g$ suppressed if they are clear from the context) but we will also sometimes replace $dV$ with $dA$ in the case of surfaces and sometimes use $ds$ for arclength measure. 
\item The standard Sobolev space consisting of functions in $L^2(\om)$ with weak gradient also in $L^2(\om)$ will be denoted $H^1(\Omega)$.  We use $H^1_0(\om)$ to denote the closure in $H^1(\om)$ of $C_0^\infty(\om)$, where $C_0^\infty(\om)$ denotes the space of smooth functions with compact support in the interior of $\om$.
\item Eigenvalues will be indexed starting with the index zero.    The Steklov spectrum of $(\om,g)$ will thus be written as:
\[\stek(\om,g): 0=\sigma_0(\om,g)\leq \sigma_1(\om,g)\leq \sigma_2(\om,g)\leq \dots\]
and the Laplace spectrum of a closed Riemannian manifold $(M,g)$ will be expressed as 
\[0=\lambda_0(M,g)\leq \lambda_1(M,g)\leq \lambda_2(M,g)\leq \dots\]
\end{enumerate}
We will write simply $\sigma_k$ rather than $\sigma_k(\om,g)$ if $(\om,g)$ is fixed.

\noindent{\bf CAUTION!} We caution that there are two frequently used indexing conventions in the literature:  indexing the lowest eigenvalue by zero as we are doing here or by one.   The convention we are using is convenient for the Steklov spectrum of connected manifolds.  In that case $\sigma_0=0$ and $\sigma_1$ is the lowest non-zero eigenvalue.    For consistency, we are using this indexing convention even, for example, in the setting of the Dirichlet spectrum of a compact manifold with boundary, in which case the lowest eigenvalue $\lambda_0^D$ is non-zero.

While the Steklov spectrum $\stek(\om,g)$ is the spectrum of the Dirichlet-to-Neumann operator $\DtN$ on $C^\infty(\Sigma)$, it is common to refer to the harmonic extensions to $\om$ of the eigenfunctions of $\DtN$ as \emph{Steklov eigenfunctions}.  We will follow that practice here.

\subsection{Examples}\label{subsec: examples}

There are very few Riemannian manifolds for which the Steklov eigenvalues can be computed explicitly.  Here we give two simple but illustrative examples.

\begin{ex}\label{ex: ball}
Consider the unit ball $\B^{d+1}$ in $\R^{d+1}$.   Let $P_k$ denote the space of homogeneous harmonic polynomials of degree $k$ on $\R^{d+1}$.  For $p\in P_k$ expressed in spherical coordinates as $p(r,\theta)=r^kh(\theta)$, observe that $\partial_\nu p =\frac{\partial p}{\partial r} =kp$ on the boundary sphere $\sph^d$.  Thus $P_k$ consists of Steklov eigenfunctions with eigenvalue $k$.    Since the spherical harmonics -- i.e., the restrictions to $\sph^d$ of all the homogeneous harmonic polynomials -- span $L^2(\sph^d)$, we conclude that the Steklov spectrum of $\B^{d+1}$ consists precisely of the non-negative integers, and the Steklov eigenspace associated with $k$ is given by $P_k$ (more precisely, by the restrictions to $\B^{d+1}$ of the elements of $P_k$).   

Compare this with the spectrum of the Laplacian $\Delta$ on $\sph^d$.   The latter consists of all $k(k+d-1)$ for $k=0,1,2,\dots$ with corresponding eigenspaces the $k$th degree spherical harmonics.    In particular, the Dirichlet-to-Neumann operator and the Lapacian have the same eigenspaces.   
In fact, the relationship between the Laplacian and the Dirichlet-to-Neumann operator is completely explicit in this case:
\[\Delta_{\sph^d}=\DtN_{\B^{d+1}}^2+(d-1)\DtN_{\B^{d+1}}.\]
 In particular, in the case of the disk $\B^2$, we have $\DtN_{\B^{2}}=\sqrt{\Delta_{\sph^1}}$.    
\end{ex}

We identify a few features of the example above, some unique to balls, others completely general.

\begin{itemize}
\item For $k$ large, the $\sigma_k$-eigenfunctions decay towards zero rapidly on compact subsets of the interior.   This property is completely general and will be discussed in Section~\ref{sec:eigenfunctions}.
\item Comparison with the Laplace eigenvalues of the sphere $\sph^d$ yields \[\sigma_k(\B^{d+1})=\sqrt{\lambda_k(\sph^d)}+ O(1).\]  As we will discuss in the next subsection, this relationship between the Steklov eigenvalues of a manifold and the Laplace eigenvalues of its boundary also holds more generally.
\item Special to this example:  Girouard, Karpukhin, Levitin and Polterovich \cite{GKLP2021} showed that 
Euclidean balls are the \emph{only} compact Riemannian manifolds for which the Dirichlet-to-Neumann operator commutes with the boundary Laplacian, and disks are the only surfaces for which $\DtN=\sqrt{\Delta_\Sigma}$.    \	 
\end{itemize}

\begin{ex}\label{example: cylinder}
Cylinders over compact manifolds are among the simplest and at the same time most useful examples.  Let $\om$ be a cylinder $\om=C_L= [0,L]\times M$ where $M$ is a connected, closed $d$-dimensional Riemannian manifold and $L\in\R^+$. Denote the eigenvalues of the Laplace--Beltrami operator $\Delta_M$ by
$0=\lambda_0(M)<\lambda_1(M)\le \lambda_2\le\cdots\nearrow\infty$, and let 
$(\varphi_k)_{k=0}^{\infty}\subset C^\infty(M)$ be a corresponding orthonormal basis of eigenfunctions.  Since $\Sigma:=\pam$ consists of two copies of $M$, we have 
\begin{equation}\label{eq: sig vs N}\lambda_{2k}(\Sigma)=\lambda_{2k+1}(\Sigma)=\lambda_k(M)\end{equation}
for all $k=0,1,2,\dots$.   
The Steklov eigenvalues of the cylinder are given by $0$, $\frac{2}{L}$ and for each $k\geq 1$,
\[
\sqrt{\lambda_k(M)}\tanh(\sqrt{\lambda_k(M)} \frac{L}{2})\qquad\text{and}\qquad \sqrt{\lambda_k(M)}\coth (\sqrt{\lambda_k(M)} \frac{L}{2}).
\]
Using the variables $t\in [0,L]$ and $x\in M$, the corresponding eigenfunctions are
\[
1;\ t;\ \cosh\left(\sqrt{\lambda_k(M)}(t-L/2)\right)\varphi_k(x);\ \sinh\left(\sqrt{\lambda_k(M)}(t-L/2)\right)\varphi_k(x).
\]

(i) We first consider the asymptotics of the Steklov eigenvalues as $k\to\infty$. Because $\tanh(t)=1+O(t^{-\infty})$ and $\coth(t)=1+O(t^{-\infty})$ it follows from the asymptotic growth rate of $\lambda_k$ given by the Weyl law that
\[\sigma_k(C_L)=\sqrt{\lambda_k(\Sigma)}+O(k^{-\infty}).\]

    (ii)  Next we let the length $L$ of the cylinder vary and consider the limiting behaviour as $L\to 0$:   The eigenvalue $2/L\to+\infty$, while
\[\sqrt{\lambda_k(M)}\tanh(\sqrt{\lambda_k(M)} \frac{L}{2})\to 0\qquad\text{and}\qquad\sqrt{\lambda_k(M)}\coth (\sqrt{\lambda_k(M)} \frac{L}{2})\to+\infty.\] 

In particular, the number of very small eigenvalues increases and for each index $k\in\N$,
\[\sigma_k(C_L)\xrightarrow{\,L\to 0\,}0.\]

(iii) Finally let $L\to\infty$.   Then $2/L\to 0$ and
\[
\sqrt{\lambda_k(M)}\tanh(\sqrt{\lambda_k(M)} \frac{L}{2})\to\sqrt{\lambda_k(M)}\qquad\text{and}\qquad \sqrt{\lambda_k(M)}\coth (\sqrt{\lambda_k(M)} \frac{L}{2})\to\sqrt{\lambda_k(M)}.\]
Taking into account Equation~\eqref{eq: sig vs N}, it follows that 
\[\sigma_k(C_L)\xrightarrow{\,L\to \infty\,}\sqrt{\lambda_k(\Sigma)}.\]
\end{ex}

\begin{itemize}
\item This example illustrates that we can find metrics on the underlying manifold for which the $k$th Steklov eigenvalue is arbitrarily small while keeping the volume of the boundary fixed.   As we will discuss in Subsection~\ref{subsec: prelim bounds}, one can construct metrics with similar behaviour on every compact manifold.

\item The fact that $\sigma_k(C_L)=\sqrt{\lambda_k(\Sigma)}+O(k^{-\infty})$ is a feature of this example that is \emph{not} common to all Riemannian manifolds.  In fact, it is enough to look at the ball $\Omega=B(0,1)\subset\R^n$, with $n\geq 3$ to see that this is not true.  However, we will see in the next subsection that a weaker relationship does hold in general.  
\end{itemize}

This simple example has many applications.   A sampling:

\begin{itemize}
\item It will be used in Section~\ref{section:boundshigher} to show that the exponent on $k$ in some bounds for $\sigma_k$ is optimal (see Theorem~\ref{globalestimate}). 
\item It can be used to deduce results on $\lambda_k$ from results on $\sigma_k$ and vica versa. See  Examples 5.1 and 5.2 in~\cite{CoGi2021}.   
\item The earliest known examples of non-isometric Steklov isospectral manifolds arose from the observation (see the earlier survey \cite{GiPo2017}) that the Steklov spectrum of the cylinder $[0,L]\times M$ depends only on $L$ and the Laplace spectrum of $M$; thus any pair of Laplace isospectral closed Riemannian manifolds yields a pair of cylindrical Steklov isospectral manifolds.

\end{itemize}

\begin{remark} The example of a dumbbell is classical for the Laplacian.  The Steklov spectrum of a dumbbell is addressed by Bucur, Henrot and Michetti \cite{BuHeMi2021} in all dimensions. See also the Ph.D. thesis of Michetti \cite{Mi20222}. 

\end{remark}

\subsection{Asymptotic behaviour of eigenvalues}\label{subsec: prelim asympt}

For compact Riemannian manifolds $(\om,g)$ with smooth boundary, the Dirichlet-to-Neumann operator $\DtN=\DtN_{(\om,g)}:C^\infty(\pam)\to C^\infty(\pam)$ is an elliptic pseudodifferential operator of order one.   As shown in \cite{LeUh1989},
the symbol of $\DtN$ is completely determined by the Riemannian metric in an arbitrarily small neighborhood of the boundary. Since the asymptotics of the spectrum of a pseudodifferential operator depend only on the symbol, this yields the following theorem:

\begin{thm}\cite{HiLu2001}, \cite[Theorem 2.5]{GPPS2014}\label{thm: hislop lutzer} Suppose $(\om,g)$ and $(\om',g')$ are compact Riemannian manifolds with boundary.   If some neighborhood of $\pam$ is isometric to a neighborhood of $\pam'$, then 
\[\sigma_k(\om,g)-\sigma_k(\om',g') = O(k^{-\infty}).\]
\end{thm}
(As discussed in Section~\ref{sec:inv probs pos}, a much stronger statement holds in the case of surfaces.)

The principal symbol of $\DtN_{(\om,g)}$ depends only on the boundary $(\Sigma, g)$ and in fact coincides with the principal symbol of $\sqrt{\Delta_\Sigma}$, where $\Delta_\Sigma$ is the Laplace-Beltrami operator of the boundary $\Sigma$ of $\om$ with the induced Riemannian metric. However, the subprincipal symbols of these two operators are different in general.  

The Weyl law for the Steklov eigenvalues (see Section~\ref{sec:inv probs pos}) yields 
\begin{equation}\label{eq: prelim Weyl}\sigma_k(\om,g)=\frac{2\pi}{|\B^d|}\left(\frac{k}{|\Sigma|_g}\right)^{1/d} \,+\,O(1). \end{equation}
As noted in \cite{GKLP2021}, a comparison with the Weyl law for the Laplacian yields
\begin{equation}\label{eq: sigma = lambda +O(1)}\sigma_k(\om,g)=\sqrt{\lambda_k(\Sigma,g)} + O(1)\end{equation}
where the $\lambda_j$'s are the eigenvalues of $\Delta_\Sigma$, the Laplacian on $\Sigma=\pam$.    

Further relationships between Steklov eigenvalues and Laplace eigenvalues of the boundary will be manifest in various parts of this paper (see, for example,Theorem~\ref{comparison}).

\subsection{Variational characterisation of eigenvalues}
In most cases, it is impossible to compute the Steklov spectrum of a manifold explicitly. Instead, one resorts to variational characterisations of eigenvalues in order to obtain lower and upper bounds. 

Let $(\Omega,g)$ be a compact Riemannian manifold with boundary $\Sigma=\partial \Omega$.  In contrast to the previous subsection, we do not require the boundary to be smooth.   For example, we allow Lipschitz boundary.   The Rayleigh--Steklov quotient of a function $u \in H^1(\Omega)$ is given by 
\begin{equation}\label{eq:Rayleigh S}
R(u)= \frac{\int_{\Omega}\vert \nabla u\vert^2 dV_{(\Omega,g)}}{\int_{\Sigma}u^2dV_{(\Sigma,g))}}.
\end{equation}

Denote by $\mathcal E(k)$ the set of all $k$-dimensional subspaces of $H^1(\Omega)$.  Let $\mathcal H(k)\subset \mathcal E(k)$ consist of those $k$-dimensional subspaces of $H^1$ that are orthogonal to the constant functions on $\Sigma$.     The following equation gives two convenient formulations of the variational characterisation
of the Steklov eigenvalues $\sigma_k(\om)$ for all $k\in\N$:
\begin{equation}\label{eq:Stek Rayleigh min max}
\sigma_k(\Omega)=  \min_{E\in \mathcal E(k+1)}\max_{0\not=u \in E}R(u)=\min_{V\in \mathcal H(k)}\max_{0\not=u \in V}R(u).
\end{equation}
Letting $u_0, u_1, u_2,\dots$ be Steklov eigenfunctions for $\sigma_0,\sigma_1, \sigma_2,\dots$, then the minimum in the first formulation is obtained by $E$=span$(u_0,\dots, u_k)$ and in the second formulation by $V$=span$(u_1,\dots, u_k)$.

\subsection{Conformal invariance in dimension two}\label{subsec: conf invar}

The study of the Steklov spectrum of surfaces often employs different techniques than in higher dimensions due to the following conformal invariance property:

\begin{prop}\label{prop: conf invar}
Let $\om$ be a compact surface with boundary $\Sigma$.   Suppose $g$ and $g'$ are Riemannian metrics on $\om$ satisfying both of the following conditions:
\begin{enumerate}
\item $g'$ is conformally equivalent to $g$, i.e., $g'=\tau g$ for some positive function $\tau\in C^\infty(\om)$;
\item $\tau\equiv 1$ on $\Sigma$.
\end{enumerate}
Then the Dirichlet-to-Neumann operators of $(\om,g)$ and $(\om,g')$ coincide, and 
\[\stek(\om,g)=\stek(\om,g').\]
\end{prop}

\begin{proof}
Since we are in dimension two, the Laplace-Beltrami operators associated with $g$ and $g'$ satisfy $\Delta_{g'}=\frac{1}{\tau}\Delta_g$; thus the condition that a function be harmonic depends only on the conformal class of the metric.  Moreover, since the metrics agree on $\Sigma$, the unit normals to the boundary agree and thus the Dirichlet-to-Neumann operators are identical.
\end{proof}

We will use the following convenient language introduced by Fraser and Schoen:
 
 \begin{defn}\label{sigmaisom} We say two compact Riemannian surfaces $(\om_1,g_1)$ and $(\om_2,g_2)$ are $\sigma$-\emph{isometric} if there exists a diffeomorphism $\Phi:\om_1\to \om_2$ such that $\Phi^*g_2=\tau g_1$ where $\tau\in C^\infty(\om_1)$ satisfies $\tau_{|\partial\om_1}\equiv 1$.   
 
 \end{defn}
 
 \begin{cor}\label{cor:sigmaisom} Suppose $(\om_1,g_1)$ and $(\om_2,g_2)$ are $\sigma$-isometric Riemannian surfaces.    Then they have the same Steklov spectrum.   Moreover, for $\Phi$ as in Definition~\ref{sigmaisom}, the restriction $\Phi |_{\pam_1}:\pam_1\to \pam_2$ intertwines the Dirichlet-to-Neumann maps of $(\om_1,g_1)$ and $(\om_2,g_2)$.
 
 \end{cor}

One can give another proof that $\stek(\om_1,g_1)=\stek(\om_2,g_2)$ when $(\om_1,g_1)$ and $(\om_2,g_2)$ are $\sigma$-isometric by appealing to the variational characterisation~\eqref{eq:Stek Rayleigh min max} of the eigenvalues.  The numerator of the Rayleigh quotient $R(u)$ in Equation~\eqref{eq:Rayleigh S} is the Dirichlet energy.
In the case of surfaces, the Dirichlet energy is independent of the conformal class.    Since the denominator of the Rayleigh quotient depends only on the metric restricted to the boundary, the Steklov spectra agree.

A particularly simple and interesting class of $\sigma$-isometric metrics is obtained for surfaces of revolution.
\begin{prop}\cite{Br2019}
Let $\Omega\subset\R^3$ be a surface of revolution with connected boundary $\partial\Omega=\partial\D\times\{0\}$. Then
$\sigma_k(\Omega)=\sigma_k(\D)$ where $\D$ is a disk of the same boundary length.  Moreover, the Dirichlet-to-Neumann maps of $\om$ and $\D$ coincide (when one identifies the boundary circles of $\om$ and $\D$).
\end{prop}
The proof of this proposition is surprisingly direct, since conformal parametrisations for surfaces of revolutions are known since Liouville and it follows that any surface of revolution is $\sigma$-isometric to the unit disk.
This leads to a gigantic family of surfaces in $\R^3$ that have exactly the same DtN map. 
\begin{ques}\label{ques:conformal}
Describe the class of all smooth compact surfaces $\Omega\subset\R^3$ with boundary $\partial\Omega=\partial\D$ that admit a conformal parametrisation $\Phi:\D\to\Omega$ such that $|\Phi'|\equiv 1$ on $\partial\D$.
\end{ques}

\begin{remark}\label{rem:con sing}
A Riemannian metric $g$ on a surface $\om$ is said to have an isolated conical singularity at an interior point $p\in\om$ if in a sufficiently small geodesic disk centered at $p$ with complex coordinate $z$, the metric is expressed in the form
\[g=|z|^{2(\alpha-1)}\varphi(z) |dz|^2\]
for some real number $\alpha>0$ (with $\alpha\neq 1$) and some positive smooth function $\varphi$.  The cone angle at $p$ is $2\pi\alpha$.  In case $\varphi\equiv 1$, then the metric in this geodesic disk is isometric to the standard flat cone with cone angle $2\pi\alpha$.    
A Riemannian metric on $\om$ that is smooth except for isolated interior conical singularities is conformally equivalent to a smooth metric although the conformal factor has value $0$ or $\infty$ (depending on the cone angle) at the conical singularities. Moreover, under our assumption that all the singularities are at interior points, one can choose the conformal factor to be identically one on the boundary.  We extend the notion of  $\sigma$-isometry to allow conformal equivalences of this type; i.e., we allow the conformal factor $\tau$ in Definition~\ref{sigmaisom} to take on the values $0$ and $\infty$ at finitely many interior points.  

For compact Riemannian manifolds with isolated interior conical singularities, the Steklov spectrum is well-defined via the variational characterisation of eigenvalues~\ref{eq:Stek Rayleigh min max}.  Moreover, Corollary~\ref{cor:sigmaisom} continues to hold with the extended definition of $\sigma$-isometry.

The ability to conformally remove conical singularities without affecting the Steklov spectrum plays a role in the construction of smooth metrics that maximise normalised Steklov eigenvalues on surfaces; see Section~\ref{sec:existence}.
\end{remark}

\subsection{Steklov eigenvalue bounds}\label{subsec: prelim bounds}

Since rescaling a Riemannian metric has the effect of rescaling all the Steklov eigenvalues, one must choose a scale-invariant normalisation in order to address eigenvalue bounds.   The most commonly used normalisation is via the boundary volume:  $\sigma_k(\om,g)|\Sigma|_g^{\frac{1}{d}}$. A second normalisation is via the volume of the manifold $\om$ itself:  $\sigma_k(\om,g)|\om|_g^{\frac{1}{d+1}}$. Motivated by the isoperimetric ratio, 
Karpukhin and M\'etras recently introduced a normalisation involving both the boundary volume and the volume of $\om$; see Section ~\ref{section:boundshigher}.

Note that in the case of surfaces $\om$, $\sigma$-isometries (see Definition~\ref{sigmaisom}) allow one to adjust the area arbitrarily without affecting the Steklov eigenvalues.  Thus the boundary length normalisation is the most natural one in this case, although the area normalisation has occasionally been used in dimension two when restricting, say, to the case of plane domains. Due to the conformal invariance of the Steklov spectrum in dimension two, the techniques used for addressing eigenvalue bounds for surfaces typically differ substantially from higher dimensions.

We first address here the non-existence of lower eigenvalue bounds, other than the trivial bound of zero, for manifolds of arbitrary dimension.  \begin{prop}\label{prop: no lower bound}  Let $\om$ be a compact $d+1$-dimensional manifold with boundary, and let $k\in \N$.  Then     
\begin{enumerate}
\item[1.] \cite[Subsection 2.2]{GiPo2010};  \cite[Section 4]{GiPo2017} For every $\epsilon >0$, there exists a Riemannian metric $g_\epsilon$ on $\om$ such that $\sigma_k(\om,g_\epsilon)<\epsilon$.   The metrics can be chosen so that $|\om|_{g_\epsilon}=|\om|$ and $|\pam|_{g_\epsilon}=|\pam|$, independently of $\eps$.
\item[2.] \cite[Proposition 2.1]{CoElGi2019}. Moreover, if  $d+1\geq 3$, then given any Riemannian metric $g$ on $\om$, the metrics $g_\epsilon$ as above can be chosen so that they are conformally equivalent to $g$ and coincide with $g$ on the boundary $\Sigma$.
\end{enumerate}
\end{prop}

\begin{proof}
(1)  Fix $L>0$ and let $\eta>0$.  Choose a metric $h_\eta$ such that $(\om, h)$ contains a cylinder $C_{L,\eta}$ isometric to $[0,L]\times \B^{d}_\eta$ with lateral boundary $[0,L]\times \sph^{d-1}_\eta$ contained in $\Sigma$. (Here $\B^{d}_\eta$ is a Euclidean ball of radius $\eta$.) See Figure~\ref{fig:cylinder}. 
\begin{figure}
  \centering
  \includegraphics[width=3cm]{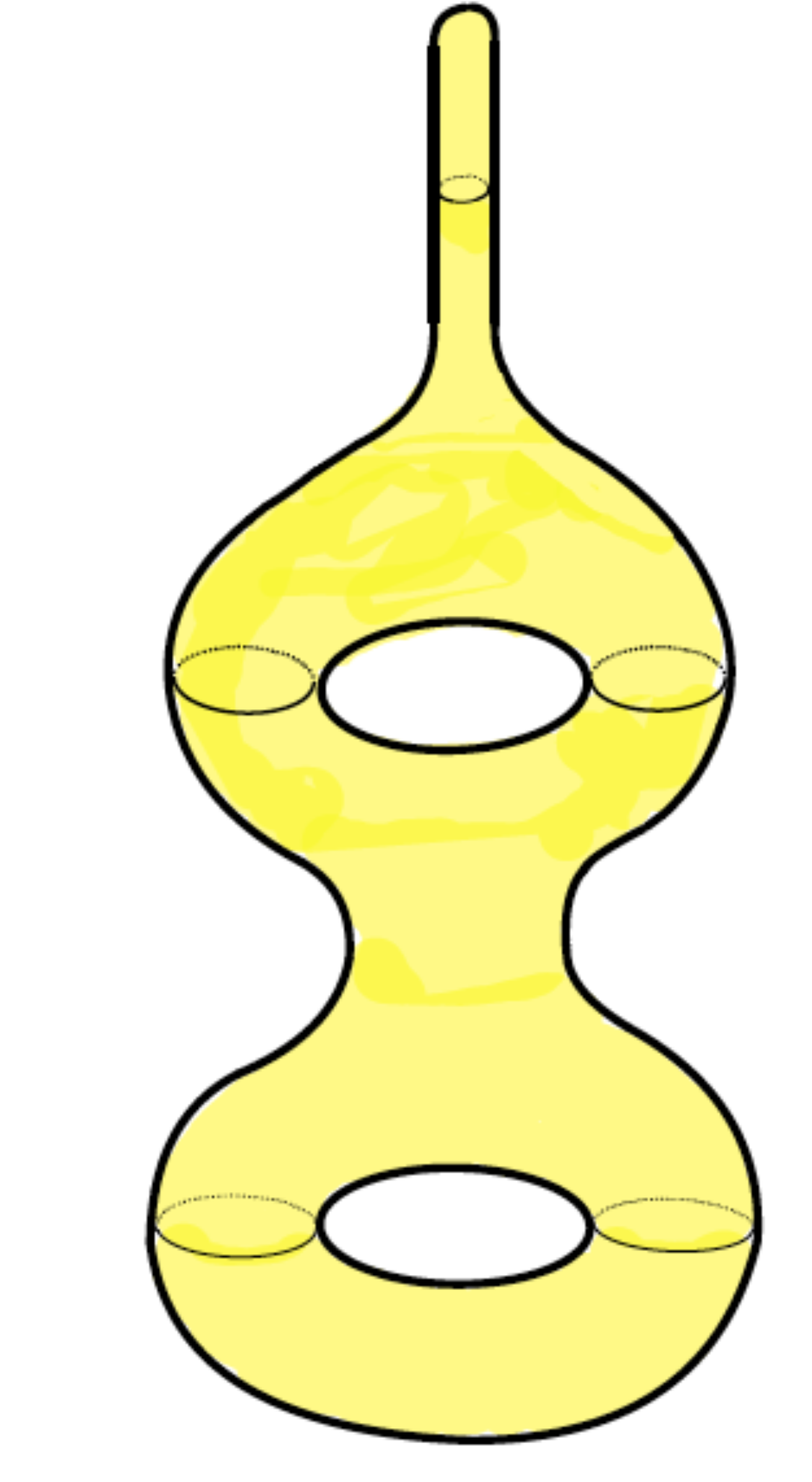}
  \caption{}
  \label{fig:cylinder}
  \end{figure}

Let $E\subset\mathcal{E}(k)$ be the subspace spanned by the constant function 1 together with $u_1, \dots, u_k$, where $u_j\equiv 0$ off $C_{L,\eta}$ and $u_j(t,x) = \sin(\frac{j\pi t}{L})$ for $(t,x)\in C_{L,\eta}$.   Then an easy computation shows that there exists a constant $A$ depending only on $L$, $k$ and $d$, not on $\eta$, such that the Rayleigh--Steklov quotient satisfies $R(u)\leq A\eta$ for all $u\in E$. Thus by Equation~\eqref{eq:Stek Rayleigh min max}, $\sigma_k(\om, h_\eta)\leq A\eta$, and the first statement in (1) follows by choosing $g_\epsilon=h_\eta$ with $\eta$ sufficiently small. To guarantee that the measure of $\Omega$ and of its boundary are independent of $\eps$, one then needs to perform some scaling and local perturbations.

(2) Fix $g$ and observe that if $\beta$ is a non-negative smooth function on $\om$ then the Riemannian metric $g'=\beta g$ has 
Dirichlet energy 
\[\int_{\Omega}\vert \nabla u\vert^2_g \,\beta^{d-1}\,dV_{(\Omega,g)}.\]
The desired result is achieved by fixing a point $p\in \Sigma$ and a small neighborhood $U$ of $p$ in $\om$ and choosing $\beta$ so that: (i) $\beta\equiv 1$ both on $\Sigma$ and on the complement of $U$ in $\om$; (ii) $0<\beta<1$ in  $\operatorname{int}(U)$; and (iii) $\beta$ is very close to zero in  $\operatorname{int}(U)$ away from a small neighborhood of $\partial{U}$.   One then applies Equation~\eqref{eq:Stek Rayleigh min max} choosing $E\subset\mathcal{E}(k)$ to consist of functions supported in $U$. See \cite{CoElGi2019} for details.
\end{proof}

As discussed at the beginning of this subsection, the commonly used eigenvalue normalisations in the literature involve only the volume and the volume of the boundary of the Riemannian manifold.   With respect to any such normalisation, Proposition~\ref{prop: no lower bound} implies the following:

\begin{cor}\label{cor: no lower bound}
Let $\om$ be a compact manifold with boundary.   Then for each $k\in \N$ there exist Riemannian metrics on $\om$ for which the $k$th normalised Steklov eigenvalue is arbitrarily small. 
\end{cor}

Section~\ref{section:surfaces} will focus on upper bounds for eigenvalues on surfaces, normalised by boundary length.    Such bounds always exist and depend only on the topology of the surface.   However, for manifolds $\om$ of higher dimension, Colbois, El Soufi and Girouard \cite{CoElGi2019} showed that  within every conformal class of metrics, the normalised eigenvalue $\sigma_1(\om,g)|\Sigma|_g$ can be made arbitrarily large.  Very interesting questions arise concerning eigenvalue bounds when one imposes geometric constraints in these higher dimensional settings, which is the subject of Section~\ref{section:boundshigher}.

\subsection{Mixed eigenvalue problems}\label{subsec: mixed} 

While our focus in this survey will be on the Steklov problem, it will sometimes be useful to consider mixed Steklov-Neumann or mixed Steklov-Dirichlet problems. In particular, as we will discuss in the next subsection, one can sometimes obtain bounds on Steklov eigenvalues of a Riemannian manifold $\om$ by comparing the Steklov eigenvalues of $\om$ with the eigenvalues
of mixed
problems on well-chosen domains in $\om$.

Given a decomposition $\pam=\partial_S(\om)\,\sqcup\,\partial_N(\om)$, one defines the mixed Steklov-Neumann problem
\begin{gather*}
  \Delta f=0\mbox{ in } {\om},\\
  \partial_\nu f=\sigma f\mbox{ on } \partial_S(\om),\qquad
  \partial_\nu f=0 \mbox{ on }\ \partial_N(\om).
\end{gather*} 
The eigenvalues of this mixed problem form a discrete sequence
\[0=\sigma_0^N( \om) \le \sigma_1^N( \om)\leq\sigma_2^N(\om)\leq\cdots\nearrow\infty,\]
and for each $k\geq 1$ the $k$-th eigenvalue is given by
\begin{equation}\label{eq:Rayleigh S-N}
 \sigma_k^{N}(\om)=\min_{E\in\mathcal {E}(k+1)}\max_{0\neq u\in E}\frac{\int_{\om}|\nabla u|^2\,dV_\om}{\int_{\Sigma}|u|^2\,dV_{\partial_S(\om)}}.\end{equation}

Similarly, the mixed Steklov-Dirichlet problem is given by
\begin{gather*}
  \Delta f=0\mbox{ in } {\om},\\
  \partial_\nu f=\sigma f\mbox{ on } \partial_S(\om),\qquad
   f=0 \mbox{ on }\ \partial_D(\om)
\end{gather*} 
relative to a decomposition $\pam=\partial_S(\om)\,\sqcup\,\partial_D(\om)$, and the eigenvalues form a discrete sequence
\[0<\sigma_0^D( \om) \le \sigma_1^D( \om)\leq\sigma_2^D(\om)\leq\cdots\nearrow\infty,\]
Their variational characterisation is given by 
\begin{equation}\label{eq:Rayleigh S-D}
 \sigma_k^{D}(\om)=\min_{E\in\mathcal {E}_0(k+1)}\max_{0\neq u\in E}\frac{\int_{\om}|\nabla u|^2\,dV_\om}{\int_{\Sigma}|u|^2\,dV_{\partial_S(\om)}},
\end{equation}
where $\mathcal {E}_0(k+1)$ consists of all $(k+1)$-dimensional subspaces of $\{f\in H^1(\om)\,:\,f=0\,\text{ on }\partial_D(\om)\}$.

\begin{ex}[Mixed Steklov--Neumann and Steklov--Dirichlet eigenvalues of cylinders]\label{example:mixedcylinders}~

(i) Let $C_L$ be the cylinder in Example~\ref{example: cylinder}.  Consider the mixed problem on $C_L$ with Steklov condition at $t=0$ and Neumann condition at $t=L$:
\begin{gather*}
  \Delta f=0\mbox{ in } {C_L},\\
  \partial_\nu f=\sigma f\mbox{ on } \{0\}\times M,\qquad
  \partial_\nu f=0 \mbox{ on } \{L\}\times M.
\end{gather*}
The eigenvalues are 
\[\sigma_k^N(C_L)=\sqrt{\lambda_k(M)}\tanh(\sqrt{\lambda_k(M)}L),\quad k\geq 0,\] 
with corresponding eigenfunctions
\[\varphi_k(x)\cosh(\sqrt{\lambda_k(M)}(L-t)).\]
In particular $\sigma_0^N(L)=0$, with constant eigenfunction.
Notice that for each index $k$,
\[\lim_{L\to 0}\sigma_k^N(C_L)=0\qquad\text{and}\qquad\lim_{L\to+\infty}\sigma_k^N(C_L)=\sqrt{\lambda_k(M)}.\]	

(ii) Next consider the Steklov--Dirichlet problem on $C_L$ with Steklov condition at $t=0$ and Dirichlet condition at $t=L$. The eigenvalues are
$\sigma_0^D(L)=1/L$ with corresponding eigenfunctions
  $1-t/L$ and for each $k\geq 1$,
  \[\sigma_k^D(L)=\sqrt{\lambda_k(M)}\coth(\sqrt{\lambda_k(M)}L),\]
  with corresponding eigenfunctions
  \[\varphi_k(x)\sinh(\sqrt{\lambda_k(M)}(L-t)).\]

For each index $k$, we have
\[\lim_{L\to 0}\sigma_k^D(C_L)=+\infty\qquad\text{and}\qquad\lim_{L\to+\infty}\sigma_k^D(C_L)=\sqrt{\lambda_k(M)}.\]
Observe that
\[\sigma_k^D=\sqrt{\lambda_k(M)}+O(k^{-\infty})\qquad\text{and}\qquad\sigma_k^N=\sqrt{\lambda_k(M)}+O(k^{-\infty}).\]
\end{ex}

\begin{ex} \label{annuli} The mixed Steklov-Dirichlet and Steklov-Neumann eigenvalues on annular domains.

In $\R^{d+1}$ ($d\ge 2$), let $B_1$ and $B_L$ be the balls centered at the origin of radius $1$ and $L$, respectively, with $L>1$. Consider the annulus $\om_L:=B_L \setminus B_1$ with Steklov condition on $\partial B_1$ and Dirichlet (resp. Neumann) condition on $\partial B_L$.
Because of the symmetries of the problem, the eigenvalues for both mixed problems have multiplicity.
We denote the distinct eigenvalues respectively by
\[
\sigma_{(0)}^D(\Omega_L)< \sigma_{(1)}^D(\Omega_L)<\sigma_{(2)}^D(\Omega_L)<...
\]
and
\[
\sigma_{(0)}^N(\Omega_L)< \sigma_{(1)}^N(\Omega_L)<\sigma_{(2)}^N(\Omega_L)<...
\]
where $\sigma_{(k)}^D(\Omega_L)$ and $\sigma_{(k)}^N(\Omega_L)$ have the multiplicity of the $k$th eigenvalue of the Laplacian on the sphere $\mathbb S^d$.
  
Then, we have (\cite[Proposition 4 and 5]{CoVe2021})
  \[
  \sigma_{(k)}^D(\Omega_L)= \frac{k}{L^{2k+d-1}-1} +\frac{(k+d-1)L^{2k+d-1}}{L^{2k+d-1}-1}
  \]
  and
  \[
  \sigma_{(k)}^N(\Omega_L)=k \frac{(d+k-1)(L^{2k+d-1}-1)}{kL^{2k+d-1}+(k+d-1)}
  \]
  
  In particular,
for $k\ge 0$,
\begin{equation} \label{Dirichlet}
	\lim_{L\to \infty} \sigma^D_{(k)}(\Omega_L)=k+d-1,
\end{equation}
and for $k>0$,
\begin{equation} \label{Neumann}
	\lim_{L\to \infty} \sigma^N_{(k)}(\Omega_L)=k+d-1.
\end{equation}

  It may appear surprising that for every $k>0$, \[\lim_{L\to\infty}\,\left(\sigma_{(k)}^N(\Omega_L)-\sigma_{(k)}^D(\Omega_L)\right)=0.\] Intuitively, the reason is that each Steklov-Neumann eigenfunction of $\Omega_L$ is obtained by separation of variables as a product of a radial function with the corresponding eigenfunction of the sphere.  Observe that the denominator in the Rayleigh quotient is an integral over the inner boundary sphere only.   When the annulus is very large,  the radial function must eventually decay towards zero as the radius grows.

This example is used in the proof of Theorem \ref{upperrev}.
 
\end{ex}

The calculation of the (non-mixed) Steklov spectrum of an annulus $\Omega_L$ is much more complicated.  The asymptotics are given in \cite[Proposition 3.1]{FrSc2019} by Fraser and Schoen, and $\sigma_1$ is given in \cite[Theorem 4.1]{Ft2022} by Ftouhi. In dimension 2, it was already done by Dittmar in \cite{Di2004} and presented in \cite{GiPo2017}.

\subsection{Dirichlet-Neumann bracketing}\label{subsec:Dir Neum bracketing} 

Let $\om$ be a compact Riemannian manifold, let $\Sigma=\pam$, and let $A\subset \Omega$ be an open neighborhood of $\Sigma$ in $\om$.    
We denote by $\partial_IA$ the intersection of the boundary of $A$ with
the interior of $\Omega$ and we suppose that $\partial_IA$ is smooth.  We have \[\partial A=\Sigma\sqcup \partial_IA.\]  Consider the mixed Steklov-Neumann and mixed Steklov-Dirichlet eigenvalue problems on $A$, where we impose the Steklov condition on $\Sigma$ and Neumann or Dirichlet conditions on $\partial_IA$.  Comparing the variational formulae~\eqref{eq:Rayleigh S-N} and \eqref{eq:Rayleigh S-D} (with $A$ playing the role of $\om$) with the variational formula~\eqref{eq:Rayleigh S}, we obtain the following bracketing for  
each $k$:

\begin{gather}\label{ineq:compmixed}
\sigma_k^N(A) \le \sigma_k(\Omega) \le \sigma_{k}^D(A).
\end{gather}

\begin{remark} Suppose that the boundary $\Sigma$ of $\om$ has $b$ connected components.   Any sufficiently small neighborhood of $\Sigma$ will then have $b$ connected components and thus will satisfy $\sigma_k^N(A)=0$ for $k\leq b-1$.   This suggests that the global geometry of $\Omega$ can have a greater impact on the first $b-1$ non-zero Steklov eigenvalues of $\om$ and leads to interesting questions (see open question \ref{question lower}).  Nonetheless, the following consequence of Inequality~\eqref{ineq:compmixed} illustrates again that the effect on the Steklov eigenvalues of the geometry far away from the boundary is limited.   (Compare with Theorem~\ref{thm: hislop lutzer}.) 
\end{remark}

\begin{prop}\label{prop: dir neum bracket}

Let $(\om,g)$ and $(\om',g')$ be compact Riemannian manifolds with boundary.  Suppose that some neighborhood of $\pam$ in $\om$ is isometric to a neighborhood of $\pam'$ in $\om'$.  Identify these neighborhoods and call them $A$.   Then with boundary conditions chosen as above, we have 
\[|\sigma_k(\om)-\sigma_k(\om')|\leq \sigma_k^D(A)-\sigma_k^N(A).\]
\end{prop}

\begin{ex}[Manifolds with cylindrical boundary neighborhood]\label{cylindricalend}
Given $L\geq 0$, let $\Omega_L$ be a compact manifold with connected boundary $\Sigma$ such that a neighborhood of the boundary is isometric to $\Sigma\times [0,L]$. It follows from Example~\ref{example:mixedcylinders} and the bracketing inequality~\eqref{ineq:compmixed} that
\begin{gather}\label{ineq:DNbracketCylinder}
    \sqrt{\lambda_k(\Sigma)}\tanh(\sqrt{\lambda_k(\Sigma)}L)\le \sigma_k(\Omega) \le\sqrt{\lambda_k(\Sigma)}\coth(\sqrt{\lambda_k(\Sigma)}L).
    \end{gather}

    This inequality is replete with interesting consequences that will lead to interesting questions.   
    For instance, the definition of $\tanh$ and $\coth$ gives $\tanh(x)<1<\coth(x)$ and one checks that as $x\to\infty$,
    \[\tanh(x),\coth(x)=1+O(x^{-\infty}).\]
    Thus for each $N\in\N$,
    $\lim_{x\to\infty}(1-\tanh(x))x^N=0= \lim_{x\to\infty}(1-\coth(x))x^N$.

    Now the Weyl Law for the eigenvalues of the Laplace operator implies that $\lambda_k\sim c(n)k^{2/n}$ as $k\to\infty$.  Since  Inequality~(\ref{ineq:DNbracketCylinder}) says
    \[\sqrt{\lambda_k(\Sigma)}\left(\tanh(\sqrt{\lambda_k(\Sigma)}L)-1\right)\le \sigma_k(\Omega)-\sqrt{\lambda_k(\Sigma)} \le\sqrt{\lambda_k(\Sigma)}\left(\coth(\sqrt{\lambda_k(\Sigma)}L)-1\right).\]
    it follows that
    $|\sigma_k-\sqrt{\lambda_k(\Sigma)}|=O(k^{-\infty})$.
    In other words, for manifolds $\Omega_L$ with cylindrical boundary components, the Steklov eigenvalues are intimately linked to the Laplace eigenvalue of the boundary:
    \[\sigma_k=\sqrt{\lambda_k(\Sigma)}+O(k^{-\infty})\]
    a much stronger link that in the general case of Equation~\eqref{eq: sigma = lambda +O(1)}.

    Another interesting consequence of inequality~\eqref{ineq:DNbracketCylinder} is that manifolds $\Omega_L$ containing a cylindrical neighborhood of the boundary of length $L$ have precisely controlled  Steklov eigenvalues. For each fixed $k$, as $L$ goes to infinity, we have
    \[\sigma_k=\sqrt{\lambda_k(\Sigma)}+O(L^{-\infty}).\] 
    In particular, when $L$ is very large, not only are the Steklov spectrum and the spectrum of $\sqrt{\Delta_{\Sigma}}$ asymptotically close but in fact \emph{all} their eigenvalues are very close.
            \end{ex}

  \subsection{Variational eigenvalues for Radon measures}
\label{subsection:introvareigenradon}

 The reader may notice the similarity between the Rayleigh--Steklov quotient $R(u)$ in Equation~\eqref{eq:Rayleigh S} and the standard Rayleigh quotient for the Neumann eigenvalues on $\om$, i.e., the eigenvalues of the Laplace--Beltrami operator with Neumann boundary conditions.   The latter is given by 
 \[\frac{\int_{\Omega}\vert \nabla u\vert^2 dV_{(\Omega,g)}}{\int_{\om}u^2\,dV_{(\Omega,g)}}.\]
 Comparing the two Rayleigh quotients, the only difference is in the denominators; in one case the integral is with respect to the volume measure on the boundary $\Sigma$, in the other it is with respect to the volume measure on $\om$.
 
This observation led Kokarev~\cite{Ko2014} to introduce variational eigenvalues associated to any nonzero Radon measure $\mu$ on $\Omega$.
Let $u\in C^\infty(\Omega)$ with $\int_{\Omega}u^2\,d\mu\neq 0$. The \emph{Rayleigh--Radon quotient} of $u$ is defined to be
\[R_{\mu}(u):=\frac{\int_\Omega|\nabla u|^2\,dV_{(\Omega,g)}}{\int_{\Omega}u^2\,d\mu}.\]
It is then natural to define the \emph{variational eigenvalues} by
\begin{equation}\label{def:VarEigenRadon}
\lambda_k(\Omega,g,\mu) := 
\inf_{F_{k+1}}\sup_{f\in
F_{k+1}\setminus\{0\}} R_\mu(f),
\end{equation}
where the infimum is taken over all $(k+1)$-dimensional subspaces
$F_{k+1}\subset C^\infty(\Omega)$ such that the image of $F_{k+1}$ in $L^2(\Omega,\mu)$ is also $(k+1)$-dimensional.
We will sometimes write $\lambda_k(\mu)$ for $\lambda_k(\Omega,g,\mu)$ when no confusion is created.

This setting encompasses many well-known eigenvalue problems.

\begin{ex}\label{ex: vareigen weighted lap}
Let $\beta:\Omega\longrightarrow\R$ be a continuous positive function. For $\mu=\beta\,dV_\om$ the variational eigenvalues
$\lambda_k(\om,g,\mu)$ are the eigenvalues of the following non-homogeneous weighted Laplace problem with Neumann boundary conditions:
\[\begin{cases}
  \Delta u=\lambda \beta u&\text{in }\Omega,\\
  \partial_{\nu}u=0&\text{on }\partial\Omega.
\end{cases}\]
(If the boundary $\partial\Omega$ is empty, i.e., if $\om$ is a closed manifold, these are simply the eigenvalues of the weighted Laplace operator $\beta^{-1}\Delta$, without any boundary condition.)
In the special case that $\Omega$ is a surface, the Laplace operator induced by the conformally equivalent metric $\beta g$ satisfies $\Delta_{\beta g}=\beta^{-1}\Delta$, so the variational eigenvalues $\lambda_k(\om,g,\mu)$ are the Neumann eigenvalues of $\Delta_{\beta g}$ with Neumann boundary conditions.  In arbitrary dimension, if $\beta$ is a constant function, the eigenvalues $\lambda_k(\Omega,g,\mu)=\beta^{-1}\lambda_k(\Omega,g)$ where the $\lambda_k(\om,g)$ are the Neumann eigenvalues of $\Delta_\Omega$. 
 
\end{ex}
\begin{ex}\label{ex:vareigenSteklov}
Let $(\Omega,g)$ be a compact Riemannian manifold with non-empty boundary $\Sigma$ and let $\iota:\Sigma\to\Omega$ be the inclusion.  Let $\mu=\iota_\star dV_{\Sigma,g}$ be the push-forward of the boundary measure. That is, for each open set $A\subset\Omega$,
\[\mu(A)=|A\cap\Sigma|_{\Sigma,g}.\]
Then the variational eigenvalues of the measure $\mu=\iota_\star dV_\Sigma$ are the Steklov eigenvalues of $\Omega$, i.e., $\lambda_k(\mu)=\sigma_k(\Omega)$.
\end{ex}
\begin{ex}\label{ex:varweightedSteklov}
Let $(\Omega,g)$ be a compact Riemannian manifold with non-empty boundary, let $\iota:\Sigma\to\Omega$ be the inclusion, and let $0\leq \rho\in L^{\infty}(\Sigma, dV_\Sigma)$.  
Then the variational eigenvalues of the measure $\mu=\iota_\star (\rho\,dV_\Sigma)$ are the eigenvalues $\lambda_k(\mu)=\sigma_k(\om,g,\rho)$ for the so-called \emph{weighted Steklov problem} with density $\rho$ given by:  
\[\begin{cases}
\Delta u= 0\,\,&\mbox{in}\,\,\om,\\
\partial_\nu u =\sigma_k\,\rho u \,\,\,&\mbox{on}\,\,\partial\om.
\end{cases}\]
Note that if $\rho$ is a strictly positive density, then the weighted Steklov eigenvalues are the eigenvalues of $\frac{1}{\rho}\DtN$ where $\DtN$ is the Dirichlet-to-Neumann operator of $(\om,g)$. 
\end{ex}
\begin{ex}
 Let $(\Omega,g)$ be a compact Riemannian manifold with non-empty boundary $\Sigma$ and let $\beta:\Omega\longrightarrow\R$ be a continuous positive function. Consider the measure $\mu=\beta\,dV_\Omega+\iota_\star dV_\Sigma$. Then the  variational eigenvalues $\lambda_k(\mu)$ are the eigenvalues of the \emph{dynamical spectral problem}:
 \[\begin{cases}
  \Delta u=\lambda \beta u&\text{in }\Omega,\\
  \partial_{\nu}u=\lambda u&\text{on }\partial\Omega.
\end{cases}\]
Similar problems were    studied in~\cite{BeFr2005} and used in~\cite{GiHeLa2021} for the study of isoperimetric type inequalities for Steklov eigenvalues, as we will see in Section~\ref{section:surfaces}. Its eigenvalues form an unbounded sequence
$0=\sigma_0\leq\sigma_{1,\beta}\leq\sigma_{2,\beta}\leq\cdots\to+\infty$.
\end{ex}

\begin{ex}\label{example:transmission}
Let $(M,g)$ be a closed $(d+1)$-dimensional manifold and let $\Omega\subset M$ be a domain with smooth boundary $\Sigma$. Let $\iota:\Sigma\to M$ be the inclusion in the ambient manifold $M$ and $\mu=\iota_{\star}dV_{\Sigma}$ be the boundary measure of $\Omega$. Then the variational eigenvalues of $\mu$ are the eigenvalues of the following transmission eigenvalue problem:
\[\begin{cases}
  \Delta u=0&\text{in }M\setminus\Sigma,\\
  (\partial_{\nu^+}+\partial_{\nu^-})u=\tau u&\text{on }\Sigma.
\end{cases}\]
They form an unbounded sequence
$0=\tau_0\leq\tau_1\leq\tau_2\leq\cdots\to+\infty$. It follows directly from the variational definition that $\sigma_k(\Omega)\leq\tau_k$ for each index $k$.
\end{ex}

At this point, variational eigenvalues of Radon measures are merely a convenient tool to keep track of various eigenvalue problems. However, by restricting to classes of Radon measures that have good functional properties, Girouard, Karpukhin and Lagac\'e~\cite{GiKaLa2021} were able to formulate conditions such that the variational eigenvalues $\lambda_k(\mu)$ form a non-negative discrete unbounded sequence and depend continuously on $\mu$. See Appendix~\ref{Section:Radon} for some details. This will be useful when considering the saturation of isoperimetric-type inequalities in subsection~\ref{subsec:homolarge}.

\subsection{Upper bounds, regularity and uniform approximation of domains}
\label{sec:upperboundsunifapprox}

Over the years, many people have proved various upper bounds for Steklov eigenvalues in terms of geometric and topological features of a manifold. These results are usually stated in the class of smooth compact manifolds with boundary. In particular, the boundary of these manifolds are themselves smooth. However, it is interesting to know which of these results can be extended to manifolds with boundary that are not smooth. A particularly interesting case is that of bounded domains with Lipschitz boundary in a complete Riemannian manifold $M$.   Any bounded Lipschitz domain $\om$ in $\R^n$ can be nicely approximated by a sequence of domains $\om_j$ with smooth boundary; see e.g., Verchota \cite{Ve1984} for details. Mitrea and Taylor \cite{MiTa1999} observed that the analogous statement holds for bounded Lipschitz domains in complete Riemannian manifolds.

Several recent results address stability of Steklov eigenvalues and of weighted Steklov eigenvalues  under suitable domain perturbations. (See Example~\ref{ex:varweightedSteklov} for the notion of weighted Steklov eignevalues.)  Bucur and Nahon gave a sufficient condition for stability in the case of plane domains.    Shortly after, Bucur, Giacomini and Trebeschi \cite[Theorem 4.1]{BuGiTr2020} obtained a stability result for Steklov eigenvalues of bounded domains in $\R^n$.   Karpukhin and Lagac\'e, using \cite[Lemma 3.1]{GiLa2021}, extended these results to domains in arbitrary complete Riemannian manifolds. 

In the case of connected surfaces $\om$ with smooth boundary, there exist upper bounds, depending only on the topology of the surface and on the perimeter-normalised Steklov eigenvalues.  These will be discussed in the next section.  The stability results allow these bounds to be extended to domains with Lipschitz boundary and to weighted Steklov eigenvalues:

\begin{cor}\label{rem: Lipschitz bound}\cite{KaLa2022} (See also \cite{ADGHRS2023}.) 
Let $\om=\om_{\gamma,b}$ be the orientable surface of genus $\gamma$ with $b$ boundary components and let 
\[\sigma^*_k(\gamma,b)=\sup_g\,\sigma_k(\om,g)|\partial\om|_g\]
where the supremum is over all smooth Riemannian metrics on $\om$.   Let $D$ be a bounded domain with Lipschitz boundary in a complete Riemannian surface $(M,g)$.  If $D$ is orientable of genus $\gamma$ with $b$ boundary components, then $\sigma_k(D,g)|\partial D|_g\leq \sigma^*_k(\gamma,b)$.   Moreover, the same eigenvalue bounds hold for all weighted Steklov eigenvalue problems on $D$; i.e., $\sigma_k(\om,g,\rho)\|\rho\|_{L^\infty}\leq \sigma^*_k(\gamma,b)$ for all non-negative densities $\rho\in L^1(\partial D)$ that are not identically zero.
\end{cor}

In the higher-dimensional setting, there do not exist upper bounds on the normalised Steklov eigenvalues that depend only on the topology.  However, as will be discussed later in this survey, bounds are known for domains in complete Riemannian manifolds in $M$ subject to geometric constraints on $M$.   As a consequence of the stability results, using any of the normalisations of eigenvalues discussed in Subsection~\ref{subsec: prelim bounds}, one has:

\begin{prop}\cite[Theorem 3.5]{KaLa2022} Let $(M,g)$ be a complete Riemannian manifold.  Then any normalised Steklov eigenvalue bound that is valid for all bounded domains in $(M,g)$ with smooth boundary is also valid for all domains in $(M,g)$ with Lipschitz boundary.
\end{prop}

Under the hypotheses of the proposition, Karpukhin and Lagac\'e also prove \cite[Theorem 1.5]{KaLa2022} that the bounds extend to weighted normalised Steklov eigenvalues, provided that one uses the normalisation introduced in \cite{KaMe2021}.  (See Definition~\ref{def: conf ext pair} later in this survey for the definition of this normalisation for weighted Steklov eigenvalues.)

\begin{remark}  The stability results cited above for Steklov eigenvalues $\sigma_k$ are not uniform in $k$.  Thus while they allow eigenvalue bounds to be generalized from manifolds with smooth boundary to manifolds with Lipschitz boundary, they do not have applications to the challenging question of Steklov asymptotics for manifolds with non-smooth boundary.  See Section~\ref{sec:inv probs pos} for recent results in the setting of Euclidean domains.

\end{remark}

\section{Upper bounds for Steklov eigenvalues on surfaces and homogenisation}
\label{section:surfaces}

The initial impetus for studying Steklov eigenvalues from a geometric point of view came from Weinstock's 1954 paper~\cite{We1954}. He proved that among all simply-connected planar domains $\Omega\subset\R^2$ of prescribed perimeter $L=|\partial\Omega|$, the first nonzero eigenvalue $\sigma_1$ is maximal if and only if $\Omega$ is a disk. 
\begin{thm}[Weinstock, 1954]\label{thm:Weinstock}
Let $\Omega\subset\R^2$ a bounded simply-connected domain with smooth boundary. Then
\begin{equation}\label{ineq:weinstock}
\sigma_1L\leq 2\pi,
\end{equation}
with equality if and only if $\Omega$ is a disk. 
\end{thm}
The proof of Weinstock's inequality~\ref{ineq:weinstock} is prototypical and it will be useful to have it fresh in our mind.  The first step is to use the Riemann mapping theorem to obtain a conformal diffeomorphism $\Phi:\Omega\to\D$. The regularity of the boundary $\partial\Omega$ implies that $\Phi$ extends to a diffeomorphism $\overline{\Phi}:\overline{\Omega}\to\overline{\D}$. Now the first nonzero Steklov eigenvalue of the unit disk $\D$ is 1, and it has multiplicity two. That is, $\sigma_1(\D)=\sigma_2(\D)=1<\sigma_3(\D)$. The corresponding eigenspace is the span of the coordinate functions $\pi_1,\pi_2:\D\to\R$ defined by $\pi_i(x)=x_i$. Precomposing these functions with $\Phi$ leads to functions $u_i:=\pi_i\circ\Phi:\Omega\to\R$ that Weinstock wants to use in the variational characterisation of $\sigma_1(\Omega)$. In order to do so, one must ensure that these functions are admissible: $\int_{\partial\Omega}u_i\,ds=0$ for both $i=1,2$. This is not true for an arbitrary conformal diffeomorphism $\Phi$, but the group of conformal automorphisms of $\D$ is rich enough to ensure the existence of a $\Phi$ for which this holds. This is in fact the easiest occurrence of the now classical \emph{center of mass method} which was put forward by Hersch in~\cite{He1970} following earlier work by Szeg\H{o}. Using this well-chosen conformal diffeomorphism $\Phi$, we see that
\begin{equation*}
    \sigma_1(\Omega)\int_{\partial\Omega}u_i^2\,ds\leq\int_{\Omega}|\nabla u_i|^2\,dA.
\end{equation*}
Summing over $i=1,2$ and using the conformal invariance of the Dirichlet energy (which holds because $\Omega$ is 2-dimensional), this leads to the sought inequality:
\begin{equation*}
    \sigma_1(\Omega)\int_{\partial\Omega}1\,ds\leq\int_{\D}|\nabla\pi_1|^2+|\nabla\pi_2|^2\,dA=2\pi.
\end{equation*}

\begin{remark}
The Weinstock inequality also holds for simply-connected domains with Lipschitz boundary. See the discussion in Section~\ref{sec:upperboundsunifapprox}. Moreover, it also holds for compact simply-connected Riemannian surfaces with boundary.
\end{remark}

\subsection{Upper bounds for multiply-connected domains and surfaces}

Several results for higher-ranked eigenvalues $\sigma_k$ of multiply-connected domains and compact surfaces with boundary were also obtained by various authors, who replaced the conformal equivalence obtained from the Riemann mapping theorem by proper holomorphic covers $\Phi:\Omega\to\D$ that are known as Ahlfors maps. The bounds that are obtained in this way are not sharp in general. They depend on the degree of the cover $\Phi$, which in turn depends on the genus $\gamma$ of the surface and on the number $b$ of connected components of its boundary $\partial\Omega$. See in particular the results of Girouard and Polterovich~\cite{GiPo2012}. More recently, this was improved by Karpukhin, who obtained the following in~\cite[Theorem 1.4]{Ka2017}.
\begin{thm}\label{thm:karp hersch}
Let $\Omega$ be a compact oriented Riemannian surface of genus $\gamma$ with $b$ boundary components. Then for each $p,q\in\N$ the following holds:
\begin{gather*}
    \sigma_p\sigma_qL^2\leq
    \pi^2
    \begin{cases}
    (p+q+2\gamma+2b-3)^2&\text{ if }p+q\equiv 1(\text{ mod }2)\\
    (p+q+2\gamma+2b-2)^2&\text{ if }p+q\equiv 0(\text{ mod }2)\\
    \end{cases}
\end{gather*}
In particular, setting $p=q=k$ leads to
\begin{gather}\label{ineq:karpukhinuppersurface}
    \sigma_kL\leq 2\pi(\gamma+b+k-1).
\end{gather}
\end{thm}
The proof of this result is based on results of Yang and Yu~\cite{YaYu2017} that allow comparison of Steklov problems on differential $p$-forms with eigenvalues of the Laplace--Beltrami operator on the boundary. (We will discuss this result further in Section~\ref{stek.forms}). In the case of surfaces, the boundary is a union of circles and the eigenvalues of the tangential Laplacian are known explicitly.
\begin{remark}
It is common knowledge that using trial functions in variational characterisations of eigenvalues rarely leads to sharp upper bounds.   However,
\begin{enumerate}
    \item For $\gamma=0$, $b=1$, and $k=1$,  one recovers Weinstock's result, in which case the bound (\ref{ineq:karpukhinuppersurface}) is sharp.
    \item For $\gamma=0$, $b=1$, and arbitrary $k$,  one recovers from~ (\ref{ineq:karpukhinuppersurface}) a result of Hersch, Payne and Schiffer~\cite{HePaSc1975}, which is known to be sharp thanks to a construction of Girouard and Polterovich~\cite{GiPo2010}.
    \item For $\gamma=0$, $b=2$ and $k=1$, the best upper bound is known thanks to the work of Fraser and Schoen~\cite{FrSc2016}.  It is attained by the so-called \emph{critical catenoid}, for which $\sigma_1L\approx 4\pi/1.2$. We will discuss the critical catenoid further in  Example~\ref{ex.cat}.
        \end{enumerate} 
    \end{remark}    
    \begin{remark}\label{rem:has cgr}~
    \begin{enumerate}
    \item In early upper bounds on $\sigma_k$, the index $k$ appeared as a multiplicative rather than an additive factor. The first upper bound to feature additivity is due to Hassannezhad~\cite{Ha2011}, who proved the existence of $A,B>0$ such that
    $\sigma_kL\leq A\gamma + Bk$. Note in particular that the number of connected coponents of the boundary does not appear in this inequality.
    \item The dependence on the genus is essential, since Colbois, Girouard and Raveendran~\cite{CoGiRa2018} have constructed surfaces with arbitrarily large normalised first eigenvalue $\sigma_1L$ and with connected boundary.
    \end{enumerate}
\end{remark}

\subsection{Upper bound for surfaces of genus 0}\label{section:uppergenus0}

For $\eps\in(0,1)$, let $A_\eps:=B(0,1)\setminus B(0,\eps)$ be a planar annulus. It was observed in~\cite{GiPo2017} that for $\eps>0$ small enough, $\sigma_1(A_\eps)L(\partial A_\eps)>2\pi$. This shows that the simple-connectedness assumption in Weinstock's result is genuinely necessary, which raises the question of how large the perimeter-normalised eigenvalue $\sigma_1(\Omega)L(\partial\Omega)$ can be among all bounded planar domains $\Omega\subset\R^2$ with smooth boundary, without assuming simple-connectivity. 
The Riemann mapping theorem is not available in this case. We have seen above that one could use the Ahlfors map instead, in which case the resulting upper bounds depend on the number of connected components of its boundary. In~\cite{Ko2014} Kokarev proposed instead to use the stereographic parametrisation 
$\Phi:\R^2\to \Sp^2$. The group of automorphisms of $\Sp^2$ is rich enough to ensure the existence of a conformal diffeomorphism $F:\Sp^2\to\Sp^2$ for which the functions $u_i:=\pi_i\circ F\circ\Phi:\Omega\to\R$ satisfy $\int_{\partial\Omega}u_i\,ds=0$ for $i\in\{1,2,3\}$ as well as $u_1^2+u_2^2+u_3^2=1$ identically on $\partial\Omega$. This is the Hersch renormalisation trick again, as in the proof above of the Weinstock inequality (Theorem~\ref{thm:Weinstock}). Because the stereographic projection is a conformal map, it follows as above that
\[\sigma_1L\leq \int_{F(\Phi(\Omega))}\sum_{i=1}^3|\nabla\pi_i|^2\,dA.\]
An easy computation shows that $\sum_{i=1}^3|\nabla\pi_i|^2=2$ pointwise, so that
\[\sigma_1L\leq 2\text{Area}(F(\Phi(\Omega))    )<2\text{Area}(\Sp^2)=8\pi.\]
This led Kokarev \cite[Theorem A1]{Ko2014} to the following result.
\begin{thm}\label{thm:KokarevGenus0}
Let $\Omega$ be a compact surface with boundary of genus $\gamma=0$. Then
\begin{gather}\label{ineq:Kokarev}
\sigma_1L< 8\pi.
\end{gather}
\end{thm}
It is then natural to investigate the sharpness of inequality~\eqref{ineq:Kokarev}, which amounts to the construction of surfaces of genus 0 with large enough perimeter-normalised $\sigma_1L$.  

\subsection{Interlude: using homogenisation theory to obtain large Steklov eigenvalues}\label{subsec:homolarge}

In~\cite{GiHeLa2021}, Girouard, Henrot and Lagac\'e studied homogenisation of the Steklov problem by periodic perforation. The  results are expressed in terms of the Neumann and dynamical eigenvalue problems, which we now recall. 
Let $\Omega\subset\R^{d+1}$ be a bounded domain with smooth boundary $\partial\Omega$. Recall that the Neumann spectral problem is
\[\begin{cases}
  \Delta f=\lambda f&\text{in }\Omega,\\
  \partial_{\nu}f=0&\text{on }\partial\Omega.
\end{cases}\]
Its spectrum consists of an unbounded sequence of real numbers
\[0=\lambda_0<\lambda_1\leq\lambda_2\leq\cdots\to+\infty.\]
The \emph{dynamical spectral problem with parameter $\beta\in [0,\infty)$} is
\begin{gather}\label{eq:dynamiceigen}
\begin{cases}
  \Delta f=\beta \sigma f&\text{in }\Omega,\\
  \partial_{\nu}f=\sigma f&\text{on }\partial\Omega.
\end{cases}
\end{gather}
Its spectrum consists of an unbounded sequence of real numbers
\[0=\sigma_0<\sigma_{1,\beta}\leq\sigma_{2,\beta}\leq\cdots\to+\infty.\]
Readers are invited to look at the work of von Below and Fran\c{c}ois~\cite{BeFr2005} for details.
For $\beta=0$, the dynamical eigenvalues of $\Omega$ coincide with Steklov eigenvalues: $\sigma_{k,0}(\Omega)=\sigma_k(\Omega)$. As $\beta\to+\infty$, the relative importance of the spectral parameter in the boundary condition seems to disappear. This is captured in the following result~\cite{GiHeLa2021}.
\begin{thm}\label{thm:dynamiclargebeta}
For each $k\in\N$, the eigenvalue $\sigma_{k,\beta}$ depends continuously on $\beta$ and satisfies
\[\lim_{\beta\to+\infty}\beta \sigma_{k,\beta}=\lambda_k.\]
\end{thm}
For the purpose of studying Steklov eigenvalues, the importance of the dynamical eigenvalue problem~\eqref{eq:dynamiceigen} is that it appears as the limit problem for periodic homogenisation by perforation. Given $\eps>0$, let $\Omega^\eps\subset\Omega$ be the domain obtained by removing balls $B(p,r_\eps)\subset\Omega$ of radius $r_\eps>0$ centered at a point $p$ of the periodic lattice $\eps\Z^{d+1}$. See Figure~\ref{figure:Omegaeps}.
\begin{figure}
  \centering
  \includegraphics[width=9cm]{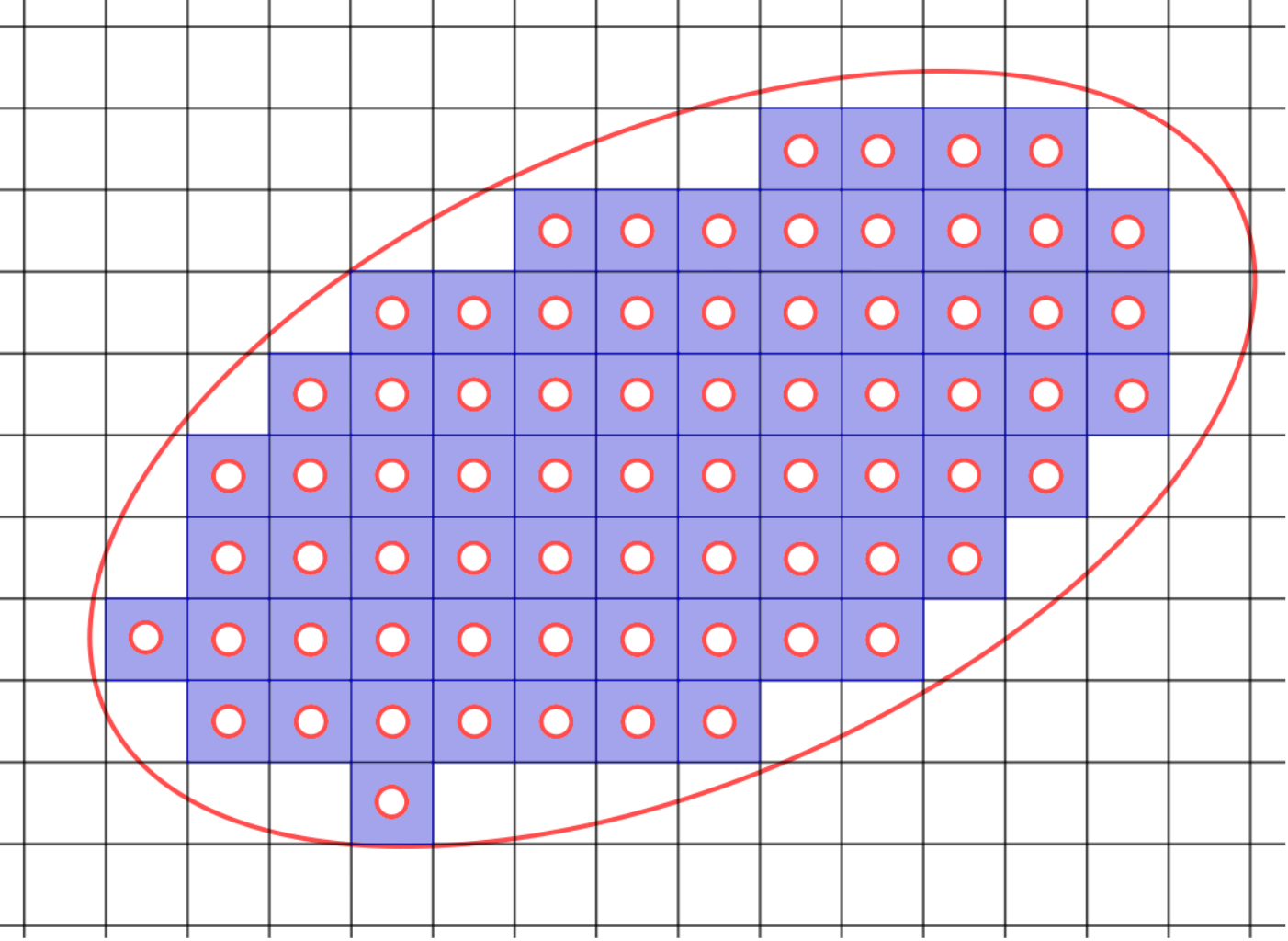}
  \caption{The perforated domain $\Omega^\eps\subset\R^2$ in the planar case.}
  \label{figure:Omegaeps}
\end{figure}
The behaviour of the Steklov eigenvalues $\sigma_k(\Omega^\eps)$ as $\eps\to 0$ then depends on the choice of radius $r_\eps\in (0,\eps)$. The following result should be compared with the classical \emph{crushed ice problem}. See for instance the work of Rauch and Taylor~\cite{RaTa1975}, and of Cioranescu and Murat~\cite{CiMu1982}.
\begin{thm}[\cite{GiHeLa2021}]\label{thm:PeriodicHomoeConstantRadius}
Let $\Omega\subset\R^{d+1}$ be a bounded domain with smooth boundary $\partial\Omega$. Let $T^\eps\subset\Omega$ be the union of all balls $B(p,r_\eps)$ with $p\in\eps\Z^d$ that are included in $\Omega$, and consider the perforated domain $\Omega^\eps:=\Omega\setminus\overline{T^\eps}$.
The asymptotic behaviour of $\sigma_k(\Omega^\eps)$ as $\eps\to 0$ depends on the parameter\footnote{It is implicitly understood that $r_\eps$ is chosen so that this limit exists.}
\begin{gather}\label{eq:betaregimehomo}
    \beta:=\frac{1}{|\Sp^d|}\lim_{\eps\to 0}r_\eps^{d}/\eps^{d+1}\in[0,+\infty].
\end{gather}
\begin{description}
\item[Small-holes regime] If $\beta=0$, then 
\[|\partial T^\eps|\to 0\quad\text{ and }\quad \sigma_k(\Omega^\eps)\to\sigma_k(\Omega).\]
\item[Large-holes regime] If $\beta=+\infty$, then 
\[|\partial T^\eps|\to \infty\quad\text{ and }\quad\sigma_k(\Omega^\eps)\to 0.\]
\item[Critical regime] If $\beta\in (0,\infty)$, then
\[|\partial T^\eps|\to \beta |\Omega|\quad\text{ and }\quad
    \sigma_k(\Omega^\eps)\to\sigma_{k,\beta}.\]
\end{description}
\end{thm}
\begin{remark}
The convention that we use for the dynamical eigenvalue problem is slightly different from that of~\cite{GiHeLa2021}, where the constant $A_d:=|\Sp^d|$ was built into the problem itself, while here we have simply introduced it in the definition of the constant~\eqref{eq:betaregimehomo} controlling the homogenisation regime. This is clearer for our purpose, and in particular will be compatible with the situation where $\beta$ is a density function rather than a constant.
\end{remark}
\begin{remark}
The large-holes regime could be deduced from known upper bounds for $\sigma_k$. For instance, it follows from the earlier work of Colbois, Girouard and El Soufi~\cite{CoElGi2011} that any domain $\Omega\subset\R^{d+1}$ with boundary of large measure has small Steklov eigenvalues. This follows for instance from inequality~\eqref{ineq:ceg} in the next section.
\end{remark}
The critical regime is particularly interesting for planar domains.
Indeed, in this case $d=1$ and it follows that
\begin{gather}\label{eq:limithomoplanar}
\sigma_k(\Omega^\eps)L(\partial\Omega^\eps)\xrightarrow{\eps\to0}
\sigma_{k,\beta}(\Omega)(L(\partial\Omega)+\beta |\Omega|)\xrightarrow{\beta\to+\infty}\lambda_k(\Omega)|\Omega|,
\end{gather}
where the second limit is a direct consequence of Theorem~\ref{thm:dynamiclargebeta}. 
This suggests a link between maximisation of perimeter-normalised Steklov eigenvalues and maximisation of area-normalised Neumann eigenvalues. In particular, the Szeg\H{o}--Weinberger inequality states that for each bounded planar domain $\Omega\subset\R^2$, 
\[\lambda_1(\Omega)|\Omega|\leq\lambda_1(\D)\pi\cong 3.39\pi,\]
hence the largest possible value for the RHS of~\eqref{eq:limithomoplanar} is $\lambda_1(\D)\pi$.
A simple diagonal argument then leads to the existence of a family $\Omega^\eps\subset\R^2$ with $\sigma_1(\Omega^\eps)L(\partial\Omega^\eps)\xrightarrow{\eps\to 0}\lambda_1(\D)\pi$. In combination with
the above Theorem~\ref{thm:KokarevGenus0} of Kokarev, this shows that
\[\lambda_1(\D)\pi\leq \sup\{\sigma_1(\Omega)L(\partial\Omega)\,:\,\Omega\subset\R^2\}\leq 8\pi.\]

In this perforation procedure, the balls $B(p,r_\eps)$ all have the same radius $r_\eps$. Could one obtain larger Steklov eigenvalues by relaxing this constraint? Indeed, given a constant $\alpha>d$ and a positive continuous function $\beta:\Omega\to\R$, Girouard, Karpukhin and Lagac\'e~\cite{GiKaLa2021} introduced the family of functions $r_{\eps,\alpha}:\Omega\to\R$ defined by
\[r_{\eps,\alpha}(p)=\left(\frac{\eps^\alpha}{|\Sp^d|}\beta(p)\right)^{1/d}.\]
This function now determines the radii of the various balls to be removed from $\Omega$, with the situation where $\beta$ is constant and $\alpha=d+1$ corresponding to the critical regime in Theorem~\ref{thm:PeriodicHomoeConstantRadius}.
Given $\eps>0$, let $\Omega^\eps\subset\Omega$ be the domain obtained by removing balls $B(p,r_{\eps,\alpha}(p))\subset\Omega$ of radius $r_{\eps,\alpha}(p)>0$ centered at point $p$ of the periodic lattice $\eps\Z^{d+1}$. The behaviour of the Steklov eigenvalues $\sigma_k(\Omega^\eps)$ as $\eps\to 0$ depends on the choice of density function $\beta$ and on the parameter $\alpha$.
It is convenient to discuss the results in terms of variational eigenvalues of Radon measures. See Section~\ref{subsection:introvareigenradon} for a quick overview and Appendix~\ref{Section:Radon} for more details. The inclusion $\iota:\partial\Omega^\eps\to\overline{\Omega^\eps}$ allows the definition of the push-forward measure \[\mu_\alpha^\eps=\iota_*dV_{\partial\Omega^\eps}\] on $\overline{\Omega^\eps}$, which we call the boundary measure of the perforated domain $\Omega^\eps$.  The corresponding renormalised probability measures are
\[\overline{\mu_\alpha^\eps}=\frac{\mu_\alpha^\eps}{|\partial\Omega^\eps|}.\]
It follows from the definition~(\ref{def:VarEigenRadon}) of variational eigenvalues and from~Example~\ref{ex:vareigenSteklov} that
\[\lambda_k(\Omega^\eps,\overline{\mu_\alpha^\eps})=\sigma_k(\Omega^\eps)|\partial\Omega^\eps|. \]
The behaviour of these measures as $\eps\to 0$ depends on the parameter $\alpha>d$. This is expressed in terms of weak-$\star$ convergence to limit measures. Girouard, Karpukhin and Lagac\'e~\cite[Theorem 6.2]{GiKaLa2021} proved the convergence of the corresponding variational eigenvalues.
\begin{thm}\label{thm:HomoBdryGKL}
Let $\Omega\subset\R^{d+1}$ be a bounded domain with smooth boundary $\Sigma$ and let $\beta:\Omega\to\R$ be a positive continuous function. For each $\eps>0$ small enough, let $T^\eps\subset\Omega$ be the union of all balls $B(p,r_{\eps,\alpha}(p))$ with $p\in\eps\Z^{d+1}$ that are included in $\Omega$, and consider the perforated domain $\Omega^\eps:=\Omega\setminus\overline{T^\eps}$.
\begin{description}
\item[Small-holes regime] If $\alpha>d+1$, then the measures $\overline{\mu_\alpha^\eps}$ concentrate on the boundary $\Sigma$:
\[\overline{\mu_\alpha^\eps}\xrightarrow{\eps\to0}\overline{dV_{\Sigma}}.\]
Moreover,
\[\sigma_k(\Omega^\eps)|\partial\Omega^\eps|=\lambda_k(\Omega^\eps,\overline{\mu_\alpha^\eps})\xrightarrow{\eps\to0}\lambda_k(\Omega,\overline{dV_{\Sigma}})=\sigma_k(\Omega)|\Sigma|.\]
\item[Large-holes regime] If $\alpha\in (d,d+1)$, then the boundary becomes negligible in the limit and $\beta dV_\Omega$ dominates:
\[\overline{\mu_\alpha^\eps}\xrightarrow{\eps\to0}\overline{\beta dV_{\Omega}}.\]
Moreover,
\[\sigma_k(\Omega^\eps)|\partial\Omega^\eps|
=
\lambda_k(\Omega^\eps,\overline{\mu_\alpha^\eps})\xrightarrow{\eps\to0}\lambda_k(\Omega,\overline{\beta dV_{\Omega}})
=
\lambda_k(\Omega,\beta)\int_{\Omega}\beta\,dV_\Omega.\]
\item[Critical regime] If $\alpha=d+1$, then both the boundary and interior measures persist in the limit:
\[\overline{\mu_\alpha^\eps}\xrightarrow{\eps\to0}\overline{dV_{\Sigma}+\beta dV_{\Omega}}.\]
Moreover,
\[\sigma_k(\Omega^\eps)|\partial\Omega^\eps|=\lambda_k(\Omega^\eps,\overline{\mu_\alpha^\eps})\xrightarrow{\eps\to0}\lambda_k(\Omega,\overline{dV_{\Sigma}+\beta dV_{\Omega}})=\sigma_{k,\beta}(\Omega)(|\Sigma|+\int_{\Omega}\beta\,dV_\Omega).\]
\end{description}
\end{thm}
The proof of Theorem~\ref{thm:HomoBdryGKL} is based on continuity properties of variational eigenvalues associated to the Radon measures $\overline{\mu_\alpha^\eps}$. See Theorem~\ref{thm:ContinuityRadon}.
For a simply-connected domain $\Omega\subset\R^2$, particularly interesting density functions $\beta\in C^\infty(\Omega)$ can be obtained by considering conformal maps $\Phi_\delta$ from the disk $\D$ to punctured spheres $C_\delta=\Sp^2\setminus B(p,\delta)$. In particular, if $\Omega=\D$ these maps can be constructed explicitly as a composition of the stereographic parametrisation with homotheties of the plane. The pullback of the round metric $g_{\Sp^2}$ is of the form $\Phi_\delta^{\star}g_{\Sp^2}=\beta_{\delta}g_{\text{eucl}}$, for some positive $\beta_\delta\in C^\infty(\overline{\D})$. It follows from conformal invariance of the Dirichlet energy that the Neumann eigenvalues of $C_\delta$ can be expressed as weighted Neumann eigenvalues of $\Omega$:
\[\lambda_k(C_\delta,g_{\Sp^2})=\lambda_k(\Omega,\beta_{\delta}dV_\Omega).\]
Now, it is well known that the Neumann eigenvalues of a punctured closed Riemannian manifold converge to the eigenvalues of the manifold as the radius of the puncture goes to zero (see for instance the work of Ann\'e~\cite{An1986}). In particular, $\lambda_k(C_\delta)\xrightarrow{\delta\to 0}\lambda_k(\Sp^2)$.
Combining this observation with the large-hole regime of~Theorem~\ref{thm:HomoBdryGKL} leads to a family of perforated domains $\Omega^{\eps}\subset\Omega$ such that
\[\sigma_k(\Omega^\eps)L(\partial\Omega^\eps)\xrightarrow{\eps\to0}
\lambda_k(\Omega,\beta_\delta)\int_{\Omega}\beta_\delta\,dV_\Omega=\lambda_k(C_\delta)|C_\delta|\xrightarrow{\delta\to 0}\lambda_k(\Sp^2)\times 4\pi.\]
Since $\lambda_1(\Sp^2)=2$, this shows that Kokarev's inequality~\eqref{ineq:Kokarev} is sharp and provides a complete solution for the isoperimetric problem for $\sigma_1$ of planar domains. 
\begin{thm}[\cite{Ko2014,GiKaLa2021}]
\label{thm:IsopStekOnePlanar}
Let $\Omega\subset\R^2$ be a bounded domain with sufficiently regular boundary. Then,
$\sigma_1(\Omega)L(\partial\Omega)<8\pi$.
Moreover there exists a family $\Omega^\eps\subset\R^2$ such that
\[\sigma_1(\Omega^\eps)L(\partial\Omega^\eps)\xrightarrow{\eps\to 0}8\pi.\]
\end{thm}

Instead of using the round metric $g_{\Sp^2}$ on the sphere, one can use an arbitrary metric $g$ and proceed exactly as above to obtain planar domains $\Omega^\eps$ such that $\sigma_k(\Omega^\eps)L(\partial\Omega^\eps)\xrightarrow{\eps\to 0}\lambda_k(\Sp^2,g)\text{Area}(\Sp^2,g)$. The best upper bound for these area-normalised eigenvalues were obtained by Karpukhin, Nadirashvili, Penskoi and Polterovich in~\cite{KNPP2020}:
\[\lambda_k(\Sp^2,g)\text{Area}(\Sp^2,g)\leq 8\pi k.\]
This shows that $\sup\{\sigma_k(\Omega)L(\partial\Omega)\,:\Omega\subset\R^2\}\geq 8\pi k$. However, $8\pi k$ is also an upper bound, as we will see shortly (see Theorem~\ref{thm: GiKaLa 8pi}).

\subsection{Best upper bounds for Steklov eigenvalues and conformal eigenvalues}

Thus far we have been discussing homogenisation in the Euclidean setting where we have a periodic procedure for perforating a domain. In order to address domains in compact Riemannian manifolds, Girouard and Lagac\'e~\cite{GiLa2021} extended the homogenisation procedure of~\cite{GiHeLa2021} to the non-periodic setting by using Vorono\u{\i} tessellations associated to maximal $\eps$-separated subsets in a closed Riemannian manifold $(M,g)$.
\begin{thm}[\cite{GiLa2021}]\label{thm:homoclosed}
Let $(M,g)$ be a closed Riemannian manifold and let $\beta:M\to\R_{>0}$ a continuous function. There exists a family $\Omega^\eps\subset M$ such that
$|\partial\Omega^\eps|\xrightarrow{\eps\to 0}\int_{M}\beta\,dV_g$, $|\Omega^\eps|\xrightarrow{\eps\to 0}|M|_g$ and for each $k\in\N$,
\begin{gather}\label{eq:limithomomanifold}
    \sigma_k(\Omega^\eps)\xrightarrow{\eps\to 0}\lambda_k(\beta^{-1}\Delta_g).
\end{gather}
\end{thm}

For surfaces, this suggests a link between maximisation of perimeter-normalised Steklov eigenvalues for domains $\Omega\subset M$ and maximisation of area-normalised eigenvalues of the Laplace operator. This is best expressed by introducing the \emph{conformal eigenvalues} of a closed Riemannian manifold $(M,g)$ of dimension $d+1$. They are defined by
\begin{gather}\label{eq:defConformalEigenvalues}
\lambda_k^*(M,[g]):=\sup_{h\in[g]}\lambda_k(M,h)\text{Vol}(M,h)^{2/{d+1}}.
\end{gather}
They were introduced and studied in~\cite{CoEl2003}, following work of Korevaar~\cite{Ko1993} who showed that they are finite.

In their paper~\cite{KaSt2020}, Karpukhin and Stern discovered a link between the Steklov eigenvalues of domains in a closed surface and the eigenvalues of the Laplacian on that surface. In particular they proved that for any domain $\Omega\subset M$, the following strict inequalities hold:
\begin{gather}\label{eq:karp stern}
    \sigma_1(\Omega,g)L(\partial\Omega)<\lambda_1^*(M,[g])\qquad \text{ and }\qquad\sigma_2(\Omega,g)L(\partial\Omega)<\lambda_2^*(M,[g]).
\end{gather}
It follows from Theorem~\ref{thm:homoclosed} that these inequalities are sharp and raises the question of whether similar inequalities hold for arbitrary index $k$.
Girouard, Karpukhin and Lagac\'e answered this question in the affirmative as follows:
\begin{thm}[\cite{GiKaLa2021}]\label{thm: GKL sharp}
Let $M$ be a closed surface. Then for each $k\in\N$ and for each domain $\Omega\subset M$,
\begin{gather}\label{ineq:StekConfEigen}
    \sigma_k(\Omega)L(\partial\Omega)\leq\lambda_k^*(M,[g]).
\end{gather}
Moreover, for each $k\in\N$, there exists a family of domains $\Omega^\eps\subset M$ such that \[\sigma_k(\Omega^\eps)L(\partial\Omega^\eps)\xrightarrow{\eps\to0}\lambda_k^*(M,[g]).\]
\end{thm}
The family $\Omega^\eps$ is obtained as a direct consequence of Theorem~\ref{thm:homoclosed} while the upper bound is proved using continuity properties of the variational eigenvalues associated to Radon measures. Indeed, given a domain $\Omega\subset M$ with boundary $\Sigma$, let $\mu=\iota_{\star}dA$ be the boundary measure of $\Omega$ in $M$. Recall from Example~\ref{example:transmission} that the variational eigenvalues of $\mu$ are the transmission eigenvalues $\tau_k(\Omega)$ of $\Omega\subset M$.
One can construct a family $g_\eps=\beta_{\eps}g$ of conformal Riemannian metrics that concentrates on $\partial\Omega$ in the weak-$\star$ sense:
$dV_{g_\eps}\xrightarrow{\eps\to 0}\mu$. It was proved in~\cite{GiKaLa2021} that the corresponding variational eigenvalues also converge:
\[\lambda_k(M,g_\eps)=\lambda_k(M,g,dV_{g_\eps})\xrightarrow{\eps\to 0}\lambda_k(M,g,\mu)=\tau_k(\Omega)\geq\sigma_k(\Omega).\]
See Theorem~\ref{thm:ContinuityRadon}. Inequality~\eqref{ineq:StekConfEigen} now follows from the definition~\eqref{eq:defConformalEigenvalues} of the conformal eigenvalue $\lambda_k^*(M,[g])$.

In view of the strict inequalities~\eqref{eq:karp stern}, we ask the following:
\begin{ques}\label{ques:karpstern}
Can inequality~\eqref{ineq:StekConfEigen} be improved to a strict inequality for $k\geq 3$?
\end{ques}

The conformal eigenvalues of many surfaces are known explicitly. For instance, it was proved in~\cite{KNPP2020} that $\lambda_k^*(\Sp^2,[g_{\Sp^2}])=8\pi k$. Because planar domains are conformally equivalent to spherical domains, this implies a complete solution of the isoperimetric problems for $\sigma_k$.
\begin{thm}[\cite{GiKaLa2021}]\label{thm: GiKaLa 8pi}
Let $\Omega\subset\R^2$ be a bounded domain with sufficiently regular boundary. Then,
$\sigma_k(\Omega)L(\partial\Omega)\leq8\pi k$. 
Moreover there exists a family $\Omega^\eps\subset\R^2$ such that
\[\sigma_k(\Omega^\eps)L(\partial\Omega^\eps)\xrightarrow{\eps\to 0}8\pi k.\]
\end{thm}

\begin{remark}\label{rem:b to infty}
For $M$ a closed surface, define
\begin{equation}\label{eq:lambdak*M}
 \lambda_k^*(M)=\sup_g\,\lambda_k(M,g)\operatorname{Area}(M,g)   
\end{equation}
where the supremum is over all Riemannian metrics on $M$.  Similarly, for $\om$ a compact surface with boundary, define 
\begin{equation}\label{eq:sigmak*M}
 \sigma_k^*(\om)=\sup_g\,\sigma_k(\om,g)L(\partial\om,g)  
\end{equation}
where again the supremum is over all Riemannian metrics on $\om$.  

Now let $M_b$ be the compact surface with boundary  obtained by removing $b$ disjoint disks from $M$. Then as a consequence of Theorem~\ref{thm: GKL sharp} and its proof by homogenisation, we have $\sigma_k^*(M_b)\leq \lambda_k^*(M)$ for all $b$ and 
\[\lim_{b\to\infty}\,\sigma_k^*(M_b)=\lambda_k^*(M).\]
In particular, letting $M$ be the 2-sphere, we have in the notation of Remark~\ref{rem: Lipschitz bound} that 
\[\lim_{b\to\infty}\sigma^*_1(0,b)=8\pi.\]
\end{remark}

\subsection{Stability and quantitative isoperimetry}

Whenever a sharp inequality is known, it becomes interesting to investigate the case of almost equality.
For instance, the case of equality in Weinstock's theorem (Theorem~\ref{thm:Weinstock}) states that for simply-connected domains, $\sigma_1(\Omega)L(\partial\Omega)=2\pi$ if and only if $\Omega$ is a disk. This raises the question of whether a domain $\Omega$ having $\sigma_1(\Omega)L(\partial\Omega)$ near $2\pi$ implies that $\Omega$ is near a disk, and if so, in which sense.

In~\cite{BuNa2021}, Bucur and Nahon gave a negative answer to that question. Let us state a particularly striking case of their result here.
\begin{thm}\label{thm:BucurNahon}
Let $\Omega\subset\R^2$ be a simply-connected planar domain. Then there exists a family $\Omega^\eps$ of simply-connected planar domains such that
$\Omega^\eps\xrightarrow{\text{Hausdorff}}\Omega$ while for each $k\in\N$,
\[\sigma_k(\Omega^\eps)L(\partial\Omega^\eps)\xrightarrow{\eps\to0} 2\pi\sigma_k(\D).\]
\end{thm}
In particular, $\sigma_1(\Omega^\eps)L(\partial\Omega^\eps)\to 2\pi$, while the domains $\Omega^{\epsilon}$ approach a domain $\Omega$ which could be very different from a disk.
\begin{remark}\label{rem:BucurNahon}
The proof of Theorem~\ref{thm:BucurNahon} is based on a boundary-homogenisation method. Let $\Phi:\Omega\to\D$ be a conformal diffeomorphism. Then because $\Omega$ is smooth, $\Phi$ extends to a diffeomorphism up to the boundary, and one can define the pullback measure $\mu:=\Phi^\star(ds)=|\Phi'(s)|\,ds$ on $\partial\Omega$. Because the Dirichlet energy is conformally invariant, the Steklov problem on $\D$ is isospectral to a weighted Steklov problem on $\Omega$, which we express using the variational eigenvalue associated to this Radon measure:
$\lambda_k(\Omega,\mu)=\sigma_k(\D)$. See Section~\ref{subsection:introvareigenradon} for a quick overview of variational eigenvalues of Radon measures and Example~\ref{ex:varweightedSteklov} for the notion of weighted Steklov problem. The proof is then based on perturbations $\Omega^\eps$ of $\Omega$ by small oscillations of its boundary that lead to smooth approximations of the measure $\mu$ by the boundary measures of $\Omega^\eps$, so that $\sigma_k(\Omega^\eps)\to\lambda_k(\Omega,\mu)$.
\end{remark}
The above result shows that the first perimeter-normalised Steklov eigenvalue is not stable under Hausdorff perturbations of the domain. However, the story is completely different if we change the notion of proximity between $\Omega^\eps$ and $\Omega$ that we use, as we will see shortly.
 Let $\Omega$ be a bounded simply-connected planar domain such that $L(\partial\Omega)=2\pi$. Let $\Phi:\D\to\Omega$ be a conformal diffeomorphism and let $\mu:=\Phi^{\star}ds=\Theta\,ds$,
where $\Theta(z)=|\Phi'(z)|\in C^\infty(\partial\D)$. Because the group of conformal automorphisms of $\D$ is rich enough, it is possible to choose $\Phi$ so that the measure $\mu$ has its center of mass at the origin: $\int_{\partial\D}\pi_i\,d\mu=0$ for each $i\in\{1,2\}$. 
This follows from a topological argument that was introduced by Hersch~\cite{He1970} following work of Szeg\H{o}.
Observe that
\[\lambda_1(\D,\mu)=\sigma_1(\Omega).\]
The following is a reformulation of \cite[Proposition 3.1]{BuNa2021}.
\begin{thm}
There are constants $\delta_0\in (0,1)$ and $C>0$ with the following property.
Let $\Theta\in L^\infty(\partial\D)$ be a positive density such that $\int_{\partial\D}\pi_i\,\Theta ds=0$ for $i\in\{1,2\}$ and such that $\int_{\partial\D}\Theta\,ds=2\pi$.
If $\lambda_1(\D,\Theta\,ds)>\delta_0$, then
\begin{gather}\label{ineq:wstab}
\lambda_1(\D,\Theta\,ds)\leq \frac{1}{1+C\|\Theta-1\|_{H^{-1/2}}^2}.
\end{gather}
\end{thm}
This beautiful result shows that in this particular norm, stability is restored for the Weinstock inequality after transplantation to the disk by an appropriate conformal map. Indeed, if $\Omega^\eps$ is a family of bounded simply-connected planar domains such that $L(\partial\Omega^\eps)=2\pi$ and $\sigma_1(\Omega^\eps)\xrightarrow{\eps\to 0} 1$, then it follows from~\eqref{ineq:wstab} that in the dual Sobolev space $H^{-1/2}(\partial\Omega)$, the corresponding densities $\Theta_\eps$ will satisfy
\[\lim_{\eps\to 0}\Theta_\eps=1.\]
This should be compared with the recent paper~\cite{KNPS2021} by Karpukhin, Nahon, Polterovich and Stern, where the stability of isoperimetric inequalities for eigenvalues of the Laplace operator on surfaces is studied.

The situation for arbitrary bounded planar domains $\Omega\subset\R^2$ is quite different. Indeed, Theorem~\ref{thm:IsopStekOnePlanar} states that $\sigma_1L<8\pi$ is a sharp upper bound, but since the inequality is strict, there does not exist a planar domain realising this bound. However we can still obtain interesting information regarding planar domains $\Omega\subset\R^2$ such that $\sigma_1(\Omega)L(\partial\Omega)$ is close to $8\pi$.
In this case again, there is a flexibility in the geometry of the maximising sequence. Indeed, it was proved in~\cite{GiKaLa2021} that one can start with any simply-connected domain $\Omega\subset\R^2$ with smooth boundary, and construct a family of domains $\Omega^\eps\subset\Omega$, obtained by perforation, such that \[\sigma_1(\Omega^\eps)L(\partial\Omega^\eps)\xrightarrow{\eps\to0}8\pi.\]
However, one may still obtain geometric information on maximising sequences in this situation.
The following result is a corollary of~\cite[Theorem 2.1]{GiKaLa2021}.
\begin{thm}\label{thm:8pi strict}
Let $\Omega$ be a compact surface of genus $0$ with  $\sigma_1(\Omega)L(\partial\Omega)\geq 6\pi$ and such that the number of connected components of its boundary is $b$. Then,
\begin{gather*}
\sigma_1(\Omega)L(\partial\Omega)\leq 8\pi-6\pi\exp(-2b).
\end{gather*}
Moreover, if $\Omega\subset\R^2$ is a bounded planar domain the following also holds:
\begin{gather*}
\sigma_1(\Omega)L(\partial\Omega)\leq 8\pi-6\pi\exp(-\frac{L(\partial\Omega)}{\diam(\Omega)})
\end{gather*}
\end{thm}
The proof is based on a careful quantitative adaptation of Kokarev's result (Theorem~\ref{thm:KokarevGenus0}). The second inequality proves in particular that any maximising sequence $\Omega^\eps\subset\R^2$ for $\sigma_1L$ has a boundary with an unbounded number of connected components in the limit. Moreover, if the maximising sequence is normalised by requiring that $L(\partial\Omega^\eps)=1$, then its diameter must tend to 0 in the limit.

Let us conclude this section by mentioning a recent preprint of Karpukhin and Stern~\cite{KaSt2021} where a quantitative improvement of Theorem~\ref{thm: GKL sharp} is presented for some surfaces: the sphere $\Sp^2$, the projective plane $\R\mathbb{P}^2$, the torus $\mathbb{T}$ and the Klein bottle $\mathbb{K}$.

\section{Geometric bounds in higher dimensions}
\label{section:boundshigher}

Because of scaling properties, meaningful bounds for eigenvalues require some type of constraint on the size of the underlying manifold. In the previous sections we have seen that prescribing the area of a surface or the length of its boundary is often sufficient to obtain interesting upper bounds. For manifolds of dimension at least 3, the situation is more complicated. For instance it was proved by Colbois, El Soufi and Girouard ~\cite{CoElGi2019} that any compact connected manifold $(\Omega,g_0)$ admits conformal perturbations $g=\delta g_0$ with $\delta\equiv 1$ on the boundary and with $\sigma_1(\Omega,g)$ arbitrarily large. In other words, prescribing the geometry of the boundary is not enough to bound $\sigma_1$, even while staying in a fixed conformal class. This shows that, in order to get upper bounds for the Steklov spectrum in dimension higher than $2$, we need additional geometric information. 

Much less is known regarding lower bounds, and in Subsection \ref{Lower}, we will summarise what is known.

Another indication of the flexibility of the Steklov spectrum is given in
the paper~\cite{Ja2014} by Jammes, where it is shown that any finite part of the Steklov spectrum can be prescribed within a given conformal class $(\Omega,[g_0])$ provided that $\Omega$ has dimension at least three. In view of the work of Lohkamp~\cite{Lo1996} it is therefore natural to ask, in the case of manifolds of dimension at least 3, whether it is possible in a fixed conformal class to simultaneously prescribe a finite part of the Steklov spectrum, the volume of $(\Omega,g)$, and the volume of $\partial \Omega$. We will explain in Remark \ref{rem:prescription} below that the answer is no.

 The outline of this section is as follows.

\begin{itemize}

\item Subsection~\ref{Lower}: Lower bounds for eigenvalues. 
\item Subsection~\ref{Upperexamples}: Upper bounds for eigenvalues: basic results and examples.
\item Subsection~\ref{UpperDomains}: Upper bounds: the case of domains in a Riemannian manifold.
 \item Subsection~\ref{UpperRiemannian}: Metric upper bounds for Riemannian manifolds.
  
\item Subsection~\ref{UpperRev}: Upper and lower bounds: the case of manifolds of revolution.
\item Subsection~\ref{UpperSub}: Upper and lower bounds: the case of submanifolds of Euclidean space.
\end{itemize}

\subsection{Lower bounds for eigenvalues} \label{Lower}

Regarding lower bounds for eigenvalues, and in particular for the first nonzero eigenvalue $\sigma_1$, it is useful to recall briefly a couple of facts about the spectrum of the Laplacian. If $(M,g)$ is a closed connected Riemannian manifold of dimension $d$ (resp. if $(\Omega,g)$ is a compact connected Riemannian manifold of dimension $d$ with boundary), there are two main ways to find a lower bound for the first nonzero eigenvalue $\lambda_1(M,g)$ (resp. the first nonzero eigenvalue $\mu_1(\Omega,g)$  for the Laplacian with the Neumann boundary condition):

\smallskip
- One way is to compare $\lambda_1(M,g)$, respectively $\mu_1(\Omega,g)$, with an isoperimetric constant by applying the celebrated Cheeger inequality \cite{Ch1969} 
\begin{equation}
	\lambda_1(M,g) \ge \frac{h_c^2}{4},
\end{equation}
respectively 
\begin{equation}
	\mu_1(\om,g) \ge \frac{h_c^2}{4}.
\end{equation}
Here, $h_c$ is the classical Cheeger constant associated to a compact Riemannian manifold $M$
defined by
\begin{equation} \label{ctcheeger}
h_c(M)=\inf_{\vert A\vert \le \frac{\vert M\vert}{2}}\frac{\vert \partial_IA \vert}{\vert A\vert}
\end{equation}
where the infimum is over subsets $A$ of $M$ with smooth boundary such that $\vert A\vert \le \frac{\vert M\vert}{2}$ (the definition of $h_c$ is word for word the same for $(\Omega,g)$).
Here, for a compact Riemannian manifold $M$, and a domain $A\subset M$, we write $\partial_IA$ for the interior
boundary of $A$, that is the intersection of the boundary of $A$ with
the interior of $M$.

This gives a relation between the spectrum of the Laplacian and the geometry of $(M,g)$ through the Cheeger constant $h_c$. In general, this quantity is difficult to estimate, let alone compute. However the Cheeger constant is a very good geometric measure of the spectral gap $\lambda_1(M,g)$. Indeed Buser~\cite{Bu1982} proved that for a closed Riemannian manifold, with only the additional hypothesis of a lower bound on the Ricci curvature of $(M,g)$, one also gets an upper bound on $\lambda_1(M,g)$ in terms of $h_c$. The Buser inequality states that if $Ric (M,g)\ge -(d-1)a^2$ (where $a\ge 0$), then
\begin{equation} \label{Thm:Buser}
    \lambda_1(M,g)\le 2a(d-1)h_c+10h_c^2.
\end{equation}

Note that a similar upper bound does not exist for $\mu_1(\Omega,g)$; see \cite[Example 1.4]{Bu1982}.
\smallskip

- A second way is to give a lower bound for the lowest non-zero eigenvalue directly in terms of geometric invariants of $(M,g)$.  The foundational work of Li and Yau establishes lower bounds in terms of the Ricci curvature and the diameter both for the eigenvalue $\lambda_1(M,g)$ of any  connected closed Riemannian manifold~\cite[Theorem 7]{LiYa1979} and for the Neumann eigenvalue $\mu_1(\Omega,g)$ in the case of a compact manifold $\Omega$ with boundary~\cite[Theorem 9]{LiYa1979}.  In the latter case, they impose the additional hypothesis that the boundary is convex in the sense that the principal curvatures of $\partial \Omega$ are non-negative.  This was
generalised by Chen ~\cite[Theorem 1.1]{Ch1990}, where the convexity condition is replaced by the hypothesis of an interior $\delta$-rolling 
condition (which means that every point on the boundary is on the boundary
of a ball of radius $\frac{\delta}{2}$ whose interior lies entirely inside $\Omega$ and whose closure meets $\partial \Omega$
only at the given point). The important point is that, in addition to conditions on the geometry inside $\Omega$, we need some control of the geometry of the boundary.

\subsubsection{Lower bounds of the Steklov spectrum via geometric constants.}\label{subsec.lowerboundgeom}
Proposition \ref{prop: no lower bound} shows that one can easily construct Riemannian metrics with small eigenvalues under local deformation. In order to find a lower bound for the Steklov spectrum of a compact manifold $\Omega$ with boundary, it is thus natural to impose a geometric condition on the boundary $\Sigma=\partial \Omega$ comparable to the convexity assumption of \cite[Theorem 9]{LiYa1979}. 
This is precisely the celebrated conjecture proposed by Escobar in \cite{Es1999}.

\begin{conj}[Escobar]\label{conj:escobar}
Let $\Omega$ be a smooth compact connected Riemannian manifold of dimension
$\ge 3$ with boundary $\Sigma=\partial \Omega$. Suppose that the Ricci curvature of $\Omega$ is non-negative
and that the second fundamental form $\rho$ of $\Sigma$ is bounded below by $c>0$.
Then $\sigma_1(\Omega) \ge c$, with equality if and only if $\Omega$ is the Euclidean ball of radius $\frac{1}{c}$.
\end{conj}
Note that Example \ref{example: cylinder}, with $L\to 0$, shows that convexity of the boundary does not imply any lower bound. One really needs the hypothesis $c>0$.

In ~\cite[Theorem 1]{Es1997},  Escobar proved the conjecture in dimension 2, and moreover proved that in higher dimensions, $\sigma_1(\Omega) > \frac{c}{2}$. However, the conjecture itself remains open when the dimension is at least 3

Important progress was made recently by Xia and Xiong in \cite[Theorem 1]{XiXi2019}. The authors show that the conclusion of Escobar's conjecture is true if we impose non-negative sectional curvature $K_g$ instead of Ricci curvature. Specifically, assume that
\begin{equation}
	K_g\ge 0;\ \rho\ge cg_{\Sigma}>0,
\end{equation}
where $\rho$ denotes the second fundamental form of the boundary $\Sigma$ and $g_{\Sigma}$ the restriction of $g$ to $\Sigma$.
Then $\sigma_1(\Omega) \ge c$ with equality if and only if $\Omega$ is isometric to the Euclidean ball of radius $\frac{1}{c}$.

These types of results lead naturally to the question of a possible generalisation when the curvature is not necessarily non-negative, while keeping some restriction on the geometry of the boundary $\Sigma =\partial \Omega$ and on the geometry of $\Omega$ near the boundary. It turns out that we can get partial results by first addressing another natural question: \emph{Is it possible to relate the Steklov spectrum of $\Omega$ to the Laplace spectrum of its boundary $\Sigma= \partial \Omega$?} This question was considered by  Wang and Xia in \cite{WaXi2009} and Karpukhin in \cite{Ka2017}. Important progress is given in the paper by Provenzano and Stubbe ~\cite{PrSt2019}, who consider the problem only for domains in $\R^{d+1}$. The ideas may be generalised to the Riemannian context: this was done by Xiong \cite{Xi2018} and by Colbois, Girouard and  Hassannezhad in \cite{CoGiHa2020} as we describe next.

Let $d\in\mathbb N$ and let $\alpha,\beta,\kappa_-,\kappa_+\in\mathbb R$ and $\delta_0>0$ be such that $\alpha\leq\beta$ and $\kappa_-\leq\kappa_+$. Consider the class $\mathcal{C}=\mathcal{C}(d,\alpha,\beta,\kappa_-,\kappa_+,\delta_0)$ of smooth compact Riemannian manifolds $\Omega$ of dimension $d+1$ with nonempty boundary $\Sigma=\partial \Omega$ satisfying the following hypotheses:
\begin{itemize}
	\item[(H1)] The rolling radius of $\Omega$ satisfies $\delta(\Omega)\geq \delta_0$.
	\item[(H2)] The sectional curvature $K$ satisfies
	$\alpha\leq K\leq \beta$ on the tubular neighbourhood
	\[\Omega_{\delta}=\{x\in \Omega\,:\,d(x,\Sigma)<\delta\}.\]
	\item[(H3)] The principal curvatures of the boundary $\Sigma$ satisfy
	$\kappa_-\leq\kappa_i\leq\kappa_+$. 
\end{itemize}

Let $b$ be the number of connected components of $\Sigma$. The spectrum of the Laplacian on $\Sigma$ is denoted by
	$0=\lambda_0(\Sigma)=\lambda_1(\Sigma)=...=\lambda_{b-1}(\Sigma)<\lambda_{b}(\Sigma)\le ...$
 
\begin{thm}\label{comparison} \cite[Theorem 3]{CoGiHa2020}
There exist explicit constants $D= D(d,\alpha,\beta,\kappa_-,\kappa_+,\delta_0)$ $ B= B(d,\alpha,\kappa_-,\delta_0)$ such that each manifold $\Omega$ in the class $\mathcal{C}$ satisfies the following inequalities for each~$k\in\mathbb N$,
\begin{gather}
\lambda_k\leq \sigma_k^2+D\sigma_k,\label{ineq:main1}\\
\sigma_k\leq B+\sqrt{B^2+\lambda_k}.\label{ineq:thm:mainstek}
\end{gather}
In particular, for each $k\in\mathbb N$,
$|\sigma_k-\sqrt{\lambda_k}|<\max\{D,2B\}$.
\end{thm}

 The hypotheses $H_1,H_2,H_3$ of Theorem \ref{comparison} seem quite strong. In \cite[Examples 37, 38, 39]{CoGiHa2020}, it is explained why they are necessary.

Under the hypotheses of Theorem \ref{comparison} one can take
\[D=\frac{d+1}{\tilde \delta}+\sqrt{|\alpha|+\kappa_-^2}
\qquad\mbox{ and }\qquad
B=\frac{1}{2\tilde \delta}+\frac{d+1}{2}\sqrt{|\alpha|+\kappa_-^2},\]
where $\tilde \delta\leq \delta_0$ is a positive constant depending on $\delta_0$, $\beta$, and $\kappa_+$.

Inequality (\ref{ineq:thm:mainstek}) gives an upper bound for $\sigma_k(\Omega)$ in terms of $\lambda_k(\Sigma)$ and of the geometry of $\Omega$. We will see in the sequel many other ways to obtain upper bounds for the eigenvalues in terms of the geometry of $\Omega$.

Inequality (\ref{ineq:main1}) implies  that for all $k$,
\begin{equation} \label{In: geometry}
	\sigma_k(\Omega)\ge \frac{2\lambda_k(\Sigma)}{\sqrt{4\lambda_k(\Sigma)+D^2}+D}.
\end{equation}

This inequality is interesting only for $k$ larger than or equal to the number of connected components of the boundary $\Sigma=\partial \Omega$. In particular, if the boundary is connected, this gives a lower bound on $\sigma_1(\Omega)$ in terms of $\lambda_1(\Sigma)$ and of the geometry of $\Omega$. Since, as shown in \cite{LiYa1979}, $\lambda_1(\Sigma)$ is bounded explicitly in terms of the geometry of $\Sigma$, one thus obtains a lower bound for $\sigma_1(\om)$ depending only on the geometry of $\om$ and of its boundary $\Sigma$. To our knowledge, this is the only general lower bound in terms of the classical geometric invariants of $\Omega$ and $\Sigma$ like the curvature. However it involves the constant $D$ from Theorem~\ref{comparison} that is not fully explicit. Note also that Example \ref{example: cylinder} shows that, in general, $\sigma_1$ depends on the inner geometry of the manifold $\Omega$, and not only on the geometry near the boundary. This leads to the question

\begin{ques}\label{question lower}
Under the hypotheses (H1)-(H3) of Theorem \ref{comparison} and  the assumption of connected boundary $\Sigma$, is it possible to find an \emph{explicit} lower bound for $\sigma_1(\Omega)$ in terms of geometric invariants of $\Omega$ and of it boundary $\Sigma$?

When the boundary $\Sigma$ is not assumed to be connected, is it also possible to find an explicit lower bound of this type?
\end{ques}
Note that Inequality (\ref{In: geometry}) will be used below in (\ref{partiallower}), in the case in which the boundary is connected.

\subsubsection{Lower bounds for the Steklov spectrum via isoperimetric constants.}

\noindent
\textbf{A Cheeger inequality}. Jammes \cite[Theorem 1]{Ja2015} proves a Cheeger-type inequality. 
Besides the classical Cheeger constant denoted by $h_c(\Omega)$, Jammes introduces another isoperimetric constant, denoted by $h_j(\Omega)$ and defined by
\begin{equation} \label{ctjammes}
h_j(\Omega)=\inf_{\vert A\vert \le \frac{\vert \Omega\vert}{2}}\frac{\vert \partial_IA \vert}{\vert A \cap \partial \Omega\vert}.
\end{equation}
(Recall that $\partial_IA$ is defined following Equation~\eqref{ctcheeger}.)
Then, Jammes shows that
\begin{equation}\label{ineq:stek cheeger}
	\sigma_1(\Omega) \ge \frac{1}{4}h_c(\Omega)h_j(\Omega).
\end{equation}
This estimate is optimal in the sense that the two isoperimetric constants $h_c(\Omega)$ and $h_j(\Omega)$ both have to appear. Intuitively, the classical Cheeger constant $h_c(\Omega)$ allows one to measure how large a hypersurface needs to be in order to disconnect $\om$ into two substantial pieces (consider the celebrated example of the Cheeger dumbbell, see for example \cite[Sections 2, 3]{Co2017}). As shown in Inequality (\ref{Thm:Buser}), having small $h_c$ has a strong influence on the spectrum of the Laplacian. The new constant $h_j(\Omega)$ will be small if there are two parts of the boundary $\Sigma$ of $\om$ that are close together in $\om$ but far apart with respect to intrinsic distance in $\Sigma$. As in the proof of Part 1 of Proposition \ref{prop: no lower bound} this tends to create small eigenvalues for the Steklov spectrum.

In \cite{Ja2015}, Jammes also discusses the optimality of Inequality~\eqref{ineq:stek cheeger}, and we present two examples that are related to Example \ref{example: cylinder}.

First in \cite[Example 4]{Ja2015}, to see the necessity of the presence of $h_j$, Jammes considers the family of cylinders $\Omega_n= M\times[0,\frac{1}{n}]$ where $M$ is a closed Riemannian manifold. Example \ref{example: cylinder} shows that $\sigma_1(\Omega_n)\sim \frac{1}{n}$ as $n\to \infty$. It is also easy to see that $h_j(\Omega_n) \to 0$ as $n\to \infty$. However we will see that $h_c(\Omega_n)$ is uniformly bounded from below, which shows the necessity of the presence of $h_j$. To bound $h_c(\Omega_n)$ from below, if $A \subset \Omega_n$ is a domain, we have to bound the ratio $\frac{\vert \partial_IA\vert}{\vert A \vert}$ from below. To this aim, the author glues $2n$ copies of $\Omega_n$ along their boundaries in order to get a closed manifold $M \times S^1$, and associates to $A$ a domain $A'\subset M \times S^1$ obtained by reflecting $A$ along the boundary. We have $\vert A'\vert=2n \vert A\vert$ and $\vert \partial A'\vert =2n \vert \partial_IA\vert$. This implies
$\frac{\vert \partial_IA\vert }{\vert A \vert}= \frac{\vert \partial A'\vert}{\vert A'\vert}$, but $A'$ is a domain in a fixed manifold
$M \times S^1$, so 
 $\frac{\vert \partial A'\vert}{\vert A'\vert}\ge h_c(M \times S^1)$, which gives a lower bound on $h_c(\Omega_n)$.

 Next \cite[Example 5]{Ja2015} shows the necessity of the presence of $h_c$. The author considers the product $M \times [0,L]$ as in Example \ref{example: cylinder}, but with $L\to \infty$. For $L$ large enough, we have $\sigma_1(\Omega)=\frac{2}{L}\to 0$ and $h_c\to 0$ as $L\to \infty$. However one can see by projection on the boundary that for $A\subset \Omega$, the ratio
 $\frac{\vert \partial_IA \vert}{\vert A \cap \partial \Omega\vert}$ is bounded from below by $1$, so that $h_j \ge 1$.

However, although very interesting, Inequality~\eqref{ineq:stek cheeger} (\cite[Theorem 1]{Ja2015}) is not optimal in the sense that there does not exist a Buser-type inequality (i.e., an analogue of Inequality~\eqref{Thm:Buser} for this constant $h_c(\Omega)h_j(\Omega)$. As the next examples show, it is easy to deform a Riemannian manifold $(\Omega,g)$ to make the Cheeger constant arbitrarily small without affecting the first eigenvalue $\sigma_1$ too much. In Example \ref{cylindricalend}, the Cheeger-Jammes constant $h_j$ stays bounded and in Example \ref{Ex: hyperbolic}, it becomes also small.
\begin{ex}
Example \ref{cylindricalend} above shows that the Steklov eigenvalues of a Riemannian manifold $(\Omega,g)$ with connected boundary admitting a neighourhood that is isometric to $\Sigma \times [0,L]$ are bounded in terms of $L$ and the eigenvalues of the Laplacian on $\Sigma$. 

Let us consider a one parameter family of metrics $h_{\epsilon}^2g$ conformal to $g$ with $h_{\epsilon}=1$ on $\Sigma \times [0,L/2]$ and $h_{\epsilon}=\frac{1}{\epsilon}$ outside $\Sigma \times [0,L]$. Because the metric is fixed on $\Sigma \times [0,L/2]$, the Steklov spectrum is not affected too much by the conformal factor.
It is even possible to keep the curvature uniformly bounded by a careful smoothing.

On the other hand, the Cheeger constant $h_c$ of $(\Omega,h_{\epsilon}^2g)$ becomes small as $\epsilon \to 0$: to see this, it suffices to consider a domain $A$ outside $\Sigma \times [0,L]$. The behaviour of $\frac{\vert \partial_IA \vert_{h_{\epsilon}^2g}}{\vert A\vert_{h_{\epsilon}^2g}}$ is proportional to $\epsilon$ and $h_c(\Omega,h_{\epsilon}^2g) \to 0$ as $\epsilon \to 0$.

The Cheeger-Jammes constant $h_j(\Omega,h_{\epsilon}^2g)$ is bounded from above: it suffices to consider domains $A \subset \Sigma \times [0,L/2]$ where the metric is constant.
\end{ex}
\begin{ex} \label{Ex: hyperbolic}
For convenience in this example, we will refer to hyperbolic surfaces (curvature $-1$) of genus one with one boundary component as hyperbolic tori. We construct a family $\Omega_{\epsilon},\ 0<\epsilon<\frac{1}{2}$, of hyperbolic tori with the following properties:  The boundary $\Sigma_{\epsilon}$ of $\Omega_{\epsilon}$ has length $\ge 1$; the first nonzero eigenvalue $\sigma_1(\Omega_{\epsilon})$ of $\Omega_{\epsilon}$ is uniformly bounded from below as $\epsilon \to 0$; however, the  Cheeger constants satisfy $h_c(\Omega_{\epsilon})\to 0$ and $h_j(\Omega_{\epsilon})\to 0$ as $\epsilon \to 0$. This shows that, even when the curvature is bounded, small $h_c(\Omega)$ and $h_j(\Omega)$ do not imply small $\sigma_1$. That is, there is no Cheeger-Buser type upper bound for $\sigma_1$.
\begin{figure}[h]
  \centering
  \includegraphics[width=10cm]{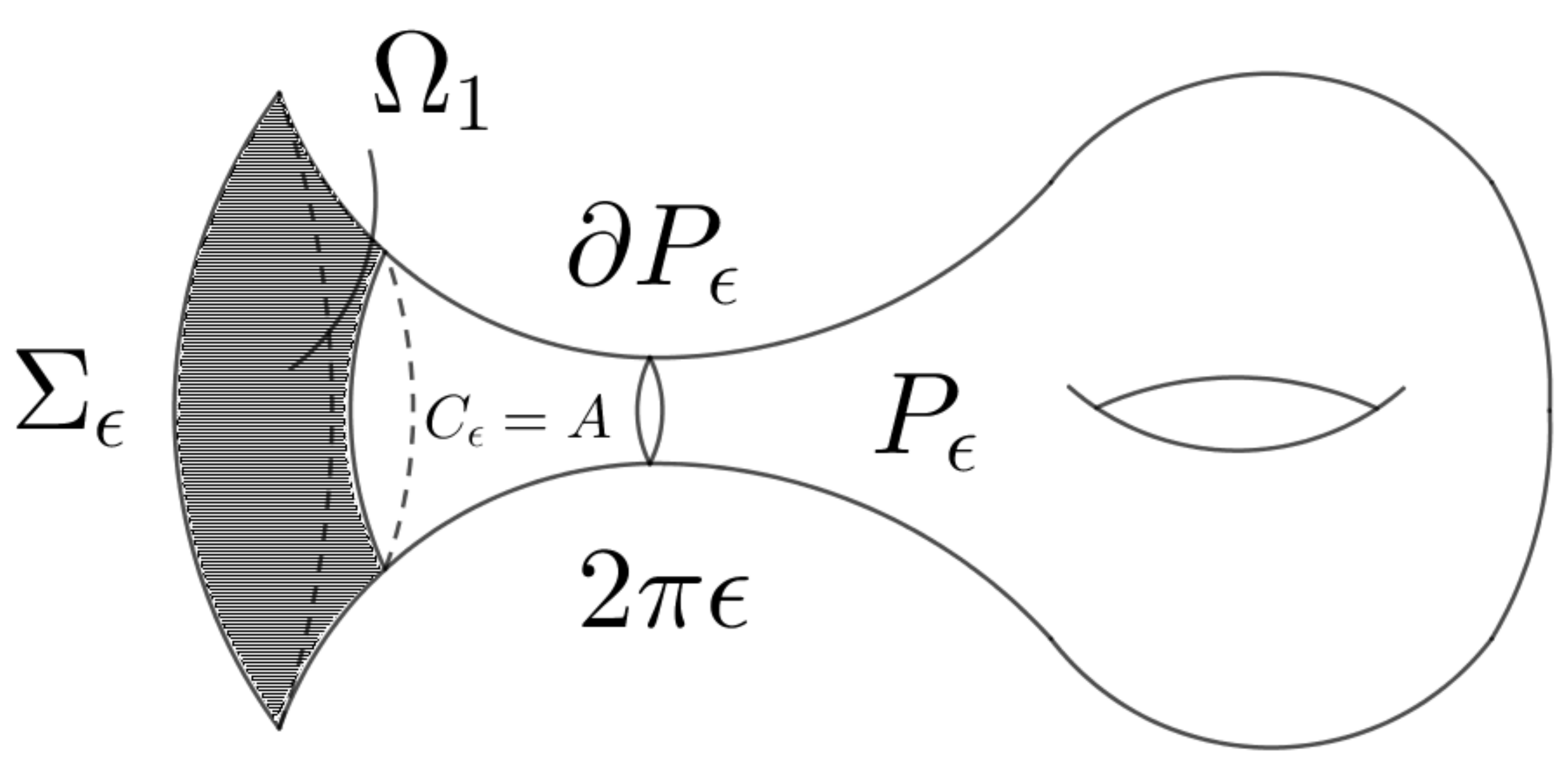}
  \caption{There is no Buser-type inequality for $\sigma_1$.}
  \label{fig:nobuser}
  \end{figure}
The surface $\Omega_{\epsilon}$ is built by gluing together a hyperbolic cylinder $C_{\epsilon}$ and a hyperbolic torus $P_{\epsilon}$ along a common geodesic boundary component of length $2\pi \epsilon$, as explained, for example, in \cite[Chapter 3]{Bu1992} (see Figure~\ref{fig:nobuser}).

The construction is given as follows:  For $\epsilon>0$, let $P_{\epsilon}$ be a hyperbolic torus with geodesic boundary of length $2\pi\epsilon$.  Let  $L_\epsilon=1+\arcosh \frac{1}{\epsilon}$. Endow the cylinder $C_{\epsilon}:=[0,L_{\epsilon}]\times S^1$ with the hyperbolic metric given in Fermi coordinates by $dr^2+\epsilon^2 \cosh^2 r$, $0\le r \le L_\epsilon, 0\le \theta \le 2\pi$. The boundary component corresponding to $r=0$ has length $2\pi \epsilon$.  Glue this boundary component to the boundary of $P_{\epsilon}$ to obtain $\Omega_{\epsilon}$. The other boundary component $\Sigma_{\epsilon}$ of $C_{\epsilon}$, corresponding to $r=1+\arcosh \frac{1}{\epsilon}$, is of length $\vert \Sigma_{\epsilon} \vert >1$, and becomes the boundary of $\Omega_{\epsilon}$, which is connected.

One can show that the neighborhood $\Sigma_1$ of $\Sigma_{\epsilon}$ given by $\{r:\arcosh \frac{1}{\epsilon}\le r \le 1+ \arcosh \frac{1}{\epsilon}\}$ is uniformly quasi-isometric to $[0,1] \times S^1$ as $\eps\to0$. Because $\Omega_{\epsilon}$ has one boundary component, $\sigma_1(\Omega_{\epsilon})$  is uniformly bounded from below and above using Dirichlet-to-Neumann bracketing (Proposition \ref{prop: dir neum bracket} and Example \ref{cylindricalend}). 

However, consider the domain $A=C_{\epsilon} \subset \Omega_{\epsilon}$. One can see that $|\partial_IA|=2\pi \epsilon$, and $\vert A\cap \partial \Omega_{\epsilon}\vert=\vert \Sigma_{\epsilon}\vert \ge 1, \vert A\vert \ge 1$, so that both $h_c$ and $h_j$ tend to $0$ as $\epsilon \to 0$.

\end{ex}

In Remark 2 of \cite{Ja2015}, the author observes that one can slightly change the definition of the constant $h_j$ of (\ref{ctjammes}) and consider
\begin{equation} \label{ctescobar}
h_e(\Omega)=\inf_{\vert A \cap \partial \Omega\vert \le \frac{\vert \partial \Omega\vert}{2}}\frac{\vert \partial_IA\vert}{\vert A \cap \partial \Omega\vert}.
\end{equation}
This corresponds to the Cheeger--Escobar constant defined in \cite{Es1999} and Jammes observes that Inequality \ref{ineq:stek cheeger} is true with the same proof.

In case $\Sigma=\partial \Omega$ is connected, Theorem   \ref{comparison} allows one to get another Cheeger-type inequality. If $h_c(\Sigma)$ denotes the classical Cheeger constant of $\Sigma$, we have $\lambda_1(\Sigma) \ge \frac{h^2_c(\Sigma)}{4}$, and estimate (\ref{ineq:main1}) leads to 
 \begin{equation}\label{partiallower}
	 \sigma_1(\Omega)\ge \frac{h^2_c(\Sigma)}{2\sqrt{h^2_c(\Sigma)+D}}.
 \end{equation}
However, the geometry of $\Omega$ appears strongly through the term $D$ of Theorem \ref{comparison}. Regarding upper bounds, it was mentioned after Theorem \ref{comparison}
that Inequality \eqref{ineq:thm:mainstek} gives an upper bound for $\sigma_k(\Omega)$ in terms of $\lambda_k(\Sigma)$ and of the geometry of $\Omega$. Combining this with Inequality 
 (\ref{Thm:Buser}), this shows that $\sigma_1(\Omega)$ is bounded from above by a term involving the Cheeger constant and the Ricci curvature of $\Sigma$, along with the geometry of $\Omega$ through the term $B$ that appears in~\eqref{ineq:thm:mainstek}. But this upper bound depends on a lot of geometric invariants, leading to the following question:

\begin{ques}\label{question cheeger}
Can one define a different Cheeger-type isoperimetric constant $h'$ for which $\sigma_1$ satisfies a Buser-type inequality as in (\ref{Thm:Buser})? I.e., is there is an upper bound for $\sigma_1$ in terms of this new isoperimetric constant $h'$ and of a lower bound on the Ricci curvature of the manifold $\Omega$.
\end{ques}

\noindent
\textbf{Higher order Cheeger inequalities}. In \cite{HaMi2020}, Hassannezhad and Miclo proved a lower bound for the $k$th Steklov eigenvalue in terms of what they call a $k$th Cheeger-Steklov constant in three different situations. In the context of Riemannian manifolds, this inequality extends the Cheeger inequality by Jammes and we will describe it below. In section \ref{cheegerdiscrete}, we will briefly describe another aspect of this work regarding the Steklov problem on graphs. The third aspect, concerning measurable state spaces, is beyond the scope of this paper and will not be described.

Intuitively, the idea is the following: for the Cheeger-Steklov inequality on a manifold $\Omega$, we consider a family of domains $A$ and $\Omega \setminus A$, and define the isoperimetric bounds $h_c$ and $h_j$ by minimising some isoperimetric ratio on the family of domains $A$. The authors use the same strategy, but in order to estimate the eigenvalue $\sigma_k$, they consider a decomposition of $\Omega$ into $k+1$ disjoint domains, and they also define two isoperimetric constants by minimisation.

More precisely, let $\Omega$ be a compact Riemannian manifold with boundary $\partial \Omega=\Sigma$ and consider the family $\mathcal A$ of non-empty open domains of $\Omega$ with piecewise smooth boundary.
The authors introduce for each $k \ge 1$ the set $\mathcal A_k$ of all $k+1$-tuples $(A_1,...,A_{k+1})$ of mutually disjoint elements of $\mathcal A$. The k-th order Cheeger constant will be defined  by minimisation on $\mathcal A_k$.

First, for each domain  
$B\in \mathcal A$, the authors introduce the isoperimetric constants $\eta(B)$ and $\eta'(B)$ defined by
\[
\eta(B)= \frac{\vert \partial_I B\vert}{\vert B\vert};\ \ \eta'(B)=
\frac{\vert \partial_I B\vert}{\vert B \cap \partial \Omega\vert}.
\]

Consider now one of the k+1-tuples $(A_1,...,A_{k+1})\in\mathcal{A}_k$. For $A\in\{A_1,...,A_{k+1}\}$ the authors define
\[
\rho(A)= \inf \{\eta(B):B\in \mathcal A;\ B\subset A;\ B\cap \partial_I A=\emptyset\}
\]
and 
\[
\rho'(A)= \inf \{\eta'(B):B\in \mathcal A;\ B\subset A;\ B\cap \partial_I A=\emptyset\}.
\]

Note that the condition $B\cap \partial_IA=\emptyset$
means that $\rho(A)$ is exactly the usual Cheeger constant of $A$ associated to the Dirichlet Laplacian for functions equal to $0$ on $\partial_I A$ (see \cite{Bu1980}, (1.5) p. 30, \cite{Ch1984}, Theorem 3, p. 95).

  The Cheeger-Steklov constant of order $k$ is defined as
\[
i_k(\Omega)=\min_{(A_1,...,A_{k+1})\in \mathcal A_k}\max_{l\in \{1,...,k+1\}}\rho(A_l)\rho'(A_l).
\]
Hassannezhad and Miclo prove a Cheeger-Steklov inequality of order $k$  \cite[Theorem C]{HaMi2020}, which states that there exists a positive universal constant $C$ such that for each $k\ge 1$ one has 
\begin{equation} \label{highercheegermanifold}
    \sigma_k(\Omega)\ge \frac{C}{k^6}i_k(\Omega).
\end{equation}

The proof of the inequality is quite complicated. The authors do not work directly on the Steklov problem, but they use the fact that a similar inequality exists for the Laplacian with Neumann boundary condition (see \cite[Theorem 26]{HaMi2020} referring to a previous result by Miclo). The authors use the fact that one can approximate the Steklov eigenvalues by the eigenvalues of the Laplacian with density with Neumann boundary condition, where the density accumulates near the boundary \cite[Theorem 24]{HaMi2020}. The authors also give examples in the spirit of the already mentioned examples by Jammes to see the necessity of the presence of the two constants $\rho$ and $\rho'$ in the definition of $i_k(\Omega)$.

A natural question is whether the term $\frac{1}{k^6}$ in the inequality (\ref{highercheegermanifold}) is optimal. The authors also show the following estimate \cite[Proposition 1]{HaMi2020}:
\begin{equation} \label{better}
\sigma_{2k+1}(\Omega)\ge \frac{C'}{\log^2(k+2)} i_{k+1}(\Omega).
\end{equation}
This means that the dependence on $k$ is only of order $\frac{1}{\log^2(k+2)}$ if the eigenvalue $\sigma_{2k+1}$ is estimated in terms of the Cheeger constant of order $k+1$. In \cite[ Remark 1]{HaMi2020}, the authors ask the following question. 

\begin{ques}\label{ques:logbase2}
Is the coefficient $\frac{C'}{\log^2(k+2)}$ of $i_{k+1}(\Omega)$ in Estimate (\ref{better}) sharp?
\end{ques}

In comparable situations (combinatorial Laplacian by Lee, Oveis Gharan and Trevisan~\cite{LGOT2014} and Markov operator by Miclo~\cite[pp. 336-337]{Mi2015}) one can show it is sharp.

More generally, this leads to the following question.

\begin{ques}\label{ques:higherordercheeger}
Is it possible to have a higher order Cheeger inequality for which the coefficient of $i_k(\om)$ does not depend on $k$?
\end{ques}

\noindent
\textbf{Other lower bounds}. For other lower bounds in some specific situations, see also \cite[Theorem 2.1]{Ve2018} by Verma and \cite[Theorem 1.3]{HaSi2020} by Hassannezhad and Siffert.

\subsection{Upper bounds for eigenvalues: basic results and examples} \label{Upperexamples}

This subject was introduced in Section 4 of the survey \cite{GiPo2017},  and we invite the reader to look there for the state of the art until 2014. 

The question of finding upper bounds for eigenvalues in terms of  geometric invariants is closely related to the question of constructing large eigenvalues under geometric restrictions and we will often present both stories in parallel. These questions depend also on the choice of normalisation. For a Riemannian manifold $\Omega$ of dimension $d+1$ with boundary $\Sigma =\partial \Omega$, the most common normalisation is with respect to the volume of the boundary: we consider the normalised eigenvalues $\sigma_k(\Omega)\vert \partial \Omega \vert^{1/d}$. Another normalisation, used in particular when $\Omega$ is a domain in a complete Riemannian manifold $M$, is with respect to the volume of $\Omega$; we consider the normalised eigenvalues $\sigma_k(\Omega)\vert \Omega\vert^{1/(d+1)}$. 

Recently, in \cite{KaMe2021}, Karpukhin and M\'etras proposed a normalisation involving both the volumes of $\Omega$ and of $\Sigma$, with the normalised eigenvalues given by $\sigma_k(\Omega) \vert \Sigma\vert\vert\Omega\vert^{\frac{1-d}{d+1}}$. See Example~\ref{example:KarpukhinMetrasNorm} for discussion of why this specific normalisation is natural. One of the interests in this last normalisation is that it was shown in the paper~\cite{CoElGi2011} by Colbois, El Soufi and Girouard and in~\cite{Ha2011} by Hassannezhad that these normalised eigenvalues are bounded from above within the conformal class of any Riemannian metric (see inequality (\ref{secondformulation}) in Theorem \ref{asma2011}), which allows one to investigate maximal Riemannian metrics in this context. Specifically, let $I_g(\Omega)$ denote the isoperimetric ratio: if $(\Omega,g)$ is a compact Riemannian manifold of dimension $d+1$, with boundary, 
	\[
	I_g(\Omega)=\frac{\vert \partial \Omega \vert_g}{\vert \Omega\vert_g^{\frac{d}{d+1}}}.
	\]
Then we have the following \cite[Theorem 4.1]{Ha2011}:
\begin{thm}\label{asma2011}
    Let $(M,g_0)$ be a complete Riemannian manifold of dimension $d+1$ with $Ric_{g_0}(M) \ge -a^2d$ for a constant $a\ge 0$. Let $\Omega \subset M$ be a relatively compact domain with $C^1$ boundary and $g$ be any metric conformal to $g_0$. Then we have
    \begin{equation}
	   \sigma_k(\Omega,g) \vert \partial \Omega\vert_g^{\frac{1}{d}}\le \frac{A_d \vert \Omega\vert_{g_0}^{\frac{2}{d+1}} a^2+B_dk^{\frac{2}{d+1}}}{I_g(\Omega)^{\frac{d-1}{d}}}
    \end{equation}
    where $A_d$ and $B_d$ are constants depending only on $d$. An equivalent formulation is
    \begin{equation} \label{secondformulation}
	   \sigma_k(\Omega,g)\vert \partial \Omega \vert_g\vert \Omega\vert_g^{\frac{1-d}{d+1}} \le  A_d \vert \Omega\vert_{g_0}^{\frac{2}{d+1}}a^2+B_dk^{\frac{2}{d+1}}.
    \end{equation}
\end{thm}
In the situation where $M$ is a closed Riemannian manifold, inequality~\eqref{secondformulation} provides a uniform upper bound since $|\Omega|_{g_0}<|M|_{g_0}$. See section~\ref{UpperDomains} for further discussion.

\begin{remark} \label{rem: domain} Theorem \ref{asma2011} is established for domains in complete manifolds and not directly for manifolds with boundary. The reason for this is that constructions as in \cite{GrNeYa2004} or \cite{CoMa2008} used in the proof of the theorem are not established for manifolds with boundary without additional conditions on the geometry of the boundary.

A compact Riemannian manifold with smooth boundary $(\Omega,g_0)$ with $Ricci_g(\Omega)\ge -a^2d$ can always be seen as a domain of a complete manifold $M$ (without boundary) by extension of the manifold $(\Omega,g_0)$. However, in general, we cannot keep the same lower bound $-a^2d$ for the Ricci curvature of the complete manifold $M$. Despite this, Inequality (\ref{secondformulation}) shows that $\sigma_k(\Omega,g)\vert \partial \Omega \vert_g\vert \Omega\vert_g^{\frac{1-d}{d+1}}$ is bounded from above for $g\in [g_0]$, but the control of the bound is not explicit. The proof of Theorem \ref{asma2011} only requires that the curvature of the ambient manifold $M$ satisfy the specified lower curvature bound on a neighbourhood of $(\Omega,g_0) \subset M$ of radius at most the diameter of $(\Omega,g_0)$ and on such a neighbourhood, the Ricci curvature is bounded from below, however in general with a different bound than $-a^2d$.

 This question of controlling the curvature when extending a manifold with boundary to a complete manifold is discussed in detail by Pigola and Veronelli in \cite{PiVe2020}.
\end{remark}
 	
\begin{remark}\label{rem:prescription}
Let $(\Omega,g_0)$ be a compact Riemannian manifold with boundary. It follows from Theorem~\ref{asma2011} and from Remark \ref{rem: domain} that one cannot use a conformal perturbation $g\in[g_0]$ to prescribe at the same time the kth eigenvalue ($k>0$) and the volumes of $(\Omega,g)$ and of $(\partial \Omega,g)$. 
\end{remark}

Let us now give a series of examples to provide intuition. 
They will allow us to show that certain upper bounds presented below are sharp in the sense that all ingredients entering into the inequalities are necessary.

When we get upper bounds for $\sigma_k$ (in particular for $k=1$), the question arises of finding a maximising metric. These examples allow us to show that, sometimes, the expected metric (for example the ball in the case of domains) is not a maximiser.

\begin{ex} \label{confdef}A general question is to understand and distinguish the spectral effects of geometric perturbations of a manifold near its boundary from the spectral effects of perturbations occurring deep inside the manifold. In dimension $d+1\ge 3 $, one can construct examples with fixed boundary and arbitrarily large $\sigma_1$ staying within the conformal class \cite[ Theorem 1.1 (ii)]{CoElGi2019}.
Let $(\Omega,g)$ be a compact, connected Riemannian manifold of dimension $d+1\geq 3$ with smooth boundary $\Sigma$.
 There exists a one-parameter family of 
 Riemannian metrics $g_\varepsilon$  conformal to $g$  that coincide with $g$ on $\Sigma$ such that 
 \[\sigma_{1}(\Omega,g_\varepsilon) \to\infty \quad as\quad \varepsilon\to 0.\]

In this example, $\sigma_1(\Omega) \vert \Sigma \vert^{1/d}$ and $\sigma_1(\Omega) \vert \Omega \vert^{1/(d+1)}$ tend to $\infty$ as $\epsilon \to 0$, but this is not the case for $\sigma_1(\Omega) \vert \Sigma\vert\vert\Omega\vert^{\frac{1-d}{d+1}}$ because the conformal class is fixed (see Inequality (\ref{secondformulation}) and Remark~\ref{rem: domain} above).
\end{ex}

Example~\ref{confdef} leads to the question of whether there are similar examples for which $\sigma_1(\Omega) \vert \Sigma\vert\vert\Omega\vert^{\frac{1-d}{d+1}} \to \infty$ still keeping the boundary fixed (but not the conformal class). In full generality, this is unknown. However, it is possible to obtain examples in specific situations.

\begin{ex}\label{CianciGirouard}
 A family of examples are constructed by Cianci and Girouard ~\cite[Theorem 1.1]{CiGi2018}. The authors consider a compact manifold $\Omega$ of dimension $d+1\ge 4$ with connected boundary $\Sigma$ having a very particular property: there exists a Riemannian metric $g_{\Sigma}$ on $\Sigma$ which admits a unit Killing vector field $\xi$ with dual $1$-form $\eta$ whose exterior derivative is nowhere $0$. The odd-dimensional spheres have this property. But not every Riemannian manifold supports a Killing vector field: a vector field $X$ on a Riemannian manifold is a Killing field if and only if the $1$-parameter group generated by $X$ consists of local isometries.

Under this condition, the authors show that there exists a family $(g_{\epsilon})_{\epsilon>0}$ of Riemannian metrics on $\Omega$ which coincide with $g_{\Sigma}$ on $\Sigma$ such that $\vert (\Omega, g_{\epsilon})\vert =1$ and $\sigma_1(\Omega, g_{\epsilon}) \to \infty$ as $\epsilon \to 0$. As $\Sigma$ is fixed and $\vert (\Omega, g_{\epsilon})\vert =1$, the normalised eigenvalue $\sigma_1(\Omega,g_{\epsilon}) \vert \Sigma\vert\vert(\Omega, g_{\epsilon})\vert^{\frac{1-d}{d+1}} \to \infty$ as $\epsilon \to 0$.

\end{ex}

Let us observe that the geometry of the example is very special: as the boundary $\Sigma$ is fixed, its intrinsic diameter is fixed. However, its extrinsic diameter, that is its diameter in $\Omega$, goes to $0$ as $\epsilon \to 0$.

It would be very surprising if the existence of a Killing field was an important condition, but it is not clear how to avoid this hypothesis.
\begin{ques}\label{ques:intrinsicextrinsic}
Let $\Omega$ be a compact Riemannian manifold of dimension $d+1\ge 3$ with boundary $\Sigma$. Is it possible to construct a family $(g_{\epsilon})_{\epsilon>0}$ of Riemannian metrics on $\Omega$ that stays constant on $\Sigma$ and satisfies $\sigma_1(\Omega,g_{\epsilon}) \vert \Sigma\vert\vert(\Omega, g_{\epsilon})\vert^{\frac{1-d}{d+1}} \to \infty$ as $\epsilon \to 0$?
\end{ques}

 Note that fixing the Riemannian metric on the boundary is a strong constraint. If we relax it by requiring that the volume $\vert\Sigma\vert$ is prescribed instead, and if the dimension $d+1$ is $\ge 4$, then Example \ref{cylindricalend} allows us to construct such examples, at least when the boundary is connected. It suffices to introduce a cylindrical metric $\Sigma_{\epsilon} \times [0,1]$ near the boundary, with $\vert \Sigma_{\epsilon}\vert=1$, $\lambda_1(\Sigma_{\epsilon})\to \infty$ (which is possible by the work of Colbois and Dodziuk~\cite{CoDo1994} because the dimension of $\Sigma_{\epsilon}$ is $\ge 3$). This method may not be used in dimension $d+1=3$, but the following example shows another, more elaborate, approach.

\begin{ex}\label{example:Karpukhin}The authors would like to thank Mikhail Karpukhin for suggesting the following construction, which leads to Riemannian metrics $g$ on the ball $\B^3$ such that both $|\partial\B^3|_g=1$ and $|\B^3|_g=1$ are prescribed, while $\sigma_1(\B^3,g)$ is arbitrarily large.

Let $M=\Sp^3$ be equipped with any Riemannian metric $g$ of volume one. The homogenisation construction of Girouard and Lagac\'e~\cite{GiLa2021} provides a family of domains $\Omega^\eps\subset M$ such that $|\partial\Omega^\eps|\xrightarrow{\eps\to 0}|M|_g=1$ and
$\sigma_1(\Omega_\eps)\xrightarrow{\eps\to 0}\lambda_1(M,g)$.
See Theorem~\ref{thm:homoclosed}.
 The domains $\Omega^\eps$ are obtained by removing a finite number of disjoint small balls $B(p_i,r_\eps)$ from $M$. One then removes a finite number of thin tubes connecting the boundaries of these balls to each other sequentially, leading to a family of connected domains $\Omega_\delta^\eps\subset\Omega^\eps\subset M$, where $\delta$ represents the width of the excised tubes. As $\delta\to 0$, the tubes collapse to curves meeting the boundary of the balls $B(p_i,r_\eps)$ perpendicularly. It follows from the work of Fraser and Schoen~\cite[Theorem 1.2]{FrSc2019} that $\sigma_1(\Omega_\delta^\eps)\xrightarrow{\delta\to 0}\sigma_1(\Omega^\eps)$. 
 
 The connecting tubes being chosen sequentially means that the corresponding graph, with vertices corresponding to the balls $B(p_i,r_\eps)$  and with edges corresponding to the connecting tubes, is a linear graph. It follows that the excised region is itself diffeomorphic to a ball, and because $M=\Sp^3$, its complement $\Omega_\delta^\eps$ is also diffeomorphic to the ball $\B^3$. Using these diffeomorphisms to pull back the metric $g$ to the ball $\B^3$ and using a simple diagonal argument leads to a sequence of Riemannian metrics $g_m$ on $\overline{\B}^3$ such that $|\B^3|_{g_m}\xrightarrow{m\to+\infty}|M|_g=1$, $|\partial\B^3|_{g_m}\xrightarrow{m\to+\infty}|M|_g=1$ and 
 $\sigma_1(\B^3,g_m)\xrightarrow{m\to+\infty}\lambda_1(M,g)$. 
Because the metric $g$ is arbitrary, the conclusion follows from the well-known fact that $\lambda_1$ is not bounded above on $\Sp^3$. See for instance the work of Bleecker~\cite{Bl1983}. 
\end{ex}

Let us come back to the question of bounding Steklov eigenvalues for domains in a closed Riemannian manifold $M$ of dimension $d+1$. Theorem~\ref{asma2011} provides a good upper bound for the Karpukhin-M\'etras normalisation. 
\begin{ex}\label{example:KarpukhinMetrasNorm}   In \cite{GiLa2021}, the authors consider a continuous density $\beta>0$ on $M$ and use Theorem~\ref{thm:homoclosed} to construct a family $\Omega_{\epsilon}\subset M$ of domains with 
\[\sigma_1(\Omega_\eps)\to\lambda_1(\beta^{-1}\Delta_g),\quad\vert\partial\Omega_\eps\vert\to\int_M\beta\,dV_g,\quad\vert\Omega^\eps|\to |M|.\] 
In particular for $\beta>0$ constant, this leads to
\begin{gather*}
\sigma_1(\Omega_{\epsilon})\vert\partial\Omega_{\epsilon}\vert^{1/d}\to\beta^{-1+1/d}\lambda_1(M)|M|^{1/d};\\
\sigma_1(\Omega_{\epsilon})\vert \Omega_{\epsilon}\vert^{1/(d+1)} \to \beta^{-1}\lambda_1(M)|M|^{1/(d+1)};\\
\sigma_1(\Omega_{\epsilon})\vert\partial\Omega_{\epsilon}\vert\vert\Omega_\epsilon\vert^{\frac{1-d}{1+d}}\to\lambda_1(M)|M|^{2/(d+1)}.
\end{gather*}
Using \cite[Theorem 1.2]{FrSc2019}, the domains $\Omega_\epsilon$ can be chosen to be connected. See Example~\ref{exfraserschoen} and Example~\ref{example:Karpukhin} for other applications. This is another indication of the importance of the mixed normalisation proposed by Karpukhin and M\'etras \cite{KaMe2021}. In fact, it is easy to show using the same method that the functional $\sigma_1(\Omega)|\partial\Omega|^a|\Omega|^b$ (where $ad+b(d+1)=1$ to obtain scaling invariance) is not bounded above for any other normalisation.
Indeed for $\beta>0$ constant it follows as above that
\[\sigma_1(\Omega)|\partial\Omega|^a|\Omega|^b\to\beta^{a-1}\lambda_1(M)|M|^b,\]
so that $a=1$ is the only possibility keeping the functional bounded.
\end{ex}
\begin{ex}\label{JadeBrisson}
The previous example is based on homogenization theory.
A simpler approach to obtaining large eigenvalues for domains in compact manifolds was developed by Brisson in \cite{Br2022}, where she considers a closed connected submanifold $N\subset M$ of positive codimension. Let $T_{\epsilon}$ be the tubular neighborhood of $N$ defined by $T_{\epsilon}=\{x\in M:dist(x,N)<\epsilon\}$ where $dist$ denotes the Riemannian distance. For $\epsilon$ small enough the boundary $\Sigma_{\epsilon}:=\partial T_{\epsilon} $ is a smooth submanifold. The author considers the domain $\Omega_{\epsilon}=M\setminus T_{\epsilon}$ whose boundary is $\Sigma_{\epsilon}$. She investigates the properties of $\sigma_k(\Omega_{\epsilon})$, in particular the asymptotics of $\epsilon\sigma_k(\Omega_{\epsilon})$. Among other things, the following result is proved \cite[Corollary 1.6]{Br2022}: 
If the dimension $d+1$ of $M$ is $\ge 3$ and the dimension $n$ of $N$ satisfies $0<n\le d-1$, then 
\begin{equation}
\lim_{\epsilon \to 0} \vert \Sigma_{\epsilon}\vert^{1/d}\sigma_{1}(\Omega_{\epsilon})=\infty.
\end{equation}
\end{ex}

 \subsubsection{Surgery methods}
\begin{ex}\label{exfraserschoen}
In their papers ~\cite{FrSc2019,FrSc2020}, among other things, Fraser and Schoen propose many very enlightening constructions having consequences for the optimisation problem. Roughly speaking, they give a way to do surgery on a compact Riemannian manifold without affecting the Steklov spectrum too much. We give below some typical results and a couple of applications, and invite the interested readers to consult the two papers.

\begin{enumerate}
\item
The paper \cite{FrSc2019} mainly concerns the first nonzero eigenvalue in dimension higher than $2$. It shows in particular that it is not possible to generalise the Weinstock inequality to domains in Euclidean space of dimension higher than 2, even with strong hypotheses on the topology of the domains.
\begin{thm} \label{contractible} \cite[Theorem 1.1 ]{FrSc2019}. Let $\B$ denote the unit ball in the Euclidean space $\R^{d+1}$ of dimension $d+1\ge 3$. Then there exists a smooth contractible domain $\Omega$ in $\R^{d+1}$ with $\vert \partial \Omega\vert =\vert \partial  \B \vert$ and $\sigma_1(\Omega)>\sigma_1( \B)$.
\end{thm}
The proof of this result is hard, but the idea is simple and beautiful. First, in \cite[Section 3]{FrSc2019}, the authors study the spectrum of the annulus $\B_1\setminus \B_{\epsilon}$ where $\B_{\rho}$ is the ball of radius $\rho$ in $\R^{d+1}$, $d\ge 2$. They show that for all $k$ and $\epsilon$ small enough 
\[
\sigma_k(\B_1)\vert \partial \B_1\vert^{1/d} < \sigma_k(\B_1\setminus \B_{\epsilon})\vert \partial (\B_1\setminus \B_{\epsilon})\vert^{1/d}.
\] 
In particular, for $k=1$, there exists an annulus with normalised first nonzero eigenvalue greater than the first normalised eigenvalue of the ball.  But an annulus is not contractible. In order to obtain a contractible example, the authors perform a surgery: they remove a thin cylinder connecting the inner sphere to the outer sphere and consider its complement which is now contractible. The main difficulty, addressed in \cite[Section 4]{FrSc2019}, is to show that this operation does not affect the spectrum too much.

Some of the results were generalised in \cite[Section 3]{Ho2021} by Hong, who performed higher-dimensional surgery. As a corollary, Hong generalises Theorem \ref{contractible}:
\begin{thm}\cite[Corollary 1.2]{Ho2021}  Let $\B$ denote the unit ball in the Euclidean space of dimension $\ge 3$. Then for each $k$, there exists a smooth contractible domain $\Omega$ with $\vert \partial \Omega\vert = \vert \partial  \B\vert$ and $\sigma_j(\Omega)>\sigma_j( \B)$ for each $1\le j\le k$.\end{thm}

\item
In \cite{FrSc2020}, the authors further develop the behaviour of the spectrum under surgery. For example, Theorem 1.1 of that paper says the following: let $\Omega_1,...,\Omega_s$ be compact $(d+1)$-dimensional Riemannian manifolds with boundary. Given $\epsilon>0$, there exists a Riemannian manifold $\Omega_{\epsilon}$, obtained by appropriately gluing $\Omega_1,...,\Omega_s$ together along their boundaries, such that
$\lim_{\epsilon \to 0} \vert \partial \Omega_{\epsilon}\vert =\vert \partial(\Omega_1 \sqcup ... \sqcup \Omega_s) \vert$ and
$\lim_{\epsilon \to 0}\sigma_k(\Omega_{\epsilon})=\sigma_k(\Omega_1 \sqcup ... \sqcup \Omega_s)$ for all $k$.

The way to glue appropriately two manifolds $\Omega_1$ and $\Omega_2$ together is described in a sequence of lemmas in \cite[Section 4.2]{FrSc2020}. Roughly speaking, the idea is to make a hole on the boundary of both manifolds and to join these holes by a catenoid. The technical difficulty is to do this without perturbing the spectrum of $\Omega_1$ and $\Omega_2$ too much.

\item
The authors also apply the previous construction to a single manifold $\Omega$ (\cite[Theorem 4.8]{FrSc2020}). They make two holes on the boundary of $\Omega$ and join them by a catenoid. They obtain a family of manifolds $\Omega_{\epsilon}$ such that for all $k$ 
\[
\lim_{\epsilon \to 0} \vert \partial \Omega_{\epsilon}\vert=\vert \partial \Omega\vert;\ \lim_{\epsilon \to 0}\sigma_k(\Omega_{\epsilon})=\sigma_k(\Omega).
\]
\end{enumerate}
\end{ex}

 \subsection{Upper bounds: the case of domains in a Riemannian manifold.}\label{UpperDomains}

Let us start with an example which probably gives the simplest available upper bound for Steklov eigenvalues. Consider a domain in the round sphere $\Omega\subset\Sp^{d+1}$ such that each coordinate function $\pi_i$ satisfies \begin{gather}
\int_{\partial\Omega}\pi_i\,dA=0.
\end{gather} 
In particular any domain such that $-\Omega=\Omega$ satisfies this condition. It follows that the coordinate functions can be used as trial functions in Equation~\eqref{eq:Stek Rayleigh min max}, so that \[\sigma_1(\Omega)\int_{\partial\Omega}\pi_i^2\,dA\leq\int_{\Omega}|\nabla\pi_i|^2\,dV.\]
Summing over $i$ and using that $\sum_{i=1}^{d+2}\pi_i^2\equiv 1$ and $\sum_{i=1}^{d+1}|\nabla\pi_i|^2\equiv d+1$ leads to 
\begin{gather}
    \sigma_1(\Omega)|\partial\Omega| \leq (d+1)|\Omega|.
\end{gather}
Note that this does not violate scale invariance because the size of the sphere $\Sp^{d+1}$ is fixed. In particular, there are no homotheties in the ambient space $\Sp^{d+1}$. In~\cite{FrSc2019} Fraser and Schoen considered an arbitrary domain $\Omega\subset\R^{d+1}$. They precomposed the coordinate functions $\pi_i$ with an appropriate conformal diffeomorphism $\Omega\to\Sp^{d+1}$ and used H\"older's inequality to obtain
\begin{gather}\label{ineq:FraserSchoenPreResult}
    \sigma_1(\Omega)|\partial\Omega|\leq (d+1)|\Sp^{d+1}|^{\frac{2}{d+1}}|\Omega|^{\frac{d-1}{d+1}}.
\end{gather}
Their argument also applies verbatim to domains $\Omega\subset\Sp^{d+1}$, and in that case inequality~(\ref{ineq:FraserSchoenPreResult}) is sharp.
\begin{prop}\label{prop:bestupperboundDomainSphere}
Let $\Omega\subset\Sp^{d+1}$ be a domain with smooth boundary. Then
\[\sigma_1(\Omega)|\partial\Omega||\Omega|^{\frac{1-d}{d+1}}\leq (d+1)|\Sp^{d+1}|^{\frac{2}{d+1}}.\]
This inequality is sharp: there exists a family $\Omega^\eps\subset\Sp^{d+1}$ such that
\[\sigma_1(\Omega^\eps)|\partial\Omega^\eps||\Omega^\eps|^{\frac{1-d}{d+1}}\xrightarrow{\eps\to 0}
(d+1)|\Sp^{d+1}|^{\frac{2}{d+1}}.\]
\end{prop}
The family $\Omega^\eps$ is constructed using homogenisation techniques. See Theorem~\ref{thm:homoclosed}.
Notice that for $d=1$ one recovers Kokarev's inequality (\ref{ineq:Kokarev}).

For domains in an arbitrary closed Riemannian manifold $M$, uniform upper bounds for all eigenvalues $\sigma_k$ are obtained from Theorem~\ref{asma2011}. Indeed, inequality (\ref{secondformulation}) shows that for any domain $\Omega \subset M$,	
\begin{equation}\label{thirdformulation}
    \sigma_k(\Omega,g)\vert \partial \Omega \vert_g\vert \Omega\vert_g^{\frac{1-d}{d+1}} \le  A_d \vert M\vert_{g_0}^{\frac{2}{d+1}}a^2+B_dk^{\frac{2}{d+1}}.
\end{equation}
In fact, the best upper bound for $\sigma_k(\Omega)|\partial\Omega||\Omega|^{\frac{1-d}{1+d}}$ can be expressed in terms of the best upper bound for the eigenvalues of the weighted Laplace operator $\beta^{-1}\Delta_g$. Let us introduce
\begin{gather}\label{eq:bestdensityeigenvalues}
\lambda_k^\#(M,g):=\sup_{0<\beta\in C^0(M)}\lambda_k(\beta^{-1}\Delta_g)\int_M\beta\,dV_g.
\end{gather}
This should be compared with the definition of conformal eigenvalues $\lambda_k^\star(M,[g])$ (see equation~\eqref{eq:defConformalEigenvalues}). For surfaces, it follows from the conformal invariance of the Laplace operator that $\lambda_k^\star(M,[g])=\lambda_k^\#(M,g)$, and the following result is a direct generalisation of Theorem~\ref{thm: GKL sharp}. It is proved in exactly the same way. 
\begin{thm}\label{thm: GKL higherd}
Let $(M,g_0)$ be a closed Riemannian manifold of dimension $d+1$ with $Ric_{g_0}(M) \ge -a^2d$ for a constant $a\ge 0$. Then for each $g\in[g_0]$, each $k\in\N$ and each domain $\Omega\subset M$,
\begin{gather} 
    \sigma_k(\Omega)|\partial\Omega||\Omega|^{\frac{1-d}{1+d}}\leq\lambda_k^\#(M,g)|M|_g^\frac{1-d}{1+d}.
\end{gather}
Moreover, for each $k\in\N$, there exists a family of domains $\Omega^\eps\subset M$ such that \[\sigma_k(\Omega^\eps)|\partial\Omega^\eps||\Omega^\eps|^{\frac{1-d}{1+d}}\xrightarrow{\eps\to0}\lambda_k^\#(M,g)|M|_g^\frac{1-d}{1+d}.\]
\end{thm}
The following is then obtained from~\eqref{thirdformulation}.
\begin{cor}
Let $(M,g_0)$ be a closed Riemannian manifold of dimension $d+1$ with $Ric_{g_0}(M) \ge -a^2d$ for a constant $a\ge 0$. Then
\[\lambda_k^\#(M,g)|M|_g^\frac{1-d}{1+d}\leq A_d \vert M\vert_{g_0}a^2+B_dk^{\frac{2}{d+1}}.\]
In particular, using the constant density $\beta\equiv 1$ leads to the following upper bound for the eigenvalues of the Laplace operator:
\begin{equation} 
    \lambda_k(M)\vert M \vert_g^{\frac{2}{d+1}} \le  A_d \vert M\vert_{g_0}^{\frac{2}{d+1}}a^2+B_dk^{\frac{2}{d+1}}.
\end{equation}
\end{cor}
\begin{remark}
The reader is invited to look at the very recent papers \cite{KaSt2022} by Karpukhin and Stern and \cite{Pe2022} by Romain Petrides for further investigation of links with harmonic maps to spheres.
\end{remark}

Let us come back to the question of bounding $\sigma_k(\Omega)$ for domains in complete manifolds. For a domain $\Omega \subset \R^{d+1}$, Fraser and Schoen combined~\eqref{ineq:FraserSchoenPreResult} with the classical isoperimetric inequality and obtained the following~\cite[Proposition 2.1]{FrSc2019}:
\begin{equation}\label{eq:fraserschoen}
    \sigma_1(\Omega)\vert\Sigma \vert^{1/d}\le \frac{(d+1)^{1/d}\vert \mathbb \Sp^{d+1}\vert^{\frac{2}{d+1}}}{\vert \mathbb \B^{d+1}\vert^{\frac{d-1}{d(d+1)}}}.
\end{equation}

More generally, if $\Omega$ is a domain in a complete manifold $(M,g_0)$ with non-negative Ricci curvature and $g$ is a Riemannian metric conformal to $g_0$, then Inequality \eqref{secondformulation} of Theorem \ref{asma2011} shows that 
\[
\sigma_k(\Omega,g)\vert \partial \Omega\vert_g \vert \Omega\vert_g^{\frac{1-d}{1+d}} \le B(d)k^{\frac{2}{d+1}}.
\]
Using the classical isoperimetric inequality, this implies that for domains $\Omega$ of the Euclidean space $\R^{d+1}$, of the hyperbolic space or of a hemisphere of the sphere $\Sp^{d+1}$, we have 
\begin{equation}\label{ineq:ceg}
	\sigma_k(\Omega)\vert \Sigma\vert^{\frac{1}{d}}\le C_dk^{\frac{2}{d+2}},
\end{equation}
where $C_d$ is a constant depending only on the dimension (see also \cite[Theorem 1.2]{CoElGi2011}). This implies that on these spaces, \emph{large boundary} implies \emph{small $\sigma_1$}. This leads to the question of whether these results may be generalised. In \cite[Examples 6.2, 6.3]{CoElGi2011}, it is shown that there does not exist an upper bound in full generality, independent of the normalisation. However, these examples are very specific constructions, and we can ask the following:
\begin{ques}\label{question:maxdom2}
Let $(M,g)$ be a complete Riemannian manifold of dimension $\ge 3$ of infinite volume, with Ricci curvature bounded from below. Discuss if one can construct domains $\Omega \subset M$ with arbitrarily large first nonzero normalised eigenvalue (for the different normalisations mentioned in this section).
\end{ques}

One can show that, in certain cases, by adding some hypotheses, we can get sharper inequalities for domains in Euclidean space, or more generally in non-compact rank-1 symmetric spaces such as the hyperbolic space. One can also characterise the case(s) of equality.

	It is known that the Weinstock inequality is not true for non simply-connected domains of the plane~\eqref{ineq:weinstock} and cannot be generalised without conditions in higher dimension. Because Fraser and  Schoen (Theorem~\ref{contractible}) showed the existence of contractible domains with larger first nonzero eigenvalue than the ball but with the same boundary area as the ball, it appears difficult to generalise Weinstock inequality to domains in $\R^{d+1}$ ($d+1\ge 3$) that are not convex. For convex domains, however, in \cite[Theorem 3.1]{BFNT2021}, Bucur, Ferone, Nitsch and Trombetti were able to show that this inequality remains true. 
	
	\begin{thm} Let $\Omega$ be a bounded convex domain in $\R^{d+1}$ and let $\Omega^*$ be the ball in $\R^{d+1}$ with $\vert \partial \Omega^*\vert =\vert \partial \Omega\vert$. Then
	\begin{equation}\label{weinstockconvex}
		\sigma_1(\Omega) \le \sigma_1(\Omega^*)
	\end{equation}
	with equality if and only if $\Omega$ is a ball.
	\end{thm}
	
	\begin{ques}\label{ques:weinstockhnsn}
	Is such an inequality also true for convex domains in hyperbolic space or in the sphere?  
	\end{ques}
	
	A quantitative version of Inequality \eqref{weinstockconvex} was given in \cite[ Theorem 1.1]{GLPT2020} by Gavitone, La Manna, Paoli and Trani.

 If we consider a normalisation with the volume $\vert \Omega\vert$, the situation is different. First, there is an inequality, due to F. Brock \cite{Br2001} for domains of the Euclidean space $\R^{d+1}$ in all dimensions:

 \begin{thm} \label{thm: brock} Let $\Omega \subset \R^{d+1}$ and let $\Omega^{\star}$ be the Euclidean ball such that $\vert \Omega\vert=\vert \Omega^*\vert$. Then
 \begin{equation} \label{thm:brock}
     \sigma_1(\Omega)\le \sigma_1(\Omega^*).
 \end{equation}
 Equality holds if and only if $\Omega$ is isometric to $\Omega^{\star}$.
     \end{thm}

     A quantitative version of inequality \eqref{thm:brock} was given in \cite{BrDeRu2012} by Brasco, De Phillippis and Ruffini.
	 
	 Inequality \eqref{thm:brock} was extended by Raveendran and Santhanam  to rank-1 symmetric spaces of noncompact types \cite[Theorem 1.1]{BiSa2014}. 
\begin{thm}\label{binoysanthanam1}
Let $M$ be a noncompact rank-1 symmetric space with sectional curvature $-4 \le K(M)\le -1$ (a typical example being the hyperbolic space). Let $\Omega \subset M$ be a bounded domain with smooth boundary. Then
\begin{equation} \label{BiSa}
    \sigma_1(\Omega)\le \sigma_1(\Omega^*)
\end{equation}
where $\Omega^*\subset M$ is a geodesic ball  such that $\vert \Omega\vert =\vert \Omega^*\vert$. Equality holds if and only if $\Omega$  is isometric to $\Omega^*$.
\end{thm}
Note that such a result is not true for rank-1 symmetric spaces of compact type. In \cite[Theorem 1.4]{CaRu2019},    Castillon and Ruffini construct a counterexample on the sphere. The domain $\Omega$ consists of the intersection of two antipodal geodesic balls, and for a ball $B$ with $\vert B\vert =\vert \Omega\vert$, they show that $\sigma_1(\Omega)>\sigma_1(B)$.

\noindent
\textbf{Doubly connected domains.} The study of eigenvalue problems for the Laplacian on doubly connected domains is a classical problem (optimal placement of an obstacle). Recently it has been studied for the Steklov eigenvalues, much of the time for the first nonzero eigenvalue, and with various boundary conditions like mixed Steklov-Dirichlet or Steklov-Neumann. These problems are usually defined in annular domains with a hole having a spherical shape. Different boundary conditions can be imposed on the inner and outer boundary, and, also, different optimisation problems can be studied.

A first result of this kind was obtained by Verma and Santhanam \cite[Theorem 1.1]{SaVe2020}. For $d+1>2$, they introduce a ball $B_1 \subset \R^{d+1}$ of radius $R_1$ and a ball $B_2 \subset \R^{d+1}$ of radius $R_2>R_1$ such that $\bar B_1\subset B_2$. They consider the mixed Steklov-Dirichlet problem on $B_2\setminus \bar B_1$ with Dirichet boundary condition on $\partial B_1$ and Steklov boundary condition on $\partial B_2$. Under these assumptions, they show that annular domains (concentric balls) maximise the first eigenvalue $\sigma_0^D$ of $B_2\setminus \bar B_1$.

In \cite[Theorem 1.2]{Ft2022}, Ftouhi gives another proof of this result, including in dimension $d+1=2$ which was left open in \cite{SaVe2020}. In the same paper, the author shows that for the first nonzero Steklov eigenvalue $\sigma_1(B_2\setminus \bar B_1)$, the maximum is also achieved uniquely when the balls are concentric (Theorem 1.1).

In \cite[Theorem 1]{Se2021}, Seo generalises the result of \cite{SaVe2020} to domains in rank one symmetric spaces. However, in the compact case, one needs to add the assumption that the radius $R_2$ of the larger ball $B_2$ is less that half of the injectivity radius of the space. For example, for a sphere, this means that the domain is contained in a hemisphere.

In this context, another natural question is to study the behaviour of the first eigenvalue when the center of $B_1$ is moving radially outward from the center of $B_2$ (with the Dirichlet boundary condition on $\partial B_1$). This question is studied by Hong, Lim and Seo ~\cite[Theorems 1-4]{HoLiSe2022} for domains in $\R^{d+1}$. In Section 6, the authors perform numerical estimates and conjecture that $\sigma_0^D(B_2\setminus \bar B_1)$ is decreasing as $B_1$ moves in the direction of the boundary. This was proved by Gavitone and Piscitelli \cite[Theorem 1.2]{GaPi2021}. In fact, their result is more general: instead of the ball $B_2$, the authors consider a domain $\Omega\subset \R^{d+1}$ which is connected and centrally symmetric with respect to a point.

This last result leads to a natural question for doubly connected domains in $\R^{d+1}$: fix the inner ball $B_1$ (with Dirichlet boundary condition on $\partial B_1$) and replace $B_2$ by a domain $\Omega$. Under certain conditions, one may hope that $\sigma_0^D(\Omega \setminus B_1)$ is maximal when $\Omega$ is a ball. This was investigated in \cite{PaPiSa2021} by Paoli, Piscitelli and Sannipoli  and in \cite{GPPS2021} by the same three authors and Gavitone.

In \cite[Definition 3.1]{PaPiSa2021} the authors introduce the concept of a nearly spherical set, which intuitively means a set close to a ball. The authors show that, among all nearly spherical sets having the same volume, the first eigenvalue $\sigma_0^D$ is maximal exactly when $\Omega$ is a ball \cite[Theorem 3.2]{PaPiSa2021}. 

In \cite[Theorem 1.1]{GPPS2021}, a similar type of result is shown when $\Omega$ is a convex domain satisfying the following condition: the convex domain $\Omega$ has to be contained in another ball of radius $R$, where $R$ depends on the radius $R_1$ of the inner ball $B_1$.

\begin{ques}\label{ques:manyquestionsdoublyconnected} These results lead to many new questions.  
\begin{itemize}
\item For example, most of the above results concern the mixed Steklov-Dirichlet problem. What about the mixed Steklov-Neumann problem (for the first nonzero eigenvalue $\sigma_1$)?

\item 
To what extent are the restrictions imposed on $\Omega$ in \cite{PaPiSa2021} and \cite{GPPS2021} necessary? 

\item
In \cite{PaPiSa2021} and \cite{GPPS2021}, the authors consider the spectrum normalised by the volume of the domain. Is it possible to obtain similar results with a normalisation by the volume of the boundary of the domain?

\item 
Is it possible to extend some of the results to domains in rank one symmetric spaces, as was done in \cite{Se2021} for the Steklov-Dirichlet problem on $B_2 \setminus B_1$?
\end{itemize}
\end{ques}

Recently, there also appeared interesting inequalities with normalisation on the diameter of the domain. We now discuss some of these.

In \cite[Theorem 1.2]{BiSa2014}, the authors obtain an inequality for domains in Cartan-Hadamard manifolds. The estimate was made more explicit recently by Li, Wang and Wu in \cite[Theorem 1.1]{LiWaWu2020}:

\begin{thm}\label{liwangwu} Let $(M,g)$ be a complete, simply-connected Riemannian manifold  of dimension $d+1$ and let $\Omega \subset M$ be a bounded domain with Lipschitz boundary. Let $M_{\kappa}$ be the $d+1$-dimensional space form of constant sectional curvature $\kappa \le 0$ and $\Omega^* \subset M_{\kappa}$ be a geodesic ball such that $\vert \Omega^*\vert =\vert \Omega \vert$. If the sectional curvature of $(M,g)$ is $\le \kappa$ and the Ricci curvature of $(M,g)$ is $\ge dKg$ with $K\le 0$, then
\begin{equation} \label{LiWaWu}
    \sigma_1(\Omega)\le \left(\frac{sn_{K}(\diam(\Omega))}{sn_{\kappa}(\\diam(\Omega))}\right)^{2d} \sigma_1(\Omega^*)
\end{equation} 
where $sn_0(t)=t$ and for $k<0$, $sn_k(t)=\frac{1}{\sqrt{-k}}\sinh(\sqrt{-k}t)$.
\end{thm}

  \begin{ques}\label{ques:diamterm}
 The term $\frac{sn_{K}(\diam(\Omega))}{sn_{\kappa}(\diam(\Omega))}$ may become very large when $\diam(\Omega)$ becomes large. Is it possible to establish a better estimate or to construct an example of a domain $\Omega$ with large diameter and $\sigma_1(\Omega)$ large?
  \end{ques}
   
  \begin{ques}\label{ques:cartanhadamard}
Can we get estimates like (\ref{LiWaWu}) for the other eigenvalues of domains in Cartan-Hadamard manifolds? The methods used in \cite{CoElGi2011} or \cite{Ha2011} do not seem to apply.
\end{ques}

In \cite[Proposition 4.3]{BoBuGi2017}, Bogosel, Bucur and Giacomini obtain an upper bound involving the diameter $\diam(\Omega)$ of the domain. 
\begin{thm} \label{diam1}There exists a constant $C_d>0$ such that for every $k \in \mathbb N$ and for every bounded connected Lipschitz open set
$\Omega$ in $\R^{d+1}$
  \begin{equation}
	\sigma_k(\Omega)\le C_d \frac{k^{\frac{d+3}{d+1}}}{\diam(\Omega)} 
\end{equation} 
\end{thm}
It would be interesting to investigate the optimality of the power of $k$.

From this theorem, we deduce that Euclidean domains with large diameter have small eigenvalues. It is known that this is not the case in general for arbitrary compact Riemannian manifolds with boundary.  Indeed by Inequality \ref{ineq:compmixed} and Proposition \ref{prop: dir neum bracket}, one can always modify the interior of the manifold without changing the spectrum very much. However, Theorem~\ref{diam1} leads to the following question:
	
	\begin{ques}\label{ques:diameterhnsn}
	Is it possible to get inequalities similar to those in Theorem \ref{diam1} for domains in the hyperbolic space or the sphere?
	\end{ques}

Regarding the importance of the diameter, Al Sayed, Bogosel, Henrot and Nacry proved the following inequality in \cite[Proposition 2.2]{ABHN2021}:
 
\begin{thm} \label{diam2} Let $\Omega$ be a convex domain in $\R^{d+1}$.
Then, there exists
an explicit constant $C = C(d, k)$ depending only on the dimension $d+1$ and on $k$ such that
\begin{equation}
	\sigma_k(\Omega)\le C \frac{\vert \Omega \vert^{\frac{1}{d}}}{\diam(\Omega)^{\frac{2d+1}{d}}}.
\end{equation}
\end{thm}
         
This shows that when the diameter is fixed, if the volume of a convex set tends to $0$ then all the eigenvalues tend to $0$.

\begin{ques}\label{ques:specconvtozero}
Let $\Omega_{\epsilon} \subset \R^{d+1}$ be a family of domains with fixed diameter (without convexity assumption). If the volume of $\Omega_{\epsilon}$ tends to $0$ as $\epsilon \to 0$, can one say that all the eigenvalues of $\Omega_{\epsilon}$ tend to $0$?
 \end{ques}

In the same paper, the authors show that when considering domains in $\R^{d+1}$ with fixed diameter, the ball is never a maximum for the kth eigenvalue $\sigma_k$ \cite[Theorem 3.2]{ABHN2021}.

	\subsection{Metric upper bounds for Riemannian manifolds.}\label{UpperRiemannian}
	
	We will now present some recent metric estimates: they do not directly involve the curvature. Let us begin with an estimate in terms of diameter and injectivity radius in the spirit of \cite{Be1979} by Berger and \cite{Cr1980} by Croke. This is \cite[Theorem 5]{CoGi2022} by Colbois and Girouard.
	
	For a compact Riemannian manifold $\Omega$ of dimension $d+1$ with boundary $\Sigma$, and for $x,y \in \Omega$, we denote by $d_{\Omega}(x,y)$ the distance in $\Omega$. The diameter of $\Sigma$ in $\Omega$ is defined by
\[\diam_\Omega(\Sigma):=\sup\left\{d_{\Omega}(x,y)\,:\,x,y\in\Sigma\right\}.\]

Let $\Sigma_1,\cdots,\Sigma_b$ be the connected components of $\Sigma$. On each $\Sigma_j$, there is the extrinsic distance $d_{\Omega}$ and the corresponding diameter $\\diam_\Omega(\Sigma_j)$. We will also consider the intrinsic distance $d_{\Sigma_j}(x,y)$ on $\Sigma_j$ and the intrinsic diameter $\diam(\Sigma_j)$ of $\Sigma_j$.

\begin{thm}\label{thm:upperboundDiam} (\cite[Theorem 5]{CoGi2022})
  Let $\Omega$ be a smooth connected compact Riemannian manifold of dimension $d+1$ with boundary $\Sigma$. Then, for each $j=1,\cdots,b$ and each $k\geq 1$,
  \begin{gather}\label{ineq:upperboundDiam}
    \sigma_k(\Omega)
    \leq
    K_d\frac{|\Omega|}{\diam_\Omega(\Sigma_j)^{2}}\frac{1}{\min(\inj(\Sigma_j)^d,diam_{\Omega}(\Sigma_j)^d)}k^{d+1},
  \end{gather}
  where $K_d$ is an explicit constant depending on the dimension of $\Omega$.
\end{thm}

To obtain upper bounds of a metric nature which have optimal exponent of $k$, we need to introduce the metric concepts of packing and growth constant, as is done in \cite{GrNeYa2004}. These constants avoid the need to introduce restrictions on the curvature. We need also to introduce a constant measuring how the boundary $\Sigma$ of $\Omega$ is distorted in $\Omega$. This is another way to compare the intrinsic diameter $\diam(\Sigma_j)$ and the extrinsic diameter $\diam_\Omega(\Sigma_j)$.

For $x\in \Omega$, let
		\[
		B^{\Omega}(x,r)=\{y \in \Omega:d_{\Omega}(x,y)<r\}.
		\]	
		For $x\in \Sigma_j$, let
		\[
		B^{\Sigma_j}(x,r)=\{y\in \Sigma:d_{\Sigma_j}(x,y)<r\}.
		\]
		
		The \emph{growth constant} $A$ is the smallest value such that for each $x \in \Sigma_j$, $j=1,...,b$, and $r>0$, $\vert B^{\Sigma_j}(x,r)\vert \le Ar^{d}$.
		
		The \emph{extrinsic packing constant} $N$ is the smallest value such that, for each $r>0$ and each $x\in \Sigma_j$, the extrinsic ball $B^{\Omega}(x,r)\cap \Sigma_j$ can be covered by $N$ extrinsic balls of radius $\frac{r}{2}$ centered at points $x_1,...,x_N \in \Sigma_j$.
		
		We measure the distortion of the boundary as follows: first, observe that for $x,y \in \Sigma_j$, $d_{\Omega}(x,y)\le d_{\Sigma_j}(x,y)$. Then if $\Sigma_j$ is a connected component of $\Sigma$, the distortion $\Lambda_j$ of $\Sigma_j$ is
		\[
		\Lambda_j=\inf\{c>0:d_{\Sigma_j}(x,y)\le cd_{\Omega}(x,y); x,y \in \Sigma_j\}.
		\]
		If $b$ is the number of connected components of $\Sigma$, the distortion $\Lambda$ of $\Sigma$ is
		\[
		\Lambda=\max\{\Lambda_1,...,\Lambda_b\}.
		\]
		
		Note that one can express the packing constant $N$ in terms of the distortion $\Lambda$ and the intrinsic packing constant of $(\Sigma,d_{\Sigma})$ \cite[ Lemma 13]{CoGi2022}. This is useful in discussing below the sharpness of the main estimate.
		
		With these definitions, we get  (\cite[Theorem 1]{CoGi2022})
		
		\begin{thm} \label{globalestimate}Let $\Omega$ be a connected compact Riemannian manifold of dimension $d+1$ with boundary $\Sigma$. For each $k\ge 1$,
		\begin{equation}\label{ineq}
			\sigma_k(\Omega)\le 512b^2N^3A\Lambda^2\frac{\vert \Omega\vert}{\vert\Sigma\vert^{\frac{d+2}{d}}}k^{2/d}
		\end{equation}
		Moreover, the exponent $2/d$ on $k$ is now optimal.
		\end{thm}
		
		\begin{remark} \label{rem: necessity}
		Of course, the estimates in Theorems \ref{thm:upperboundDiam} and \ref{globalestimate} are not sharp in the sense that there are no cases of equality, even for the first nonzero eigenvalue $\sigma_1$. This comes from the fact that we use metric constructions.
However, the different quantities appearing in the right-hand side of this inequality cannot be removed. Precisely, for each such quantity, we construct a family of examples where all the other such quantities are constant or bounded, and where the first eigenvalue becomes arbitrarily large. 

\begin{enumerate}
\item 
Example \ref{confdef} shows that the presence of the volume $\vert \Omega\vert$ is necessary in both Theorems.  

\item
Example \ref{JadeBrisson} shows that the the injectivity radius is necessary in Theorem \ref{thm:upperboundDiam} and that the volume $\vert \Sigma\vert$ is necessary in Theorem \ref{globalestimate}. 

\item
In Example \ref{CianciGirouard}, the boundary is fixed. This shows that the extrinsic diameter $diam_{\Omega} \Sigma_j $ of the boundary component is necessary in Theorem \ref{thm:upperboundDiam}. This also shows that the distortion and the growth constant $A$ are necessary in Theorem \ref{globalestimate}. Note that the distortion and the growth constant (with respect to the extrinsic distance) are closely related.

\item
		The number $b$ of connected components of the boundary is necessary in Theorem \ref{globalestimate}: in \cite{CoGi2014}, we constructed a sequence of compact surfaces $\Omega_i$ with $\vert \partial \Omega_i \vert \sigma_1(\Omega_i) \to \infty$. The surface was locally a product $S^1\times [0,1]$ near each boundary component, so that $\Lambda=1$. After renormalisation, we obtain a sequence with $\vert \partial \Omega_i\vert=1$, $\sigma_1(\Omega_i) \to \infty$, $\vert \Omega_i\vert$ uniformly bounded from below and above, $\Lambda=1$, and the packing constant of each boundary component of $\Omega_i$ and the growth constant are independent of $i$ because the surface $\Omega_i$ is a product near each boundary component. This shows that $b$ has to go to $\infty$. We can adapt this construction to higher dimensional manifolds as well.
		
		\item
		In Theorem \ref{globalestimate}, the packing constant $N$ is necessary and the exponent $2/d$ is optimal. This last fact is surprising, because in comparison with the Weyl law, we could expect to have the exponent $1/d$. These two facts are consequences of Example \ref{example: cylinder}. For each closed Riemannian manifold $M$ of dimension $d$, we consider the cylinder $\Omega_L=[0,L]\times M  $. In this example, the distortion $\Lambda=1$ and the growth constant $A$ is fixed. For $L$ small enough, we have
		\[
		\sigma_k(\Omega_L)= \sqrt{\lambda_k(M)}\tanh(\sqrt{\lambda_k(M)}\frac{L}{2}).
		\]	
		As $\vert\Omega_L \vert=L \vert M \vert$, Inequality \eqref{ineq} becomes	
		\[
		\sqrt{\lambda_k(M)}\tanh(\sqrt{\lambda_k(M)}\frac{L}{2}) \le 2048N^3A\frac{L\vert M\vert}{\vert M\vert^{\frac{d+2}{d}}}k^{2/d}
		\]
		and, after division by $L$,
		\[
		\sqrt{\lambda_k(M)}\frac{1}{L}\tanh(\sqrt{\lambda_k(M)}\frac{L}{2}) \le 2048N^3A\frac{\vert M\vert}{\vert M\vert^{\frac{d+2}{d}}}k^{2/d}.
		\] 	
		Then we let $L$ tend to $0$. As $\lim_{L\to 0} \frac{1}{L}\tanh(\sqrt{\lambda_k(M)}\frac{L}{2})=\sqrt{\lambda_k(M)}$, and we get
		\begin{equation} \label{ineq1}
		\lambda_k(\Sigma) \le 2048N^3A\frac{\vert M\vert}{\vert M\vert^{\frac{d+2}{d}}}k^{2/d}
		\end{equation}
		
		Now the Weyl law for $\lambda_k$ implies that the exponent of $k$ cannot be smaller than $2/d$.
At the same time, we see that we need control of the packing constant $N$ of $\Omega$: if $N^3$ is bounded, Inequality \eqref{ineq1} would lead to a universal inequality for $\lambda_1(M)$, which is impossible, as  
		on each closed manifold $M$ of dimension $d\ge 3$, it is possible to construct a family of Riemannian metrics $g_{\epsilon}$ on $M$, $0<\epsilon <1$, with $\vert (M,g_{\epsilon})\vert =1$ and  $\lambda_1(M,g_{\epsilon})\to \infty$ as $\epsilon\to 0$ (see \cite{CoDo1994}). In fact, in this case, the intrinsic packing constant of $(M,g_{\epsilon})$ tends to $\infty$ as $\epsilon \to 0$  and this implies that the extrinsic packing constant $N$ of $\Omega_L=[0.L]\times (M,g_{\epsilon})$ tends to $\infty$ as $\epsilon \to 0$.
		
		\item
		The exponent $d+1$ of $k$ in Theorem \ref{thm:upperboundDiam} is probably not optimal.
		
\end{enumerate}
\end{remark}

	In dimension higher than $2$, we have obtained robust geometric estimates, where all the geometric and metric ingredients appearing in the inequalities are necessary. However, these estimates are too general to be sharp. The next subsection presents a setting where one can get sharp estimates: Riemannian manifolds of revolution and in particular hypersurfaces of revolution in Euclidean space.
	
\subsection{Upper and lower bounds: the case of manifolds of revolution} \label{UpperRev}

A manifold of revolution $\Omega$ of dimension $d+1$ is a warped product $\Omega= [0,L] \times \Sp^{d} $ with a Riemannian metric $g(r,p)=dr^2+h^2(r)g_{0}$ where $g_0$ is the usual canonical metric on $\Sp^{d}$. The function $h$ is smooth and satisfies $h(r)> 0$ in $[0,L[$. 

The situation where $h(L)>0$ corresponds to $\Omega$ being a topological cylinder, as expected from the product structure.
If instead $h(L)=0$, then $\Omega$ is homeomorphic to a ball. In order to obtain a smooth manifold, we also need to impose that $h(L)=h^{(2k)}(L)=0$ for $k >0$ and $h'(L)=-1$.

First, note that for the spectrum of the Laplacian, to our knowledge, there do not exist many contributions in the context of compact Riemannian manifolds of revolution. However, we can mention the paper \cite{AbFr2002} by Abreu and Freitas where the authors study the first eigenvalue for $S^1$-invariant abstract Riemannian metrics on the sphere and compare it with the situation where the metric is realised as the pullback of the canonical Euclidean metric in $\R^3$ by some $S^1$-invariant embedding (see also \cite{CoDrEl2008} for some generalisations). A series of papers by Ariturk (\cite{Ar2014, Ar2016, Ar2018}) is very inspiring. They mainly concern maximisation of the first eigenvalue of surfaces of revolution in Euclidean space with Dirichlet boundary conditions. However, the initial paper \cite{Ar2014} is not yet published to our knowledge.

Regarding the Steklov problem, for specific manifolds of revolution, such as the Euclidean ball or the cylinder $[0,L] \times \Sp^{d} $, it is possible to explicitly calculate the eigenvalues; see examples in \cite{GiPo2017} and Example \ref{example: cylinder}. The hope is that for more general revolution manifolds, it remains possible to get sharp estimates on the Steklov eigenvalues.

We discuss lower and upper bounds for the Steklov eigenvalues of such metrics. We will take $d+1 \ge 3$. The case of surfaces ($d=1$) is distinct, well-understood, and explained in \cite[Theorem 1.1]{FaTaYu2015} by Fan, Tam and Yu; we discuss this work further in Section \ref{applic.bounds}. They obtain the maximum of the $k$th normalised Steklov eigenvalue of all rotationally symmetric metrics on $[0,L] \times \mathbb \Sp^1$ for $k>2$, and the supremum for $k=2$.

We will consider two situations: we will first consider the general case of Riemannian metrics of revolution, and subsequently, the special
case of Euclidean hypersurfaces of revolution. 

\subsubsection{Riemannian metrics of revolution} If $\Omega$ is a Riemannian manifold of revolution of dimension $d+1\ge 3$ and fixed boundary (with one or two connected components) and without other assumptions, one can find a family of revolution metrics with that boundary and arbitrarily small eigenvalues $\sigma_{k}$ and another family of revolution metrics with that boundary and arbitrarily large $\sigma_{1}$. It is even possible to construct these families so that they are conformal to any initial Riemannian metric $g$ given on $\Omega$. In order to show these facts, we just have to adapt \cite[Theorem 1.1 and Proposition 2.1]{CoElGi2019} (see Example \ref{confdef}).
Roughly speaking, if $g(r,p)=dr^2+h^2(r)g_{0}$ is a revolution metric on $\Omega$ and $f=f(r)$ a positive smooth function taking the value $1$ on $\partial \Omega$, we consider the conformal metric $g_f(r,p)=f^2(r)g(r,p)$. (Note that, after a change of variable, the Riemannian metric may be written $dr^2+\tilde h^2(r)g_0$.)
The Rayleigh quotient $R(u)$ of a function $u$ on $\Omega$ is given by
\begin{equation}
	R(u)=\frac{\int_{\Omega} \vert du\vert^2_gf^{d-1}dV_{(\Omega,g)}}{\int_{\Sigma}u^2dV{(\Sigma,g)}}
\end{equation}
 Taking $f$ close to $0$ in the interior of $\Omega$ allows one to obtain as many arbitrarily small eigenvalues as one wishes. Taking $f$ large inside $\Omega$ leads to large first nonzero eigenvalue $\sigma_1(\Omega, f^2g)$.

Therefore, in order to get control over the Steklov spectrum, one has to add some geometric hypotheses.

In \cite{Xi2022}, Xiong considers the case of a revolution metric $g(r,p)=dr^2+h^2(r)g_{0}$ on a ball with constraint on the Ricci curvature and on the convexity of the boundary.
Because of the symmetries of revolution, the spectrum of $\Omega$ comes with multiplicity. In the sequel we will denote by $\sigma_{(k)}(\Omega)$ the Steklov eigenvalues of $\Omega$ counted without multiplicity, that is
\[
\sigma_{(0)}=\sigma_0=0<\sigma_{(1)}(\Omega) < \sigma_{(2)}(\Omega) <...
\]
The multiplicity of $\sigma_{(k)}$ is the multiplicity of the $k$th eigenvalue $\lambda_{(k)}$ of $\mathbb \Sp^d$. See \cite[Proposition 2 ]{CoVe2021}.

In Example \ref{ex: ball}, it was shown that for the Euclidean ball of radius $R$ in $\R^{d+1}$, we have $\sigma_{(k)}=\frac{k}{R}$ with multiplicity expressed in terms of binomial coefficients, namely $C_{d}^{d+k}-C_{d}^{d+k-2}$ when $k\ge 2$.
Note that for a metric of revolution with two boundary components, isolated examples with larger multiplicity may appear; the multiplicity of $\sigma_{(k)}$ is not always the same.

With the above notations, Xiong shows

\begin{thm}\label{thm:xiong1} \cite[ Theorems 2 and 3]{Xi2022} Suppose that $\Omega$ has nonnegative Ricci curvature and strictly convex boundary. Then $\sigma_{(k)}$ satisfies
\begin{equation} 
	\sigma_{(k)}(\Omega,g) \ge k\frac{-h'(0)}{h(0)}.
\end{equation}

If $\Omega$ has nonpositive Ricci curvature and strictly convex boundary, then $\sigma_{(k)}$ satisfies
\begin{equation} 
	\sigma_{(k)}(\Omega,g) \le k\frac{-h'(0)}{h(0)}.
\end{equation}
Moreover, in each case, we have equality if and only if $h(r)=L-r$ (that is, $\Omega$ is isometric to the Euclidean ball of radius $L$).  
\end{thm}

In \cite{Xi2021}, Xiong investigates the same problem, but for the difference and the ratio of successive eigenvalues.

\begin{thm} \label{thm:xiong2}\cite[Theorem 2 and 5]{Xi2021} Suppose that $\Omega$ has nonnegative Ricci curvature and strictly convex boundary, then for $k\ge 0$
\begin{equation} 
	\sigma_{(k+1)}(\Omega,g)- \sigma_{(k)}(\Omega,g)\ge k\frac{-h'(0)}{h(0)}
\end{equation}
  and for $k\ge1$
\begin{equation}
    \frac{\sigma_{(k+1)}(\Omega,g)}{\sigma_{(k)}(\Omega,g)}\le \frac{k+1}{k}.
\end{equation}

If $\Omega$ has nonpositive Ricci curvature and strictly convex boundary, then for $k\ge 0$
\begin{equation} 
	\sigma_{(k+1)}(\Omega,g)- \sigma_{(k)}(\Omega,g)\le k\frac{-h'(0)}{h(0)}
\end{equation}
and for $k\ge1$
\begin{equation}
    \frac{\sigma_{(k+1)}(\Omega,g)}{\sigma_{(k)}(\Omega,g)}\ge \frac{k+1}{k}.
\end{equation}
In all these situations, equality holds if and only if $h(r)=L-r$ (that is, $\Omega$ is isometric to the Euclidean ball of radius $L$). 
\end{thm}

These results lead to different kinds of questions:

\begin{ques}\label{ques:rev1}
What can be said for revolution metrics with other geometric constraints such as $\vert Ricci \vert \le a^2$?
\end{ques}  

\begin{ques}\label{ques:rev2} Can we get similar results for revolution metrics on manifolds with two boundary components?
\end{ques}

\subsubsection{Euclidean hypersurfaces of revolution} A particular case of a revolution manifold is when $\Omega$ is a (d+1)-dimensional hypersurface of revolution in Euclidean space $\R^{d+2}$. Without lost of generality, we fix the boundary to be $\mathbb \Sp^{d}\times \{0\} \subset \R^{d+1} \times \{0\}$ if the boundary has one connected component and $\left(\mathbb \Sp^{d}\times \{0\}\right) \cup \left(\mathbb \Sp^{d}\times \{\delta\}\right)$ if the boundary has two connected components. Note that we are assuming both boundary components are isometric to the unit sphere, in particular they have the same volume.  

The fact that $\om$ is a hypersurface has a strong consequence for the induced Riemannian metric: Consider the meridian curve $c$ of the hypersurface of revolution, that is the curve in $\om$ cut by a 2-dimensional half-plane whose edge is the axis of revolution ( i.e., the $x_{d+2}$ axis in the standard coordinates on $\R^{d+2}$). Denote by $L$ its length and introduce a parametrisation by arc-length $c(r)=(h(r),x_{d+2}(r)$, $0\le r \le L$, where $h(r)$ denotes the distance to the $x_{d+2}$ axis.

Then, we can write the metric of $\Omega$ as $g(r,p)=dr^2+h^2(r)g_{0}$ where $g_0$ is the canonical metric on $\mathbb \Sp^{d}$ and $r \in [0,L]$. For each $r$ we have $\vert h'(r)\vert \le 1$. This implies
\begin{equation}
	1-r \le h(r) \le 1+r;\ h(L)-r \le h(L-r) \le h(L)+r.
\end{equation}

If $\Omega$ has one boundary component, $h(L)=0$ and $0\le h(L-r) \le r$.

If $\Omega$ has two boundary components, $h(L)=1$ and $1-r \le h(L-r)\le 1+r$.           

\medskip
\noindent
\textbf{Case with one boundary component}: This case is now well understood. For each $k\ge 1$, we have the sharp inequalities
\begin{equation}
	k\le \sigma_{(k)}(\Omega) < k+d-1.
\end{equation}
Precisely, for lower bounds, Colbois, Girouard and Gittins show the following \cite[Theorem 1.8]{CoGiGi2019}:

\begin{thm} Let $\Omega$ be a hypersurface of revolution in $\R^{d+2}$, ($d\ge 2$), with connected boundary $\mathbb \Sp^{d} \times \{0\}$. Then for each $k\ge 1$, $\sigma_{(k)}(\Omega) \ge \sigma_{(k)}(\mathbb B^{d+1})=k$, where $\mathbb B^{d+1}$ denotes the unit ball in $\R^{d+1}$. For each given $k$, we have equality if and only if $\Omega =\mathbb B^{d+1} \times \{0\}$.
\end{thm}

For upper bounds, Colbois and Verma show~\cite[Theorem 1]{CoVe2021}:

\begin{thm} \label{upperrev} Let $\Omega$ be a hypersurface of revolution in $\R^{d+2}$ with one boundary component $\mathbb \Sp^{d} \times \{0\}$. Then, for $d\ge 2$ and for each $k \ge 1$, we have
\begin{equation}
	\sigma_{(k)}(\Omega) < k+d-1.
\end{equation}
Moreover, the result is sharp. For each $\epsilon >0$ and each $k \ge 1$, there exists a hypersurface of revolution $\Omega_{\epsilon}$ with one boundary component such that $\sigma_{(k)}(\Omega_{\epsilon}) > k+d-1-\epsilon$.
\end{thm}
 
In order to understand the geometry behind this estimate,  we can consider Formulas (\ref{Dirichlet}) and (\ref{Neumann}) of Example \ref{annuli}. The hypersurfaces of revolution with $\sigma_k$ close to $k+d-1$ contain annuli $\Omega_L$ with $L \to \infty$. The bracketing formula (\ref{ineq:compmixed}) shows that the $k$th eigenvalue $\sigma_{(k)}$ converges to $k+d-1$ as $L\to \infty$.

\begin{remark}
    The situation for $d=1$ is also interesting: all surfaces of revolution that share the same connected boundary are Steklov isospectral. Indeed they are $\sigma$-isometric in the sense of Definition~\ref{sigmaisom}, hence the observation follows from Corollary~\ref{cor:sigmaisom}.
\end{remark}

\medskip
\noindent
\textbf{Case with two boundary components}:
The situation with two boundary components is more complicated. For lower bounds, we have a result comparable to the case with one boundary component   \cite[Theorem 1.11]{CoGiGi2019}.

\begin{thm} \label{lowerrev2}Let $\Omega$ be a hypersurface of revolution in $\R^{d+2}$ with boundary $\left(\mathbb \Sp^{d} \times \{0\}\right) \cup \left(\mathbb \Sp^{d} \times \{\delta\}\right)$. Let $L$ be the arc length of the meridian. If $L\ge 2$, then for each $k\ge 1$,
\begin{equation}
	\sigma_k(\Omega) \ge \sigma_k(\mathbb B^{d+1} \sqcup \mathbb B^{d+1} ) \label{inlowerrev2}
\end{equation}
In particular, if $\delta \ge 2$ this is always true.

For $\delta \ge 2$, Inequality \eqref{inlowerrev2} is sharp: for each $k$ and each $\epsilon >0$, there exists a connected hypersurface $\Omega_{\epsilon,k}$ such that $\sigma_{k}(\Omega_{\epsilon,k})\le \sigma_k(\mathbb B^{d+1}  \sqcup \mathbb B^{d+1})+\epsilon$.
\end{thm}

Note that, if we do not require the hypersurface to be connected, it suffices to take the disjoint union of two unit balls to have the sharpness. We obtain sharpness in the connected case by considering hypersurfaces formed by taking two copies of the unit ball, perforating each with a small hole about its center, and joining the two punctured balls by a thin cylinder.

When $\delta <2$, it appears to be more difficult to find a lower bound for $\sigma_{(k)}$. For $k=1$, the union of two balls gives the lower bound $0$.
But for larger $k$, using Example \ref{example: cylinder}, it is easy to see that, when $\delta$ becomes small, the cylinder $[0,\delta] \times \mathbb \Sp^{d}$ has plenty of small eigenvalues and the union of two balls no longer gives the lower bound.

\begin{ques}\label{ques:2bc1} Find a sharp lower bound for $\sigma_{k}(\Omega)$, where $\Omega$ is a hypersurface of revolution with boundary components $\left(\mathbb \Sp^{d} \times \{0\}\right) \cup \left(\mathbb \Sp^{d} \times \{\delta\}\right)$ and $\delta<2$.
\end{ques}

\begin{ques}\label{ques:2bc2} Find a sharp upper bound for $\sigma_{k}(\Omega)$, where $\Omega$ is a hypersurface of revolution with boundary components $\left(\mathbb \Sp^{d} \times \{0\}\right) \cup \left(\mathbb \Sp^{d} \times \{\delta\}\right)$.
\end{ques}

\begin{ques}\label{ques:2bc3}
In both cases (one or two boundary components) can we obtain results as in Theorem \ref{thm:xiong2}, that is, sharp bounds for the difference $\sigma_{(k+1)}-\sigma_{(k)}$ and the ratio $\frac{\sigma_{(k+1)}}{\sigma_{(k)}}$?
\end{ques}

  We have seen that considering hypersurfaces of revolution of Euclidean space allows one to obtain some sharp estimates. An intermediate situation between general Riemannian manifolds and hypersurfaces of revolution is to consider submanifolds of Euclidean space, and this is the object of the next subsection.
  
   \subsection{Upper and lower bounds: the case of submanifolds of Euclidean space}\label{UpperSub}
  
We consider the following situation: we fix a closed (not necessarily connected) submanifold $\Sigma$ of dimension $d$ of Euclidean space $\R^m$ and consider all possible $d+1$-dimensional submanifolds $\Omega \subset \R^m$ with boundary $\Sigma$. In \cite[Theorem 1.11]{CoGiGi2019}, Colbois, Girouard and Gittins prove the inequality
\begin{equation}
	\sigma_k(\Omega) \le A_{\Sigma} \vert \Omega \vert k^{2/d}
\end{equation}
where $A_{\Sigma}$ depends on the geometry of $\Sigma$. Moreover, the authors keep open the question of constructing $\Omega$ with $\sigma_1$ arbitrarily large or small (with $\Sigma$ fixed).
We can apply Inequality (\ref{ineq}) in this context, and this makes the dependence on the geometry of $\Sigma$ more explicit:
\begin{equation}
			\sigma_k(\Omega)\le 512b^2N^3A\Lambda^2\frac{\vert \Omega\vert}{\vert\Sigma\vert^{\frac{d+2}{d}}}k^{2/d}.
		\end{equation}
  The distortion $\Lambda$ of $\Sigma$ in $\Omega$ is bounded from above by the distortion of $\Sigma$ in $\R^m$. For the same reason as for Riemannian manifolds (Theorem \ref{globalestimate} and Point 5 of Remark \ref{rem: necessity}), the exponent is optimal: we can also consider in $\R^m$ a product of the type $M\times [0,L]$ as in Remark \ref{rem: necessity}. Similarly, one needs to control $\vert \Sigma\vert$: the construction of Brisson ~\cite{Br2022} and Example \ref{JadeBrisson} also works for a submanifold.
The only question is about the presence of $\vert \Omega \vert$. It does not seem possible to adapt the construction of \cite{CoElGi2019} in Example \ref{confdef} to submanifolds with fixed boundary, and it turns out that the question is open:

\begin{ques}\label{open:arbitrarilylarge}
Given a $d$-dimensional compact submanifold $\Sigma$ in $\R^m$, is it possible to construct a family of $(d+1)$- dimensional submanifolds $\Omega$ of $\R^m$ with boundary $\Sigma$ for which $\sigma_1(\Omega)$ becomes arbitrarily large?
\end{ques}

We can add different kinds of conditions in order to avoid the presence of $\vert \Omega\vert$. We can take advantage of being in the Euclidean space to introduce the \emph{intersection index} already considered in \cite{CoDrEl2010}.
For a compact immersed submanifold $N$ of dimension $q$ in $\R^{q+p}$, almost every
$p$-plane $\Pi$ in $\R^{q+p}$ is transverse to $N$, meaning that the intersection $\Pi \cap N$
consists of a finite number of points.
\begin{defn}
\label{def:index}
The intersection index of $N$ is
\begin{equation}
\label{eq:index}
i_p(N) := \sup_{\Pi} \{\# (\Pi \cap N)\},
\end{equation}
where the supremum is taken over the set of all $p$-planes $\Pi$ that are transverse to $N$ in $\R^{q+p}$.
\end{defn}

In \cite[Corollary 1.5]{CoGi2021}, the following result is proved by Colbois and Gittins:
\begin{thm}
\label{thm:3}
Let $d \geq 2$.
Let $\Sigma$ be a $d$-dimensional,
closed, smooth submanifold of $\R^m$. Let $\inj(\Sigma)$ denote the injectivity radius of $\Sigma$. There exist constants $A_d, B_d>0$ depending only on $d$ such that
for any compact $(d+1)$-dimensional submanifold $\Omega$ of $\R^m$ with boundary $\Sigma$ and for $k \geq 1$,
\begin{equation}\label{eq:T3}
  \sigma_k(\Omega) \leq A_d \frac{i(\Omega)}{\inj(\Sigma)} + B_d \, i(\Omega)\left(\frac{i(\Sigma) k}{\vert \Sigma \vert}\right)^{1/d},
\end{equation}
where $i(\Sigma) = i_{m-d}(\Sigma)$ and $i(\Omega) = i_{m-d-1}(\Omega)$.
\end{thm}

In some sense, the index $i(\Omega)$ plays the role of the volume of $\Omega$, and it is also an open question to decide if we really need it. Note that the exponent of $k$ is optimal with respect to the Weyl asymptotics: to our knowledge, this is one of the few situations where one can get an optimal exponent in the upper bound without control of the curvature.
By the results of \cite{Br2022} in Example \ref{JadeBrisson} (see also the example in \cite[Section 4]{CoGi2021}), we need to take the injectivity radius into account for this inequality.

In contrast to the case of revolution manifolds, there do not exist many sharp inequalities in the case of submanifolds of the Euclidean space for which we can characterise cases of equality. One notable exception is the paper \cite{IlMa2011} by Ilias and Makhoul. In order to get a sharp inequality, they have to take the mean curvature into account. We state here part of   \cite[Theorem 1.2]{IlMa2011}.

\begin{thm} Let $\Omega$ be an immersed compact submanifold of dimension $d+1$ in Euclidean space $\R^m$ with boundary $\Sigma$. Let $H$ denote the mean curvature of $\Sigma$. 
\begin{enumerate}
\item
If $m>d+1$, then
\begin{equation} \label{ineqmean}
	\sigma_1(\Omega)\le (d+1)\frac{\vert \Omega\vert}{\vert \Sigma\vert^2}\int_{\Sigma} \vert H\vert^2dV_{\Sigma}.
\end{equation}
We have equality in (\ref{ineqmean}) if and only if $\Omega$ is a minimal submanifold of the ball $B(\frac{1}{\sigma_1})$ of radius $\frac{1}{\sigma_1}$ in $\R^m$, $\Sigma \subset \partial B(\frac{1}{\sigma_1})$, and $\Omega$ reaches the boundary of the ball orthogonaly. 
\item
If $m=d+1$, we have the same inequality, and we have equality if and only if $\Omega$ is a ball.
\end{enumerate}
\end{thm}

(We remark that the equality statement in the first item says that $\om$ is a free boundary minimal surface in the ball.  Such surfaces will be discussed in Section~\ref{sec:existence}; see in particular Definition~\ref{def.fbms}.)

The authors also observe that one can prove a similar result for submanifolds of the sphere and ask the question:
\begin{ques}\label{ques:eucl1}
Is it possible to get a similar inequality for submanifolds of hyperbolic space?
\end{ques}

Another natural question is

\begin{ques}\label{ques:eucl2}
Is it possible to generalise Inequality (\ref{ineqmean}) to other eigenvalues? Note that a similar question for the spectrum of the Laplacian was solved only recently (and partially) by Kokarev  in~\cite[Theorem 1.6]{Ko2020}.
\end{ques}

\noindent
For other results of this kind, see the two recent papers \cite{ChSh2018,ChSh2022} by Chen and Shi.

The question of constructing manifolds $\Omega$ with prescribed boundary $\Sigma$ and with arbitrarily small Steklov eigenvalue $\sigma_1$, or even $\sigma_k$, is also interesting in this context. In general, it is easy to deform a Riemannian metric in order to construct small eigenvalues (see Proposition~\ref{prop: no lower bound}). However, one cannot realise the first construction of Proposition \ref{prop: no lower bound} if the boundary is prescribed (as a Riemannian manifold). Moreover, the second construction requires a conformal perturbation, and so is not accessible if we consider submanifolds. Colbois, Girouard and M\'etras~\cite{CoGiMe2020} gave a construction of a family $\Omega_\eps\subset\R^n$ of $(d+1)$-dimensional submanifolds such that $\partial\Omega_\eps=\Sigma$ for each $\eps$ and $\sigma_k(\Omega_\eps)\xrightarrow{\eps\to 0}0$ for each $k$. This should be compared with Open Problem~\ref{open:arbitrarilylarge}.
The proof of the general case is rather technical, but its essence is captured by~\cite[ Example 4.1]{CoGiMe2020}, which we reproduce here.

\begin{ex}   \cite[Example 4.1]{CoGiMe2020} We consider the unit circle $\mathbb \Sp^1 \subset \R^2$ as a submanifold of dimension $d=1$ in $\R^3$ and we want to fill it with a surface $\Omega$ with $\sigma_1(\Omega)$ small. We will construct $\Omega$ as the graph in $\R^3$ of a function $f$ on the unit disc $\overline{\D}$. Given a smooth function $f:\overline{\D}\rightarrow\mathbb R$ vanishing on the circle $S^1=\partial\D$, let
\[\Omega_f:=\mbox{Graph of }f=\{(x,y,f(x,y))\,:\,(x,y)\in\overline{\D}\}.\]
As a test function we will use the restriction $u_f$ to $\Omega_f$ of the function defined by $\tilde{u}(x,y,z)=x$.  Note that the norm of $u_f$ on $\mathbb S^1=\partial \Omega_f$ will be a positive constant independent of $f$.

The key point is that in order to obtain $\nabla u_f$, we have to find the orthogonal projection of $\nabla\tilde{u}=(1,0,0)$ on the tangent space of $\Omega_f$. The idea is to choose a family $f_n$ of functions such that this projection tends to $0$ almost everywhere as $n \to \infty$. Concretely, the Dirichlet energy of $u_f$ is given by
\[\int_{\Omega_f}|\nabla u_f|^2=\int_{\D} \frac{1 + f_y^2}{\sqrt{1 + f_x^2+f_y^2}}
\,dxdy.\]

For $n\in\mathbb N$, define $f=f_n:\overline{\D}\rightarrow\mathbb R$ by
\[f(x,y)=\sin(nx)(1-x^2-y^2).\]

A direct computation shows that
\[\lim_{n\to\infty}\int_{\Omega_{f_n}}|\nabla u_{f_n}|^2=0.\] This implies that the Rayleigh quotient of $u_{f_n}$ 
on $\Omega_{f_n}$ tends to $0$ as $n\to \infty$, which means $\sigma_1(\Omega_{f_n})$ tends to zero as well.
 
\end{ex}

\section{Optimising eigenvalues and applications to minimal surfaces}
\label{sec:existence}

In this section we address both the existence of metrics that maximise normalised eigenvalues and geometric implications.    These introductory remarks give an overview.

Let $(M,g)$ be a closed Riemannian manifold and $(N,h)$ an arbitrary Riemannian manifold.  Recall that a map $\tau:M\to N$ is harmonic if and only if it is a critical point of the Dirichlet energy $E$.  Here $E$ is defined by
\begin{equation}\label{eq:dir en}E(\tau)=\frac{1}{2}\int_M\,|d \tau|^2 \,dV_g. 
\end{equation}
(We emphasise that $|d\tau|$ depends on both Riemannian metrics $g$ and $h$.  In local coordinates, it is given by $|d\tau|^2_p=\sum\,h^{kl}(\tau(p))g_{ij}(p)\,\partial_{x^i}(\tau_k)\,\partial_{x^j}(\tau_l)$.) A classical result of Eells and Sampson \cite{EeSa1964} says that an isometric immersion $\tau:(M,g)\to (N,h)$ is minimal (i.e., its mean curvature vanishes) if and only if $\tau$ is harmonic. 

In 1966, only two years after the aforementioned result of Eells and Sampson, Takahashi \cite{Ta1966} observed for any closed Riemannian manifold $(M,g)$ that an isometric immersion $u: M\to \sph^m\subset \R^{m+1}$ of $(M,g)$ into a sphere is minimal (equivalently harmonic) if and only if all the coordinate functions $u_i:M\to \R$ are Laplace eigenfunctions with the same eigenvalue.   Thirty years later, Nadirashvili proved for smooth surfaces that any metric on $M$ maximising the normalised first eigenvalue admits such an isometric minimal immersion into a sphere. El Soufi and Ilias later extended this result to higher dimensions and to more general critical metrics.  (We will define the notion of ``critical metric'' used by El Soufi and Ilias after the statements of the two theorems below.)

\begin{thm}\cite{Na1996}, \cite{ElIl2000}, \cite{ElIl2008}\label{thm.nada min} Let $M$ be a closed manifold of dimension $n$.  Suppose that a smooth Riemannian metric $g_0$ on $M$ is a critical point of the functional  $g\to \bla_k(M,g)=\lambda_k(M,g)|M|_g^{2/n}$.   Then (after possibly rescaling $g_0$) there exist $m\in\Z^+$ and linearly independent $\lambda_k(M,g_0)$-eigenfunctions $u_1,\dots, u_{m+1}$ such that $\sum_{i=1}^{m+1}u_i^2\equiv 1$.    The map $u=(u_1,\dots, u_{m+1}):M\to \sph^{m}$ is an isometric minimal immersion.
\end{thm}

El Soufi and Ilias \cite{ElIl2008} also proved a converse:    suppose $u:(M,g_0)\to\sph^m$ is an isometric minimal immersion.   By Takahashi's result, there exists $\lambda$ such that the coordinate functions of $u$ are $\lambda(M,g_0)$-eigenfunctions.  Let $k$ satisfy either $\lambda_{k-1}(M,g_0)<\lambda =\lambda_k(M,g_0)$ or $\lambda_k(M,g_0)=\lambda<\lambda_{k+1}(M,g_0)$.  Then $g_0$ is a critical point of the functional $g\to\lambda_k(M,g)|M|_g$.

Theorem~\ref{thm.nada min} not only gives an important application of spectral theory to minimal surface theory but also yields applications in the opposite direction.  Indeed, Nadirashvili proved this theorem as part of a program to seek a maximising metric for the first normalised eigenvalues on the 2-torus $T^2$ and the Klein bottle.    

  El Soufi and Ilias showed that $\bla_k$-conformally critical metrics for $\bla_k$  (i.e., critical points for the restriction of $\bla_k(M,\cdot)$ to a conformal class of metrics), always yield harmonic maps $u$ into spheres (but $u$ need not be an isometric or even a conformal immersion).

\begin{thm}\cite[Theorem 4.1]{ElIl2008}\label{thm: El harmonic} 
\begin{enumerate}
\item Let $M$ be a closed manifold.   Suppose that the Riemannian metric $g_0$ on $M$ is a $\bla_k$-conformally critical metric.  Then there exist linearly independent $\lambda_k(M,g_0)$-eigenfunctions $u_1,\dots, u_{m+1}$  such that $\sum_{i=1}^mu_i^2\equiv 1$.    The map $u=(u_1,\dots, u_{m+1}): M\to \sph^{m}$ is harmonic with constant energy density $\frac{1}{2}|du^2|=\frac{1}{2}\lambda_k$\textbf{}.   

Conversely, suppose that $(M,g_0)$ admits a harmonic map $u:M\to \sph^m$ whose coordinate functions are $\lambda(M,g_0)$ eigenfunctions for some $\lambda$.  Let $k$ satisfy either $\lambda_{k-1}(M,g_0)<\lambda =\lambda_{k}(M,g_0)$ or $\lambda_{k}(M,g_0)=\lambda <\lambda_{k+1}(M,g_0)$.  Then $g_0$ is a conformally critical metric for $\bla_k$.   
\end{enumerate}
\end{thm}

El Soufi and Ilias defined the notion of critical metric used in Theorems~\ref{thm.nada min} and \ref{thm: El harmonic} as follows:  They first showed \cite[Theorem 2.1(i)]{ElIl2008} that if $\{g_t\}_t$ is a family of Riemannian metrics on $M$ depending analytically on the parameter $t\in (-\epsilon,\epsilon)$, then the map $t\to \lambda_k(M,g_t)$ admits left and right derivatives at each $t$.  They then defined a metric $g$ on $M$ to be $\lambda_k$-critical if for each volume-preserving deformation $\{g_t\}_t$ depending analytically on $t\in (-\epsilon,\epsilon)$ with $g_0=g$, one has 
\begin{equation}\label{eq:crit pt deriv }\left(\frac{d}{dt}_{t=0^-}\lambda_k(g_t)\right)\left(\frac{d}{dt}_{t=0^+}\lambda_k(g_t)\right)\leq 0.\end{equation}
Equivalently, for each such deformation, either 
\begin{equation}\label{eq:crit pt def}\lambda_k(g_t)\leq \lambda_k(g)+o(t) \mbox{\,\,\,or\,else\,\,\,} \lambda_k(g_t)\geq \lambda_k(g)+o(t).\end{equation}

\begin{remark}\label{rem:crit different}Recently, Karpukhin and M\'etras \cite{KaMe2021} verified that the results above remain valid if one modifies the notion of critical metric (referred to in \cite{KaMe2021} as extremal metrics) by requiring that~\eqref{eq:crit pt def} hold for all deformations $g_t$ that are $C^\infty$ in $t$ and not necessarily analytic. 
\end{remark}

Singularities can arise when one tries to find $\bla_k$-critical or conformally critical metrics.   Kokarev \cite[``Regularity Theorem'']{Ko2014} addressed regularity properties of $\bla_k$ conformally critical metrics (and generalizations) in the case of surfaces.   Under suitable hypotheses, he showed that the singularities are isolated conical singularities whose cone angles are multiples of $2\pi$.  (See Remark~\ref{rem:con sing} for the definition of conical singularity). The existence of a harmonic map $u$ as in Theorem~\ref{thm: El harmonic} continues to hold in this case but $u$ will have branch points.   See also the work of Nadarashvili-Sire \cite{NaSi2015} and Petrides \cite{Pe2014}.

Theorem~\ref{thm.nada min} gave impetus to the very challenging study of existence of optimising metrics.   This study was deeply influenced by stunning developments in the Steklov setting.  

The groundbreaking work \cite{FrSc2016} of Fraser and Schoen first established an analogue of Nadirashvili's theorem in the Steklov setting: metrics maximising the perimeter-normalised Steklov eigenvalues among all Riemannian metrics on a given surface give rise to free boundary minimal surfaces in Euclidean balls.  The article went on to develop innovative techniques to address the very difficult question of existence of metrics that maximise the first normalised perimeter-normalised Steklov eigenvalue.  In particular, they found explicit maximising metrics in the case of the M\"obius band and the annulus.  
For orientable surfaces of genus zero with $b$ boundary components, they proved existence subject to a ``gap hypothesis'', which is currently known to be satisfied at least for infinitely many choices of $b$.  (See Subsection~\ref{max exist} below.)  Remarkably, at the time that Fraser and Schoen proved this result, the analogous question of existence of maximisers for the first area-normalised Laplace eigenvalue on closed surfaces had been resolved only for $\sph^2$ (Hersch \cite{He1970}), $\R{\mathbf P}^2$ (Li and Yau \cite{LiYa1982}),the 2-torus \cite{Na1996} and the Klein bottle  \footnote{In the case of the Klein bottle, Jakobson, Nadirashvili and Polterovich constructed the metric in \cite{JaNaPo2006} after existence was proved in \cite{Na1996}; El Soufi, Giacomini and Jazar then proved uniqueness in \cite{ElGiJa2006}. Subtle parts in the arguments for the torus and Klein bottle were completed through the work of Girouard \cite{Gi2009} and of Cianci, Karpukhin and Medvedev \cite{CiKaMe2019}.}.

Fraser's expository article \cite{Fr2020}  gives an excellent summary of \cite{FrSc2016} and related works, including the ideas behind the proofs and many applications. We also recommend \cite{FrSc2013} (which is a mix of exposition and then-new results) and the expository article \cite{Li2020} by Li. 

Influenced by the ideas in \cite{FrSc2016}, Petrides proved -- again under a ``gap'' hypothesis -- the existence of maximising metrics for the first area-normalised Laplace eigenvalue $\bla_1$ on closed, orientable surfaces of \emph{arbitrary} genus and for the first perimeter-normalised Steklov eigenvalue $\bsig_1$ on compact, oriented surfaces of \emph{arbitrary genus} with an arbitrary number $b\geq 1$ of boundary components (Ph.D. thesis \cite{Pe2015}; see also \cite{Pe2014}, \cite{Pe2018}, \cite{Pe2019}. These results were extended to include non-orientable surfaces first in the Laplace setting by Matthiesen and Siffert \cite{MaSi2021} and then, by similar methods, in the Steklov case by Matthiesen and Petrides \cite{MaPe2020}. Petrides observed that infinitely many closed surfaces satisfy the gap hypothesis and thus admit maximising metrics; Karpukhin and Stern\cite{KaSt2020} later verified the gap hypothesis for existence of $\bsig_1$-maximisers for infinitely many closed surfaces with boundary. It is currently an open question whether the gap hypotheses for $\bla_1$-maximisers and for $\bsig_1$-maximisers are always satisfied.  In all the existence results for maximisers of normalised Laplace eigenvalues on closed surfaces, the maximising metrics are proven to be smooth modulo finitely many conical singularities; for maximisers of normalised Steklov eigenvalues on surfaces, the maximising metrics can be chosen to be smooth. 

For $k>1$, Petrides in his thesis \cite{Pe2015} also proved existence under gap hypotheses of $\bla_k$-maximisers on all closed, orientable surfaces and of $\bsig_k$-maximisers on all compact, orientable surfaces with boundary. These results were later extended to non-orientable surfaces in \cite{Pe2019}, and \cite{MaPe2020}, respectively. In contrast to the case $k=1$, however, it is known that maximisers of higher order eigenvalues do not always exist and it remains an open problem to find surfaces that admit such maximisers.

Concerning the problem of maximising normalised Laplace eigenvalues within a conformal class of metrics on closed surfaces, Petrides \cite{Pe2014} proved the existence of $\bla_1$-conformally maximal metrics (with possible conical singularities) in every conformal class.  A quite different independent proof was given by Nadarashvili-Sire \cite{NaSi2015}.  The question of existence of $\bla_k$-conformal maximisers for $k>1$ was addressed by Nadarashvili and Sire \cite{NaSi2015_2} and Petrides \cite{Pe2015,Pe2018}.   Karpukhin, Nadarashvili, Penskoi and Polterovich \cite{KNPP2020} gave new proofs of these results both for $k=1$ and higher $k$ and provided excellent exposition.

Fraser and Schoen \cite{FrSc2013} provided a Steklov analogue of the first item in Theorem~\ref{thm: El harmonic} in the case of metrics that maximize the first normalized Steklov eigenvalue in a conformal class.  Recently Karpukhin and M\'etras \cite{KaMe2021} obtained Steklov analogues of Theorems~\ref{thm.nada min} and \ref{thm: El harmonic} in all higher dimensions and for more general critical metrics. Petrides addressed the question of existence of $\bsig_k$-conformal maximisers, again subject to a gap condition, for all $k\geq 1$ in \cite{Pe2019}.   

This section is organised as follows:

\begin{itemize}

\item Subsection~\ref{fbms}:  From maximising metrics to free boundary minimal surfaces.
\item Subsection~\ref{max exist}:  Existence of maximising metrics for the first normalised eigenvalue.
\item Subsection~\ref{subsec.higher?} Do maximising metrics exist for higher normalised eigenvalues?
 \item Subsection~\ref{subsec.karp-stern}  From Steklov to Laplace: asymptotics of free boundary minimal surfaces.
\item Subsection~\ref{subsec.embedded}: Spectral index of embedded free boundary minimal surfaces.
\item Subsection~\ref{applic.bounds}  Applications to Steklov eigenvalue bounds.
\item Subsection~\ref{subsec: fbms higher dim} Higher dimensions.
\end{itemize}

\subsection{From maximising metrics to free boundary minimal surfaces}\label{fbms}

\begin{defn}\label{def.fbms}~
\begin{enumerate}
\item 
 Let $\om$ be a compact Riemannian manifold with boundary and let $\B^n$ be the unit Euclidean ball in $\R^n$ and $\sph^{n-1}$ its boundary.    Recall that a smooth map $\tau:\om\to \B^n$ is proper if $\tau(\om)\cap \partial \sph^{n-1} =\tau(\partial \om)$.  \item One says that a properly and isometrically immersed submanifold $\tau:\om\to \B^n$ is a \emph{free boundary minimal submanifold} of $\B^n$ if it is a critical point of the volume functional under all deformations $\{\tau_t\}$ of $\tau$ subject only to the constraint that $\tau_t(\partial\om)\subset \sph^{n-1}$ for all $t$.   (The image of $\partial\om$ may vary freely within $\sph^{n-1}$.)  Equivalently, $\tau$ satisfies the following two conditions:
\begin{itemize}
\item $\tau$ is harmonic on the interior of $\om$ (i.e., its component functions are harmonic);
\item $\tau(\om)$ meets $\sph^{n-1}$ orthogonally.
\end{itemize}
The second condition, together with the fact that $\tau$ is assumed to be an isometric immersion, says that for $p\in\partial\om$ and for $\nu_p\in T_p(\om)$ the outward pointing unit normal vector $\partial\om$ at $p$, the image $d\tau_p(\nu_p)$ coincides with the outward unit normal to $\sph^{n-1}$ in $T_{\tau(p)}(\B^n)$.  

When the isometrically immersed free boundary minimal submanifold $\om$ is viewed as a subset of $\B^n$, it is common to say simply that $\om$ is a free boundary minimal surface in $\B^n$, suppressing mention of the inclusion mapping.
\item More generally, a proper smooth map $\tau:\om\to \B^n$ is said to be a \emph{free boundary harmonic map} if it is a critical point for the Dirichlet energy under deformations $\{\tau_t\}$, again subject only to the constraint that $\tau_t(\partial\om)\subset \sph^{n-1}$.  Equivalently,
\begin{itemize}
\item $\tau$ is harmonic on the interior of $\om$;
\item $\tau(\om)$ meets $\sph^{n-1} $ orthogonally. 
\end{itemize}

\end{enumerate}
\end{defn}

Throughout this section, we adopt the convention that all immersions and branched immersions will be understood to be proper, whether or not this is explicitly stated.

Fraser and Schoen first observed a relationship between free boundary minimal surfaces in Euclidean balls and the Steklov problem analogous to the relationship identified by Takahashi between minimal surfaces in spheres and the Laplace eigenvalue problem.

\begin{prop}\label{fbms iff stek}\cite[Lemma 2.2]{FrSc2011} A properly immersed submanifold $\tau: \om \to \B^n$ with the induced Riemannian metric is a free boundary minimal submanifold if and only if for each $i=1,\dots, n$, the composition $u_i:=x_i\circ\tau$ is a Steklov eigenfunction of $\om$ with eigenvalue one.   (Here $x_1,\dots, x_n$ are the standard coordinate functions of $\R^n$.)
\end{prop}

\begin{proof}
Minimality of the immersion is equivalent to harmonicity of $x_i\circ\tau$, $i=1,\dots, n$.   Since the outward unit normal to $\partial \B^n$ is the radial vector field $\mathbf{r}=\sum_{i=1}^n\,x_i\frac{\partial}{\partial x_i}$, the second condition in Definition~\ref{def.fbms} is equivalent to  
\[\partial_{\nu} (x_i\circ\tau) =(\partial_{\,\mathbf{r}} x_i)\circ\tau\]
i.e., \[\partial_{\nu} (x_i\circ\tau) = x_i\circ\tau \mbox{\,\,on\,\,}\partial\om.\]
\end{proof}

We emphasise that the proposition places no assumptions on the dimension of $\om$.

Proposition~\ref{fbms iff stek} tells us for any free boundary minimal surface $S$ in $\B^n$ that the coordinate functions in $\R^n$ restrict to Steklov eigenfunctions on $S$ with eigenvalue 1.  However, it does not tell us where the value 1 occurs in the spectrum.    

\begin{defn}\label{spec.index} A properly immersed free boundary minimal submanifold $\tau: \om \to \B^n$ with the induced Riemannian metric $g_0$ is said to have \emph{spectral index} $k$ if $k$ is the minimum element of $\Z^+$ such that $\sigma_k(\om,g_0)=1$.    
\end{defn}

We now focus on the case that $(\om,g)$ is a surface. In this case, the property of  harmonicity of a map $\tau: (\om,g)\to \B^n$ depends only on the conformal class of $g$ due to conformal invariance of the Laplacian.  In particular, if $\tau:(\om,g)\to \B^n$ is a proper (possibly branched) conformal immersion that is harmonic on the interior of $\om$ and if its image $\tau(\om)$ meets $\sph^{n-1}$ orthogonally, then $\tau(\om)$ with the Riemannian metric induced from $\B^n$ will be a free boundary minimal surface.  Thus we will refer to such maps as free boundary (branched) conformal minimal immersions.   Note that the map $\tau:\om \to \tau(\om)$ is a (branched) conformal cover of the minimal surface $\tau(\om)$. 

The following relationship between the area and the boundary for 2-dimensional free boundary minimal surfaces $S$ in $\B^n$, proven in \cite[Theorem 5.4]{FrSc2011}, will be used several times later in this section:
\begin{equation}\label{2A=L}|S|=\frac{1}{2} |\partial S|\geq \pi.\end{equation}   
Brendle \cite[Theorem 4]{Br2012} showed that the lower bound $\pi$ is obtained only for a flat equatorial disk.

\begin{nota}\label{nota.sig}
Given a smooth (not necessarily orientable) compact surface $\om$ with boundary, let 
\[\sigma_k^*(\om)=\sup\{\sigma_k(g)|\partial \om|_g\}\]
where the supremum is over all smooth Riemannian metrics $g$ on $\om$.   
 
Letting $\om_{\gamma,b}$ be the orientable surface of genus $\gamma$ with $b$ boundary components, we will write
\[\sigma_k^*(\gamma,b):=\sigma_k^*(\om_{\gamma,b}).\]
\end{nota}

\begin{remark}\label{rem.2A=mL}
Suppose $(S,g_0)$ is an embedded free boundary minimal surface in $\B^n$ of spectral index $k$  where $g_0$ is the induced Riemannian metric. By Definition~\ref{spec.index}), we have $\sigma_k(S,g_0)=1$.  Thus Equation~\eqref{2A=L} yields
\begin{equation}\label{eq.2A=L index}|S|_{g_0}=\frac{1}{2} |\partial S|_{g_0}=\frac{1}{2}\sigma_k(S,g_0)|\partial S|_{g_0}\leq\sigma_k^*(S).\end{equation}

More generally, suppose that $(\om,g)$ is a compact surface with boundary and that $\tau:(\om,g) \to\B^n$ is a proper (possibly branched) conformal free boundary minimal immersion whose restriction to $\pam$ is an isometric immersion. (Thus in the language of Definition~\ref{sigmaisom}, the metric $g$ is $\sigma$-isometric to the metric induced from $\B^n$ by the immersion $\tau$.)  Let $S$ denote the image of $\tau$ and let $m$ be the order of the resulting (branched) covering $\tau:\om\to S$.  If $\tau$ has spectral index $k$ then we have 
\begin{equation}\label{eq.2A=mL index}|S|_{g}=\frac{1}{2} |\partial S|_{g}=\frac{1}{2m}\sigma_k(\om,g)|\pam|_g\leq \frac{1}{2m}\sigma_k^*(\om).\end{equation}
\end{remark}

We now turn to the main results of this subsection.   We first consider an example.

\begin{ex}\label{disk fbms}
Let $\om$ be homeomorphic to a disk.   By Weinstock's Theorem~\ref{thm:Weinstock}, $\sigma_1^*(\om)=2\pi$ and the maximum is achieved by the Euclidean metric $g_0$.   Moreover, up to $\sigma$-isometry and rescaling, $g_0$ is the unique maximiser. (See Definition~\ref{sigmaisom} for the notion of $\sigma$-isometry.)  Scaling the metric so that the radius of the disk is one, then $\sigma_1(\om,g_0)=1$ and the coordinate functions $x_1$ and $x_2$ form an orthonormal basis of the $\sigma_1$-eigenspace.   The map $(x_1,x_2)\to (x_1,x_2,0)$ isometrically embeds $\om$ into $\B^3$ as the equatorial disk, a free boundary minimal surface.  

Fraser and Schoen \cite[Theorem 2.1]{FrSc2015} proved moreover that if $\om$ is a topological disk and $u:\om\to \B^n$ is any proper branched conformal minimal immersion whose image meets $\sph^{n-1}$ orthogonally, then the image $u(\om)$ is an equatorial plane disk. 

Note:   Any free boundary minimal surface in $\B^n$ can also be viewed as a free boundary minimal surface in $\B^{n+1}$ by including $\B^n$ into $\B^{n+1}$ as the equatorial $n$-ball.    When speaking of uniqueness, one usually means unique modulo such trivial inclusions.  Thus the statement above is often expressed by saying that the equatorial plane in $\B^3$ is the unique free boundary minimal surface in any $\B^n$ that is the image of a proper branched conformal minimal immersion of a disk.
\end{ex}

For all other connected surfaces, Fraser and Schoen proved the following powerful theorem:

\begin{thm}\label{thm.fs}
Let $\om$ be a compact smooth surface with boundary that is not homeomorphic to a disk.   Suppose $g_0$ is a Riemannian metric on $\om$ satisfying 
\[\sigma_1(g_0)|\partial \om|_{g_0}=\sigma_1^*(\om).\]
Then
\begin{enumerate}
\item \cite[Proposition 5.2]{FrSc2016}
The multiplicity $\operatorname{mult}\sigma_1(g_0)$ is at least three.  After rescaling $g_0$ so that $\sigma_1(g_0)=1$,   
there exist linearly independent eigenfunctions $u_1,\dots, u_n$ for $\sigma_1(g_0)$ such that $u:=(u_1,\dots, u_n)$ is a proper branched conformal immersion of $\om$ into the unit ball $\B^n$ that restricts to an isometric immersion from $\pam$ into the unit sphere $\sph^{n-1}$.   The image $u(\om)$ is a free boundary minimal surface.   

\item \cite[Proposition 8.1]{FrSc2016} If, moreover, $\om$ is orientable and has genus zero, then $n=3=\operatorname{mult}\sigma_1(g_0)$ and $u: \om \to \B^3$ is an embedding.  (Note by Remark~\ref{rem.2A=mL} that the embedded minimal surface has area $\sigma_1^*(\om)$.)

\end{enumerate}
\end{thm}

A priori, in the case of higher genus, $n$ may be strictly smaller than $\operatorname{mult}\sigma_1(g_0)$.  (See, however, a conjecture by Kusner and McGrath~\ref{conj:mcKu} later in this section.) 

\begin{remark}\label{rem: fs explanation}
The fact that one obtains a branched \emph{conformal} immersion here as opposed to a branched \emph{isometric} immersion as in Theorem~\ref{thm.nada min} is due simply to the invariance of Steklov eigenvalues on surfaces under conformal changes of metric away from the boundary.   The pullback by the branched immersion $u$ of the metric on $u(\om)$ induced by the Euclidean metric is another maximising metric $g$ for $\sigma_1^*(\om)$ that is $\sigma$-isometric to $g_0$.    If $u$ has branch points, then $g$ will have conical singularities and the conformal factor will be singular at the cone points.  As noted in Remark~\ref{rem:con sing}, these isolated singularities do not affect the conclusion that $\stek(\om,g)=\stek(\om,g_0)$.
\end{remark}

 For an outline and key ideas of the interesting proof of the first statement, see Subsection 1.4.2 of Fraser's expository article \cite{Fr2020} referenced above.    The second statement is a special case of \cite[Proposition 8.1]{FrSc2016}; we give the full statement of that proposition later in this subsection as Proposition~\ref{first implies embed}.

Fraser and Schoen generalised the first statement of Theorem~\ref{thm.fs} to higher eigenvalues on all surfaces:

\begin{thm}\label{thm.fs2}\cite[Proposition 5.2]{FrSc2016}
Let $\om$ be a compact smooth surface with boundary.   Suppose $g_0$ is a Riemannian metric on $\om$ satisfying 
\[\sigma_k(g_0)|\partial \om|_{g_0}=\sigma_k^*(\om).\]   Then the multiplicity of $\sigma_k(g_0)$ is at least three.  After rescaling $g_0$ so that $\sigma_k(g_0)=1$,   there exist linearly independent eigenfunctions $u_1,\dots, u_n$ for $\sigma_k(g_0)$ such that $u:=(u_1,\dots, u_n)$ is a proper branched conformal minimal immersion of $\om$ into the unit ball $\B^n$ that restricts to an isometric immersion from $\pam$ into $\sph^{n-1}$.   The image $u(\om)$ is a free boundary minimal surface.   
\end{thm}

We now comment on the size of $n$ in Theorems~\ref{thm.fs} and \ref{thm.fs2}.    Necessarily $n\leq  \operatorname{mult}(\sigma_k)$ where $k$ is the eigenvalue under consideration.
Karpukhin, Kokarev, and Polterovich \cite[Theorem 1.1] {KaKoPo2014} (see also Jammes \cite[Theorem 1.5]{Ja2014} and Fraser-Schoen \cite[Theorem 2.3]{FrSc2016}) obtained multiplicity bounds on the $k$th Steklov eigenvalue of any compact Riemannian surface of genus $\gamma$ with $b$ boundary components:
\begin{equation}\label{eq.mult bound}
\begin{cases}
\operatorname{mult}(\sigma_k)\leq \min(4\gamma +2k +1, 4\gamma +2b+k)\hspace{1cm}	\om \mbox{\,orientable,}\\
\operatorname{mult}(\sigma_k)\leq \min(2\gamma +2k +1, 2\gamma +2b+k)\hspace{1cm}	\om \mbox{\,non-orientable.}
\end{cases}
\end{equation}
In particular for every Riemannian surface of genus zero, one has  $\operatorname{mult}(\sigma_1)(\om,g)\leq 3$.   It was not previously known whether this maximum is attained (except in the case of the disk, where the answer is no).     In the proof of Theorem~\ref{thm.fs}, the authors showed directly that the multiplicity has to be at least three for any metric realising $\sigma_1^*(\om)$.    This fact together with Theorem~\ref{fs max exist} and Proposition~\ref{lem: inf many b} below prove that the multiplicity bound of three is indeed attained for infinitely many surfaces of genus zero.

We end this subsection by  stating Fraser and Schoen's partial analogue of Theorem~\ref{thm: El harmonic} for conformally  maximising metrics.

\begin{thm}\cite[Proposition 2.8]{FrSc2013}\label{thm: FrSc conf class}
Let $(\om,g_0)$ be a compact Riemannian surface with boundary and suppose that 
\[\sigma_k(\om,g_0)|\partial\om|_{g_0}=\sup_{g\in [g_0]}\,\sigma_k(\om,g)|\partial\om|_g\]
where $[g_0]$ is the conformal class of $g_0$.   Then there exist linearly independent $\sigma_k(\om,g_0)$-eigenfunctions $u_1,\dots, u_m$ such that $\sum_{j=1}^mu_j^2= 1$ on $\partial \om$.  Thus 
$u:=(u_1,\dots, u_m):(\om,g_0)\to \B^m$ is a free boundary harmonic map.
\end{thm}

\begin{remark}\label{rem: eigenfunc to fb harm}
It is immediate from Definition~\ref{def.fbms} that any proper map $u: (\om,g_0) \to \B^m$ by $\sigma_k$-eigenfunctions is a free boundary harmonic map.  However, given a Riemannian surface $(\om,g_0)$ and a free boundary harmonic map $\tau: (\om,g)\to \B^m$, one \emph{cannot} conclude that the component functions are Steklov eigenfunctions since the map $p\mapsto |d\tau_p\nu_p|$ need not be constant on $\partial\om$.   The question of a converse to Theorem~\ref{thm: FrSc conf class}, as well as analogues in higher dimensions, will be addressed in Subsection~\ref{subsec: fbms higher dim}.
\end{remark}
 
\subsection{Existence of maximising metrics for the first eigenvalue}\label{max exist}~

\subsubsection{Steklov maximisers on surfaces of genus zero.} 

The first general result concerning the existence of metrics realising $\sigma_1^*(\om)$ for a large class of surfaces $\om$ was proved by Fraser and Schoen:

\begin{thm}\label{fs max exist}\cite[Theorem 1.1]{FrSc2016}. 
Let $\om_{0,b}$ be the orientable surface of genus zero with $b$ smooth boundary components.   In the notation of~\ref{nota.sig}, assume that 
\begin{equation}\label{fs gap} \sigma_1^*(0, b) > \sigma_1^*(0, b-1).
\end{equation}
 Then there exists a smooth Riemannian metric $g$ on $\om_{0,b}$ such that \[\sigma_1(g)|\partial\om_{0,b}|_g=\sigma_1^*(\om_{0,b}).\]
\end{thm}

(If $b=1$, the right-hand side of Inequality~\eqref{fs gap} isn't defined, but in that case we already know that the flat disk realises $\sigma_1^*(0,1)$.)

See \cite[Theorem 1.4.4]{Fr2020} for an outline of the innovative and deep proof of Theorem~\ref{fs max exist}.  We make only brief comments here.  Let $\{g_j\}$ be a sequence of Riemannian metrics on $\om:=\om_{0,b}$ such that $\sigma_1(g_j)|\om|_{g_j}\to \sigma_1^*(\om)$.   One must show the following:
\begin{itemize}
\item[(i)] The conformal classes of these metrics stay within a compact subset of the moduli space of conformal structures.  This is where the gap inequality~\eqref{fs gap} plays an important role. 
\item[(ii)] The boundary measures do not degenerate.   The proof involves a carefully chosen choice of maximising sequence $\{g_j\}$. 

\end{itemize}

\begin{remark}\label{rem:fs gap}
Fraser and Schoen showed that the non-strict version of Inequality~\eqref{fs gap} always holds.  In fact, they showed more generally that for all $\gamma\geq 0$, one has 
\begin{equation}\label{eq: nonstrict gap}  \sigma_1^*(\gamma, b) \geq \sigma_1^*(\gamma, b-1)
\end{equation} (See the beginning of the proof of \cite[Proposition 4.3]{FrSc2016}.)    Thus if the strict inequality holds for $b$, then $\sigma_1^*(0,b)> \sigma_1^*(0,b')$ for all $b'<b$.    

Proposition 4.3 of \cite{FrSc2016} asserted that if $\sigma_1^*(0,b-1)$ is achieved by a smooth metric, then Inequality~\eqref{fs gap} holds.   This statement, induction on $b$ and  Theorem~\ref{fs max exist} led to the conclusion that a smooth maximising metric exists for all $b$.  However, Fraser and Schoen later communicated a gap in their proof of the strict inequality in Proposition 4.3 (see \cite[Appendix]{GiLa2021}).     
See also Remark~\ref{rem: errors} below.
 \end{remark}

While the general case of the gap condition remains open, one does have the following:

\begin{prop}\label{lem: inf many b}
Let $\mathcal{B}$ denote the set of positive integers for which the gap condition~(\ref{fs gap}) is satisfied.  Then $\mathcal{B}$ is infinite.
\end{prop}

\begin{proof}
The proposition follows immediately from Theorems~\ref{thm:IsopStekOnePlanar} and \ref{thm:8pi strict}.
\end{proof}

Theorems~\ref{thm.fs} and \ref{fs max exist} together imply the following:

\begin{cor}\label{cor: fbms genus 0} Let $\om$ be an orientable smooth surface of genus zero with $b$ boundary components where $b\in \mathcal{B}$ in the notation of Proposition~\ref{lem: inf many b}.   Then there exists an embedding of $\om$ in $\B^3$ as a free boundary minimal surface such that the metric $g_0$ on $\om$ induced by the Euclidean metric satisfies \[\sigma_1(g_0)|\partial\om|_{g_0}=\sigma_1^*(\om).\]
\end{cor}

(It was originally asserted in \cite{FrSc2016} that a sequence of free boundary minimal surfaces associated with maximising metrics converge to a double disk as the number of boundary components $b$ approaches $\infty$. Following communication from Fraser and Schoen regarding a subtle error in this assertion, Girouard and Lagac\'e \cite[Appendix]{GiLa2021} conjectured that these surfaces converge in the varifold sense to the boundary sphere $\sph^2$.   This conjecture was later affirmed along with a remarkable generalisation addressed in Subsection~\ref{subsec.karp-stern}.)

Fraser and Schoen also found explicit maximising metrics on two surfaces of genus zero:

\begin{thm}\label{fs max exist cat and mob}\cite[Theorem 1.1]{FrSc2016}. 
Let $\om$ be either an annulus or a M\"obius band.  Then there exists an explicit smooth Riemannian metric $g$ on $\om$ such that $\sigma_1(g)|\partial\om|_g=\sigma_1^*(\om)$.
\end{thm}

 Their constructions of the maximisers and some of the ideas behind the proof of Theorem~\ref{fs max exist cat and mob} are described in the following example.  Note in particular the interesting role played by Theorem~\ref{thm.fs}.

\begin{ex}\label
{ex.cat}\label{ex.mobius}~
(i)Except for the equatorial disk, the earliest known example of a free boundary minimal surface in $\B^3$ is the so-called critical catenoid, the intersection with $\B^3$ of the unique catenoid about the $x_3$ axis centered at the origin that meets the unit sphere $\sph^2$ orthogonally. See Figure~\ref{fig:catenoid}.  Fraser and Schoen obtained an approximate value of $\frac{4\pi}{1.2}$ for its first normalized eigenvalue (thus improving the previously best known bound of $4\pi$ for $\sigma_1^*(0,2)$ given in Theorem~\ref{thm:karp hersch}).   Observe that the gap condition $\sigma_1^*(0, 2) > \sigma_1^*(0, 1) =2\pi$ holds, so Theorem~\ref{fs max exist} guarantees existence of a maximiser for $\sigma_1^*(0,2)$.   In \cite[Theorem 6.2]{FrSc2016}, Fraser and Schoen showed that the critical catenoid is the \emph{unique}, up to congruence, free boundary minimal annulus embedded in $\B^3$ -- in fact in any $\B^n$ -- whose coordinate functions are first Steklov eigenfunctions.    Comparing this uniqueness statement with Theorem~\ref{thm.fs},  they concluded that the critical catenoid does indeed realise $\sigma_1^*(0,2)$.     

\begin{figure}
  \centering
  \includegraphics[width=3cm]{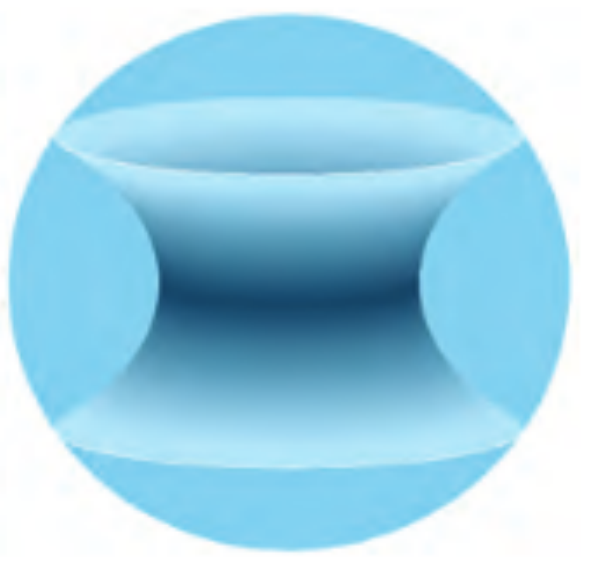}
 \caption{Critical catenoid drawn by Emma Fajeau (originally appeared in Notices of the American Mathematical Society \cite{BMWE2018}).}
  \label{fig:catenoid}
  \end{figure}
  
(ii)  Fraser and Schoen explicitly constructed a free boundary minimal embedding of the M\"obius band into $\B^4$ that is invariant under an action of $\sph^1$ by rotations.  Moreover they showed (i) that it is the unique free boundary minimal M\"obius band in any $\B^n$ that is $\sph^1$ invariant and (ii) that any free boundary minimal M\"obius band in $\B^n$ whose coordinate functions are first eigenfunctions must be $\sph^1$-invariant.  While Theorem~\ref{fs max exist} only applies to orientable surfaces, they separately proved existence and regularity of a maximising metric on the M\"obius band.  It thus followed that the metric on the M\"obius band $\om$ induced by the $\sph^1$ free boundary minimal embedding in $\B^4$ must realise $\sigma_1^*(\om)$.   As a further consequence, they found that $\sigma_1^*(\om)=2\pi\sqrt{3}$.     
\end{ex}

An early version of the article \cite{FrSc2016} containing Theorem~\ref{fs max exist} was first posted on the arXiv in 2012.  The innovative ideas introduced there quickly ignited major developments on the existence of maximising metrics for the normalised eigenvalues both for the Laplace eigenvalue problem on closed surfaces and the Steklov problem on compact surfaces with boundary. We will discuss the Laplace setting first and then return to the Steklov setting to address surfaces of genus greater than zero.

\subsubsection{Laplace maximisers.} 

Analogous to Notation~\ref{nota.sig}, for a closed surface $M$ we write 
\begin{equation}\label{nota.lam}\lambda_k^*(M)=\sup\{\lambda_k(g)|M|_g\}\text{ and  } \lambda_k^*(\gamma)=\lambda_k^*(M_{\gamma})\end{equation}
where the supremum is over all Riemannian metrics $g$ on $M$ and where $M_\gamma$ is the orientable closed surface of genus $\gamma$. 

Motivated by Theorem~\ref{fs max exist}, Petrides proved existence of maximising metrics for the first normalised eigenvalue on orientable surfaces of arbitrary genus provided a gap condition holds.

\begin{thm}\cite{Pe2014}\label{thm: lap max exists} Let $M_\gamma$ be a connected smooth oriented closed surface of genus $\gamma$.  Assume
\begin{equation}\label{eq.lagap}\lambda_1^*(\gamma)>\lambda_1^*(\gamma -1). \end{equation} 
   Then there exists a Riemannian metric $g_0$ on $M_\gamma$, smooth except for possibly a finite number of cone points, such that 
\[\lambda_1(M_\gamma,g_0)|M_{\gamma}|_{g_0}=\lambda_1^*(\gamma),\]
and there exists an associated branched minimal immersion by first eigenfunctions of $M_\gamma$ into a sphere.

Moreover, the set $\mathcal{B}_L$ of positive integers $\gamma$ that satisfy the gap condition~\eqref{eq.lagap} is infinite.
\end{thm}

Colbois and El Soufi  \cite{CoEl2003} earlier showed that the non-strict version of Inequality~\eqref{eq.lagap} always holds.  Petrides noted that the final statement of the theorem is immediate from the lower bound 
\begin{equation}\label{eq:Brma}\lambda_1^*(\gamma)\geq \frac{3}{4}(\gamma-1),\end{equation}
 which -- as pointed out in \cite{FrSc2013} -- follows from work of Brooks and Makover \cite{BrMa2001} (see also Buser-Burger-Dodziuk \cite{BuBuDo1988}).

Matthiesen and Siffert extended Petrides' result to the non-orientable setting.  Following their notation, we will use the superscript $K$ (as in Klein) to denote non-orientable.  Thus $M_\gamma^K$ will denote the closed non-orientable surface of genus $\gamma$ ($\gamma$ is sometimes referred to as the ``non-orientable genus''), and we write
\[\lambda_k^*(\gamma)^K:=\lambda_k^*(M_\gamma^K).
\]

 \begin{thm}\label{thm: lap max exists 2}\cite[Theorem 1.2]{MaSi2021}
Let $M^K_\gamma$ denote the smooth closed non-orientable surface of non-orientable genus $\gamma$.   Assume 
\begin{equation}\label{eq:lagapK}\lambda_1^*(\gamma)^K>\max\left\{\lambda_1^*(\gamma-1)^K, \lambda_1^*\left(\left\lfloor\frac{\gamma-1}{2}\right\rfloor\right)^K\right\}.\end{equation}
Then there exists a Riemannian metric $g_0$ on $M^K_\gamma$, smooth except for possibly a finite number of cone points, such that 
\[\lambda_1(M^K_\gamma,g_0)|M^K_\gamma|_{g_0}=\lambda_1^*(\gamma)^K.\]
 \end{thm}

As in the orientable case, we have

\begin{prop}\label{prop:gapK inf many}
Inequality~\eqref{eq:lagapK} holds for infinitely many $\gamma\in \Z^+$.
\end{prop}

\begin{proof}
Given a closed non-orientable manifold $M$, every Riemannian $g$ on $M$ pulls back to a Riemannian metric $\tilde{g}$ on the two-fold orientable cover $\tilde{M}$, and the Laplace spectrum of $(M,g)$ is contained in that of $(\tilde{M},\tilde{g})$.     In view of Inequality~\eqref{eq:Brma}, we thus have
\[\lim_{\gamma\to\infty}\,\lambda_1^*(\gamma)^K=\infty\]
from which the proposition follows.
\end{proof}

Explicit maximising metrics are known only for the four surfaces mentioned earlier ($\sph^2$, $T^2$, $\R {\mathbf P}^2$ and the Klein bottle) and the orientable surface of genus two.   For the latter, Nayatani and Shoda \cite{NaSh2019} verified that a singular Bolza metric realises $\lambda_1^*(2)$ as conjectured in \cite{JLNNP2005}.

\subsubsection{Steklov maximisers on surfaces of higher genus.}

Petrides generalised Theorem~\ref{fs max exist} to surfaces of genus $>0$.    

\begin{thm}\label{mp max exists}\cite[Th\'eor\`eme 16]{Pe2015},\cite[Theorem 1]{Pe2019}
 Let $\om:=\om_{\gamma,b}$ be the compact orientable surface of genus $\gamma$ with $b$ smooth boundary components.  If 
 \begin{equation}\label{eq.gap}
\sigma_1^*(\gamma,b)> \max\{\sigma_1^*(\gamma-1, b+1), \sigma_1^*(\gamma, b-1)\},
\end{equation}
then there exists a smooth Riemannian metric $g$ on $\om$ such that 
\[\sigma_1(g)|\partial\om|_g=\sigma_1^*(\om).\]
\end{thm}

(On the right-hand side of Inequality~\eqref{eq.gap}, if either of the two quantities is of the form $\sigma_1^*(\gamma', b')$ with $\gamma'<0$ or $b' <1$, then we replace $\sigma_1^*(\gamma', b')$ by zero.)
The non-strict version of Inequality~\eqref{eq.gap} always holds.

While the proof of Theorem~\ref{mp max exists} is motivated by and uses many ideas from the proof of Theorem~\ref{fs max exist}, it differs from that of Theorem~\ref{fs max exist} even in the genus zero case and thus provides a new proof of the latter theorem. 

Matthiesen and Petrides \cite{MaPe2020} proved an analog of Theorem~\eqref{mp max exists} for non-orientable surfaces.  
Let $\om^K_{\gamma,b}$ be the non-orientable surface of non-orientable genus $\gamma$ with $b$ boundary components.    Write 
\[\sigma_1^*(\gamma,b)^K=\sigma_1^*(M_{\gamma, b}^K).\]
   
\begin{thm}\cite{MaPe2020}\label{mp max exists K}
In the notation of the preceding paragraph, if 
  \begin{equation}\label{eq.gap.nonorient}
\sigma_1^*(\gamma, b)>\max\left\{\sigma_1^*\left(\left\lfloor\frac{\gamma-1}{2}\right\rfloor, b\right),\, \sigma_1^*(\gamma-1,b)^K,\, \sigma_1^*(\gamma, b-1)^K\right\},
\end{equation}
 then there exists a smooth Riemannian metric $g$ on $M_{\gamma, b}^K$ that realises $\sigma_1^*(\gamma, b)$.

 \end{thm}
 
The gap inequality~\eqref{eq.gap.nonorient} is not explicitly stated in \cite{MaPe2020} but is explained in Subsection 2.2 of the article.   

  \begin{remark}\label{rem: errors} 
In \cite{MaPe2020}, the authors put forth a proof of both Inequality~\eqref{eq.gap} and Inequality~\eqref{eq.gap.nonorient}.   The argument followed similar lines to a proof put forward in \cite{MaSi2019} of the gap inequality \eqref{eq.lagap}.  However, as we were completing this survey, we learned that the authors had very recently discovered a computational error in the proofs of all these gap inequalities.  They are working to determine whether the proofs can be corrected.   Thus while it is widely believed that all these gap inequalities hold in full generality, the problem remains open both in the Laplace and the Steklov settings at the time of this writing.      \end{remark}

Karpukhin and Stern proved that there are infinitely many pairs $(\gamma,b)$ for which the gap inequality~\eqref{eq.gap} holds and thus for which $\sigma_1^*(\gamma,b)$ is realised by a smooth metric.  More precisely, they showed:

\begin{prop}\cite[Theorem 1.8]{KaSt2020}\label{prop: KS gap inf} Let $\mathcal{B}_L$ be the infinite set defined in Theorem~\ref{thm: lap max exists}.   For each $\gamma\in \mathcal{B}_L$, there exists infinitely many $b\in \Z^+$ such that the pair $(\gamma,b)$ satisfies the Steklov gap inequality~\eqref{eq.gap}.
\end{prop}

At the time they proved this result, the set $\mathcal{B}_L$ was understood to be all of $\Z^+$ and thus the condition that $\gamma\in \mathcal{B}_L$ does not explicitly appear in their statement.   

We include the short proof.
\begin{proof} 
By Remark~\ref{rem:b to infty}, we have $\sigma_1^*(\gamma,b)\leq \lambda_1^*(\gamma)$ for all $b$ and 
\begin{equation}\label{eq:limsig}\lim_{b\to\infty}\,\sigma_1^*(\gamma,b)= \lambda_1^*(\gamma).\end{equation}
The hypothesis that $\gamma\in \mathcal{B}_L$ thus implies for all sufficiently large $b$  that \begin{equation}\label{eq:half gap}\sigma_1^*(\gamma, b)>\lambda_1^*(\gamma-1) \geq \sigma_1^*(\gamma-1, b+1). \end{equation}

Comparing ~\eqref{eq:half gap} with the Steklov gap inequality, it remains to show that 
\begin{equation}\label{eq: wts}\sigma_1^*(\gamma,b)>\sigma_1^*(\gamma, b-1)\end{equation} for infinitely many $b$.  Suppose the contrary.  Since the non-strict version of the Steklov gap inequality is known to hold, the map $b\to \sigma_1^*(\gamma, b)$ is monotone non-decreasing and thus must eventually be constant.  Equation~\eqref{eq:limsig} thus implies that 
\begin{equation}\label{eq: equality}\sigma_1^*(\gamma, b)=\lambda_1^*(\gamma)\end{equation} for all sufficiently large $b$.  Letting $b_0$ be the smallest $b$ that satisfies Equation~\eqref{eq: equality}, we have \begin{equation}\label{eq:b0}\sigma_1^*(\gamma,b_0)>\sigma_1^*(\gamma, b_0-1).\end{equation}  
The facts that $b_0$ satisfies Equation~\eqref{eq: equality} and that $\gamma\in \mathcal{B}_L$ together imply that $b_0$ satisfies Equation~\eqref{eq:half gap} and thus (by~\eqref{eq:b0}) also the Steklov gap inequality~\eqref{eq.gap}.    
Hence Theorem~\ref{mp max exists} yields a smooth metric $g$ on $\om_{\gamma,b_0}$ such that 
\begin{equation}\label{eq:contra} \sigma_1^*(\om_{\gamma,b_0},g)|\om_{\gamma,b_0}|_g=\sigma_1^*(\gamma,b)=\lambda_1^*(\gamma).\end{equation}
One may fill in the holes in $\om_{\gamma,b_0}$ and extend the Riemannian metric $g$ to a Riemannian metric on the closed surface $M_\gamma$.    Equation~\eqref{eq:contra} then contradicts the sharp inequality~\eqref{eq:karp stern}.
\end{proof}
\begin{remark}\label{rem.Mape} Suppose that $b_i$, $i=1,2$ are distinct positive integers such that each of $\sigma_1^*(\gamma,b_i)$, $i=1,2$ is realised by a Riemannian metric $g_{\gamma, b_i}$ . Theorem~\ref{thm.fs} yields 
free boundary minimal surfaces $S_i$ in balls $\B^{n_i}$ together with conformal branched coverings $u_i:(\om_{\gamma,b_i},g_{\gamma, b_i})  \to S_i$ that restrict to regular Riemannian coverings $\pam_{\gamma,b_i}\to \partial S_i$. In case $\gamma=0$, these branched coverings are homeomorphisms (see Corollary~\ref{cor: fbms genus 0}) and the $S_i$'s are distinct, in fact mutually non-homeomorphic.  In contrast, if $\gamma\geq 1$, one could have $S_i=S_j$ for some $i\neq j$.   However, Matthiesen and Petrides \cite[Theorem 1.1]{MaPe2020} showed that, for fixed $\gamma$, there can be only finitely many $b_i$ for which $(\om_{\gamma,b_i}, g_{\gamma, b_i})$ is a branched covering of a given minimal surface $S$.   The proof is an elementary consequence of Equation/Inequality~\eqref{eq.2A=mL index} along with the existence of upper bounds for $\sigma_1^*(\gamma,b)$ that are independent of $b$ as in Remark~\ref{rem:has cgr}.

 \end{remark}

 \subsection{Do maximising metrics exist for higher normalised eigenvalues?}\label{subsec.higher?}

For $k\geq 2$, we are not aware of examples of surfaces with boundary, respectively closed surfaces, for which the existence of maximisers for the $k$th normalised Steklov, respectively Laplace, eigenvalue has been established.  In contrast to the case $k=1$, the following theorem proves non-existence of maximising metrics for the higher eigenvalues on the disk $\om_{0,1}$.

\begin{thm}\label{no max}\cite[Theorem 2.3]{FrSc2020} 
For $k\geq 2$, the value $\sigma_k^*(0,1)$ is not attained by any smooth Riemannian metric.    
\end{thm}

The theorem is an immediate consequence of Theorem~\ref{thm.fs2} and the fact that the equatorial disk is the only free boundary minimal disk in any Euclidean ball as noted in Example~\ref{disk fbms}.

The case $k=2$ was proven earlier by Girouard and Polterovich \cite{GiPo2010}.   

\begin{remark}\label{rem: bubbles} Recall that the Hersch-Payne-Schiffer inequality states that $\sigma_k^*(0,1)\leq 2\pi k$.  In the article just cited, Girouard and Polterovich   constructed for each $k\geq 2$ a family of simply-connected plane domains $\om_{\epsilon}$ such that $\lim_{\epsilon\to 0} \sigma_k(\om_{\epsilon})|\pam_{\epsilon}|\to 2\pi k$.   The domains $\om_{\epsilon}$ converge as $\epsilon\to 0$ to a union of $k$ touching disks.  As conjectured by Nadirashvili \cite{Na2002} in 2002 and recently proved by Karpukhin, Nadirashvili, Penskoi and Polterovich \cite{KNPP2021}, analogous behaviour occurs in the case of the Laplace eigenvalues of metrics on the 2-sphere: $\lambda_k^*(\sph^2)=8k\pi$, and there exists a  maximising sequence of metrics converging to the union of $k$ touching spheres each with the standard round metric. Moreover, for $k\geq 2$, the value $8k\pi$ is not achieved within the class of all metrics on $\sph^2$ that are smooth except possibly for finitely many conical singularities.
\end{remark}

\begin{ques}\label{ques:any higher maximisers} Can one find examples of compact surfaces with boundary and integers $k>1$ for which a $\sigma_k$-maximising metric exists.

\end{ques}
 For arbitrary orientable compact surfaces with smooth boundary and for each positive integer $k$,  Petrides introduced an expression that we will refer to as $\gap_k(\om)$ (defined below).  The gap precludes the type of ``bubbling'' phenomenon just described in the case of the disk.   
 
Generalising the results described in the previous subsection, Petrides proved the following:

\begin{thm}\label{pe gap thm}\cite[Th\'eor\`eme 16]{Pe2015},\cite[Theorem 1]{Pe2019} Let $\om$ be an orientable surface with smooth boundary.  If $\gap_k(\om) >0$, then there exists a smooth Riemannian metric on $\om$ realising $\sigma_k^*(\om)$.
\end{thm}

\begin{remark}\label{rem:pe2018}
In \cite{Pe2018}, Petrides proved an analogous result in the Laplace setting, but we will focus here only on the Steklov case.
\end{remark}
 
The definition of $\gap_k(\om)$ appears in the right-hand side of inequality (0.2) in \cite{Pe2019}.     Inequality~\eqref{eq.gapk} below corrects a typo in that expression.   We wish to thank Petrides both for providing the corrected expression and for the following more intuitive definition of $\gap_k$.

\begin{defn}\label{def.gap} Let $\mathcal{S}$ be the collection of all surfaces $\hat{\om}$ that can be obtained from $\om$ by cutting along a non-empty finite collection of embedded closed curves $\tau:[0,1]\to \om$ with $\tau(]0,1[)\subset \operatorname{int}(\om)$ and $\tau(0),\tau(1)\subset \partial\om$.   We identify two surfaces if they are homeomorphic.  Set 
\[\gap_k(\om)=\sigma_k^*(\om)-\max_{\hat{\om}\in \mathcal{S}}  \sigma_k^*(\hat{\om}).\]  
\end{defn}

\begin{remark}\label{rem.gap}  For $\om_{\gamma,b}$ as in Notation~\ref{nota.sig}, a cut in $\om_{\gamma,b}$ along a single curve $\tau$ as above  will result in one of the following:
\begin{enumerate}
\item A surface with two components $\om_{\gamma_1,b_1}$ and $\om_{\gamma_2,b_2}$ satisfying $\gamma_1 +\gamma_2\leq \gamma$ and $b_1+b_2=b+1$ .   The curve $\tau$ appears as a boundary component of both.   
\item The connected surface $\om_{\gamma-1,b+1}$.    
\item The connected surface $\om_{\gamma, b-1}$.
\end{enumerate}
Cuts that yield surfaces of either of the first two types arise from curves $\tau$ both of whose endpoints lie on a single boundary component; for the third type, the endpoints of $\tau$ must lie on different boundary components.  
\end{remark}

When $k=1$, the condition ``$\gap_1(\om) >0$'' is equivalent to Inequality~\eqref{eq.gap}.

Beginning with a surface $\om=\om_{\gamma, b}$, one sees from Remark~\ref{rem.gap} that the set $\mathcal{S}$ in Definition~\ref{def.gap} consists of all surfaces distinct from $\om_{\gamma,b}$ itself of the form $\om_{\gamma_1,b_1}\sqcup\dots \sqcup \om_{\gamma_s,b_s}$, where $s\in \Z^+$,  $\gamma_1+\dots \gamma_s\leq \gamma$ and $\gamma_1+\dots +\gamma_s+b_1+\dots b_s\leq \gamma+b +s-1$.  Inducting on the number of components and taking into account the effect of the length normalisation in computing $\sigma_k^*(\hat{\om})$ when $\hat{\om}$ is not connected, one can express the condition $\gap_k(\om_{\gamma,b})>0$ as follows:

\begin{equation}\label{eq.gapk} \sigma_k^*(\gamma,b) >\max_{\substack{i_1+\dots +i_s=k\\i_j\geq 1\,\,\,\forall j\\\gamma_1+\dots +\gamma_s\leq \gamma\\\gamma_1+\dots +\gamma_s+b_1+\dots +b_s\leq \gamma+b+s-1\\(\gamma_1,b_1)\neq (\gamma, b)\,\, \rm{if}\,\,\,s=1
}}\, \sum_{q=1}^s\,\sigma_{i_q}^*(\gamma_q,b_q)
\end{equation}
Fraser and Schoen \cite[Corollary 4.11]{FrSc2020} proved that the non-strict verstion of Inequality~\eqref{eq.gapk} is always satisfied.

By Theorem~\ref{pe gap thm}, any positive results on the following question would provide answers to Open Question~\ref{ques:any higher maximisers}:

\begin{ques}\label{ques:gapk} For given $k$, which compact surfaces $\om$ satisfy $\gap_k(\om)>0$?

\end{ques}

We emphasise that when $k=1$, it is widely expected that the gap inequality is always satisfied. Proposition~\ref{prop: KS gap inf} lends support to this expectation.  In contrast, Theorems~\ref{no max} and \ref{pe gap thm} show that the gap condition fails for all $k\geq 2$ in the case of the topological disk, and the question is wide open for other surfaces.

\subsection{Existence results for conformally maximising metrics}\label{subsec:conf max}

We begin with the case of the Laplacian.  Given a closed surface $M$ and a conformal class $[g]$ of metrics on $M$, let 
\begin{equation}\label{eq;lak[g]}
\lambda_k^*(M,[g])=\sup_{h\in g}\,\lambda_k(M,h)|M|_h  
\end{equation}

Nadarashvili and Sire first put forth an argument in \cite{NaSi2015} to prove existence of a metric realising $\lambda_1^*(M,[g])$ under the condition that $\lambda_1^*(M,[g])>8\pi$.  In \cite{NaSi2015_2}, they similarly put forward an argument to realise $\lambda_k^*(M,[g])$ for $k\geq 2$ under the assumption of a gap condition.  Influenced by the results of Fraser and Schoen, Petrides \cite{Pe2014} gave a complete and very different existence proof for a maximiser of $\lambda_1^*(M,[g])$ without a gap hypothesis.   His proof of existence of a conformal maximiser was the first step in his proof of Theorem  ~\ref{thm: lap max exists} discussed above.  Similarly, as a step in the result cited in Remark~\ref{rem:pe2018}, he proved existence modulo a gap condition of metrics realising $\lambda_k^*(M,[g])$ for every $k\geq 2$. 

Karpukhin, Nadarashvili, Penskoi and Polterovich's resolution \cite{KNPP2020} of Nadarashvili's conjecture (see Remark~\ref{rem: bubbles} above) enabled them to give a simpler expression for the gap condition.   They also completed the arguments put forward by Nadarashvili and Sire (using results of Grigor'yan, Nadarashvili and Sire \cite{GrNaSi2016}).  We state here the version appearing in \cite{KNPP2020}.

\begin{thm}\cite{NaSi2015, NaSi2015_2,Pe2014, Pe2018, KNPP2020}\label{thm:ex lap conf max}
Let $M$ be a closed surface and let $[g]$ be any conformal class of metrics on $M$.  Then:
\begin{enumerate}
\item For every $k\geq 1$, either there exists a 
metric $h\in[g]$ that is smooth except possibly for finitely many singularities such that 
\begin{equation}\label{eq:conf max 1}\lambda_k(M,h)|M|_h=\lambda_k^*(M,[g])>\lambda_{k-1}^*(M,[g])+8\pi\end{equation}
or else 
\begin{equation}\label{eq:conf max 2}\lambda_k^*(M,[g])=\lambda_{k-1}^*(M,[g])+8\pi.\end{equation}
\item For $k=1$, the first case always holds; i.e., there exists $h$ as above satisfying Equation~\eqref{eq:conf max 1}.
\end{enumerate}

\end{thm}
The second statement follows from the first together with Petrides' result \cite[Theorem 1]{Pe2014} that 
\begin{equation}\label{eq:>8pi}
\lambda_1^*(M,[g])>8\pi 
\end{equation}
except when $M$ is diffeomorphic to a sphere.

The article \cite{KNPP2020} gives an informal interpretation of the first item in the theorem as follows:  Either a metric realising $\lambda_k^*(M,g)$ exists or else for some $j$ with $1<j<k$, a maximising sequence of metrics for $\lambda_k^*(M,[g])$ degenerates to a disjoint union of the surface $(M,h_j)$, where $h_j\in [g]$ where $h_j$ is a maximising metric for $\lambda_{k-j}^*(M,[g])$, together with $j$ identical round spheres (''bubbles'') each of volume $\frac{8\pi}{\lambda_k^*(M,g)}$.  This informal statement is clarified in \cite[Subsection 5.1]{KNPP2020}.

Note that the first item says that a metric realising $\lambda_k^*(M,[g])$ exists provided that the gap condition $\lambda_k^*(M,[g])>\lambda_{k-1}^*(M,[g])+8\pi$ is satisfied.  The reader may find both this formulation of the first item and also the informal statement in the previous paragraph helpful in comparing Theorem~\ref{thm:ex lap conf max} with Theorem~\ref{thm:ex stek conf max} below.

Next consider the Steklov problem on compact surfaces $\om$ with boundary.  Given a conformal class $[g]$, let 
\begin{equation}\label{eq;stek[g]}
\sigma_k^*(\om,[g])=\sup_{h\in [g]}\,\sigma_k(\om,h)|\pam|_h.   
\end{equation}
As the first step in his proof of Theorem~\ref{mp max exists}, Petrides proved the following analogue of the first item in Theorem~\ref{thm:ex lap conf max}:

\begin{thm}\label{thm:ex stek conf max}\cite[Theorem 2]{Pe2019}
Let $\om$ be a connected compact surface with boundary and $[g]$ a conformal class of metrics on $\om$.  If 
\begin{equation}\label{eq: conf stek gap} \sigma_k^*(\om,[g])> \max_{1\leq j<k}\,\sigma_{k-j}^*(\om,[g]) +2\pi j
\end{equation}
then there exists a smooth Riemannian metric $h\in [g]$ such that 
\[\sigma_k(\om,h)|\pam|_h=\sigma_k^*(\om,[g]).\]
\end{thm}

As before, the stronger conclusion that the metric is smooth is due to the fact that any interior conical singularities can be conformally removed without affecting the spectrum.   

The proof uses a special choice of maximising sequence and a careful argument using the relationship between conformally critical metrics and harmonic maps. The gap hypothesis prevents disks from bubbling off.

Observe that when $k=1$, existence of a maximising metric in $[g]$ would follow from an affirmative answer to the following question:

\begin{ques}\label{ques:sig >2pi} Is $\sigma_1^*(\om,[g])>2\pi$ for every conformal class $[g]$ when the surface $\om $ is not diffeomorphic to a disk? 
\end{ques}

The inequality in the question above would be the Steklov analogue of Inequality~\eqref{eq:>8pi}.
 
  \subsection{From Steklov to Laplace: asymptotics of free boundary minimal surfaces}\label{subsec.karp-stern}

 Let $M$ be a closed surface (not necessarily orientable) and let $M_b$ be the compact surface with boundary obtained by removing $b$ disjoint disks from $M$.  Thus for example, if $M=M_\gamma$, the closed orientable surface of genus $\gamma$, then $M_b=\om_{\gamma, b}$, the orientable surface of genus $\gamma$ with $b$ boundary components.   
 
Assume
\begin{itemize}
\item[(H1)] There exists a Riemannian metric (possibly with conical singularities) realising $\lambda_1^*(M)$.
\item[(H2)] There exist infinitely many positive integers $b$ for which $\sigma_1^*(M_b)$ is realised by a Riemannian metric on $M_b$.   In what follows, we will let $\{b_j\}$ be an increasing infinite sequence of such integers and $g_j$ will denote a Riemannian metric on $M_{b_j}$ realising $\sigma_1^*(M_{b_j})$.
\end{itemize}
 
 By Propositions~\ref{lem: inf many b} and \ref{prop: KS gap inf}, there are infinitely many closed surfaces $M$ for which these hypotheses are satisfied.  

 We've seen
 \begin{itemize}
 \item[(i)] $\lim_{j\to\infty}\, \sigma_1^*(M_{b_j}) = \lambda_1^*(M)$.  (See Remark~\ref{rem:b to infty}.)   
 \item[(ii)] For each $j$, there exists a proper branched conformal minimal immersion $\tau_j:(M_{b_j},g_j)\to \B^{n_j}$ whose image $S_j$ is a free boundary minimal surface, and $\tau_j$ restricts to an isometric immersion from $\partial M_{b_j}$ to $\sph^{n_j-1}$.  (See Theorem~\ref{thm.fs}.)

   \end{itemize}
   
The multiplicity of $\sigma_1(M_{b_j}, g_b)$, and thus the value of $n_j$ in item (ii), is bounded above by an integer $n$ depending only on the genus of $M$.   (See  the multiplicity bound ~\eqref{eq.mult bound}.)  By the observation at the end of Example~\ref{disk fbms}, the free boundary minimal surfaces $S_j$ in (ii) may be viewed as free boundary minimal surfaces in $\B^n$. Thus we may replace $n_j$ by $n$ for all $j$.
 
   Karpukhin and Stern proved the following striking asymptotic result:
   
 \begin{thm}\label{thm.KS asymptotics}\cite[Theorem 1.1]{KaSt2021} Assume $M$ satisfies the hypotheses (H1) and (H2).   Up to a choice of a subsequence of $\{b_j\}$, the free boundary minimal surfaces $\tau_j:M_{b_j}\to \B^n$ given by (ii) above converge in the varifold sense to a (branched) minimal surface $\tau:M\to \sph^{n-1}$ inducing a Riemannian metric (possibly with conical singularities) on $M$ realising $\lambda_1^*(M)$.
 
 As a consequence, their supports $S_j$ converge in the Hausdorff sense to that of the limit surface, and their boundary measures converge to twice the area measure of the limit surface.
 
 \end{thm}
 
 See Subsection 2.6 of \cite{KaSt2021} for a brief summary of the concept of convergence in the varifold sense.   Compare the final statement of the theorem with Equation~\eqref{2A=L}.  
 
As an example (see \cite[Corollary 1.4]{KaSt2021}), let $M$ be the topological 2-sphere.  By the last statement in Theorem~\ref{thm: lap max exists} along with Proposition~\ref{prop: KS gap inf}, $M$ satisfies (H1) and (H2).  Note that $M_b=\om_{0,b}$.  By Remark~\ref{rem:b to infty}, we have 
\[\lim_{b\to\infty}\sigma_1^*(0,b)=8\pi =\lambda_1^*(M)=\lambda_1(\sph^2)|\sph^2|\]
 where $\sph^2$ is the round unit 2-sphere.   By Theorem~\ref{thm.fs}, the maximising metric for the first Steklov eigenvalue on $\om_{0,b_j}$ (where $b_j$ is given as in (H2)) is induced by an embedding of $\om_{0,b_j}$ as a free boundary minimal surface in $\B^3$ for each $j$. Theorem~\ref{thm.KS asymptotics} states in this case that (up to a choice of  subsequence) the resulting sequence of free boundary minimal surfaces converges in the varifold sense to the boundary sphere $\sph^2$ itself, as was conjectured in \cite{GiLa2021}.
 
The following is one of many interesting open questions raised by Karpukhin and Stern in \cite{KaSt2021}.  

\begin{ques}\label{ques:ksoq3}\cite[Open Question 3]{KaSt2021}. In the setting of Theorem~\ref{thm.KS asymptotics}, if the limiting surface in $\sph^{n-1}$ realising $\lambda_1^*(M)$ is embedded, does it necessarily follow that the minimal surfaces in $\B^{n}$ realising $\sigma_1^*(M_{b_j})$ are embedded for all sufficiently large $j$?
\end{ques}

\subsection{Spectral index of embedded free boundary minimal surfaces}\label{subsec.embedded}

Knowledge of the spectral index of a free boundary minimal surface, as defined in Definition~\ref{spec.index}, has many applications.  We have already seen one application in Remark~\ref{rem.2A=mL}, and we will discuss an application to obtaining eigenvalue bounds in the next subsection.

 A. Fraser and M. Li conjectured:

\begin{conj}\label{fraser li conj}\cite[Conjecture 3.3]{FrLi2014} If $\om$ is a properly embedded free boundary minimal hypersurface in $\B^n$, then $\sigma_1(\om)=1$; i.e,. $\om$ has spectral index one. \end{conj}

This conjecture is analogous to a longstanding conjecture of S. T. Yau \cite{Ya1982} asking whether every embedded minimal surface in $\sph^3$ is embedded by first Laplace eigenfunctions.

We emphasise that if the conjecture holds, then for any choice of $\gamma$ and $b$, Inequality~\eqref{eq.2A=L index} yields an upper bound of $\frac{1}{2}\sigma_1^*(\gamma,b)$ on the area of any embedded minimal surface in $\B^n$ of  given genus $\gamma$ with $b$ boundary components.

In support of the conjecture, Fraser and Li \cite[Corollary 3.2]{FrLi2014} proved that every properly embedded free boundary minimal hypersurface in $\B^n$ satisfies $\sigma_1(\om)\geq \frac{1}{2}$.  Batista and Cunha \cite{BaCu2016} showed that this inequality is strict.  

Much work on the conjecture has focused on the case of $n=3$.    Lee and Yeon \cite{LeYe2021} presented new approaches to the conjecture; e.g., they showed that the Gauss map of any embedded free boundary minimal annulus in $\B^3$ is one-to-one and re-interpreted the conjecture as the problem of determining the Gauss map.    
McGrath \cite[Theorem 4]{Mc2018} affirmed the conjecture for free boundary minimal surfaces in $\B^3$ that satisfy certain symmetry conditions; these conditions are satisfied by large families of examples in the literature.   Girouard and Lagac\'e \cite[Theorem 1.13]{GiLa2021} applied McGrath's result and other tools to verify the conjecture for all free boundary minimal surfaces in $\B^3$ of genus zero with tetrahedral, octahedral, or isosahedral symmetry and $b$ boundary components where $b\in \{4\}, \{6,8\}, \{8,12,20\}$, respectively.  

Recently, again in the case $n=3$, Kusner and McGrath \cite{KuMc2022} proved the conjecture for surfaces satisfying substantially weaker symmetry conditions.   These new results encompass the results of \cite{Mc2018} and \cite{GiLa2021} and apply to many more types of examples. Moreover, under their symmetry assumptions, they prove that not only is the spectral index one as conjectured but also that the coordinate functions of the embedded surface span the full $\sigma_1$ eigenspace.  They then put forth the following stronger version of the Fraser-Li conjecture:

\begin{conj}\label{conj:mcKu}\cite[Conjecture 5]{KuMc2022} The first Steklov eigenspace of any properly embedded free boundary minimal surface $S$ in $\B^n$ coincides with the span of the coordinate functions of the embedding.

\end{conj}

The first item of the following proposition yields a partial converse to the Fraser-Li conjecture:

\begin{prop}\label{first implies embed}\cite[Proposition 8.1]{FrSc2016}
Let $\om$ be a Riemannian surface of genus zero with more than one boundary component. Suppose that $\tau:\om\to \B^3$ is a branched minimal immersion satisfying the free boundary condition and that $\tau(\om)$ has spectral index one.   Then
\begin{enumerate}
\item $\tau$ is an embedding;
\item $\tau(\om)$ does not contain the origin;
\item Any ray through the origin intersects $\tau(\om)$ in at most one point;
\item $\om$ is a stable minimal surface with area bounded above by $4\pi$.    
\end{enumerate}
\end{prop}

The proof of the first three items uses the nodal domain theorem along with the maximum principle for harmonic functions.   The area bound follows from Kokarev's bound $\sigma_1^*(\om)\leq 8\pi$ for all surfaces of genus zero (see Theorem~\ref{thm:KokarevGenus0}).  

As a consequence of Corollary~\ref{cor: fbms genus 0}, and the fact \cite{GiHeLa2021} that $\lim_{b\to\infty}\, \sigma_1^*(0,b)=8\pi$ (see Remark \ref{rem:b to infty}), one has the following:

\begin{cor}
The bound $4\pi$ in Proposition~\ref{first implies embed} is sharp.
\end{cor}

The literature contains many constructions of embedded free boundary minimal surfaces in $\mathbb B^3$.  In many cases the spectral index is not known.    We give a sampling here.
\begin{itemize}
\item Folha, Pacard and Zolotareva \cite{FoPaZo2017} gave examples of surfaces of genus zero with $b$ boundary components for each large $b$, converging to the double disk as $b\to \infty$.  Although their construction is a priori different from the construction by Fraser-Schoen \cite{FrSc2016} of free boundary minimal surfaces converging to the double disk, they conjectured that their surfaces are congruent to those of Fraser-Schoen.  They also proved the existence for all large $b$ of free boundary minimal surfaces in $\mathbb B^3$ of genus one with $b$ boundary components converging as $b\to\infty$ to a double copy of the unit equatorial disk punctured at the origin.  The convergence is uniform on compact subsets of $\mathbb B^3-\{0\}$.   McGrath in the work \cite{Mc2018} cited above verified that all the surfaces in both of these families have spectral index one.
\item Ketover \cite{Ke2016} and also Kapouleas and Li \cite{KaLi2021} constructed surfaces of genus $\gamma$ with three boundary components converging as varifolds to the union of the disk and the critical catenoid  as $\gamma\to \infty$.  
\item Carlotto, Franz and Schulz \cite{CaFrSc2020} construct examples of every genus with one boundary component and with dihedral symmetry.
\end{itemize}

\subsection{Applications of free boundary minimal surfaces to finding Steklov eigenvalue bounds and maximising metrics}\label{applic.bounds}

Currently, the disk, annulus and M\"obius strip are the only surfaces with boundary for which explicit maximising metrics have been determined.    As seen in Example~\ref{ex.cat}, the interplay with free boundary minimal surfaces played a critical role in both the case of the annulus and M\"obius strip.

It is rare that one can establish the type of uniqueness statements that were used in Example~\ref{ex.cat} to find the actual maximiser of $\sigma_k^*(\om)$.    However, any embedding of $\om$ as a free boundary minimal surface of spectral index $k$ in a Euclidean ball $\B^n$ can be used to obtain a lower bound for $\sigma_k^*(\om)$, since by Equation~\eqref{2A=L} and Proposition~\ref{fbms iff stek}, we have  
\[2|\om|_g= |\partial \om|_g=\sigma_k(\om,g)|\partial \om|_g\leq\sigma^*_k(\om)\]   where $g$ is the metric  on $\om$ induced by the Euclidean metric on $\B^n$.   In particular, in view of the many examples in the literature of embedded free boundary minimal surfaces in $\B^3$, the Fraser-Li Conjecture~\ref{fraser li conj} has strong implications for eigenvalue bounds.
:

If the Fraser-Li Conjecture~\ref{fraser li conj} is affirmed, then the known examples of embedded free boundary minimal surfaces in $\B^3$ will yield lower bounds for $\sigma_1^*(\om)$ for surfaces of many topological types.

We now consider higher eigenvalues.   Interesting results on maximisers have been obtained when one restricts to a special class of metrics. Given a surface that admits a circle action, let $\mathcal{S}(\om)$ denote the set of $\mathbb S^1$-invariant smooth metrics on $\om$ and define
\[\sigma_k^{\mathbb S^1}(\om):=\sup_{g\in \mathcal{S}(\om)}\,\sigma_k(\om,g)|\partial \om|_g.\]

Fan, Tam and Yu \cite{FaTaYu2015}  considered rotationally symmetric Riemannian metrics on the topological annulus $\om_{0,2}$.  We summarise their results:
\begin{enumerate}
\item $\sigma_2^{\mathbb S^1}(\om_{0,2})=4\pi$.    No metric in $\mathcal{S}(\om_{0,2})$ attains this supremum.  As $T\to\infty$, the normalised second eigenvalue on the cylinder $[0,T]\times \sph^1$ converges to $4\pi$.  
\item For every $k>2$, they both find $\sigma_k^{\mathbb S^1}(\om_{0,2})$ and find explicit metrics  $g\in\mathcal{S}(\om_{0,2})$ that attain $\sigma_k^{\mathbb S^1}(0,2)$.
\item For $k>2$ odd,  the maximising metric for $\sigma_k^{\mathbb S^1}(\om_{0,2})$ is induced by an immersion of $\om_{0,2}$ into $\B^3$ whose image is a free boundary minimal surface, the so-called critical $k$-catenoid.
\item For $k>2$ even, the metric is induced by an immersion of $\om_{0,2}$ into $\B^4$ whose image is a free boundary minimal M\"obius strip.
\end{enumerate}

Subsequently, Fraser and Sargent \cite{FrSa2021} classified all free boundary rotationally symmetric minimal annuli and M\"obius bands in $\B^n$ for arbitrary $n$.   We will denote the topological M\"obius strip by $\text{M\"ob}$.  Some highlights of their results are:

\begin{enumerate}
\item For all $k\geq 1$, they show that $\sigma_{2k-1}^{\mathbb S^1}(\mob)=\sigma_{2k}^{\mathbb S^1}(\mob)$.   Moreover they compute this value explicitly and show that it is attained by a free boundary minimal embedding of $\mob$ in $\B^4$.    
\item Every rotationally symmetric free boundary minimal annulus or M\"obius band in any $\B^n$ is critical for some $k$-th normalised eigenvalue within the space of rotationally symmetric Riemannian metrics. 

\end{enumerate}

Example~\ref{ex.cat} shows that for both the annulus and the M\"obius band, 
the supremum of the first normalised Steklov eigenvalue is attained by a rotationally symmetric  metric.  In contrast, for all $k\geq 2$,
Fraser and Schoen \cite[Theorem 3.1]{FrSc2020} showed that 
\[\sigma_k^{\mathbb S^1}(\om_{0,2})<\sigma_k^*(\om_{0,2})\mbox{\,\,\,and\,\,\,}\sigma_k^{\mathbb S^1}(\mob)<\sigma_k^*(\mob)\]
for all $k\geq 2$.

\subsection{Higher dimensions}\label{subsec: fbms higher dim}

In this final subsection, we address Karpukhin and M\'etras' extensions to arbitrary dimension of Theorems~\ref{thm.fs} and \ref{thm: FrSc conf class}.

For motivation, we will first consider the setting of the Laplacian on closed manifolds.

\subsubsection{Laplace conformally critical metrics and $n$-harmonic maps.}

 Let $(M,g)$ be a closed Riemannian manifold.  By the variational formula for the Dirichlet energy, a map $u:(M,g)\to\sph^m$ is harmonic if and only if
\[\Delta_g u=|du|_g^2u.\]

\begin{defn}\label{def:eigenmap}   A map $u: (M,g) \to S^m$ is said to be a $\lambda$-\emph{eigenmap} if $u$ is harmonic and the component functions $u_1,\dots, u_{m+1}$ are $\lambda(M,g)$-Laplace eigenfunctions.   Equivalently, $u$ is harmonic with constant energy density $|du|^2\equiv\lambda $. \end{defn}

Now suppose that $\dim(M)=2$ and let $u:(M,g)\to \sph^m$.   Since the Dirichlet energy is a conformal invariant in dimension two, the property of being harmonic depends only on the conformal class of $g$.   
 We say that $u$ is non-degenerate if $|du|_g^2$ is nowhere zero.   In this case we can define
\[g_u=|du|^2_g g.\]
Since $\dim(M)=2$, we have 
\begin{equation}\label{eq:Delta gu} \Delta_{g_u}u = \frac{1}{|du|^2_g}\Delta_gu =u,\end{equation}
so  $u$ is an eigenmap with eigenvalue one with respect to $g_u$.    Moreover,
Theorem~\ref{thm: El harmonic} tells us that $g_u$ is $\bla_k$-conformally critical where $k$ is chosen so that either $\lambda_k(M,g_u)=1<\lambda_{k+1}(M,g_u)$ or $\lambda_{k-1}(M,g_u)<1=\lambda_{k}(M,g_u)$.

An insight of Matthiesen~\cite{Ma2019} is that in any dimension $n\geq 2$, one can get a similar result by replacing harmonic maps by $n$-harmonic maps, defined as follows:

\begin{defn}\label{def:nharm}  Let $(M,g)$ and $(N,h)$ be Riemannian manifolds with $M$ compact.  The $n$-energy of a map $u:M\to N$ is defined by 
\[E_n(u)=\int_M\,|d u|_g^n \,dV_g.\]
(We again emphasise that, as in Equation~\eqref{eq:dir en}, the integrand depends on both metrics $g$ and $h$.)
The map $u$ is said to be $n$-\emph{harmonic} if it is a critical point of $E_n$.  \end{defn}

 Note that when $n=2$, $E_n$ agrees with the Dirichlet energy (see Equation~\eqref{eq:dir en}) except for the missing coefficient $\frac{1}{2}$.  (Some authors, including Karpukhin and M\'etras, include a coefficient of $\frac{1}{n}$ in the definition of $n$-energy.  We followed the convention of \cite{Ma2019} here.)
In particular, a 2-harmonic map is the same as a harmonic map.  There are numerous references on $n$-harmonic maps; \cite{Ta1994} addresses the relationship with minimal immersions.

We will be interested in the case $n=\dim(M)$.   In this case, one has for any $m\in\Z^+$ that
\begin{itemize}
\item A map $u: (M,g)\to \R^m$ is $n$-harmonic if and only if
\begin{equation}\label{eq:nharm R} \delta_g(|du|_g^{n-2}du)=0.\end{equation}
where $\delta_g$ is the dual to $d$.
\item A map $u: (M,g)\to \sph^m$ is $n$-harmonic if and only if
\begin{equation}\label{eq:n harm}\delta_g(|du|_g^{n-2}du) =|d u|_g^n u
\end{equation}
\end{itemize}

The $n$-energy is a conformal invariant for $n$-dimensional manifolds $M$.   Moreover, we have the following:

\begin{lemma}\cite{Ma2019}\label{lem: eigenmap}
Let $(M,g)$ be an $n$-dimensional closed Riemannian manifold.  
\begin{enumerate}
\item  If $u:M\to \sph^m$ is an eigenmap with respect to $g$, then $u$ is necessarily both $n$-harmonic and harmonic. 
\item If $u:(M,g)\to \sph^m$ is a non-degenerate $n$-harmonic map, then $u$ is an eigenmap with respect to the conformally equivalent metric $g'=|du|^2_g g$.
\end{enumerate}
\end{lemma}

The first statement is immediate from Equation~\eqref{eq:n harm} since eigenmaps have constant energy density, and the proof of the second statement is similar to the case $n=2$ above.

Karpukhin and M\'etras \cite{KaMe2021} observed that Lemma~\ref{lem: eigenmap} enables Theorem~\ref{thm: El harmonic} to be reformulated as in the corollary below and also showed that the result remains valid with the notion of critical metric alluded to in Remark~\ref{rem:crit different}.  

\begin{cor}(See \cite[Theorem 1]{KaMe2021}.) \label{cor: KM Lap}
Let $(M,g)$ be a compact $n$-dimensional Riemannian manifold.
\begin{enumerate}
\item Let $k\in \Z^+$.   If there exists a $\bla_k$-conformally critical metric $g_0$ in $[g]$, then there exists $m\in \Z^+$ and a non-degenerate $n$-harmonic map $u: (M,[g])\to \sph^m$. 
\item Conversely, if $u: (M,g)\to \sph^m$ is a non-degenerate $n$-harmonic map, then there exists $k\in \Z^+$ such that $g_u =|du|^2_g g$ is a $\bla_k$-conformally critical metric and $u$ is a $\lambda_k(M,g_0)$ eigenmap.
 \end{enumerate}

\end{cor}

\subsubsection{Critical metrics for Steklov eigenvalues}

Karpukhin and M\'etras had two critical insights that enabled them to generalise Theorems~\ref{thm: FrSc conf class} and \ref{thm.fs} to higher dimensions:

\begin{enumerate}
\item For manifolds of dimension $d+1$, one must use the normalisation 
\begin{equation}\label{eq: KM normalisation}\bsig_k(\om,g)=\sigma_k(\om,g)|\om|_g^{\frac{1-d}{d+1}}|\partial\om|_g.\end{equation}
Indeed they showed that this is the \emph{only} normalisation for which smooth critical metrics can exist.  (Compare with  Example~\ref{example:KarpukhinMetrasNorm}, which illustrates another instance in which this normalisation is distinguished by its special behavior.)
\item One must introduce densities and consider the weighted Steklov spectrum.
\end{enumerate}

The new normalisation that they introduced was discussed in Section~\ref{section:boundshigher}.     Here we address the second insight above.  To adapt Matthiesen's ideas to the Steklov case, one considers free boundary $(d+1)$-harmonic maps, where as usual $d+1 =\dim(\om)$.    

Free boundary $(d+1)$-harmonic maps $u: (\om,g)\to \B^m$ are defined analogously to free boundary harmonic maps as in Definition~\ref{def.fbms}.   In view of Equation~\eqref{eq:nharm R}, the condition to be a free boundary $(d+1)$-harmonic map can be expressed as 
\begin{equation}\label{eq: fb harm ball}
\begin{cases}
\delta_g(|du|_g^{d-1}du)= 0\,\,&\mbox{in}\,\,\om,\\
\,\,\,\,\partial\nu_gu =|\partial\nu_{g}u|u\,\,\,&\mbox{on}\,\,\partial\om.
\end{cases}
\end{equation}

Now assume that $u$ is non-degenerate.  If we replace $g$ by the conformally equivalent metric $g_u=|du|_g^2 g$, then (\ref{eq: fb harm ball}) can be expressed as
\begin{equation}\label{eq: fb harm ball'}
\begin{cases}
\Delta_{g_u}u= 0\,\,&\mbox{in}\,\,\om,\\
\partial\nu_{g_u}u =|\partial\nu_{g_u}u|u \,\,\,&\mbox{on}\,\,\partial\om.
\end{cases}
\end{equation}
Equivalently, if one defines $\rho_u:\partial \om \to \R$, by 
\begin{equation}\label{eq:rho sub phi}\rho_u=|\partial\nu_{g_u}u|\end{equation}
then \eqref{eq: fb harm ball'} says that the coordinate functions of $u$ are eigenfunctions with eigenvalue one of the weighted Steklov problem on $(\om,g_u)$ with density $\rho_u$ as in Example~\ref{ex:varweightedSteklov}.   Denoting the eigenvalues by $\sigma_k(\om,g_u, \rho_u)$, one defines the \emph{spectral index} of $u$ to be the minimum $k$ for which $\sigma_k(\om,g_u, \rho_u)=1$.

Moreover, Karpukhin and M\'etras \cite[Lemma 2]{KaMe2021} showed that the density $\rho_u$ is strictly positive.   Given any Riemannian metric $g$ on $\om$ and any strictly positive density $\rho$, write 
\begin{equation}\label{eq:dtn weight}
\DtN_{g,\rho}:=\frac{1}{\rho}\DtN_{g}
\end{equation} 
where $\DtN_g$ is the Dirichlet-to-Neumann operator of $(\om,g)$.   As noted in  Example~\ref{ex:varweightedSteklov}, the weighted Steklov eigenvalues $\sigma_k(\om,g,\rho)$ are the eigenvalues of $\DtN_{g,\rho}$. 
In particular, the coordinate functions of the free boundary $(d+1)$-harmonic map  $u:(\om,g)\to \B^m$ are eigenfunctions of $\DtN_{g_u,\rho_u}$ with eigenvalue one.  Note, therefore, that they are both harmonic and $(d+1)$-harmonic with respect to $g_u$ (the latter because $(d+1)$-harmonicity is a conformal invariant in dimension $d+1$ as noted above). 

 Generalising their notion of critical metric (see Remark~\ref{rem:crit different}), Karpukhin and M\'etras define the notions of $\bsig_k$-critical pairs and $\bsig_k$-conformally critical pairs $(g,\rho)$ as follows:  
 
 \begin{defn}\label{def: conf ext pair} Let $\om$ be a $(d+1)$-dimensional compact Riemannian manifold with boundary. Denote by $\mathcal{M}_\om$ the space of all Riemannian metrics on $\om$, and denote by $C^\infty(\sig)_{>0}$ the space of positive smooth densities on the boundary $\sig$ of $\om$.
 
 \begin{enumerate} 
  \item 
 Given a pair $(g,\rho)\in \mathcal{M}_\om\times C^\infty(\sig)_{>0}$, denote by $\sigma_k(\om,g,\rho)$ the $k$th eigenvalue for the associated weighted Steklov problem. Define the normalised weighted Steklov eigenvalues by   
\begin{equation}\label{eq: weight normalis}\bsig_k(\om,g,\rho)=\sigma_k(\om,g,\rho)\,|\om|^{\frac{1-d}{d+1}} \,\|\rho\|_{L^1(\om,g)}.\end{equation}
Observe that this normalisation is the natural extension to the weighted Steklov problem of the normalisation~\eqref{eq: KM normalisation}.
 \item Given a functional $F:\mathcal{M}_\om\times C^\infty(\sig)_{>0}\to \R$, we say $(g,\rho)\in \mathcal{M}_\om\times C^\infty(\sig)_{>0}$ is an $F$-critical pair if for each deformation $\{(g_t,\rho_t)\}_t$ depending smoothly on $t\in (-\epsilon, \epsilon)$ with $(g_0,\rho_0)=(g,\rho)$, either 
 \begin{equation}\label{eq:extremal pt def}F(g_t,\rho_t)\leq F(g,\rho)+o(t) \mbox{\,\,\,or\,else\,\,\,} F(g_t,\rho_t)\geq F(g,\rho)+o(t).\end{equation}
 We say $(g,\rho)$ is an $F$-conformally critical pair if \eqref{eq:extremal pt def} holds for all such deformations $(g_t,\rho_t)$ that satisfy $g_t\in [g]$ for all $t$.

\end{enumerate}
\end{defn}

We note that Karpukhin and M\'etras use the language ``extremal'' rather than ``critical''.

Karpukhin and M\'etras \cite{KaMe2021} observed that a modification of the proof of Hassannezhad's theorem~\ref{asma2011} yields
\[\bsig_k(\om,g,\rho)\leq Ck^{2/n}\]
where $C$ is a constant depending only on the conformal class $[g]$.

\begin{remark}\label{rem:resc} ~
\begin{enumerate}
\item Observe that the Riemannian metric $g$ and the positive density $\rho$ can be rescaled independently of each other without affecting the normalised eigenvalues, i.e., 
\begin{equation}\label{eq:higherresc}\bsig_k(\alpha g, \beta \rho)=\bsig_k(g,\rho) \,\,\,\,\,\forall \,\alpha, \beta\in \R^+.\end{equation}
Moreover, the eigenvalues themselves satisfy
\[\sigma_k(\alpha g, \alpha^{-1/2}\rho)=\sigma_k(g,\rho) \]
for all $\alpha \in \R^+$.   
\item In the special case that $\dim(\om)=2$, the eigenvalues moreover satisfy the following conformal invariance property:
\[\sigma_k(\tau g,  \tau^{-1/2}\rho)=\sigma_k(g,\rho)\,\,\,\,\,\forall \,\tau\in C^{\infty}_{>0}(\om).\]
The conformal change $(g,\rho)\to (\tau g,  \tau^{-1/2}\rho)$ extends the notion of $\sigma$-isometry from metrics to metric-density pairs.  (See Definition~\ref{sigmaisom}.)  We will refer to such changes as \emph{weighted $\sigma$-isometries}.   

Moreover, again in dimension 2, the normalisation~\eqref{eq: weight normalis} coincides with the standard normalisation $\bsig_k(\om,g,\rho)=\sigma_k(\om,g,\rho) \,\|\rho\|_{L^1(\om,g)}$.   In particular, weighted $\sigma$-isometries preserve both the weighted Steklov eigenvalues, as just noted above, and also the normalised weighted eigenvalues.   In particular, if $(g,\rho)$ is $\bsig_k$-critical or conformally critical, then so is any weighted $\sigma$-isometric pair.

\item Again in dimension 2, observe that every weighted $\sigma$-isometry class of pairs $(g,\rho)$ contains pairs of the form $(g',1)$.  Indeed one can take $g'=\frac{1}{\tilde{\rho}^2}g$, where $\tilde{\rho}$ is any strictly positive smooth function on $\om$ that restricts to $\rho$ on $\partial\om$.   Moreover $g'$ is unique up to $\sigma$-isometry in the usual sense of Definition~\ref{sigmaisom}.   We caution, however, that $g'$ is not $\sigma$-isometric to $g_u$ unless $\rho\equiv 1$.   
\end{enumerate}

\end{remark}

   We now state Karpukhin and M\'etras' characterisation of conformally critical pairs.

\begin{thm}\label{thm: KM stek conf extremals}\cite[Theorems 2 and 3]{KaMe2021}
Let $\om$ be a compact manifold with boundary of dimension $d+1\geq 2$.   We use the notation of Definition~\ref{def: conf ext pair} and Remark~\ref{rem:resc}.
\begin{enumerate}
\item Suppose that $(g,\rho)$ is $\bsig_k$-conformally critical.   Then  there exists $m>1$ and a non-degenerate free boundary $(d+1)$-harmonic map $u:(\om, [g])\to \B^m$ such that $(g,\rho)=( g_u, \rho_u)$ up to rescaling as in Equation~\eqref{eq:higherresc} (and, in the case of dimension two, also up to weighted $\sigma$-isometries).  
\item Conversely, let $u: (\om, [g])\to \B^m$ be a non-degenerate free boundary $(d+1)$-harmonic map.   Then $(g_u,\rho_u)$ is a $\bsig_k$-conformally critical pair (as is any weighted $\sigma$-isometric pair in case $\dim(\om)=2$), where $k$ is the spectral index of $u$.
\end{enumerate}

\end{thm}

The coordinate functions of $u$ are eigenfunctions with eigenvalue one of the weighted Steklov problem for the pair $(g_u,\rho_u)$.   In particular, $u$ is harmonic as well as $(d+1)$-harmonic with respect to $g_u$.

Consider the case of surfaces.
In view of Remark~\ref{rem:resc}, the first statement of Theorem~\ref{thm: KM stek conf extremals} in this case  is a reformulation of Theorem~\ref{thm: FrSc conf class}.  (While Fraser and Schoen stated Theorem~\ref{thm: FrSc conf class} only for conformally maximal metrics, Karpukhin and M\'etras observe that their proof is easily modified to address more general conformally critical metrics.)

The second statement similarly guarantees that the coordinate functions of an arbitrary free boundary harmonic map $u: (\om, [g])\to \B^m$ can be realised as eigenfunctions for some $\bsig_k$ critical metric $g'$ for the unweighted Steklov problem;  one needs only choose $g'$ so that $(g',1)$ is weighted $\sigma$-isometric to $(g_u,\rho_u)$ as in Remark~\ref{rem:resc}.

\begin{ex}\label{ex:kmannulus}\cite[Theorem 6]{KaMe2021} Let $\om$ be an annulus.  Conformal classes of metrics on $\om$ are in one-to-one correspondence with cylinders $[0,T]\times S^1$ for $T\in \R^+$.    The global $\bsig_1$-maximising metric (the critical catenoid described in Example~\ref{ex.cat}) lies in the conformal class  $[g_{t^*}]$ for which $t^*=\frac{t_0}{2}$ where $t_0$ is the unique solution of the equation $t=\coth t$.  For those conformal classes $[g_t]$ with $t\geq t^*$, the authors explicitly construct a rotationally symmetric free boundary harmonic map $u_t: (\om, g_t)\to \B^3$ of spectral index one.  (For $t=t^*$, $u_t$ is the immersion to the critical catenoid.)  For each such $t$, the last statement in Remark~\ref{rem:resc} yields a $\bsig_1$-conformally critical metric $g'_t$ for the unweighted Steklov problem in the conformal class $[g_t]$ such that the coordinate functions of $u$ are $\sigma_1(\om, g'_t)$-eigenfunctions.   However, they show that for $T$ sufficiently large, $\rho_{u_t}$ takes on different constant values on the two boundary circles of $\om$, so $g'_t$ is not $\sigma$-isometric to $g_{u_t}$.   

They conjecture that these $\bsig_1$-conformally critical metrics are $\bsig_1$-conformal maximisers.
\end{ex}

Next we consider critical pairs in the full space $\mom\times C^\infty_{> 0}$.   Since any $\bsig_k$-critical pair $(g,\rho)$ is also $\bsig_k$-conformally equivalent, Theorem~\ref{thm: KM stek conf extremals} yields a free boundary $(d+1)$-harmonic map $u: (\om,[g])\to \B^m$ such that $(g,\rho)= (g_u, \rho_u)$.  As noted after the statement of Theorem~\ref{thm: KM stek conf extremals}, the coordinate functions of $u$ are harmonic with respect to $g_u$.   Karpukhin and M\'etras show that $u$ is moreover a \emph{conformal} map and thus the image is a free boundary minimal submanifold.  Denoting by $g_{\B^m}$ the Euclidean metric on $\B^m$, we moreover have 
\[u^*g_{\B^m}=|du |^2_g g=g_u\]
and $\rho_u\equiv 1$.   This argument implies the first statement of the following theorem:

\begin{thm}\cite[Theorem 5]{KaMe2021}\label{thm:KM global}
Let $\om$ be a compact manifold with boundary and let $(g,\rho)$ be a $\bsig_k$-critical pair.   Then there exists $m\in \Z^+$ and a free boundary minimal immersion $u:\om\to \B^m$ such that $(g,\rho)= (g_u, 1)$ up to rescaling as in Equation~\eqref{eq:higherresc} (and, in the case of dimension two, also up to weighted $\sigma$-isometries).  

Conversely, if $u:\om\to \B^m$ is a free boundary minimal immersion, then $(u^*g_{B^m},1)$ is a $\bsig_k$-critical pair, where $k$ is the spectral index of $u$.

\end{thm}

If $(g, 1)$ is a $\bsig_k$-critical pair, respectively $\bsig_k$-conformally critical pair, then $g$ is necessarily a $\bsig_k$-conformally critical metric.    In view of the results above, it is interesting to ask the converse:

\begin{ques}\label{ques: conf ext metric vs pair}
If $g$ is a $\bsig_k$-(conformally) critical metric, is $(g,1)$ necessarily a $\bsig_k$-(conformally) critical pair?  
\end{ques} 

We wish to thank Karpukhin and M\'etras for clarifying the discussion below.

Using \cite[Theorems 9 and 10]{KaMe2021} for conformally critical metrics and \cite[Theorem 12]{KaMe2021}, along with an analogue of \cite[Theorem 9]{KaMe2021} for globally critical metrics, Question~\ref{ques: conf ext metric vs pair} has an affirmative answer provided that the $\bsig_k$-(conformally) critical metric $g$ satisfies the condition that
\begin{equation}\label{eq:simple cond} \bsig_k(\om,g)>\sigma_{k-1}(\om,g)\,\,\mbox{or}\,\,\bsig_k(\om,g)<\sigma_{k+1}(\om,g).  \end{equation}

Consider the case of conformally \emph{maximising} metrics.    For $k=0,1,2,\dots$ and $g$ any Riemannian metric on $\om$, write    
\[\bsig_k^*(\om, [g])=\sup_{h\in [g]}\, \bsig_k(\om,h).\]
Observe that if $g$ is a conformal maximiser of $\bsig_k$ and if $\bsig^*_{k-1}(\om,[g])< \bsig_k^*(\om, [g])$, then the first inequality in \eqref{eq:simple cond} necessarily holds, so $(g,1)$ is a conformally critical pair.  Thus 
an affirmative answer to Question~\ref{ques: conf ext metric vs pair} for conformally maximising metrics would follow from an affirmative answer to the following:

\begin{ques}\label{ques: conf spec simple} Given a compact manifold $\om$ with boundary and a conformal class $[g]$ of Riemannian metrics on $\om$, do we have
\[\bsig_k^*(\om,[g])>\bsig_{k-1}^*(\om, [g]) \mbox{\,\,for all}\,\,k\geq 1?\]
\end{ques}

The analogous question for the Laplacian on closed manifolds was answered affirmatively by Colbois and El Soufi \cite{CoEl2003}.

Similarly, if $g$ is a $\bsig_k$-maximising metric on $\om$, one could conclude that $(g,1)$ is a $\bsig_k$-maximising pair if one could show that $\bsig_{k-1}^*(\om)<\bsig_{k}^*(\om)$.

\begin{remark} As we were completing this survey, the extensive article \cite{KaSt2022} appeared on the mathematics arXiv. This article contains many interesting results addressing harmonic maps and eigenvalue optimisation in dimension greater than two.

\end{remark}

\section{Discretisation and Steklov eigenvalues of graphs} 
\label{Section: discretisation}

To facilitate the study of the eigenvalues of the Laplacian on a Riemannian manifold, numerous papers in the 1980's (for example \cite{Bu1982,Br1986, Ka1985}) have shown that it is very useful to introduce rough discretisation of manifolds. Then one can compare the spectrum of the manifold (in particular the first nonzero eigenvalue) to the spectrum of the combinatorial Laplacian on a graph associated to the discretisation. This idea of rough discretisation has applications not only for the study of the spectrum but also for other geometric concepts. This is very well explained in the two books by Chavel \cite{Ch1993, Ch2001}. See in particular \cite[Chapter V]{Ch2001} for isoperimetric inequalities and \cite[Chapter VI]{Ch2001} for Sobolev inequalities.

More generally, the question behind this approach is the following: if two metric spaces are "close in large scale", in some sense, what do they have in common? One potential meaning of "close in large scale" is the existence of a rough quasi-isometry: if $(X,d_X),(Y,d_Y)$ are two metric spaces, a map $\phi:X\rightarrow Y$ is a rough quasi-isometry if there exist constants $a\ge 1,b\ge 0$
and $\epsilon>0$ such that
\[
a^{-1}d_X(x_1,x_2)-b\le d_Y(\phi(x_1),\phi(x_2))\le ad_X(x_1,x_2)+b
\]
for all $x_1,x_2 \in X$ and 
\[\bigcup_{x\in X}B(\phi(x),\epsilon)=Y.\] 
  Observe that two compact metric spaces are always roughly quasi-isometric, so when we consider compact metric spaces, we will need more information. 

Throughout this section, the metric space $X$ will be a closed Riemannian manifold $M$ or a compact Riemannian manifold with boundary $\Omega$ and $Y$ will be a subset of points of $X$ equipped with a graph structure. The goal is to compare the spectrum of a differential operator on $X$ with the spectrum of a combinatorial operator on $Y$.
It is important to note that this type of discretisation is different from those used in numerical analysis, which were extended to Riemannian manifolds by different authors (see the initial contribution by Dodziuk and Patodi \cite{DoPa1976} and the more recent approaches of Burago, Ivanov and Kurilev \cite{BuIvKu2015} and Aubry \cite{Au2013}). In the situation we will describe, the mesh $\epsilon>0$ of the discretization $Y$ will be fixed. For the finite elements method and its extension to the Riemannian setting, the mesh tends to $0$ and the quotient between the k th eigenvalue of the differential operator on the Riemannian manifold and the k th eigenvalue of its combinatorial approximation tends to $1$ as $\epsilon \to 0$ under suitable assumptions (\cite{DoPa1976}). In \cite{BuIvKu2015} and \cite{Au2013}, the authors make explicit the rate of convergence by universal functions depending on the geometry of the manifold $M$ (bounds on the sectional curvature, injectivity radius, diameter).

 The outline of this section is as follows.

\begin{itemize}

\item Subsection~\ref{Discretisation and spectrum}: Discretisation and spectrum. We will in particular recall the case of the Laplacian before looking at the case of the Steklov problem on graphs.
\item Subsection~\ref{Discrete}: Study of the spectrum of the Steklov problem on a graph.
\item Subsection~\ref{Diameter}: Lower bounds in terms of the diameter.
 
 \item Subsection~\ref{UpperGraphs}: Upper bounds for subgraphs.

\end{itemize}

\subsection{Discretisation and spectrum.} \label{Discretisation and spectrum}

\subsubsection{Case of the Laplacian.}\label{subsec:disc lap} We will recall the case of the Laplacian, because this is what we imitate in the case of the Steklov problem on graphs. First, note that the study of the spectrum of the combinatorial Laplacian is a well-established subject with a large literature (see for example \cite{BrHa2012}), which is not the case for the Steklov problem on graphs.

We use a coarse and uniform discretisation that is not sensitive to the local geometry. We fix a (large) family of compact Riemannian manifolds without boundary given by a set of geometric parameters: bounds on the curvature in order to avoid local deformations that are invisible to a discretisation, and a lower bound on the injectivity radius in order to guarantee that the volumes of small balls are comparable (this last point could be avoided by using weighted graphs). Then we choose a mesh (smaller than the injectivity radius) for the discretisation (the distance between the points of the discretisation) and discretise each manifold using the same mesh of discretisation. The goal is then to establish a uniform comparison between the spectra of the smooth Laplacian on the manifolds and the combinatorial ones on the corresponding graphs. We refer to the paper~\cite{Ma2005} by Mantuano that we describe briefly below.

 We associate a graph to a Riemannian manifold following 
\cite[Section V.3.2]{Ch2001}.
Let $(M^n,g)$ be a connected closed n-dimensional Riemannian 
manifold. We denote by $d_M$ the usual Riemannian distance.
A discretisation of $M$, of mesh $\epsilon > 0$,  is a graph $\Gamma = (V, E)$ such that the set $V$ of vertices 
is a maximal $\epsilon$-separated subset of $M$. That is, for any $v$, $ w \in V$, $v\neq w$, we have 
$d_M(v,w) \geq \epsilon$ and $\cup_{v \in V} B(v,\epsilon) = M $.
Moreover, the graph structure of $\Gamma$ is entirely determined by the 
collection of neighbours that we define as follows. 
For each $v \in  V$, a point $w \in V$ is declared to be a neighbour of $v$ (that is $v\sim w$ is an edge of $\Gamma$), if $ 0 < d_M(v,w) < 3 \epsilon$, see \cite[p.140]{Ch2001}. 
In the sequel, we denote by $N(v)$ the set of neighbours of $v$.

We denote by $i$, $1\le i\le \vert V\vert$, the vertices of $\Gamma$. If there is an edge between $i$ and $j$, we denote this by $i\sim j$. By functions on $\Gamma$ we mean functions from $V$ to $\R$; the set of functions forms a vector space of dimension $\vert V\vert$. The Laplacian $\Delta u$ of a function $u$ on $\Gamma$ is given by 
\[
\Delta u(i)=\sum_{j\sim i}(u(i)-u(j)).
\]
This is a symmetric operator on $V$ and it has a set of eigenvalues
\[
0=\lambda_0(\Gamma)< \lambda_1(\Gamma)\le ...\le \lambda_{\vert V\vert-1}(\Gamma).
\]

We have the following important result \cite[Theorem 3.7]{Ma2005}): if 
$\mathcal{M}(n, \kappa, r_{0})$ denotes the class of 
all closed $n$-dimensional Riemannian manifolds with Ricci curvature 
and injectivity radius uniformly bounded below 
(i.e. with $Ricci(M,g) \geq -(n-1) \kappa  g$, $\kappa \geq 0$ and 
$Inj(M) \geq r_{0} > 0$) and if we fix $\epsilon < \frac{r_0}{2}$,
  there exist positive constants $c$, $C$ 
  such that for all manifolds $M$ in $\mathcal{M}(n, 
\kappa, r_{0})$ and any discretisation $\Gamma$ of $M$ of mesh $\epsilon$, we have
\begin{equation} \label{discretelaplacian}
    c \leq \frac{\lambda_{k}(M) }{\lambda_{k}(\Gamma)} \leq C
\end{equation}
for all $0< k < |V|$. The crucial point is this: the constants $c$ and $C$ depend on $n$, 
$\kappa$, $r_{0}$, and $\epsilon$, but not on the particular manifold or discretisation we consider.  
\begin{remark}
Observe that $|V|$ behaves like the volume of $M$ for a \emph{fixed} mesh size $\epsilon$, for example $\epsilon=\frac{r_0}{4}$, in the sense that $|V|/|M|$ is bounded above and also bounded away from zero by constants depending on the bounds on the geometry (curvature, injectivity radius) in the class $\mathcal{M}(n, \kappa, r_{0})$ . To see this, we use the fact that the mesh is smaller than the injectivity radius and that the curvature is bounded. So the volumes of the ball of radius $\epsilon$ are uniformly comparable.
\end{remark}

\subsubsection{The Steklov spectrum on a compact Riemannian manifold with boundary.} In \cite{CoGiRa2018}, Colbois, Girouard and  Raveendran introduced a discretisation for the Steklov spectrum in the spirit of what was done for the Laplacian. 

In comparison with the case of the Laplacian, the two main difficulties are the presence of a boundary and the necessity to define what plays the role of the Steklov problem on graphs and of its spectrum.

Regarding the Steklov spectrum, the definition considered in \cite{CoGiRa2018} differs slightly from the definitions commonly used recently (as will be explained below). In \cite{CoGiRa2018}, a graph with boundary is a pair $(\Gamma,B)$ where $\Gamma=(V,E)$ is a finite, simple, connected graph, and $B\subset V$ is a non empty set of vertices called the boundary. The eigenvalues of the Steklov problem on $(\Gamma,B)$ are defined by 
\begin{equation} \label{def:var1}
    \sigma_k(\Gamma,B)= \min_F \max_{u\in F,u\not=0}\frac{\sum_{i\sim j}(u(i)-u(j))^2}{\sum_{i\in B}u(i)^2}
\end{equation}
 where $F$ is a vector subspace of dimension $k+1$ of the functions on $\Gamma$. That is, in order to define the eigenvalues, we imitate the usual min-max characterisation in the discrete setting, but \emph{without} defining a discrete Dirichlet-to-Neumann operator. Observe that there are exactly $\vert B\vert$ eigenvalues (counted with multiplicity) and that, in the particular case where $B=V$, we recover the spectrum of the usual combinatorial Laplacian.

With this definition, a relation between the Steklov spectrum of a Riemannian manifold and the spectrum of its discretisation can be established as follows. For $\kappa\geq 0$ and $r_0\in (0,1)$, we introduce the space $\mathcal M(d,\kappa,r_0)$ of Riemannian manifolds $\Omega$ of dimension $d+1$ with non empty boundary and with the following properties:

\begin{itemize}
\item The boundary $\Sigma$ admits a neighbourhood
which is isometric to the cylinder $[0,1]\times\Sigma$, with the
boundary corresponding to $\{0\}\times\Sigma$;
\item The Ricci curvature of $\Omega$ is bounded below by $-d\kappa$;
\item The Ricci curvature of $\Sigma$ is bounded below by $-(d-1)\kappa$;
\item For each point $p\in \Omega$ such that $d(p,\Sigma)>1$, $\mbox{inj}_{\Omega}(p)>r_0$;\item  For each point $p\in \Sigma$, $\mbox{inj}_\Sigma(p)>r_0$.
\end{itemize}

The first condition, requiring that a neighborhood of the boundary be cylindrical, seems to be very strong. However, in many applications, it suffices to have a Riemannian manifold which is locally quasi-isometric (in the Riemannian sense) to a cylinder near the boundary.

To discretise the elements $\Omega \in \mathcal M(d,\kappa,r_0)$ we first fix the mesh $\epsilon\in(0, \frac{r_0}{4})$. We choose a maximal $\epsilon$-separated set $V_{\Sigma}$ on the boundary $\Sigma$ of $\Omega$, that is, a discretisation of the boundary. Then we introduce a copy $V'_{\epsilon}$ of $V_{\Sigma}$ using the fact that $\Omega$ is cylindrical near the boundary:
\[
V'_{\epsilon}=\{4\epsilon\} \times V_{\Sigma}.
\]
We then introduce a maximal $\epsilon$-separated set $V_{I}$ on $\Omega \setminus (0,4\epsilon)\times \Sigma$ such that $V'_{\epsilon} \subset V_I$. Then the discretisation is given by the set $V=V_{\Sigma} \cup V_I$.

We associate to $V$ the structure of a graph with boundary $B=V_{\Sigma}$. Regarding the edges, they are of two types. First, as in Subsection~\ref{subsec:disc lap}, we join two vertices $v$ and $w$ by an edge if $d_{\Omega}(v,w)<3\epsilon$ where $d_{\Omega}$ denotes the Riemannian distance on $\Omega$. Secondly, we join each vertex $(4\epsilon,v)$ in $V'_{\epsilon}$ to the vertex $(0,v)\in B$.
We thus obtain a graph $\Gamma=(V,E)$ with boundary $B$. 

With this definition, we obtain a uniform comparison between the first nonzero Steklov eigenvalue of $\Omega \in \mathcal M(d,\kappa,r_0)$ and the first nonzero eigenvalue of the Steklov problem on the associated graph with boundary $\Gamma=(V,E)$ \cite[Theorem 3]{CoGiRa2018}. Given $\epsilon \in (0,\frac{r_0}{4})$, there exist two constants $A_1,A_2$ depending on $\kappa, r_0,d,\epsilon$ such that for any discretisation of a manifold $\Omega \in \mathcal M(d,\kappa,r_0)$,
\[
A_1 \le \frac{\sigma_1(\Omega)}{\sigma_1(\Gamma)}\le A_2
\]
where $\Gamma$ is the graph associated to the discretisation. 

If we want to take all possible eigenvalues into account, we have that for $1\le k \le \vert B\vert-1$ (with $B$ the boundary of $\Gamma$)
\begin{equation} \label{discretesteklov}
    \frac{A_1}{k} \le \frac{\sigma_k(\Omega)}{\sigma_k(\Gamma)}\le A_2
\end{equation}
 We see that for large $k$, the inequality is weaker than one might expect in view of Inequality (\ref{discretelaplacian}) in the case of the discrete Laplacian.

This result, \cite[Theorem 3]{CoGiRa2018}, has three immediate applications: \cite[Theorems 4,5,6]{CoGiRa2018}). In particular, Theorem 6 shows that one can construct a family $\Omega_l$ of surfaces with connected boundary $\Sigma_l$ such that $\lim_{l\to \infty} \sigma_1(\Omega_l)\vert \Sigma_l\vert=\infty$. 

\begin{ques}\label{ques:discretcomp} For all $1\le k \le \vert B\vert-1$,
is it possible to obtain an inequality of the form
\[
A_1 \le \frac{\sigma_k(\Omega)}{\sigma_k(\Gamma)}\le A_2?
\]
\end{ques}

\subsection{Study of the spectrum of the Steklov problem on a graph.} \label{Discrete}

 In the previous subsection we defined a possible concept of a graph with boundary, which was adapted to the discretisation of a compact Riemannian manifold with boundary. We did not define the Steklov problem on a graph, but we directly defined the eigenvalues via the min-max characterisation using the natural definition of the Rayleigh quotient. We can extend this point of view to all finite graphs: we just have to choose the boundary of a given finite graph among the vertices of the graph. The two extremal situations are the case where the boundary has only one vertex, and the only eigenvalue is $0$, and the case where the boundary consists of all the vertices of the graph, in which case we recover the usual combinatorial Laplacian: this was done by Colbois and Girouard in \cite{CoGi2014}.

A particular class of Riemannian manifolds with boundary consists of the manifolds obtained as subdomains of a larger manifold. There is a corresponding concept for graphs with boundary
which are obtained as subgraphs of a larger graph. We will consider these extensively in what follows. In this case, we cannot choose the boundary: it is precisely defined by the situation. Moreover, there are \emph{no} edges between any two vertices in the boundary. In this particular case, there is a definition of a Steklov problem on graphs, whose eigenvalues are the Steklov eigenvalues.

\begin{defn} \label{def: graphsteklov}
Let $G = (V,E)$ be a graph and let $\Omega$ be a finite domain of $V$ (that is a finite \emph{connected} subset of $V$). The vertex boundary of $\Omega$ is $B=\delta \Omega=\{i\in V, i\not \in \Omega: \exists j\in \Omega,\  i\sim j\}$. The set of edges between two subsets $\Omega_1,\Omega_2 \subset V$ is $E(\Omega_1,\Omega_2)=\{i\sim j,\  i\in \Omega_1,j\in \Omega_2\}$.
Let $\Omega \subset V$ with $\Omega\not=V$. We define a graph with boundary $\Gamma=(\bar \Omega , E')$ associated to $\Omega$ by 
\[
\bar{\Omega}=\Omega \cup \delta \Omega;\  E'=E(\Omega,\bar{\Omega}).
\]
This defines a graph $\Gamma$ with vertices $\bar{\Omega}$ and with boundary $B=\delta \Omega$. By definition, between any two elements of $B$, there is no edge. 
\end{defn}

\begin{remark}
Note that sometimes the authors do not suppose $\Omega$ to be connected, see for example \cite[page 3]{HuHuWa2017}. But in that case, the multiplicity of the eigenvalue $0$ is equal to the number of connected components of the graph (see \cite[Proposition 3.2]{HuHuWa2017}) and in order to obtain nontrivial lower bounds, $\sigma_1$ denotes the first nontrivial eigenvalue (\cite[page 4]{HuHuWa2017}).
\end{remark}

\begin{ex}\label{ex: graph1}
Let $G$ be the lattice $\mathbb Z^2$ with its usual graph structure (a vertex $(p,q)$ has $4$ neighbours $(p+1,q); (p-1,q); (p,q+1); (p,q-1)$). Let $\Omega$ be the subset of $\mathbb Z^2$ defined by
\[
\Omega=\{(p,q)\in \mathbb Z^2: -1 \le p,q\le 1\}.
\]
We have 
\[
\delta \Omega=\{(-2,q)\cup (2,q)\cup (p,-2)\cup(p,2):-1 \le p,q\le 1\}
\]
and $\bar \Omega= \Omega \cup \delta \Omega$.
In this case, the points of $\delta \Omega$ have only one edge.
\end{ex}

For domains in ambient graphs (as in Definition \ref{def: graphsteklov}), there is a natural notion of Steklov problem corresponding to the definition in the continuous case. As before, the Laplacian $\Delta u$ of a function $u$ is given by $\Delta u(i)=\sum_{j\sim i}(u(i)-u(j))$ and we introduce the normal derivative at a point $i\in B$ given by $\partial_{\nu}u(i)=\frac{\partial u}{\partial \nu}(i)=\sum_{j\sim i}(u(i)-u(j))$. Then the Steklov problem on the graph with boundary $(\Gamma,B)$ consists of finding $\sigma \in \R$ for which there exists a non-trivial solution of
\[\begin{cases}
  \Delta u(i)=0&\text{if } i\in B^c,\\
  \partial_{\nu}u(i)=\sigma u(i)&\text{if } i\in B.
\end{cases}\]
The numbers $\sigma$ are the Steklov eigenvalues and we have
\[
0=\sigma_0< \sigma_1\le ...\le \sigma_{\vert B\vert-1}.
\]

The eigenvalues have the variational characterisation
\begin{equation} \label{def:var2}
    \sigma_k(\Gamma,B)= \min_F \max_{u\in F,u\not=0}\frac{\sum_{i\sim j}(u(i)-u(j))^2}{\sum_{i\in B}u(i)^2}
\end{equation}
where $F$ is a vector subspace of dimension $k+1$ of the functions on $\Gamma$. Note that we recover the notion of the Rayleigh quotient introduced in (\ref{def:var1}).

In this situation, one can also define a discrete Dirichlet-to-Neumann operator as follows (see for example \cite[(1.3)]{HuHuWa2017} and \cite{Per2019} in a more general situation):
\[
\Lambda: \mathbb R^{\delta \Omega}\rightarrow \mathbb R^{\delta \Omega}
\]
given by 
\[
\Lambda v(i)= \partial_{\nu}u_v(i)
\]
where $u_v$ denote the harmonic extension of $v$ to $\Omega$.

\subsubsection{A Cheeger-type inequality.} \label{cheegerdiscrete}To our knowledge, the first estimate for $\sigma_1$ in the discrete case was a Cheeger-type inequality obtained independently by Hassannezhad and Miclo in \cite{HaMi2020} (as a particular case of Theorem A) and 
in the paper \cite{HuHuWa2017} by Hua, Huang and Wang. With the above notations, the authors consider a finite graph $\Omega \subset V$ and establish a Cheeger-type inequality following Escobar \cite{Es1997} and Jammes \cite{Ja2015}.

Let $\bar \Omega =\Omega \cup \delta \Omega$. If $A\subset \bar \Omega$, we denote by $\partial A$ the edge boundary of $A$, that is the subset of edges $E(A,A^c)$ and by $\partial_{\Omega}A$ the subset $\partial A\cap E(\Omega,\bar \Omega)$. For example, in Example \ref{ex: graph1}, if we take $A\subset \bar \Omega$ given by $A=\{(0,2)\}$, then $\partial A$ consists of the $4$ edges emanating from $(0,2)$ but $\partial_{\Omega}A$ consists only of the edge between $(0,2)$ and $(0,1)$. As in the continuous case (see (\ref{ctescobar}) and (\ref{ctjammes})), beside the classical Cheeger constant
\[
h_C(\bar \Omega)= \min_{\vert A\vert \le \frac{1}{2} \vert \bar \Omega\vert} \frac{\vert \partial_{\Omega}A \vert }{\vert A\vert},
\]
one has to introduce two other constants: the Escobar-type Cheeger constant 
\[
h_E(\bar \Omega)=\min_{A \subset \bar \Omega: \vert A \cap \delta \Omega\vert \le \frac{1}{2}\vert \delta \Omega \vert} \frac{\vert \partial_{\Omega}A \vert}{\vert A\cap \delta \Omega \vert},
\]
and the Jammes-type Cheeger constant
\[
h_J(\bar \Omega)=\min_{A \subset \bar \Omega: \vert A \vert \le \frac{1}{2}\vert \bar \Omega \vert} \frac{\vert \partial_{\Omega}A \vert}{\vert A\cap \delta \Omega \vert}.
\]

The main result is a Cheeger-Jammes type inequality \cite[Theorem 1.3]{HuHuWa2017} and \cite[Theorem A]{HaMi2020}):
\[
\sigma_1(\bar \Omega) \ge \frac{1}{2} h_C(\bar \Omega) h_J(\bar \Omega).
\]
This inequality has the same defect as in the continuous case: one can deform a graph far from the boundary and make $h$ very small without significantly affecting the Steklov spectrum. However, there are situations where the inequality is asymptotically sharp \cite[ Example 5.1]{HuHuWa2017}.

In \cite[Proposition 4.2]{HuHuWa2017}, the authors also establish an upper bound involving the isoperimetric constant $h_E(\bar \Omega)$: 
\[
\sigma_1(\bar \Omega) \le 2 h_E(\bar \Omega).
\]
In \cite[Example 4.2 ]{HuHuWa2017}, the authors give a situation where this inequality is sharp. In particular, one cannot avoid the factor $2$ in the inequality.

Note that the results of \cite{HaMi2020} and \cite{HuHuWa2017} are more general than what we mention, because they are established for weighted graphs.

In the same spirit of what they did in the Riemannian (and measurable) cases, Hassannezhad and Miclo establish a higher order Cheeger-type inequality for finite graphs \cite[Theorem A ]{HaMi2020}. Here also, as in the continuous case (see \ref{highercheegermanifold}), there is a term in $k^6$ in the denominator of the inequality.  Hassannezhad and Miclo also obtain a better inequality when estimating $\sigma_{2k+1}$ with respect to the Cheeger constant of order $k+1$ \cite[Proposition 1]{HaMi2020}.

\begin{ques}\label{ques:cheegerdiscrete} Same question as in the Riemannian case -- is it possible to have a higher order Cheeger inequality without the term in $k$?
\end{ques}

\subsection{Lower bound via the diameter.} \label{Diameter}  

For Riemannian manifolds, there are only a few results concerning lower bounds for the first nonzero Steklov eigenvalue. In the context of the Steklov problem on graphs, lower bounds depending on the extrinsic diameter of the boundary of a graph are established by Perrin in   \cite{Per2019}.

Let $\Gamma=(V,E)$ a finite graph, $B \subset V$ the boundary. Then, the diameter of a finite connected graph with boundary $(\Gamma, B)$ is the maximum distance between any two vertices of $(\Gamma, B)$, where the distance of two adjacent vertices is $1$. The extrinsic diameter $d_B$ of the boundary $B$ of $\Gamma$ is the maximum distance in $\Gamma$ between any two vertices of $B$.

In \cite[Theorem 1]{Per2019}, it is shown that, for a finite connected graph with boundary $(\Gamma, B)$
\[
\sigma_1(\Gamma,B)\ge \frac{\vert B\vert}{(\vert B\vert-1)^2d_B}.
\]

Note that the bound is optimal for $\vert B\vert=2$ and that its proof is very easy. The simplest graph with $\vert B\vert=2$ of diameter $d_B$ is given by $d_B+1$ vertices $0,...,d_B$ with edges $i\sim i+1$, $i=0,d_B-1$ and boundary the two extremities $B=\{0;d_B\}$. The first nonzero eigenvalue $\sigma_1$ of this graph is $\frac{2}{d_B}$.

In Theorem 2, a stronger, but much more elaborate estimate is proved:
\[
\sigma_1(\Gamma,B)\ge \frac{\vert B\vert}{\lfloor \frac{\vert B\vert}{2}\rfloor \lceil \frac{\vert B\vert}{2} \rceil d_B}.
\]
Moreover, this bound is sharp for any $B$ as $d_B \to \infty$ in the following sense. In \cite[Lemma 2]{Per2019}, the author constructs for each given $b$ a family of graphs $H^b$ with $\vert B\vert=b$, diameter $d_b$ and first nonzero eigenvalue $\sigma_1(H^b)= \frac{\vert B\vert}{\lfloor \frac{\vert B\vert}{2}\rfloor \lceil \frac{\vert B\vert}{2} (\rceil d_B-2)+b}$.

Note that from \cite[Theorems 1 and 2]{Per2019} we can deduce a similar estimate for Riemannian manifolds with boundary that we can discretise described in Section \ref{Discretisation and spectrum}. However, the estimate depends not only on the extrinsic diameter of the boundary, but also on the bounds on the Ricci curvature and on the injectivity radius that we impose.

\subsection{Upper bounds for subgraphs.} \label{UpperGraphs}

 In their paper \cite{HaHu2020}, Han and Hua study the Steklov eigenvalue problem on subgraphs of integer lattices. One can compare this problem with the study of the spectrum of Euclidean domains in the continuous case. This paper leads naturally to the same type of questions for subgraphs of more general lattices. We will briefly describe some recent results about this question because they lead to different approaches to the problem.

In \cite{HaHu2020}, the authors consider the lattice $\mathbb Z^n$ with its usual graph structure, and a finite subset $\Omega$ which defines a graph $\bar \Omega$ with boundary $B$. In \cite[Theorem 1.2]{HaHu2020}, they prove a Brock-type inequality (compare to Theorem \ref{thm: brock} in the Euclidean case):

\begin{equation} \label{In:Brock1}
    \sum_{i=1}^n \frac{1}{\sigma_i(\bar \Omega)} \ge C_1 \vert \Omega \vert^{1/n}-\frac{C_2}{ \vert \Omega\vert},
\end{equation}
 with explicit constants $C_1=(64 n^3 \omega_n^{\frac{1}{n}})^{-1}$, where $\omega_n$ is the volume of the unit ball in $\R^n$, and $C_2=\frac{1}{32n}$.

Note that in order for the right-hand side to be positive, one needs to have $\vert \Omega \vert > \left(\frac{C_2}{C_1}\right)^{n/(n+1)}$ which is satisfied if
  $\vert \Omega \vert$ is large enough in terms of $n$. As a corollary, the authors show (Corollary 1.4) that
\begin{equation}\label{In:Brock2}
    \sigma_1(\Omega) \le \frac{n}{C_1 \vert \Omega \vert^{\frac{1}{n}}- \frac{C_2}{\vert \Omega\vert}}.
\end{equation}
 In particular, as $\vert \Omega \vert \to \infty$, $\sigma_1(\Omega)\to 0$, at least at rate $\frac{C(n)}{\vert \Omega\vert^{\frac{1}{n}}}$.

The proof of the theorem is quite tricky. The authors associate to the subgraph a domain of $\R^n$, and they apply a weighted isoperimetric inequality due to Brock to this domain (see Inequality (\ref{thm:brock})). The difficulty is the construction of a convenient domain, which creates a lot of technical problems.

In \cite{Per2021}, Perrin investigates the same type of problem, but with a more geometric viewpoint, in the sense where the author does not use classical results from analysis as the Brock inequality, but works directly on the graph itself. The author observes that the crucial point is to control the growth of balls in the graph, which need not be a Euclidean lattice. For example, the Heisenberg lattice $H_3(\mathbb Z)$ of the Heisenberg group $H_3(\mathbb R)$ is convenient. 
In \cite{Per2021}, the author considers a finite subgraph $\Omega$ of a Cayley graph $\Gamma$ of a group with polynomial growth of order $D$: there exists $C_0=C_0(\Gamma)$ such that, if $B(N)$ is any ball of radius $N$ in $\Gamma$, 
\[
C_0^{-1}N^D \le \vert B(N)\vert \le C_0N^D
\]
for each $N \in \mathbb N^*$. For example, for the lattice $\mathbb Z^n$, $D=n$, but for the lattice $H_3(\mathbb Z)$, $D=4$. 

In this setting, Perrin shows that there exists a constant $C(\Gamma)$ such that, if $B$ denotes the boundary of the finite subgraph $\Omega \subset \Gamma$, we have \cite[Theorem 1]{Per2021}:
\begin{enumerate}
    \item If $D \le 2$, $\sigma_1(\bar \Omega) \le C(\Gamma) \frac{1}{\vert B\vert}$.
 \item If $D >2$,  $\sigma_1(\bar \Omega) \le C(\Gamma) \frac{\vert \Omega\vert^{\frac{D-2}{D}}}{\vert B\vert}$.
 \end{enumerate}

As a corollary \cite[Corollary 2]{Per2021}, Perrin observes that  $\sigma_1(\bar \Omega) \le \frac{1}{\vert \bar \Omega\vert^{\frac{1}{D}}}$.

The proof of Perrin's result is completely different from, and simpler than, the proof Inequality (\ref{In:Brock2}). The idea is to look directly at the graph and to construct a test function for the Rayleigh quotient, using the fact that the growth of the balls is controlled.
For the first nonzero eigenvalue $\sigma_1$, this result generalises the Inequality (\ref{In:Brock2}), but with a less explicit constant $C(\Gamma)$, and it says nothing for the next eigenvalues: it does not seem possible to use the same method to estimate the higher eigenvalues.

However, the higher eigenvalues may be estimated by other means, as in the recent paper \cite{Ts2022} by Tschanz. The main result is \cite[Theorem 5]{Ts2022}): for a subgraph $\Omega$ of a Cayley graph $\Gamma$ of polynomial growth of order $D\ge 2$, there exists a constant $C(\Gamma)>0$ such that for all $k\le \vert B\vert-1$,
\begin{equation}\label{tschanz1}
    \sigma_k(\bar \Omega)\le C(\Gamma) \frac{1}{\vert B \vert^{\frac{1}{D-1}}}k^{\frac{D+2}{D}}.
\end{equation}
In particular, for a fixed $k$, we see that a large volume of $\Omega$ implies a small eigenvalue $\sigma_k(\Omega)$ \cite[Corollary 6]{Ts2022}. This comes from the fact that there exists a constant $C=C(D)>0$ such that $\vert B\vert \ge \vert \Omega\vert^{\frac{D-1}{D}}$ (\cite[Proposition 9]{Ts2022}).

The proof of the result uses a third method; as in \cite{HaHu2020}, the author cannot work directly on the graph. But the idea is to associate to $\Gamma$ a complete Riemannian manifold $M$ of dimension $D$ modelled on it and to $\Omega$  a domain $U$ of $M$. Then, one shows that the growth of balls in $M$ is polynomial of order $D$, and the author adapts the method of \cite{CoElGi2011} to obtain a result on the domain $U$. Using \cite{CoGiRa2018}, one can transfer the estimate from $U$ to $\Omega$.

A further question is whether or not we need to have polynomial growth to get such results. This is not the case, as was shown by He and Hua in \cite{HeHu2022}. The authors study the case of subgraphs of a tree, and it is easy to produce trees where the volume of balls grows exponentially, for example the regular tree of degree $\ge 3$. The authors establish a variety of results (Theorem 1.1, 1.2, 1.5). In particular, Theorem 1.5 says the following: consider a finite tree $\Gamma=(\Omega,E)$ with boundary $B=\delta \Omega$, where the boundary denotes the vertices of degree one. Suppose that $\Gamma$ has degree bounded by $D$. Then, for any $3\le k \le \vert B\vert-1$,
\begin{equation}\label{ineghehua}
\sigma_k(\bar \Omega)\le \frac{8(D-1)^2(k-1)}{\vert B\vert}.
\end{equation}

In Theorem 1.3, the authors obtain an upper bound in terms of the diameter $L$ of the tree: 
\[
\sigma_1(\bar \Omega)\le \frac{2}{L},
\]
and they characterise the equality case.

This result leads to another natural question: a tree may have exponential growth but is very different from a space of negative curvature like a Cartan-Hadamard manifold, for example the hyperbolic space. Precisely, there does not exist a rough quasi-isometry between a tree and the hyperbolic space. The question is to decide what can be said for subgraphs of Cayley graphs roughly quasi-isometric to a Cartan-Hadamard manifold. 

A partial answer to this question is given by Tschanz in \cite{Ts2023} for a finite subgraph of a graph $\Gamma$ roughly quasi-isometric to the hyperbolic space. In order to construct such a graph $\Gamma$, the author uses a tiling of the hyperbolic space by triangles. The vertices $V$ of $\Gamma$ correspond to the triangles of the tiling and the edges to the adjacent triangles. By construction, $\Gamma$ is roughly quasi-isometric to the hyperbolic space.

In \cite[Theorem 9]{Ts2023}, it is shown that there exists a constant $C(\Gamma)$ such that for all finite subgraphs $\Omega$ of $\Gamma$ with boundary $B$, one has
\begin{equation} \label{inegtschanz2}
    \sigma_k(\bar \Omega) \le C(\Gamma)\frac{k^2}{\vert B\vert^2}.
\end{equation}
The first idea of the proof is, in some sense, comparable to the proof of \cite{HaHu2020}, that is to associate a domain of the hyperbolic space to the subgraph. Then Tschanz uses \cite[Theorem 1.2]{CoElGi2011} to get an upper bound for the Steklov spectrum of the domain, and  \cite[Theorem 3]{CoGiRa2018} to transfer the estimate from the domain to the graph. This leads to serious technical difficulties; the reason is that the domains are constructed by gluing together hyperbolic triangles, resulting in singularities along their boundaries. Thus the domains are not quasi-isometric in the Riemannian sense to a domain with cylindrical boundary. To solve this difficulty, the idea is to smooth out the domain near the boundary.

There exist numerous Cayley graphs which are roughly quasi-isometric to the hyperbolic space. For example, to each compact surface with curvature $-1$, one can associate such a graph via its fundamental group. It seems natural that the behaviour of the spectrum of subgraphs of roughly quasi-isometric domains must be comparable. This is not obvious and it leads to the following question.

\begin{ques}\label{ques:discretequasiisom}
 Let $\Gamma_1,\Gamma_2$ be two infinite roughly quasi-isometric graphs. If there exists a constant $C_1(\Gamma_1)$ such that for each finite subgraph $\Omega$ of $\Gamma_1$, $\sigma_1(\bar \Omega)\le \frac{C_1(\Gamma_1)}{\vert B\vert^{\alpha}}$, with $B$ the boundary of $\Omega$ and $\alpha>0$, is an analogous property also true for the finite subgraphs of $\Gamma_2$?
\end{ques}

 \subsection{Other contributions} 
 
 There are new contributions appearing regularly on the spectrum of the Steklov problem ongraphs that we will not describe here. For example, the setting of weighted finite graphs satisfying the Bakry-Emery curvature-dimension condition $CD(K, n)$ with
$K > 0$ and $n > 1$ is considered in \cite{ShiYu2022}.These authors considered also a discrete Dirichlet-to-Neumann operator for differential forms on graphs \cite{ShiYu20222} For infinite graphs, see the preprint \cite{HuHuWa2018}.

\section{Steklov spectrum on $p$-forms}\label{stek.forms}

There are several different notions of a Dirichlet-to-Neumann map acting on the space $\Acal(\sig)$ of smooth differential forms on the boundary $\sig$ of a compact $(d+1)$-dimensional Riemannian manifold $\om$.     The earliest such notion, introduced by Joshi and Lionheart \cite{JoLi2005}, is a pseudo-differential operator $\Acal^p(\sig)\times \Acal^{d+1-p}(\sig)\to \Acal^{d-p}(\sig) \times \Acal^{p-1}(\sig)$.    Belishev and Sharafatudinov \cite{BeSh2008} introduced an operator $\Lambda:\Acal^p(\sig)\to\Acal^{d-p}(\sig)$.   The article \cite{ShSh2013} by Sharafutdinov and Shonkwiler synthesises these two notions.   These definitions are motivated not by spectral problems but rather inverse problems such as the analog for forms of Calderon's problem, which asks the extent to which the Dirichlet-to-Neumann map of a Riemannian manifold with boundary determines the manifold.

We will focus here on two notions of a Dirichlet-to-Neumann map for $p$-forms, both of which allow the Steklov spectrum to be generalised to the setting of $p$-forms.   The first, introduced by Raulot and Savo \cite{RaSa2012}, is a non-negative elliptic, self-adjoint, pseudo-differential operator of order one.  The second, introduced by Karpukhin \cite{Ka2019}, coincides up to sign with the composition $*\circ \Lambda$, where $\Lambda$ is the operator defined by Belishev and Sharafutdinov and $*$ is the Hodge-$*$ operator of $\sig$.    This operator is self-adjoint on the space of co-closed forms and has non-negative discrete spectrum on this space.

The definition of the Dirichlet-to-Neumann map on functions relies crucially on the uniqueness of the harmonic extension to $\om$ of functions $f\in C^\infty(\sig)$.  In defining a Dirichlet-to-Neumann operator on $p$-forms when $p>0$, one must take into account the non-uniqueness of harmonic $p$-forms on $\om$ that pull back to a given smooth $p$-form on $\sig$.

Before introducing the operators, we recall the Hodge-Morrey-Friedrichs decomposition of the space $\Acal^p(\om)$ of smooth $p$-forms on $\om$.    See \cite{Sc1995} for an extensive introduction to the Hodge-Morrey-Friedrichs decomposition and its applications to boundary value problems or  \cite{Ka2019} for an overview of the aspects especially relevant here.  Recall that the Hodge Laplacian acting on the space $\Acal^p(\om)$ of smooth $p$-forms on a compact Riemannian manifold $\om$ is given by $\Delta=(d\delta+\delta d)$.

\begin{nota}\label{hodge_decomp} For any subspace ${\mathcal W}$ of $\Acal(\om)$, we denote by ${\mathcal W}_D$ and ${\mathcal W}_N$ the subspaces of elements of ${\mathcal W}$ satisfying Dirichlet and Neumann boundary conditions, respectively, i.e., 
\[{\mathcal W}_D(\om)=\{\omega\in {\mathcal W}(\om): \, i^*\, \omega =0\}\]
where $i: \sig\to \om$ is the inclusion, 
and
\[{\mathcal W}_ N(\om)=\{\omega\in {\mathcal W}(\om): \, \nu \lrcorner\, \omega =0\}\]
where $\nu$ is the outward pointing unit normal along $\sig$ and $\lrcorner$ denotes the interior product.

 Let $\delta$ denote the adjoint of the exterior differential operator.   Set 
\[\Hcal^p(\om)=\{\omega\in \Acal^p(\om): d\omega=0=\delta\omega\}.\]
Elements of $\Hcal^p(\om)$ are called harmonic fields.

\end{nota}

The Hodge-Morrey-Friedrichs decomposition says that 
\begin{equation}\label{hmf}\Acal^p(M) = d(\Acal_D^{p-1}(\om)) \oplus \Hcal^p(\om)   \oplus \delta(\Acal_N^{p+1}(\om)).\end{equation}

 In contrast to the case of closed manifolds, the space of harmonic $p$-fields is properly contained in the space of harmonic $p$-forms and both are infinite-dimensional.   However, $\Hcal^p_N(\om)$ is finite-dimensional and the Hodge-Morrey-Friedrichs decomposition further says that every de Rham cohomology class is uniquely represented by an element of $\Hcal^p_N(\om)$:
 \[H^p(\om;\R)\simeq \Hcal^p_N(\om).\]
Similarly, the relative cohomology classes of $\om$ are represented by $\Hcal^p_D(\om)$:
 \[H^p(\om,\sig;\R)\simeq \Hcal^p_D(\om).\]

\subsection{Raulot and Savo's Dirichlet-to-Neumann operator and its spectrum}

\begin{defn}\label{dtnp}\cite{RaSa2012} Let $\Acal^p(\sig)$ be the space of smooth $p$-forms on $\sig$.  Given  $\eta\in \Acal^p(\sig)$, define the \emph{harmonic tangential extension} $\widehat{\eta}$ of $\eta$ to be the unique $p$-form on $\om$ satisfying
\[\begin{cases} \lap \widehat{\eta}=0\\i^*\widehat{\eta}=\eta\\\nu\lrcorner\,\widehat{\eta}=0.
\end{cases}\]
Then Raulot and Savo's Dirichlet-to-Neumann operator $\DtNp:\Acal^p(\sig)\to \Acal^p(\sig)$ is defined by  \[\DtNp(\eta)=\nu\lrcorner \,d\widehat{\eta}.\]

Observe that when $p=0$, $\DtNp$ coincides with the usual Dirichlet-to-Neumann operator $\DtN: C^\infty(\sig)\to C^\infty(\sig)$.   For arbitrary $p$, we will denote the eigenvalues of $\DtNp$ by \[0\leq \rssig_{1,p}(\om)\leq \rssig_{2,p}(\om)\leq \dots.\]   

\end{defn}
\begin{remark}\label{sighatnota}~
(i) Zero occurs as an eigenvalue of $\DtNp$ if and only $H^p(\om;\R)\neq 0$.  In this case the eigenvalue zero has multiplicity equal to the Betti number $\beta_p(\om)=\dim H^p(\om;\R)$.   In particular, the indexing convention introduced above on the eigenvalues is inconsistent with the convention we have been using in the case of functions, where $\sigma_0(\om)=0$ and $\sigma_1(\om)$ is the first non-zero eigenvalue.    
In this case, we have $\rssig_{k,\,0}(\om)= \sigma_{k-1}(\om)$.  

We will denote the non-zero eigenvalues of $\DtNp$ by
\begin{equation}\label{hsig}\hsig_{k,p}(\om):=\rssig_{\beta_p(\om)+k,\,p}(\om).\end{equation}
Under the assumption that $\om$ is connected, we have
\begin{equation}\label{hat0}\hsig_{k,0}(\om)=\sigma_k(\om).\end{equation}

(ii) If one replaces $\lap\widehat{\eta}=0$ by $\lap\widehat{\eta}=\alpha\widehat{\eta}$ in Definition~\ref{dtnp}, where the parameter $\alpha$ lies in $\Cx -[0,\infty)$, then one obtains a parametrised family of invertible pseudo-differential self-adjoint operators introduced earlier by Carron \cite{Ca2002}.   

\end{remark}

The variational characterisation of the eigenvalues is given by
\begin{equation}
\rssig_{k,p}(\om)=\inf_{V}\,\sup_{0\neq\varphi\in V}\,\frac{\|d\varphi\|_{L^2(\om)}^2+\|\delta\varphi\|_{L^2(\om)}^2}{\|i^*\varphi\|_{L^2(\sig)}^2 }.\end{equation}
 Here $V$ runs over the collection of all $k$-dimensional subspaces of $\Acal_N^p(\om)$.  (See \cite{RaSa2012}.)

The articles \cite{Ka2017},  \cite{Kw2016}, \cite{Mi2021}, \cite{RaSa2012}, \cite{RaSa2014}, \cite{ShYu2016}, \cite{ShYu2017}, \cite{YaYu2017} and \cite{YaYu2017_2},  contain many interesting geometric bounds both from above and below for the eigenvalues.   We give only a sampling here.

\subsubsection{Upper bounds in terms of the isoperimetric ratio.}~

   A compact Riemannian  manifold $\om$ with boundary $\sig$ is said to be \emph{harmonic} if the mean value of any harmonic function on $\om$ equals its mean value on $\sig$.    

\begin{thm}\label{rsisoper} Let $\om$ be a $(d+1)$-dimensional compact oriented Riemannian manifold with smooth boundary $\sig$.  

\begin{enumerate}
\item {\rm ( Raulot-Savo \cite{RaSa2012}.)} 
\[\rssig_{1,d}(\om)\leq \frac{|\sig|}{|\om|}.\] 
 If equality holds, then $\om$ must be harmonic.     (As a partial converse, if $\om$ is harmonic, then 
 $\frac{|\sig|}{|\om|}$ lies in the spectrum of $\rsDtN_d$.)

\item {\rm (Raulot-Savo \cite{RaSa2014}.)}  If $\om$ is a Euclidean domain then 
\[\rssig_{1,p}(\om)\leq \left(\frac{p+1}{d+1}\right)\frac{|\sig|}{|\om|}\]
for $p=1, \dots, d$.   The inequality is strict when $p<\frac{d+1}{2}$.  When $p\geq \frac{p+1}{2}$, equality holds iff $\om$ is a Euclidean ball.
\end{enumerate}
\end{thm}

For the equality statement in (1), the authors show that $\om$ is harmonic if and only if the ``mean-exit time'' function $E$ on $\om$, given as the solution of $\Delta E\equiv 1$ on $\om$ and $E=0$ on $\sig$, has constant normal derivative.  In this case, $*dE$ is an eigenform with eigenvalue $\frac{|\sig|}{|\om|}$.   

Yang and Yu generalised the inequality in item (2), and Shi and Yu strengthened it further as follows:

\begin{thm}\label{yyisoper} Let $\om$ be a $(d+1)$-dimensional compact oriented Riemannian manifold with smooth boundary $\sig$.    Suppose that the space of parallel exact 1-forms on $\om$ has dimension $m>0$.  Then: 
\begin{enumerate}
\item {\rm (Yang-Yu \cite[Theorem 1.2]{YaYu2017_2}.)}  
\begin{equation}\label{yyiso}\hsig_{k,p}(\om)\leq \frac{C_{m-1}^{p}}{C_m^{p+1}+1 -k}\,\frac{|\sig|}{|\om|}\end{equation}
for $p=1,\dots m-1$ and $k=1,2,\dots C_m^{p+1}$.  Here $C_k^\ell$ denotes the binomial coefficient $\binom{k}{\ell}$.  
\item {\rm ( Shi-Yu \cite[Theorem 1.3]{ShYu2016}.)} 
$\hsig_{1,p}(\om)+\dots +\hsig_{C_m^{p+1},p}(\om) \leq C_{m-1}^p \frac{|\sig|}{|\om|}$
for $p=1,\dots m-1$.
\end{enumerate}

\end{thm}

\begin{ques}\label{ques:yangyusharp} Is Yang-Yu's inequality sharp when $k=1$ and $p<\frac{d+1}{2}$?   Is it sharp when $k>1?$

\end{ques}

\subsection{Generalising the Hersch-Payne-Schiffer Inequality.}~

\begin{nota}\label{nota.lap} Given a compact Riemannian manifold $\sig$ with $b$ connected components, we denote the eigenvalues of the Laplace-Beltrami operator of $\sig$ by 
\[0=\lambda_0(\sig)\leq\lambda_1(\sig)\leq \dots \]
Note that the first non-zero eigenvalue is $\lambda_{b}(\sig)$.
\end{nota}

The well-known Hersch-Payne-Schiffer \cite{HePaSc1975} inequality states for simply-connected plane domains $\om$ that 
\begin{equation}\label{eqhps} \sigma_j(\om)\sigma_k(\om)L(\partial\om)^2\leq \lambda_{j+k}(S^1)=\begin{cases}\pi^2(j+k)^2, \,\,\,\,\,\, \, \,\,\,\,\,\,\,j+k\,\mbox{even}\\ \pi^2(j+k-1)^2, \,\,\, j+k\, \,\mbox{odd}\end{cases}\end{equation}
The inequality was proven by the clever trick of considering the product $R(u)R(v)$ of the Steklov Rayleigh quotients of suitably chosen functions $u,v\in C^\infty(\sig)$ whose harmonic extensions $\widehat{u}$ and $\widehat{v}$ to $\om$ are harmonic conjugates.   

This method is of course special to dimension two.   Surprisingly, by using the Steklov spectrum on differential forms, Yang and Yu  generalised the Hersch-Payne-Schiffer inequality to higher dimensions:

\begin{thm}\cite{YaYu2017}\label{yyhps} Let $\om$ be a compact, oriented $(d+1)$-dimensional Riemannian manifold with smooth boundary $\sig$.    Then for every $j,k\in \Z^+$, we have
\[\sigma_j(\om)\,\hsig_{k, d-1}(\om)\leq \lambda_{j+k+\beta_d -2}(\sig)\]
where $\beta_d$ is the $d$th Betti number of $\om$.

\end{thm}

Observe that this result agrees with Equation~\eqref{eqhps} when $\om$ is a 2-dimensional simply-connected plane domain.

The key to Yang and Yu's generalisation is the observation that 
harmonic functions $\widehat{u}$ and $\widehat{v}$ on a planar domain are harmonic conjugates if and only if 
\[*d\widehat{u}=d\widehat{v}\]
(up to sign) 
where $*$ denotes the Hodge star operator.   In higher dimensions, starting with a suitably chosen function $u\in C^\infty(\sig)$ and its harmonic extension $\widehat{u}$ to $\om$, the role of $\widehat{v}$ is played by a $(d-1)$-form $\widehat{\alpha}$ such that $*d\widehat{u}= d\widehat{\alpha}$.  (One of the defining conditions placed on $u$ guarantees that $*d\widehat{u}$ is exact.)   The proof then follows the outline of Hersch, Payne and Schiffer's proof.

Subsequently, Karpukhin \cite{Ka2017} applied these results to obtain a version of the Hersch-Payne-Schiffer inequality for surfaces of arbitrary genus $\gamma$ and $b$ boundary components.    In this case, we have $d=1$ so $\hsig_{k, d-1}(\om)=\sigma_k(\om)$ and $\beta_1= 2\gamma +b-1$.  Thus Theorem~\ref{yyhps} says that
    \[\sigma_j(\om)\sigma_{k}(\om)\leq \lambda_{j +k+ 2\gamma +b-3}(\sig).\]
  By finding an upper bound for $\lambda_{p +q +2\gamma +b-2}(\sig)$ when $\sig$ is a disjoint union of $b$ circles, Karpukhin obtained the Steklov eigenvalue bound stated earlier in this survey as Theorem~\ref{thm:karp hersch}.

\subsubsection{Geometric lower bounds .}~

\begin{nota}\label{nota_pcurv}~
\begin{enumerate}
\item Let $(\om,g)$ be a compact $d+1$-manifold with smooth boundary $\sig$.  For $x\in \sig$, let
\[\kappa_1(x)\leq\kappa_2(x)\leq\dots\leq\kappa_d(x)\]
be the principal curvatures at $x$. Set
\[c_p(\om)=\inf_{x\in\sig}\,(\kappa_1(x)+\dots+\kappa_p(x)).\]
$\om$ is said to be $p$-convex if $c_p(\om)\geq 0$.    In particular, 1-convexity is equivalent to convexity.
\item Let $W^{\pp}(\om)$ be the curvature term in Bochner's formula for the Laplacian on $p$-forms:
\[\lap(\omega)=\nabla^*\nabla\omega + W^{\pp}(\om).\]
In particular, $W^{(1)}(\om)$ and $W^{(d)}(\om)$ coincide with the Ricci tensor.
\end{enumerate}
\end{nota}

\begin{thm}\label{rsthm1_3} {\rm [Raulot-Savo \cite[Theorems 1 and 3.]{RaSa2012}}\label{thm.rsc} Assume that $W^{\pp}(\om)\geq 0$.  Let $p\in \{1,\dots,d\}$ and suppose that $c_p(\om)>0$.   Then
\begin{enumerate}
\item If $p<\frac{d+1}{2}$, we have
\[\rssig_{1,p}(\om)>\frac{d-p+2}{d-p+1}\,c_p(\om).\]
\item If $p\geq\frac{d+1}{2}$, we have
\[\rssig_{1,p}(\om)\geq\frac{p+1}{p}\,c_p(\om).\]
Equality holds for the ball in Euclidean space $\R^{d+1}$.  If, moreover, $p>\frac{d+1}{2}$, then the ball is the only Euclidean domain for which equality holds.
 \end{enumerate}  
\end{thm}

The proof used a Reilly-type formula for differential forms.

\begin{remark}~
\begin{enumerate}
\item  Escobar's Conjecture, cited in Subsection~\ref{subsec.lowerboundgeom} asks whether $\sigma_1(\om)\geq c_1(\om)$ under the hypotheses that $c_1>0$ and that $\om$ has non-negative Ricci curvature.     Under the same hypotheses as Escobar's conjecture,  the case $p=1$ of Theorem~\ref{thm.rsc} says that $\rssig_{1,1}(\om)>\frac{d+1}{d}c_1(\om)$.  
\item When $p=d$, we have $c_d(\om)=dH$, where $H$ is the minimum on $\sig$ of the mean curvature.   Thus the case $p=d$ of the theorem says that if the Ricci curvature of $\om$ is non-negative and the mean curvature of the boundary is positive, then $\rssig_{1,d}(\om)\geq (d+1)H$.    The authors further prove in this case that equality holds if and only if $\om$ is a Euclidean ball.

\end{enumerate}
\end{remark}

While Theorem~\ref{rsthm1_3} assumes $p\geq 1$, Karpukhin \cite{Ka2017} applied this theorem along with Theorem~\ref{yyhps} to obtain new geometric lower bounds for the Steklov spectrum on functions:

\begin{thm}\label{kathmgeo} Assume that $W^{(2)}(\om)\geq 0$ and that $c_{d-1}(\om)> 0$. 
\begin{enumerate}
\item If $\dim(\om)=d+1\geq 4$, then
\[\sigma_k(\om)\leq \left(\frac{d-1}{c_{d-1}(\om) d}\right)\,\lambda_{m-1}(\sig)\]
for all $k\geq 1$.
\item If $\dim(\om) =3$, then for all $k\geq b+2$, where $b$ is the number of boundary components of $\om$, one has \[\sigma_k <\left( \frac{2}{3c_{d-1}(\om)}\right)\,\lambda_{k-1}.\]
\end{enumerate}

\end{thm}

The short proof in the case that $\om$ is orientable combines Theorem~\ref{rsthm1_3} with $p=d-1$ (after noting that non-negativity of $W^{(2)}$ yields the same for $W^{(d-1)}$) with Theorem~\ref{yyhps}.    The non-orientable case is proven by passing to a double cover.

 \subsubsection{Bounds in terms of geometry and the Hodge Laplace eigenvalues of the boundary.}~

\begin{nota}\label{nota.hodge}~
Let $\lap_\sig^{\pp}$ denote the Hodge Laplacian acting on smooth $p$-forms on $\sig$.  We denote the non-zero eigenvalues of $\lapsp$ by 
\[0<\hla_{1,p}(\sig)\leq \hla_{2,p}(\sig) \leq \dots\]
$\lapsp$ leaves invariant the subspaces of harmonic forms (the zero eigenspace of $\lapsp$), the exact forms,  and the co-exact forms.   For $p>0$, denote the nonzero eigenvalues of $\lapsp$ restricted to the exact $p$-forms by 
\[0<\lambda'_{1,p}(\sig)\leq \lambda'_{2,p}(\sig)\leq \dots\]
In particular, we have $\hla_{k,p}\leq \lambda'_{k,p}$.
\end{nota}

Observe that 
\begin{equation}\label{eq:hla1}
\lambda'_{k,1}(\Sigma)=\hla_{k,0}(\sig)=\lambda_{k+b-1}(\sig)
\end{equation}
where $b$ is the number of connected components of $\Sigma$ and $\lambda_j$ defined as in Notation~\ref{nota.lap}.

The following result was first proven by Raulot and Savo \cite{RaSa2012} when $R_0=0$.  The extension to the case $R_0>0$ is due to Kwong \cite{Kw2016}.

\begin{thm}\label{rasa8}\cite[Theorem 8]{RaSa2012}, \cite[Theorem 2.4 (part 4)]{Kw2016}.
We use Notation~\ref{hodge_decomp}, \ref{nota_pcurv} and \ref{nota.hodge}.   Assume that $H^p(\om,\sig;\R)=0$, that $W^{\pp}\geq R_0$, and that $c_p(\om)$ and $c_{d+1-p}(\om)$ are both non-negative.   Then
\[2\lambda'_{1,p}\geq R_0+c_p(\om)\rssig_{1,d-p}(\om)+c_{d+1-p}(\om)\rssig_{1,p-1}(\om).\]
\end{thm}

The hypotheses that the Ricci curvature is non-negative and the mean curvature is positive imply that $H^1(\om;\sig;\R)=0$.  (See \cite[Theorems 2.6.1 and 2.6.4, Corollary 2.6.2]{Sc1995}).     Thus, in view of Equation~\eqref{eq:hla1}, the special case of $p=1$ in Theorem ~\ref{rasa8} can be restated as follows:

\begin{thm}\label{rasa9} \cite[Theorem 9]{RaSa2012}, \cite[Theorem 2.4]{Kw2016}.
Assume that the Ricci curvature of $\om$ is  bounded below by a non-negative constant $R_0$ and that $\om$ has strictly convex boundary (i.e., $c_1(\om)>0$).  Letting $b$ be the number of boundary components of $\om$, we have 
\[2\lambda_b(\sig)\geq R_0+ c_1(\om)\rssig_{1,d-1}(\om)+ (d)(H) \sigma_1(\om)\]
where $H$ is the minimum value of the mean curvature of $\sig$.   If $d\geq 3$, then equality holds if and only if $R_0=0$ and $\om$ is a Euclidean ball.
\end{thm}

This theorem sharpens a result of Escobar \cite[Theorem 9]{Es1999}, who showed under the same hypotheses that 
$2\lambda_b(\sig)> (d)(H) \sigma_1(\om)$.

 \subsubsection{Manifolds isometrically immersed in $\R^m$}~
 
For compact $(d+1)$-dimensional Riemannian manifolds $\om$ with boundary that are isometrically immersed in $\R^m$ for some $m$,  Michel \cite{Mi2021} compared $\rssig_{1,p}$ and  $\rssig_{1,p-1}$ in the presence of geometric bounds.
 
 \begin{nota}\label{nota.shape} Define a tensor field $T^{(p)}$ on $\om$ as follows:  For $x\in \om$ and $\mathbf{n}$ a normal vector to $\om$ at $x$, the shape operator $S_{\mathbf{n}}$ gives rise to a self-sdjoint endomorphism $S_{\mathbf{n}}^p$ of $\Lambda^p_x(\om)$ given by 
 \[S_{\mathbf{n}}^p(\omega)(X_1,\dots, X_p)= \sum_{j=1}^p\,\omega(X_1, \dots, S_{\mathbf{n}}(X_j),\dots,X_p).\]
 Define $T^p_x$ to be the self-adjoint non-negative endomorphism of $\Lambda^p_x(\om)$ given by 
 \[T^{(p)}_x=\sum_{i=1}^r\, (S^p_{\mathbf{n}_i})^2\]
 where $\{\mathbf{n}_1,\dots, \mathbf{n}_r\}$ is an orthonormal basis for the normal bundle to $\om$ at $x$.   (The definition is independent of the choice of basis.)
 \end{nota}
 
 \begin{thm}\cite{Mi2021} Suppose that the $(d+1)$-dimensional compact Riemannian manifold $\om$ is isometrically immersed in $\R^m$.  Let $p\in\{1,\dots, d\}$.   In the notation of~\ref{nota_pcurv} and \ref{nota.shape},  if $\om$ has $p$-convex boundary and $W^{(p)}\geq T^{(p)}$, then
 \[\rssig_{1,p}(\om)\geq \rssig_{1,p-1}(\om) + \frac{c_p(\om)}{p}.\]
If $m=d+1$ and $\frac{d+3}{2}\leq p\leq d$, then equality holds when $\om$ is the Euclidean unit ball.
 \end{thm}

Note that when $m=d+1$, i.e., $\om$ is a Euclidean domain, the only relevant hypothesis is $p$-convexity since $W^{(p)}\equiv T^{(p)} \equiv 0$ in that case.

\subsection{Karpukhin's Dirichlet-to-Neumann operator and its spectrum}

In this subsection, $\om$ is assumed to be oriented.   Karpukhin's notion of the Dirichlet-to-Neumann operator on forms is a modification of that of Belishev and Sharafutdinov and is motivated by Maxwell's equations.  

\begin{defn}\label{dtnk}\cite{Ka2019} Given  $\eta\in \Acal^p(\sig)$, let $\teta\in \Acal^p(\om) $ be any solution of 
\begin{equation}
\label{eqteta}\begin{cases} \lap \teta=0 \\ i^*\teta=\eta\\\delta\teta=0
\end{cases}\end{equation}
Then Karpukhin's Dirichlet-to-Neumann operator $\kDtN_p:\Acal^p(\sig)\to \Acal^p(\sig)$ is defined by  \[\kDtN_p(\eta)=\nu\lrcorner \,d\teta.\]
\end{defn}
The author proves that $\kDtN_p$ is well-defined by first showing that the set of solutions of Equation~\eqref{eqteta} is an affine space with associated vector space $\Hcal_D(\om)$.
As was the case with $\rsDtN_0$, the operator $\kDtN_0$ coincides with $\DtN: C^\infty(\sig)\to C^\infty(\sig)$. 

Note that the condition $\nu\,\lrcorner\,\widehat{\eta}=0$ in Definition~\ref{dtnp} has been replaced here by the condition that $\teta$ be co-closed.  Observe that $i^*\omega=0 = \nu\,\lrcorner\,d\omega$ for all $\omega\in d(\Acal_D^{p-1}(\om))$, and the Hodge-Morrey-Friedrichs decomposition~\eqref{hmf} shows that $\Acal^p(\om)$ is the direct sum of $d(\Acal_D^{p-1}(\om))$ with the space of co-closed $p$-forms.

    The Hodge decomposition and the fact that all harmonic forms on a closed manifold are co-closed show that $\Acal^p(\sig)$ is the vector space direct sum of the subspaces of exact $p$-forms and of co-closed $p$-forms.  Karpukhin~\cite[Theorem 2.3]{Ka2019} showed 
\begin{itemize}
\item[(i)] $\kDtN_p$ is identically zero on the space of exact forms on $\sig$.
\item[(ii)]  The restriction to the space of co-closed forms on $\sig$ is a non-negative operator with compact resolvent and discrete spectrum.     Moreover, the kernel of the restricted operator has dimension $I_p$ given by
\begin{equation}\label{Ip}I_p:=\dim(i^*H^p(\om;\R)).\end{equation} 
\item[(iii)] $\kDtN_d \equiv 0$; thus the operator is of interest only on forms of degree at most $d-1$.
\end{itemize}
We will denote the eigenvalues of $\kDtN_p$ restricted to the space of co-closed forms by \[0\leq \ksig_{1,p}(\om)\leq \ksig_{2,p}(\om)\leq \dots\]   
and the non-zero eigenvalues by 
\[\hsigk_{k,p}(\om):=\ksig_{k+I_p,p}(\om).\]

\begin{ques}\label{ques: karpukhin pseudo} Is $\kDtN_p$ a pseudo-differential operator?
\end{ques}

\begin{ex}\cite{RaSa2014}, \cite{Ka2019}\label{pball} Let $\om=\B^{d+1}$ be the Euclidean unit ball of dimension $d+1$ and thus $\Sigma=\sph^d$.   Recall (Example~\ref{ex: ball}) that the Dirichlet-to-Neumann operator on $C^\infty(\sph^d)$ and the Laplace-Beltrami operator of $\sph^d$ have the same eigenspaces, namely the restrictions to $\sph^d$ of the homogeneous harmonic polynomials of degree $k$ for $k=0,1,2,\dots$.     
In the case of $p$-forms, somewhat similar relationships exist between the eigenspaces of the Hodge Laplacian of $\sph^d$ (computed by  Ikeda and Taneguchi  in 1978) and those of both $\DtNp$ and $\kDtN_p$ as we now describe.  
 
Let $P_{k,p}$ denote the space of all homogeneous polynomial $p$-forms of degree $k$.  (Here $\omega\in \Acal^p(\om)$ is said to be a homogeneous polynomial $p$-form of degree $k$ if, relative to the standard basis $\{dx^{i_1}\wedge\dots\wedge dx^{i_p}; \,1\leq 1_1 <\dots  i_p\leq d+1\}$ of $\Lambda^p(\R^{d+1})$, all its coefficient functions are homogeneous polynomials of degree $k$.)    
 Let \[P_{k,p}^{c\,c}= \{\omega \in P_{k,p}: \delta \omega =0\}.\] 
We first recall Ikeda and Taneguchi's computation of the eigenspaces of the Hodge Laplacian on $\sph^d$.    Recalling Notation~\ref{hodge_decomp}, let
\[H'_{k,p}=P_{k,p}^{c\,c}\cap \Hcal^p(\R^{d+1})\]
and 
\[H''_{k,p}= P_{k,p}^{c\,c}\cap \Acal^p_N(\R^{d+1}).\]
For $1\leq p\leq d-1$, the collection of subspaces $\{ i^*H'_{k,p},  i^*H''_{k,p} : k=1,2,\dots\}$ is a complete eigenspace decomposition for the Hodge Laplacian on $\Acal^p(\sph^d)$.  Moreover, the subspaces $i^*H'_{k,p}$, $k=1,2,\dots$ together span the space of exact forms, and the subspaces $i^*H''_{k,p}$ together span the co-exact forms. 

We first consider the co-exact subspaces.   Let $\eta \in  i^*H''_{k,p}$.  Since $H''_{k,p}$ consists of tangential co-closed  harmonic forms, we have $\widehat{\eta}=\teta $ in the notation of Definitions~\ref{dtnp} and \ref{dtnk}.  Thus $\DtNp$ and $\kDtN_p$ agree on this subspace.   For both operators, $i^*H''_{k,p}$ is an eigenspace with eigenvalue $p+k$. 

Next consider the exact subspaces.   It turns out that each $i^*H'_{k,p}$ is an eigenspace for $\DtNp$ with eigenvalue $(k+p-1)\frac{d+2k+1}{d+2k-1}$.  However, in contrast to the case of the co-closed forms, the tangential harmonic extensions of elements of $i^*H'_{k,p}$ are \emph{not} polynomial forms.   On the other hand, for all $k$, elements $\eta\in i^*H'_{k,p}$ have co-closed polynomial harmonic extensions $\teta$ lying in $H'_{k,p}$.    Since $d\teta=0$ and thus $\nu\,\lrcorner \,\teta=0$, the operator $\kDtN_p$ is identically zero on the exact forms, in agreement with statement (i) above.

Finally, for $p=d$, we have $\Acal^d(\sph^d) = d\,\Acal^{d-1}(S^d) \oplus \R\omega$, where $\omega$ is the volume form of $\sph^d$.   All three operators (the Hodge Laplacian, $\rsDtN_d$ and $\kDtN_d$) vanish on $\omega$ while the behaviour on the exact forms is the same as in the case $1\leq p\leq d-1$.   In particular, $\DtN^K_d\equiv 0$ in agreement with (iii).
 
 \end{ex}

The eigenvalues $\ksig_{k,p}$, $k=0,1,2\dots$, satisfy the variational characterisation (see \cite[Theorem 2.3]{Ka2019}):
\begin{equation}\label{kvar}
\ksig_{k,p}(\om)=\max_{E}\,\min_{\substack{\eta\perp E\\i^*\varphi=\eta}}\,\frac{\|d\varphi\|^2_{L^2(\om)}}{\|\eta\|^2_{L^2(\sig)}}.\end{equation}
where $E$ varies over all $k$-dimensional subspaces of $\Acal^p(\sig)$ consisting of co-closed $p$-forms.  The maximum is achieved when $E$ is spanned by the first $k$-eigenforms, $\eta$ is a $\ksig_{k,p}$-eigenfunction, and $\varphi$ satisfies Equations~\eqref{eqteta} (with $\varphi$ playing the role of $\teta$).  (Note that $d\varphi$ is independent of the choice of solution to Equations~\eqref{eqteta}.

\begin{thm}\label{kar_compare}\cite[Theorem 2.5]{Ka2019}. For $0\leq p\leq d-1$, one has 
\[\hsig_{k,p}(\om)\leq \hsigk_{k,p}(\om)\]
for all $k=1,2,\dots$
\end{thm}
The proof compares the variational characterisations of eigenvalues.

Girouard, Karpukhin, Levitin and Polterovich proved the following analogue of Equation~\eqref{eq: sigma = lambda +O(1)}.  Recall that the Hodge Laplacian $\lap_\sig^{\pp}$ on $\sig=\partial\om$ leaves invariant the subspace of co-closed $p$-forms in $\Acal^p(\sig)$.

\begin{thm}\cite[Theorem 5.6]{GKLP2021}\label{thm:ksig hodge}
 Let 
 \[0\leq \widetilde{\lambda}_0(\sig)\leq \widetilde{\lambda}_1(\sig)\leq \dots\]
 denote  the eigenvalues of $\lap_\sig^{\pp}$ restricted to the subspace of co-closed forms in $\Acal^p(\sig)$.  Then there exists a constant $C$ such that 
\begin{equation}\label{eq:karp p weyl}\left|\ksig_{k,p}(\om)-\sqrt{\widetilde{\lambda}_k(\sig)}\right|\leq C.\end{equation}
\end{thm}

As a consequence, they obtained the following Weyl law \cite[Theorem 5.8]{GKLP2021}:
\begin{equation}\label{eq: Weyl Karp p}
  \#\{\ksig_{k,p}(\om,g)<\sigma\}=\binom{d-1}{p}\frac{|\B^d||\sig|_g}{(2\pi)^d}\,\sigma^d + o(\sigma^{d})  
\end{equation}
where $d+1=\dim(\om)\geq 2$.

\begin{ques}\label{ques: Karp Weyl error}
Can the error bound in the Weyl Law~\eqref{eq: Weyl Karp p} for the asymptotics  of $\ksig_{k,p}$ be improved to $O(\sigma^{d-1})$?
\end{ques}
This would follow from an affirmative answer to Question~\ref{ques: karpukhin pseudo}.  See \cite[Remark 5.9]{GKLP2021} for further comments.

Karpukhin remarks that many of the eigenvalue inequalities for the Raulot-Savo eigenvalues $\rssig_{k,p}$ have analogues for $\ksig_{k,p}$ and in fact can be reproven using Theorem~\ref{kar_compare}.    He illustrates with the following generalised and strengthened version of Theorem~\ref{yyhps}.

\begin{thm}\label{khps}\cite[Theorem 2.7]{Ka2019}  Let $\om$ be a compact, oriented $(d+1)$-dimensional Riemannian manifold with smooth boundary $\sig$.    Then for every $j,k\in \Z^+$ and every $p\in\{0,1,\dots, d-1\}$, we have 
\[\hsigk_{j,p}(\om)\hsigk_{k, d-1-p}(\om) \leq \lambda^{c\,c}_{I_p+j+k -1 +\beta_{d-p} }(\sig).\]
Here $0\leq \lambda^{c\,c}_{1,p}(\sig)\leq \lambda^{c\,c}_{2,p}(\sig)\leq\dots$ are the eigenvalues of the Hodge Laplacian of $\sig$ acting on the space of co-closed $p$-forms.

If $j=k=1$, then equality holds for all $p\in\{0,1,\dots, d-1\}$ when $\om$ is a Euclidean unit ball.
\end{thm}

Observe that Theorem~\ref{yyhps} follows from the case $p=0$ of Theorem~\ref{khps} together with Theorem~\ref{kar_compare}.    In view of the fact that the inequality in Theorem~\ref{khps} is sharp when $j=k=1$, it is interesting to ask whether other inequalities for the Raulot-Savo eigenvalues have sharp versions for Karpukhin's eigenvalues.

\subsubsection{Conformally invariant case: $d=2p+1$}~

Suppose that the dimension $d+1$ of $\om$ is even and let $p=\frac{d-1}{2}$.  Then, analogous to the case of the Steklov spectrum on functions in dimension two, the variational characterisation of eigenvalues~\eqref{kvar} shows that the eigenvalues $\ksig_{k,p}$ are invariant under conformal changes of metric provided the conformal factor is trivial on the boundary $\sig$.

Motivated by the literature on optimizing Steklov eigenvalues on surfaces discussed in Section~\ref{sec:existence}, Karpukhin proposes:

\begin{prob}\cite{Ka2019}\label{prob:karp} Given a closed oriented Riemannian manifold $(\sig,h)$ of dimension $d=2p+1$, let $[\sig,h]_m$ denote the collection of all orientable Riemannian manifolds $(\om,g)$ with $\partial \om =\sig$, $g_{|\sig}=h$ and $\beta_{d-p}(\om)=m$.  For fixed $m$ and $k$, investigate \[\sup\,\{\ksig_{k,p}(\om,g): (\om,g)\in [\sig,h]_m\}.\] 
\end{prob}

The following theorem shows that the supremum is always finite.  

\begin{thm}\cite[Theorem 2.11]{Ka2019}\label{khpsp} Suppose $d=2p+1$.   Then with the notation and hypotheses of Theorem~\ref{khps}, we have
 \[\hsigk_{k,p}(\om)^2\leq \lambda^{c\,c}_{I_p+2k -1 +\beta_{p+1} }(\sig).\]
 Moreover, when $k\leq \frac{1}{2}\binom{2p+2}{p+1}$, equality holds for the Euclidean unit ball.

\end{thm}

The inequality is immediate from Theorem~\ref{khps}.

\begin{conj}\label{conjs:karpuhkin}\cite{Ka2019}~
\begin{itemize}
\item The inequality in Theorem~\ref{khpsp} is sharp for all $k$.
\item When $k\leq \frac{1}{2}\binom{2p+2}{p+1}$, equality holds only for the Euclidean unit ball.
\end{itemize}

\end{conj}

\section{Inverse problems: positive results}\label{sec:inv probs pos}

In this section we discuss some recent ``positive" inverse spectral results for the Steklov problem. By ``positive" we mean that we are looking for geometric information that \emph{can} be recovered from the Steklov spectrum. 

The analogous problem for the Laplacian, either on a closed manifold or with Dirichlet, Neumann, or Robin conditions on a domain, has been well-studied. The traditional approach is to look for \emph{spectral invariants} -- quantities which can be determined from the spectrum of the Laplacian and which have some geometric meaning. With these invariants, one may then try to prove results such as spectral rigidity: proving that all elements of a certain restricted class of domains or manifolds are spectrally distinguishable.

The hunt for spectral invariants of the Laplacian starts with Weyl's asymptotics. Weyl proved in 1911 (in the cases $d=1$ and $d=2$) that when $\Omega\subset\mathbb R^{d+1}$ is bounded and has sufficiently regular boundary,
\[N(\lambda_k\leq\lambda) = C_d|\Omega|\lambda^{(d+1)/2} + o(\lambda^{(d+1)/2}),\]
where $C_d$ is a dimensional constant \cite{We1911}. This estimate has been generalised to all dimensions and refined significantly over the last century. Notably, in 1980, Ivrii proved a \emph{two-term} asymptotic formula under the assumption that the set of all periodic geodesic billiards on $\Omega$ has measure zero \cite{Iv1980}:
\[N(\lambda_k\leq\lambda) = C_d|\Omega|\lambda^{(d+1)/2} + C_d'|\partial\Omega|\lambda^{d/2} + o(\lambda^{d/2}).\]
Similar asymptotics hold under similar assumptions in the case where $\Omega$ is a closed manifold (see for example Duistermaat and Guillemin \cite{DuGu1975}), or a domain in a complete manifold; we will not discuss them here.

The use of these asymptotics for inverse problems is that since the counting function is spectrally determined, so are the coefficients of the expansion on the right-hand side. This shows that the volume of $\Omega$ may be determined from the spectrum of the Laplacian, and similarly the volume of $\partial\Omega$ may be recovered under the aforementioned dynamical assumptions. 

In order to obtain more spectral invariants, we may turn to \emph{heat} invariants. It was shown by McKean and Singer in 1967 \cite{McSi1967} that for a smooth manifold with boundary of dimension $d+1$, the trace of the heat kernel, $Tr H(t)$, has a short-time asymptotic expansion as $t\to 0$ given by
\[Tr H(t) = \sum_{j=0}^{\infty}a_jt^{-(d+1)/2 + j/2}.\]
Here the quantities $a_j$ may be computed recursively, and each is an integral, over $\Omega$, of a polynomial in the curvature of $\Omega$ and its derivatives, plus an analogous integral over $\partial\Omega$ involving the principal curvatures of the boundary. See the appendix of \cite{OPS1988a} and the references therein, including \cite{Me1983}. The key point is that since
\[Tr H(t)=\sum_{k=1}^{\infty}e^{-\lambda_kt},\]
the heat trace, and thus all of the coefficients $a_i$, are spectral invariants. We call the $a_i$ the ``heat invariants". 

The first term $a_0$ is a dimensional constant times the volume of $\Omega$, so that volume is a spectral invariant. Likewise, it follows from examining $a_1$ that the volume of $\partial\Omega$ is a spectral invariant, even when the dynamical assumption need to obtain a second term in Weyl's law is not satisfied. The coefficient $a_2$ was first studied by McKean and Singer \cite{McSi1967}, who showed that if $\Omega$ is a compact surface with boundary then $a_2$ is a nonzero multiple of the Euler characteristic, and therefore the Euler characteristic is a spectral invariant. More can be done, though the calculation of the $a_i$ is recursive and some of the formulas for $a_i$, $i\geq 3$, are harder to interpret geometrically. See, for example, \cite{Gi2003} for a detailed discussion of these formulas.

A more exotic approach to spectral invariants is to make use of the \emph{spectral zeta function}, defined for $\Re s$ large by
\[\zeta(s):=\sum_{k=1}^{\infty}\lambda_k^{-s}.\]
This function may then be meromorphically continued to the complex plane by writing
\[\zeta(s)=\frac{1}{\Gamma(s)}\int_{0}^{\infty}t^{s-1}Tr H(t)\, dt\]
and taking advantage of the short-time asymptotic expansion of the heat trace. Since the spectral zeta function is spectrally determined, so are all special values of the zeta function and its derivatives. 

A particularly useful special value is the \emph{determinant of the Laplacian}, defined by
\[\det\Delta := e^{-\zeta'(0)}.\]
In a series of celebrated papers in the late 1980s, Osgood, Phillips, and Sarnak used both the determinant and the heat invariants to show that any isospectral family of planar domains is compact in a natural $C^{\infty}$ topology \cite{OPS1988,OPS1988a}.  Specifically, they used the determinant together with some calculus of variations and the first couple of heat invariants to prove compactness in the Sobolev space $H^s$ for some $s>0$ \cite{OPS1988}. They then used the heat invariants $\{a_i\}$ to bootstrap the argument and obtain compactness for every real $s$, with each successive heat invariant improving the Sobolev order by one \cite{OPS1988a}. Although they were not able to explicitly compute each heat invariant without recursion, they were able to isolate key terms which enabled them to prove compactness. See \cite[Appendix]{OPS1988a}.

\subsection{Steklov spectral invariants}

Many of the techniques used for the Laplacian are naturally used in the Steklov problem as well. For example, there are Weyl asymptotics:
\begin{equation}\label{eq:stekweyl}
N(\sigma_k\leq\sigma) = C_d|\partial\Omega|\sigma^{d} + o(\sigma^{d}),
\end{equation}
where again $C_d$ is an explicit dimensional constant. 

The asymptotics \eqref{eq:stekweyl} are due to various authors under various assumptions on smoothness of the boundary. In the case where $\Omega$ is a domain in Euclidean space and $\partial\Omega$ is $C^1$, it was proved by Agranovich in 2006 \cite{Ag2006}. 
\begin{ques}\label{ques:weylrough} Show that \eqref{eq:stekweyl} holds whenever $\partial\Omega$ is Lipschitz.
\end{ques} 
Recently, Karpukhin, Lagac\'e, and Polterovich have done this for $d+1=2$ \cite{KLP2022}.

Observe that the exponent in \eqref{eq:stekweyl} is $d$ rather than $(d+1)/2$. At least in the case of smooth boundary, this is because the Steklov spectrum is the spectrum of a pseudodifferential operator of order 1 on a $d$-dimensional manifold, rather than the spectrum of a differential operator of order 2 on a $d+1$-dimensional manifold. For the same reason, the leading term is a geometric constant times the volume of the \emph{boundary} $\partial\Omega$, showing that $|\partial\Omega|$ is a Steklov spectral invariant.

As with the Laplacian, there is sometimes a second term in these asymptotics, subject to certain assumptions concerning the dynamics of the geodesic flow on $\partial\Omega$. For example, if $\partial\Omega$ is smooth and its periodic geodesics form a set of measure zero in its cotangent bundle, then one recovers a two-term Weyl law from the results of \cite{DuGu1975}, see \cite[Remark 5.7]{PoSh2015}. However, as with the Laplacian, there are no further terms.

Although the Weyl asymptotics are the best results available in dimension $d+1>2$, the case $d+1=2$ is quite different, due to the conformal invariance of the Laplacian. Suppose that $\Omega$ is a simply connected surface with smooth boundary. In this case, it has been known since the 1970s that the Steklov spectrum has very precise asymptotic behaviour \cite{Sh1971,Ro1979}. Specifically, let $D_{|\partial D|}$ be a disk with the same perimeter as $\Omega$. Then
\[|\sigma_k(\Omega)-\sigma_k(D_{|\partial D|})| = O(k^{-\infty}).\]
In other words, the eigenvalues of $\Omega$ converge quite rapidly to a doubled arithmetic progression.

\begin{remark}
\label{rem:gpps2014}
In 2014, Girouard, Parnovski, Polterovich and Sher investigated the $d=2$ case where $\Omega$ is not simply connected \cite{GPPS2014}. Suppose that $\Omega$ is a smooth surface with boundary, which has $N$ boundary components with lengths $\ell_1,\ldots,\ell_N$, and let $U$ be a union of $N$ disks with perimeters $\ell_1,\ldots,\ell_N$. Then
\[|\sigma_k(\Omega)-\sigma_k(U)|=O(k^{-\infty}).\]
Using some elementary number theory, they were able to prove that the multiset $\{\ell_1,\ldots,\ell_N\}$ of boundary lengths may be determined algorithmically from the Steklov spectrum $\sigma(\Omega)$ \cite{GPPS2014}. This work is described in detail in \cite[Section 2]{GPPS2014}.
\end{remark}

The other methods of obtaining spectral invariants have also been used in the Steklov setting. Polterovich and Sher made a systematic study of Steklov heat invariants \cite{PoSh2015}, proving among other things that the total mean curvature of $\partial\Omega$ is a Steklov spectral invariant. This was then used to show that a ball in $\mathbb R^3$ is uniquely determined by its Steklov spectrum among all domains in $\mathbb R^3$ \cite{PoSh2015}.

The Steklov spectral zeta function may be defined similarly to the Laplace spectral zeta function. It was exploited in a notable series of papers by Julian Edward from the late 1980s and early 1990s. In his first paper he used the zeta function to show that a disk in $\mathbb R^2$ is uniquely determined among simply connected domains by its Steklov spectrum \cite{Ed1993}. (This may also be deduced either from the Weyl asymptotics or from Weinstock's inequality \ref{ineq:weinstock}, and is in fact true even among all domains since the number of boundary components is determined, see Remark \ref{rem:gpps2014}.) In a follow-up work he investigated compactness of isospectral sets of simply connected planar domains along the lines of Osgood, Phillips, and Sarnak \cite{Ed1993a}. Although he was not quite able to prove compactness of isospectral sets in the $C^{\infty}$ topology, he was able to prove it in the $H^{5/2-\epsilon}$ topology for any $\epsilon>0$, thus carrying out the first step of the program of Osgood, Phillips, and Sarnak \cite{OPS1988,OPS1988a} in the Steklov setting. 

\subsection{Isospectral compactness}

Recently, Alexandre Jollivet and Vladimir Sharafutdinov \cite{JoSh2018_2} have completed the program started by Edward.
\begin{thm}\label{thm:joshmain} \cite[Theorem 1.3]{JoSh2018_2} Suppose that $\{\Omega_i\}$ is a sequence of Steklov isospectral, compact, simply connected, possibly multisheeted planar domains. Then there exists a subsequence of $\{\Omega_i\}$ which converges to a limit $\Omega$ in the $C^{\infty}$ topology.
\end{thm}
The proof is an interesting variation on the techniques of Osgood, Phillips, and Sarnak in \cite{OPS1988,OPS1988a} and we give a sketch of it here.

The main problem is that in \cite{OPS1988a}, the authors used the Laplace heat invariants to perform their bootstrapping procedure. However, in two dimensions, all Steklov heat invariants beyond the first are \emph{non-local} and as such cannot be computed by the same recursive methods typically used to compute Laplace heat invariants. In general, as shown in \cite{PoSh2015}, only the first $d$ Steklov heat invariants are local. So there is no hope of being able to use a structure theorem for heat invariants to complete the proof. 

Nevertheless, Jollivet and Sharafutdinov are able to use a set of invariants drawn from the Steklov zeta function, which they call ``zeta-invariants", to replace the role of the heat invariants. These invariants were originally defined in a paper of Malkovich and Sharafutdinov \cite{MaSh2015}. To define these zeta-invariants, Jollivet and Sharafutdinov use the well-known fact that if $\Omega$ is a compact, simply connected planar domain, then the Steklov eigenvalue problem is equivalent to the eigenvalue problem for the operator $a\mathcal D_{\mathbb S^1}$, where
\begin{equation}\label{eq:a(z)}a(z)=|\Phi'(z)|^{-1},\ \Phi:\mathbb D\to \Omega\textrm{ is biholomorphic.}\end{equation}
In this translation, also used by Edward \cite{Ed1993,Ed1993a} and dating back much earlier \cite{Ro1979}, the inverse spectral problem is equivalent to the problem of determining $a(z)$ from the spectrum of $a\mathcal D_{\mathbb S^1}$. The zeta-invariants are defined by
\[Z_m(a):=\textrm{Tr}[(a\mathcal D_{\mathbb S^1})^{2m} - (-ia\frac{d}{d\theta})^{2m}].\]

Although Edward did not use the notation of zeta-invariants, his proof of isospectral compactness in $H^{5/2-\epsilon}$ uses the zeta-invariants $Z_1(a)$ and $Z_2(a)$ as well as the special values of the zeta function $\zeta(-1)$ and $\zeta(-3)$ \cite{Ed1993a}. To do this he computed a formula for $Z_1(a)$ in terms of the Fourier coefficients $\hat a_m$ of $a$:
\[Z_1(a)=\frac 23\sum_{n=2}^{\infty}(n^3-n)|\hat a_n|^2.\]
This formula was then generalised by Malkovich and Sharafutdinov to a formula for $Z_m(a)$ \cite{MaSh2015}.

The key novel observation of Jollivet and Sharafutdinov is a lower bound for $Z_m(a)$, namely
\begin{thm} \cite[Theorem 1.1]{JoSh2018_2}. Let $m\geq 1$ and let $b=a^m$. Then there is a constant $c_m$ depending only on $m$ such that for any smooth function $a$,
\[Z_m(a)\geq c_m\sum_{n=m+1}^{\infty}n^{2m+1}|\hat b_n|^2.\]
\end{thm}
This lower bound builds on their prior work in \cite{JoSh2018} and turns out to be enough for them to prove Theorem \ref{thm:joshmain}. In essence it fills the same role that the structure theorem for higher heat invariants plays in \cite{OPS1988a}. Given a sequence of isospectral domains $\om_k$, let $a_k$ be the corresponding sequence of functions defined as in Equation~\eqref{eq:a(z)}. The authors use their lower bound on $Z_m(a)$ to obtain control over the $H^{m+\frac 12}$-norms of $a_k$ and thus prove compactness in $H^{m+\frac 12-\epsilon}$ for each $m$.

One might also consider the case of multiply connected domains. Indeed, Osgood, Phillips, and Sarnak were able to prove an isospectral compactness theorem in that setting as well \cite{OPS1989}.
\begin{ques}\label{ques:multconnops}
Are families of multiply connected, compact, Steklov isospectral planar domains necessarily compact in the $C^{\infty}$ topology?
\end{ques}

\begin{remark} For certain restricted classes of manifolds, much stronger results than isospectral compactness may sometimes be deduced. See, for example, the papers \cite{DaKaNi2021} by Daud\'e, Kamran and Nicolau and \cite{Gendron2020} by Gendron, in the setting of warped product manifolds.
\end{remark}

\subsection{Steklov spectral asymptotics with corners}

As discussed in Remark \ref{rem:gpps2014}, Steklov spectral asymptotics in two dimensions, for $\partial\Omega$ smooth, are straightforward: the spectrum converges extremely rapidly to a doubled arithmetic progression, or a union of such progressions in the case where $\partial\Omega$ has more than one connected component. However, it has been known for a long time that the presence of corners on $\partial\Omega$ changes the picture substantially. We highlight some recent developments on asymptotics in these settings and then discuss the inverse spectral applications.

\subsubsection{Sloshing asymptotics} The first recent results on these asymptotics come in the case of a mixed Steklov-Neumann problem known as the \emph{sloshing problem}. In this problem, $\Omega$ is taken as a simply connected region in $\mathbb R^2$ whose boundary decomposes in a piecewise smooth fashion as $S\cup W$, where $S$ is a segment of the $x$-axis of length $L$ from point $A$ to point $B$, and $W$ is a curve in the lower half-plane from point $B$ to point $A$. We let $\alpha$ and $\beta$ be the interior angles of $\Omega$ at $A$ and $B$ respectively. We then consider the mixed Steklov-Neumann problem obtained by imposing Steklov conditions on $S$ and Neumann conditions on $W$. There is an analogous Steklov-Dirichlet problem as well.

Computations of these eigenvalues can be done directly in some very specific cases such as the equilateral triangle. Based on these computations and on considerable numerical evidence, Fox and Kuttler conjectured in 1983 \cite{FoKu1983} that in the case where $\alpha=\beta$,
\[\sigma_k\cdot L = \pi(k+\frac 12) - \frac{\pi^2}{4\alpha}+o(1).\]
In 2017, Levitin, Parnovski, Polterovich, and Sher proved the following generalisation of the Fox-Kuttler conjecture \cite{LPPS2017}:
\begin{thm}\label{thm:sloshingmain} \cite[Theorem 1.1]{LPPS2017}
Suppose that $\Omega$ is a sloshing domain where $W$ is piecewise $C^1$ and the angles $\alpha$ and $\beta$ are less than $\pi/2$. Then
\[\sigma_k\cdot L = \pi(k+\frac 12) -\frac{\pi^2}{8}(\frac 1{\alpha}+\frac 1{\beta}) + o(1).\]
\end{thm}
Observe that the angles produce a shift in the asymptotics of the eigenvalues. Moreover, if the domain is symmetric, the angles may be recovered from those eigenvalue asymptotics.
\begin{remark} Some comments:
\begin{itemize}
    \item There is a similar result for the Steklov-Dirichlet problem; all that changes is that the minus sign in front of the term with the angles becomes a plus sign.
    \item If $W$ is straight near the corners, the remainder estimate can be improved from $o(1)$.
    \item The angles $\alpha$ and/or $\beta$ may be allowed to equal $\pi/2$ under the assumption that $\Omega$ satisfies a local John's condition (see \cite{LPPS2017} for the definition).
\end{itemize}
\end{remark}

The proof of Theorem \ref{thm:sloshingmain} relies on a construction of quasimodes, that is, approximate Steklov eigenfunctions corresponding to approximate Steklov eigenvalues. These quasimodes are constructed by using what physicists call the method of matched asymptotic expansions: namely, one may glue a traveling wave along $S$ to solutions of model problems near its endpoints. Conveniently, appropriate model solutions near the corners were originally given by Peters \cite{Pe1950}. The gluing construction, in the end, produces a sequence of quasimodes which are accurate enough to show that there is a real Steklov eigenvalue near each quasi-eigenvalue.

A major challenge in the proof, and in any similar quasimode construction, is that it is difficult to show that this construction does not ``miss" any Steklov eigenvalue -- in essence, that the set of quasimodes is complete. In this case this difficulty is handled by combining ODE techniques based on \cite{Na1967} with domain monotonicity.

\subsubsection{Steklov spectral asymptotics for polygons}

The methods used to tackle the sloshing problem in \cite{LPPS2017} were further developed by the same authors to give precise Steklov spectral asymptotics for curvilinear polygons \cite{LPPS2019}. Throughout this subsection, we let $\Omega$ be a curvilinear polygon, piecewise smooth, with finitely many nonzero interior angles $\alpha_j$ and side lengths $\ell_j$, arranged so that the $\ell_j$ appear in clockwise order and $\alpha_j$ is the angle between $\ell_j$ and $\ell_{j+1}$.

\begin{thm}\label{thm:steklovpolymain} \cite[Theorem 1.4]{LPPS2019} Suppose that $\Omega$ is a curvilinear polygon all of whose interior angles are in $(0,\pi)$. Then there exists a sequence of quasi-eigenvalues $\{\hat\sigma_k\}$, which may be explicitly computed in terms of the side lengths and the angles, as well as an explicitly computable $\epsilon_0>0$, for which
\[\sigma_k=\hat\sigma_k + O(k^{-\epsilon_0}).\]
\end{thm}

Although $\hat\sigma_k$ may be computed in all circumstances, giving a full account is notationally complicated. Here we restrict to the case where no $\alpha_j$ has the form $\pi/n$ for $n\in\mathbb N$. In this case, we define matrices as follows:
\[A(\alpha_j) = \begin{bmatrix} \csc(\frac{\pi^2}{2\alpha_j}) & -i\cot(\frac{\pi^2}{2\alpha_j}) \\ i\cot(\frac{\pi^2}{2\alpha_j}) & \csc(\frac{\pi^2}{2\alpha_j})\end{bmatrix};\quad B(\ell_j,\sigma) = \begin{bmatrix} e^{i\ell_j\sigma} & 0 \\ 0 &e^{-i\ell_j\sigma} \end{bmatrix};\quad C(\alpha_j,\ell_j,\sigma)=A(\alpha_j)B(\ell_j,\sigma).\]
Then for any given $\Omega$ with angles $\vec\alpha=(\alpha_1,\ldots,\alpha_n)$ and side lengths $\vec\ell=(\ell_1,\ldots,\ell_n)$, we can define the parameter-dependent matrix $T(\vec\alpha,\vec\ell,\sigma)$ by
\[T(\vec\alpha,\vec\ell,\sigma) = C(\alpha_n,\ell_n,\sigma)C(\alpha_{n-1},\ell_{n-1},\sigma)\dots C(\alpha_1,\ell_1,\sigma).\]
With this, we define the following:
\begin{definition}
Suppose that no $\alpha_i$ is equal to $\pi/n$. Then the \emph{quasi-eigenvalues} $\hat\sigma_k$ of $\Omega$ are the non-negative values of $\sigma$ for which 1 is an eigenvalue of $T(\vec\alpha,\vec\ell,\sigma)$.
\end{definition}
The multiplicity of $\hat\sigma_k$ is the geometric multiplicity of 1 as an eigenvalue of $T(\vec\alpha,\vec\ell,\hat\sigma_k)$, except that if $\hat\sigma_1=0$ it has multiplicity 1. Note also that 1 is an eigenvalue of $T(\vec\alpha,\vec\ell,\sigma)$ if and only if the trace of $T(\vec\alpha,\vec\ell,\sigma)$ equals 2, which assists in computations \cite{LPPS2019}. There are similar, if more notationally involved, expressions for $\hat\sigma_k$ when one or more of the angles $\alpha_j$ has form $\pi/n$.

The condition that $T(\vec\alpha,\vec\ell,\sigma)$ have eigenvalue 1 should be viewed as a quantisation condition. Along each side of a polygon, a Steklov eigenfunction of sufficiently high energy $\sigma$ will look like a traveling wave with frequency $\sigma$. The matrices $A(\alpha)$ should be viewed as transmission matrices at the corners, and $B(\ell,\sigma)$ as matrices representing propagation along each side. The condition that 1 be an eigenvalue of $T$ is the condition that when the wave propagates all the way around the polygon, it comes back with the same phase.

The proof of Theorem \ref{thm:steklovpolymain} is again a quasimode construction, gluing together traveling waves along each side with model solutions in each corner. The model solutions in the corner are obtained by doubling the Peters solutions used in the sloshing problem. There are a number of technical complications related to the curvilinear boundary, and proving completeness is again quite difficult, but all this can be overcome.

A natural open question is the following:
\begin{ques}\label{ques:polygonsbigangles}
Can the condition that the angles of $\Omega$ are in $(0,\pi)$ be relaxed to $(0,2\pi)$?
\end{ques}
Based on numerical evidence the answer appears to be ``yes", but the proof techniques used here depend very strongly on the angles being less than $\pi$.

\subsubsection{Inverse spectral problem for polygons}

As discussed earlier, it is possible to recover the boundary lengths of a surface from its Steklov spectrum. This poses the natural question: what geometric information can be recovered from the Steklov spectrum of a polygon $\Omega$? In particular, can the side lengths and/or angles be recovered? The following theorem, due to Krymski, Levitin, Parnovski, Polterovich, and Sher \cite{KLPPS2021}, indicates that the answer is \emph{generically} yes, at least up to a natural symmetry involving the angles.

\begin{thm}\label{thm:mainklpps} \cite[Theorem 1.9]{KLPPS2021} Suppose that $\Omega$ is a curvilinear polygon, with all interior angles in $(0,\pi)$, and with no interior angle equal to $\pi/n$ for any $n$. Suppose also that the lengths $\{\ell_1,\ldots,\ell_n\}$ are incommensurable over $\{-1,0,1\}$. Then
\begin{enumerate}
    \item The side lengths $\ell_j$, in order, may be recovered from the Steklov spectrum.
    \item The vector
    \[\cos(\vec\alpha):=(\cos\frac{\pi^2}{2\alpha_1},\ldots,\cos\frac{\pi^2}{2\alpha_n})\]
    may be recovered from the Steklov spectrum.
\end{enumerate}
Moreover, the procedure to recover these is constructive.
\end{thm}

The procedure has two steps. First, the Steklov spectrum is used to reconstruct the polynomial $T(\vec\alpha,\vec\ell,\sigma)$ via the construction of an infinite product and use of the Hadamard-Weierstrass factorisation theorem. Second, the vectors $cos(\vec\alpha)$ and $\vec\ell$ are recovered from $T(\vec\alpha,\vec\ell,\sigma)$. This second recovery is fully algorithmic and can be implemented by computer.

\begin{remark} 
\begin{itemize}
    \item Observe that there is no hope of recovering $\vec\alpha$ from $T(\vec\alpha,\vec\ell,\sigma)$, as the definition of $T$ may be written so that it only involves $\cos(\vec\alpha)$. 
\item There are further results when one or more of the angles is $\pi/n$ for $n$ even; see \cite[Section 1.3]{KLPPS2021} for details.
\item The condition about incommensurability is necessary, as without this assumption one may construct two different polygons with the same $T(\sigma)$. The same is true if one or more angles are $\pi/n$ for $n$ odd. See \cite[Subsection 3.3]{KLPPS2021}.
\end{itemize}
\end{remark}

Many questions here are still open, for example:
\begin{ques}\label{ques:recoverangles}
Under the hypotheses of Theorem \ref{thm:mainklpps}, can the angles themselves, and not just $\cos(\pi^2/2\alpha)$, be recovered from the Steklov spectrum?
\end{ques}
In order to address this one will need to use some sort of non-asymptotic information from the Steklov spectrum, as the asymptotic behaviour does not distinguish between two different angles with the same value of $\cos(\pi^2/2\alpha)$.

\subsection{Inverse problems on orbifolds}\label{subsec:orb pos}

Orbifolds are a generalisation of manifolds in which certain types of singularities may occur.  The simplest examples of orbifolds, the so-called ``good orbifolds'', are quotients $\orb:=H\bs M$ of a manifold $M$ by a discrete group $H$ acting smoothly and effectively with only finite isotropy.   (``Effective'' means that the identity element of $\Gamma$ is the only element that acts as the identity transformation.)  Points in $\orb$ can be viewed as $H$-orbits of the action of $H$ on $M$, hence the name ``orbifold''.  Letting $\pi: M\to \orb$ be the projection, the boundary of $\orb$ is given by 
\begin{equation}\label{eq.boundary}\partial\orb=\pi(\partial M).\end{equation}
Singularities occur at elements of $\orb$ for which the corresponding orbit has non-trivial isotropy in $H$.  The \emph{isotropy type} of the singularity is the isomorphism class of the isotropy group.   

Arbitrary $n$-dimensional orbifolds locally have the structure of good orbifolds.  Every point in the interior of an orbifold has a neighborhood $U$ modelled on a quotient $\Gamma\bs\widetilde{U}$ of an open set $\widetilde{U} \subset\R^n$ by an effective action of a (possibly trivial) finite group.   A neighborhood of each boundary point has a similar model but with the role of $\R^n$ played by a closed half-space of $\R^n$.

 Orbifolds are stratified spaces:  the regular points form an open dense stratum, and the singular strata are connected components of the set of singular points with a given isotropy type.
 
 \begin{ex}\label{2dorb}
In the two-dimensional case, all singular strata are classified as follows:
\begin{enumerate}
\item Cone points $p$ of order $m$, $2\leq m\in \Z^+$:  A neighborhood of $p$ is modelled on the quotient of a disk by the cyclic group generated by rotation.  
\item Reflectors: These are 1-dimensional singular strata.   A neighborhood of a point in the reflector is modelled on the quotient of a disk by reflection across a diameter.
\item Dihedral points:   These are points where two reflectors come together at an angle of $\frac{\pi}{m}$ for some $m\in\Z^+$.  A neighborhood of the point is modelled on the quotient of a ball by a dihedral group of order $2m$.
  \end{enumerate}

For $n$-dimensional orbifolds, one can have singular strata of any dimension $<n$.  The only strata of co-dimension one are reflectors (modelled on quotients of a ball by reflection across a plane through the center).
\end{ex}

Riemannian metrics on a good orbifold $\mathcal{O}:=H\bs M$ are defined by $H$-invariant Riemannian metrics on $M$.  Letting $\pi:M\to \orb$ be the projection, a function $f:\orb\to\R$ is said to be $C^\infty$ if $\pi^*f\in C^\infty(M)$.  Since the Laplacian of $M$ commutes with the isometric action of $H$, one can define a Laplace operator on $C^\infty(\orb)$ by $\pi^*\Delta_\orb f =\Delta_M\circ \pi^*f$, where $\Delta_M$ is the Laplace-Beltrami operator of $M$.  On arbitrary orbifolds, Riemannian metrics are defined locally using orbifold charts, subject to a compatibility condition to obtain a global Riemannian metric.  The Laplacian is similarly defined.

The Dirichlet-to-Neumann operator and Steklov spectrum on orbifolds were introduced in \cite{ADGHRS2019} by Arias-Marco, Dryden, Gordon, Hassannezhad, Ray, and Stanhope  The Dirichlet-to-Neumann operator on compact Riemannian orbifolds is a pseudo-differential operator with discrete spectrum.  The variational characterisation of eigenvalues \ref{eq:Stek Rayleigh min max} remains valid in the orbifold setting.

\begin{remark}\label{paorb}
Any orbifold $\orb$ whose only singular strata are reflectors is a good orbifold.  Moreover, the underlying topological space $\underline{\orb}$ of $\orb$ has the structure of a manifold with boundary.    $\partial\underline{\orb}$ is the union of the orbifold boundary $\partial\orb$ and the reflector strata.  For example, a ``half-disk'' orbifold, the quotient of a disk by reflection, is pictured in Figure~\ref{fig:halfdisk}; the reflector stratum is indicated by a double line.
\begin{figure}
  \centering
  \includegraphics[width=3cm]{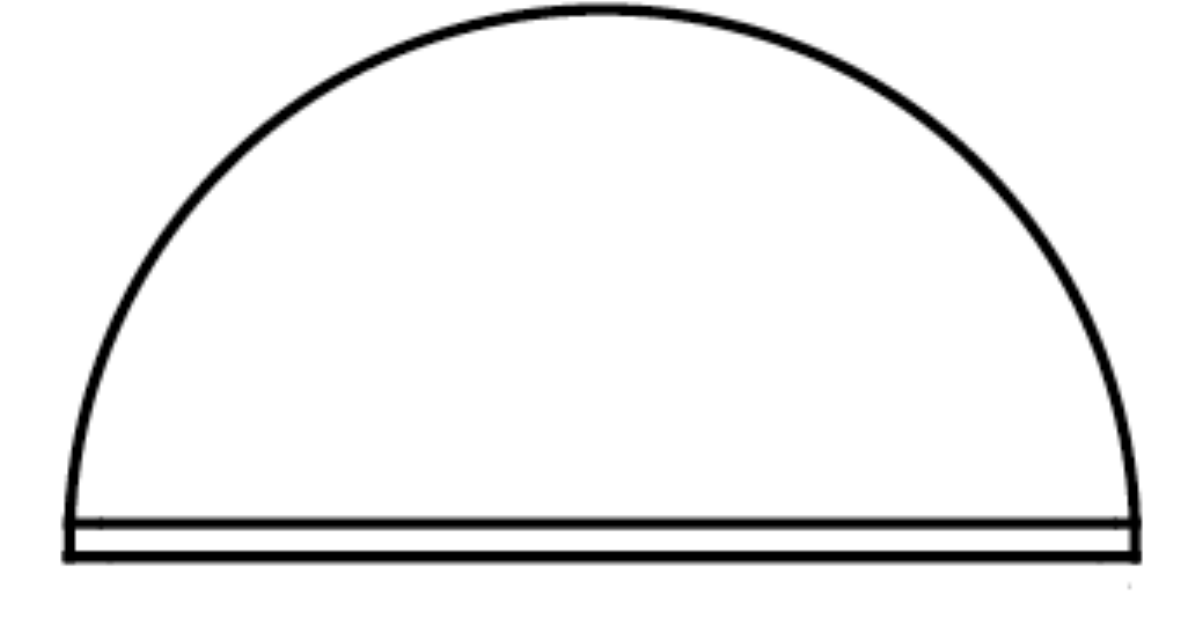}
 \caption{}
  \label{fig:halfdisk}
  \end{figure}

 Functions $f\in C^\infty(\orb)$ must have normal derivative zero on the reflectors since they pull back to reflection-invariant functions on the double $M$.    Thus the Steklov eigenvalue problem on $\orb$ is equivalent to a mixed Steklov-Neumann eigenvalue problem on the manifold $\underline{\orb}$.
\end{remark}

For both the Laplace spectrum and the Steklov spectrum, it is natural to ask the following:

\begin{itemize}\item \emph{Does the spectrum detect the presence of singularities?   Equivalently, are Riemannian orbifolds with singularities spectrally distinguishable from Riemannian manifolds?}

\end{itemize}

This question remains open even in the case of the Laplace spectrum although there are many partial results.    The heat asymptotics for closed Riemannian orbifolds were developed by Donnelly \cite{Do1976} in the case of good Riemannian orbifolds and generalised to arbitrary closed Riemannian orbifolds by Dryden, Gordon, Greenwald, and Webb (\cite{DGGW2008} and errata \cite{DGGW2017}).   From the heat asymptotics, one can show for example that the Laplace spectrum detects whether an orbifold contains any singular strata of odd co-dimension.  

In the setting of the Steklov spectrum, the question above can be split into two questions:

\begin{ques}\label{ques: orb boundary}
Does the Steklov spectrum detect the presence of singularities on the boundary of a compact Riemannian orbifold?  
\end{ques}

\begin{ques}\label{ques: orb interior}
Does the Steklov spectrum detect the presence of singularities in the interior of a compact Riemannian orbifold?
\end{ques}

These questions were studied in the case of orbisurfaces, i.e., two-dimensional orbifolds, in \cite{ADGHRS2019}.   Before stating the results, we note that every component of the boundary of an orbisurface is a closed one-dimensional orbifold that has one of the following two structures: 
\begin{itemize}
\item Type I:  a circle.   This type has no singularities.
\item Type II:  a ``half-circle" $\Z_2\bs S^1$, the quotient of a circle by a reflection.  This type has two singularities, both of which are reflector points.
\end{itemize}
For example, the half-disk orbifold in Remark~\ref{paorb} has boundary of the second type.

The following proposition and corollary generalise the results of Girouard, Parnovski, Polterovich and Sher discussed in Remark~\ref{rem:gpps2014}.
  
\begin{prop}\cite[Theorem 1.2]{ADGHRS2019}\label{adghrs}  Let $\orb$ be a compact Riemannian orbisurface.  Suppose that $\partial\orb$ consists of $r$ type I components of lengths $\ell_1,\dots,\ell_r$ and $s$ type II components of lengths $\ell_1',\dots,\ell_s'$.     Let $U$ be the disjoint union of $r$ disks of circumference $\ell_1,\dots, \ell_r$ and $s$ half-disk orbifolds of boundary length $\ell_1',\dots,\ell_s'$.  Then
\[|\sigma_k(\orb)-\sigma_k(U)| =O(k^{-\infty}).\]
\end{prop}. 

\begin{cor}\label{cor:stek orbisurface} The Steklov spectrum of a compact Riemannian orbisurface determines:
\begin{enumerate} 
\item  the number of boundary components of each type and their lengths up to an equivalence.  
\item the number of singular points on the boundary.
\end{enumerate}
\end{cor}

The equivalence relation in the statement of the corollary is generated by the following:    a single circular boundary component of length $\ell_1$ together with a pair of type II boundary components each of length $\ell_2$ is equivalent to a single type I boundary component of length $2\ell_2$ together with two type II components each of length $\frac{1}{2}\ell_1$.   The first statement cannot be improved since the disjoint union of a disk of boundary length $\ell_1$ and two`half-disk orbifolds each of boundary length $\ell_2$ is Steklov isospectral to the disjoint union of a disk and two half-disks with the equivalent boundary lengths.

Next, consider Question~\ref{ques: orb interior} for orbisurfaces.  By Example~\ref{2dorb}, there are three types of singularities to consider.  The first of these, orbifold cone points, can be viewed as a special case of conical singularities.  As discussed in Remark~\ref{rem:con sing}, they can be removed by a $\sigma$-isometry and thus are not detected by the Steklov spectrum.    Dihedral points can similarly be removed; the two reflectors that meet at the dihedral point then become a single reflector.   Thus every orbisurface is $\sigma$-isometric to an orbisurface whose only interior singularities are reflectors. It is not known whether the Steklov spectrum detects reflectors that do not intersect the orbifold boundary.    In view of Remark~\ref{paorb}, this question is a special case of the following much more general question that we pose in arbitrary dimension:

\begin{ques}\label{ques:mixed vs pure} Can a mixed Steklov-Neumann problem on a compact Riemannian manifold $\om_1$  and a pure Steklov problem on a compact Riemannian manifold $\om_2$ have the same spectrum?
\end{ques}

In dimension greater than two, Questions~\ref{ques: orb boundary} and \ref{ques: orb interior} are completely open.

One can also ask:

\begin{ques}\label{ques.sing type} Among Riemannian orbifolds with singularities, to what extent does the Steklov spectrum recognise the types of singularities?
\end{ques}

As the heat asymptotics yielded some positive inverse spectral results for the Laplacian on orbifolds, one might start with the following:

\begin{prob}\label{prob: orb heat}  Develop Steklov heat asymptotics for compact Riemannian orbifolds with boundary.
\end{prob}

\section{Inverse problems: negative results}\label{sec: isospectrality}

 In this section, we address constructions of Steklov isospectral Riemannian manifolds that are not isometric.   Such constructions enable us to identify geometric or topological invariants that are not spectrally determined.  
 
 Recall by Corollary~\ref{cor:sigmaisom} that $\sigma$-isometric Riemannian surfaces (as defined in Definition~\ref{sigmaisom}) are necessarily Steklov isospectral.  Thus when working in dimension two, our interest will be in Steklov isospectral surfaces that are not $\sigma$-isometric.   
 
The current state of the art for constructing Steklov isospectral manifolds is almost identical to that for Laplace isospectrality.   There are two general techniques: the Sunada technique (and various generalisations) and the torus action technique.  These techniques were initially introduced to construct Laplace isospectral manifolds, with or without boundary.  In the case of non-trivial boundary, the methods generally produce manifolds that are both Neumann and Dirichlet isospectral; moreover, the boundaries of the manifolds are Laplace isospectral as well.  Both methods are quite robust in that they yield pairs or families of Riemannian manifolds that are simultaneously isospectral for a wide range of spectral problems.  In particular, in the case of manifolds with boundary, the manifolds constructed by these methods are now known to be Steklov isospectral as well.  

Beyond these techniques, most constructions of either Laplace isospectral or Steklov isospectral   manifolds in the literature are via ad hoc methods, e.g., direct computations of the spectra.

There is a very large literature on Riemannian manifolds constructed via the Sunada and torus action techniques.  While most of this literature pre-dated the application to the Steklov problem, it immediately yields a vast number of examples of Steklov isospectral manifolds.  In particular, almost all known Laplace isospectral Riemannian manifolds with boundary are also Steklov isospectral.    (Plane domains are a notable exception, as we will discuss later in this section.)      

However, it is disappointing that we do not currently have methods for isospectral constructions specific to the Steklov problem.  
 In particular, the following questions remain open:

 \begin{ques}\label{ques.NotLapIsosp}
Do there exist Steklov isospectral Riemannian manifolds whose boundaries are not Laplace isospectral?
\end{ques} 
This question can be rephrased as
\emph{Does the Steklov spectrum of a Riemannian manifold determine the Laplace spectrum of its boundary?}

\begin{ques}\label{ques.NotNeumIsosp}
Do there exist Steklov isospectral Riemannian manifolds that are not Neumann isospectral (other than $\sigma$-isometric surfaces)?  
\end{ques}

The analogous question with ``Neumann'' replaced by ``Dirichlet'' is also open.

Question~\ref{ques.NotLapIsosp} was first posed in the earlier survey \cite{GiPo2017}.  We are hesitant to conjecture an answer to this question.  Throughout this survey, we have seen many relationships between the Steklov spectrum of a Riemannian manifold $(\om,g)$ and the Laplace spectrum of $(\partial\om,g)$, so this question is natural and interesting.   On the other hand, there is no obvious connection between the  Steklov spectrum and the Neumann or Dirichlet spectrum.  Thus one would expect Question~\ref{ques.NotNeumIsosp} to have a positive answer. The fact that the question is open seems primarily indicative of the challenge of coming up with new methods for constructing Steklov isospectral manifolds.

All known examples of Steklov isospectral manifolds satisfy the following notion of ``strong Steklov isospectrality'':  

 \begin{nota}\label{nota.strongStek} Given a compact Riemannian manifold $\om$ with boundary and $\alpha\in \R$, consider the \emph{Steklov eigenvalue problem with potential} $\alpha$: 
 \begin{equation}\label{eq.alpha}
\begin{cases}\Delta u=\alpha\,u \mbox{\,\,\,\,on\,\,}\om\\
\partial_\nu u=\sigma u \mbox{\,\,\,\,on\,\,}\partial\om
\end{cases}
\end{equation}
This is a well-defined problem with discrete spectrum, provided that $\alpha$ is not a Dirichlet eigenvalue for the Laplacian on $\om$.    We will denote its spectrum by 
$\stek_\alpha(\om)$. One has an orthonormal basis of $L^2(\partial\om)$ consisting of $\alpha$-Steklov eigenfunctions (more precisely, of the restrictions to $\partial\om$ of the $\alpha$-Steklov eigenfunctions).  

Observe that $\stek_0(\om)=\stek(\om)$. 
 

We will say that two compact Riemannian manifolds $\om$ and $\om'$ are \emph{strongly Steklov isospectral} if they have the same Dirichlet spectrum and 
$\stek_\alpha(\om)=\stek_{\alpha}(\om')$ for all $\alpha$ not in their common Dirichlet spectrum.  
\end{nota}
  
\begin{remark} If one fixes $\sigma$ and treats $\alpha$ as the unknown, then Equation~\eqref{eq.alpha} becomes the Robin eigenvalue problem with parameter $\sigma$.    Thus $\om$ and $\om'$ are strongly Steklov isospectral if and only if they are isospectral for the Robin problem for every parameter $\sigma$.   

\end{remark}  

We note that except when $\alpha=0$, the spectrum $\stek_\alpha(\om,g)$ of a Riemannian surface $(\om,g)$ is \emph{not} invariant under $\sigma$-isometries.   (See Definition~\ref{sigmaisom}.)

\begin{ques}\label{stek not strong}
Do there exist Steklov isospectral Riemannian manifolds that are not strongly Steklov isospectral (other than $\sigma$-isometric surfaces)?
\end{ques}

There is no reason to expect the Steklov spectrum to determine all the $\alpha$-Steklov spectra.  As in the case of Question~\ref{ques.NotNeumIsosp}, new methods are needed to produce counterexamples.

The outline of this section is as follows: 

\begin{itemize}
\item Subsection~\ref{subsec.prod}: Product manifolds.   (We will see that strong Steklov isospectrality behaves better than simple -- i.e., not strong -- Steklov isospectrality under Riemannian direct products.)  
\item Subsection~\ref{subsec.sun}: Sunada's technique. 
\item Subsection~\ref{subsec.torus}: Torus action technique.  
\item Subsection~\ref{ballquots}: Quotients of Euclidean balls.  
 \end{itemize}
 
  \subsection{Product manifolds.}\label{subsec.prod}
  
As noted in the discussion following Example~\ref{example: cylinder}, the earliest known non-trivial construction of Steklov isospectral manifolds consisted of pairs of cylinders $I\times N$ and $I\times N'$, where $N$ and $N'$ are any pair of non-isometric Laplace isospectral closed manifolds and $I$ is an interval.   One can replace $I$ in this construction by any \emph{fixed} compact Riemannian manifold with boundary to obtain Steklov isospectral manifolds  $\om\times N$ and $\om\times N'$.   (This elementary fact follows from Proposition~\ref{isosp_prod} below.)    

Strong Steklov isospectrality is preserved by much more general products:

\begin{prop}\label{isosp_prod}
Let $N$ and $N'$ be Laplace isospectral closed Riemannian manifolds and let $\om$ and $\om'$ be strongly Steklov isospectral compact Riemannian manifolds with boundary.   Then $\om\times N$ is strongly Steklov isospectral to $\om'\times N'$.
\end{prop}

\begin{proof} Denote by $\spec(N)$ the Laplace spectrum of $N$ (and of $N'$).  The  definition of strong Steklov isospectrality in Notation~\ref{nota.strongStek} implies that $\om$ and $\om'$ have the same Dirichlet spectrum, which we denote by $\spec_D(\om)$.   We then have 
\[\spec_D(\om\times N)=\{\beta+\lambda: \beta\in \spec_D(\om),\,\,\lambda\in \spec(N)\}=\spec_D(\om'\times N') .\]
 We fix $\alpha\not\in\spec_D(\om\times N)$ and solve the eigenvalue problem~\eqref{eq.alpha} by separation of variables.  For $u\in C^\infty(\om)$ and $v\in C^\infty(N)$, we find that $uv\in C^\infty(\om\times N)$ is an eigenfunction for~\eqref{eq.alpha} with eigenvalue $\sigma$ if there exists $\lambda\in \spec(N)$ such that $\Delta_N v=\lambda v$ and such that $u$ is a $\sigma$-eigenfunction of the Steklov eigenvalue problem on $\om$ with potential $\alpha-\lambda$.  The $\alpha-\lambda$ Steklov eigenvalue problem is well-defined since $\alpha-\lambda$ is not in the Dirichlet spectrum of $\om$.  Since the restrictions to $\partial(\om\times N)$ of the resulting eigenfunctions span $L^2(\partial(\om\times N))$, we conclude that
\begin{equation}\label{sepvar}\stek_\alpha(\om\times N) = \bigsqcup_{\lambda\in \spec(N)}\,\stek_{\alpha-\lambda} (\om)\,=\, \stek_\alpha( \om'\times N')\end{equation}
(By disjoint union in the middle term, we mean that the eigenvalues are repeated by multiplicity.)
\end{proof}

Equation~(\ref{sepvar}) shows that the Steklov spectrum of $\om\times N$ depends on the strong Steklov spectrum of $\om$ as well as the Laplace spectrum of $N$, so there is no analogue of Proposition~\ref{isosp_prod} for simple (i.e., not strong) Steklov isospectrality even if one assumes that $N=N'$.  

Proposition~\ref{isosp_prod} partially generalizes to warped products.    Recall that if $(M,g_M)$ and $(N,g_N)$ are Riemannian manifolds and $f\in C^\infty(M)$  is a strictly positive function, the associated \emph{warped product}  $(M,g)\times_f (N,g_N)$ (which we will denote as $M\times_f N$ if $g_M$ and $g_N$ are understood) is the Riemannian manifold 
\[M\times_f N=(M\times N,  g_M + f^2 g_N).\]
(More, precisely, the metric is given by $\pi_M^*g_M\,+\, (f\circ\pi_M)^2\,\pi_N^*g_N$ where $\pi_M:M\times N\to M$ and $\pi_N:M\times N\to N$ are the projections.)    

For any pair of closed Laplace isospectral Riemannian manifolds $(N,g)$ and $(N',g')$ and any fixed closed Riemannian manifold $(M,g)$ and function $0<f\in C^\infty(M)$, Ejiri \cite{Ej1979} showed that $M\times_f N$ and $M\times_f N'$ are Laplace isospectral.  
Motivated by Ejiri's result along with work of Xiong \cite{Xi2021} addressing the Steklov spectra of warped products, we show the following:
 
 \begin{prop}\label{prop:warp} Let $(\om,g_\om)$ be a compact Riemannian manifold with boundary, let $0<f\in C^\infty(\om)$, and let $(N,g_N)$ and $(N',g'_N)$ be Laplace isospectral closed Riemannian manifolds.  Then $\om\times_f N$ and $\om\times_f N'$ are strongly Steklov isospectral.    
 \end{prop}  
 
Under the hypotheses of Proposition~\ref{prop:warp}, it is straightforward to modify the proof of Ejiri's result to show that $\om\times_f N$ and $\om\times_f N'$ are Dirichlet isospectral and also Neumann isospectral.    The proof of Proposition~\ref{prop:warp} is similar to that of Proposition~\ref{isosp_prod}, using separation of variables.   We give a brief overview.

\begin{proof}[Sketch of proof of Proposition~\ref{prop:warp}]
We denote by $\Delta_\om$, $\Delta_N$, and $\Delta_f$ the Laplacians of $\om$, $N$, and $\om\times_f N$, respectively.  For $u\in C^\infty(\om)$ and $v\in C^\infty(N)$, a computation yields 
\[\Delta_f(uv)=\left(\Delta_\om u -\frac{n}{f}\langle\nabla f,\nabla u\rangle\right) v +\frac{u}{f^2} \Delta_N v\]
where $\langle\,,\,\rangle$ is the pointwise inner product given by $g_\om$.   
In particular, if $v$ is an eigenfunction of $\Delta_N$ with eigenvalue $\lambda$, then
\begin{equation}\label{eq:warp sep1}\Delta_f(uv)=\left(\Delta_\om u -\frac{n}{f}\langle\nabla f,\nabla u\rangle+\frac{\lambda}{f^2} u\right) v.\end{equation}

Define operators $\mathcal{A}, \,\mathcal{A}_\beta:C^\infty(\om)\to C^\infty(\om)$ for $\beta\in\R$, by 
\begin{equation}\label{eq:warp sep2}\mathcal{A}(u)=\Delta_\om u -\frac{n}{f}\langle\nabla f,\nabla u\rangle \mbox{\,\,and\,\,}\mathcal{A}_\beta (u)=\left(\mathcal{A}+\frac{\beta}{f^2}\right)(u).\end{equation}
The operator $\mathcal{A}$ is well-studied as it coincides with the so-called ``weighted'' or ``drifting'' Laplacian on the metric measure space $(\om,g_\om, \mu)$ where $\mu= f^n dV$.   (Here $dV$ denotes the Riemannian volume form associated with $g_\om$.)   The operator is known to have discrete Steklov spectrum.  See, for example, \cite{DuMaWaXi2021} and references therein.

If $v$ is an eigenfunction of $\Delta_N$ with eigenvalue $\lambda$, then  Equation~\eqref{eq:warp sep1} implies
\begin{equation}\label{eq:iff}\Delta_f(uv) = \alpha (uv) \iff \mathcal{A}_\lambda u =\alpha u.\end{equation}    Denote the Laplace eigenvalues of $N$ by $0=\lambda_0\leq \lambda_1 \leq$.  

Using Equation~\eqref{eq:iff}, one shows that
\[\spec_D(\om\times_f N)=\bigsqcup_{j=0}^\infty\,\spec_D(\mathcal{A}_{\lambda_j})=\spec_D(\om\times_f N').\]
Next, one again uses Equation~\eqref{eq:iff} to show for $\alpha\not\in\spec_D(\om\times_f N)$ that
\[\stek_\alpha(\om\times_f N)=\bigsqcup_{j=0}^\infty\,\stek_{\alpha}\,(\mathcal{A}_{\lambda_j})=\stek_\alpha(\om\times_f N').\]
Here $\stek_{\alpha}(\mathcal{A}_{\beta})$ is defined as in Notation~\ref{nota.strongStek} with $\Delta$ replaced by $\mathcal{A}_\beta$.  The proposition then follows.
\end{proof}

 \subsection{Sunada's technique}\label{subsec.sun}
 
 Toshikazu Sunada \cite{Su1985} introduced an elegant and simple technique to construct pairs of Laplace isospectral compact manifolds with a common Riemannian covering.   If the manifolds have boundary, the technique works with both Dirichlet and Neumann boundary conditions.   Moreover, the manifolds constructed by this technique are strongly isospectral:  they are isospectral with respect to all natural self-adjoint elliptic differential operators; e.g., they have the same Hodge spectra on $p$-forms for all $p$.   The technique remains the most widely used method for isospectral constructions.   
 
 \begin{defn} Subgroups $H_1$ and $H_2$ of a finite group $G$ are said to be \emph{almost conjugate} in $G$ if each $G$-conjugacy class intersects $H_1$ and $H_2$ in the same number of elements.
 
 \end{defn} 
 
 The following version of the Sunada Theorem has the same hypotheses as Sunada's original theorem (except of course that we require the boundary to be non-trivial).
  
  \begin{thm}\label{thm.sun}\cite[Theorem 2.3]{GoHeWe2021}
Assume that $H_1$ and $H_2$ are almost conjugate subgroups of a finite group $G$.   Suppose that $G$ acts by isometries on a compact $(d+1)$-dimensional Riemannian manifold $\om$ with boundary.  Then
\begin{enumerate}
\item $H_1\bs \om$ and $H_2\bs\om$ are strongly Steklov isospectral;
\item $\stek_p^{\rs}(H_1\bs \om)= \stek_p^{\rs}(H_2\bs \om)$ and $\stek_p^{K}(H_1\bs \om)= \stek_p^{K}(H_2\bs \om)$ for every $p\in \{1,\dots, d\}$, where $\stek_p^{\rs}$ and $\stek_p^K$ denote the Steklov spectra on $p$ forms defined by Raulot-Savo and by Karpukhin, respectively.  (See Section~\ref{stek.forms}.)

\end{enumerate}
 \end{thm}

  The theorem becomes trivial if $H_1$ and $H_2$ are conjugate subgroups of $G$; i.e., the isospectral quotients are isometric in that case.  When the subgroups are only almost conjugate, the quotient manifolds will usually be non-isometric, but one must always check.  
  
  If $H_1$ and $H_2$ do not act freely on $\om$, then $H_1\bs \om$ and $H_2\bs\om$ will be orbifolds with singularities; otherwise they are smooth manifolds.

  There is a huge literature on isospectral manifolds based on Sunada's theorem.  Many use a delightful method introduced by Buser \cite{Bu1986} for constructing examples using the patterns of the Schreier graphs of the coset spaces $H_1\bs G$ and $H_2\bs G$, where $(G,H_1,H_2)$ satisfies the hypothesis of Sunada's theorem.
  Figure~\ref{fig:Buser_surfaces} (reprinted from \cite{GoWeWo1992}) illustrates Buser's construction in this manner of a pair of Neumann and Dirichlet isospectral flat surfaces embedded in $\R^3$. Theorem~\ref{thm.sun} implies that they are also Steklov isospectral.   One can easily adjust the shape of the building block (a cross in this example) so that the resulting surfaces have smooth boundary. 
  
 Sunada's technique enables one to identify global invariants that are not determined by the Steklov spectrum (as well as by the many other spectra to which the technique applies.)    For example, the following invariants are not spectrally determined:
 
 \begin{itemize}
 \item the diameter of the manifold and the intrinsic diameter of its boundary (as evident in Buser's example in Figure~\ref{fig:Buser_surfaces});
 \item the fundamental group of the manifold.   (Such examples can arise when the almost conjugate subgroups $H_1$ and $H_2$ used in the construction are non-isomorphic.) 
  \end{itemize}
  \begin{figure}
    \centering
  \includegraphics[width=7cm]{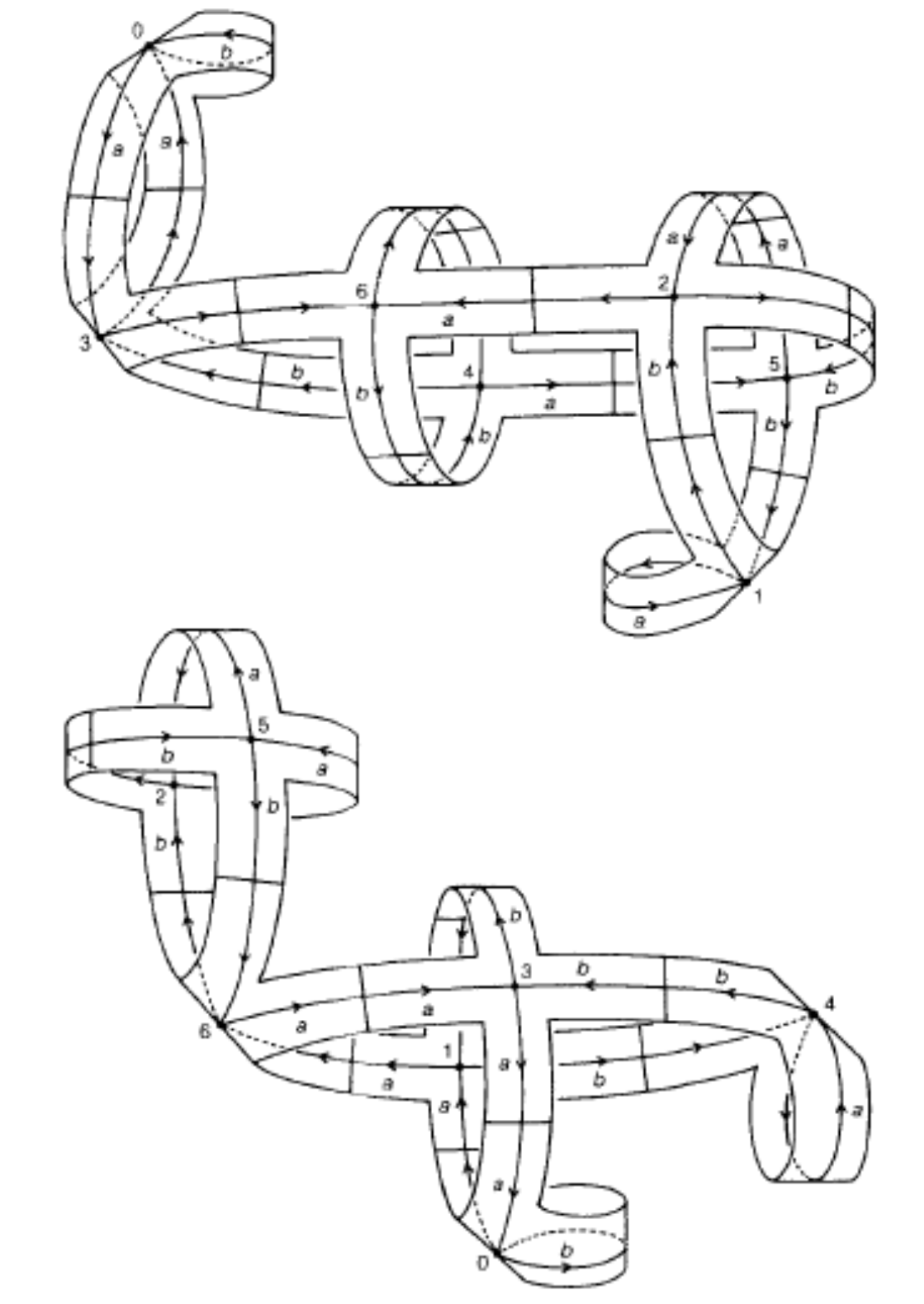}
  \caption{Isospectral surfaces (reprinted from \cite{GoWeWo1992}).}
  \label{fig:Buser_surfaces}
  \end{figure}
 
 \begin{remark}\label{sun.orb}~ 
In \cite{GoHeWe2021}, the first statement of Theorem~\ref{thm.sun} is stated in the more general setting of weighted Steklov problems.

One can further generalise Sunada's Theorem to the setting of variational eigenvalue problems associated with admissible Radon measures as follows:  In addition to the hypotheses of Theorem~\ref{thm.sun}, assume that $\mu$ is a $G$-invariant admissible Radon measure on $\om$.  The induced Radon measures on $H_1\bs \om$ and $H_2\bs\om$ will then have the same variational spectrum.  (See Subsection~\ref{subsection:introvareigenradon} and Appendix~\ref{Section:Radon}, in particular Definition~\ref{def:admissible} for the concept of variational eigenvalue problems associated with such measures.)  
 
  \end{remark}
 
 We end our discussion of Sunada's Theorem by remarking on the following open question: 
  \begin{ques}\label{ques:isospplanedomains}
Do there exist non-isometric Steklov isospectral plane domains?
\end{ques}

Sunada's technique has been used to construct pairs of Neumann and Dirichlet isospectral plane domains. We've seen that Neumann and Dirichlet isospectral manifolds constructed via Sunada's technique are generally Steklov isospectral as well.  However, there is a subtlety in the application of Sunada's technique to plane domains that does not allow us to conclude Steklov isospectrality. Consider the pair of orbifolds $\orb_1$ and $\orb_2$ in Figure~\ref{fig:planar} (reprinted from \cite{GoWeWo1992}), obtained as quotients by reflection of Buser's surfaces $\om_1$ and $\om_2$ shown in Figure~\ref{fig:Buser_surfaces}. 
\begin{figure}
  \centering
  \includegraphics[width=9cm]{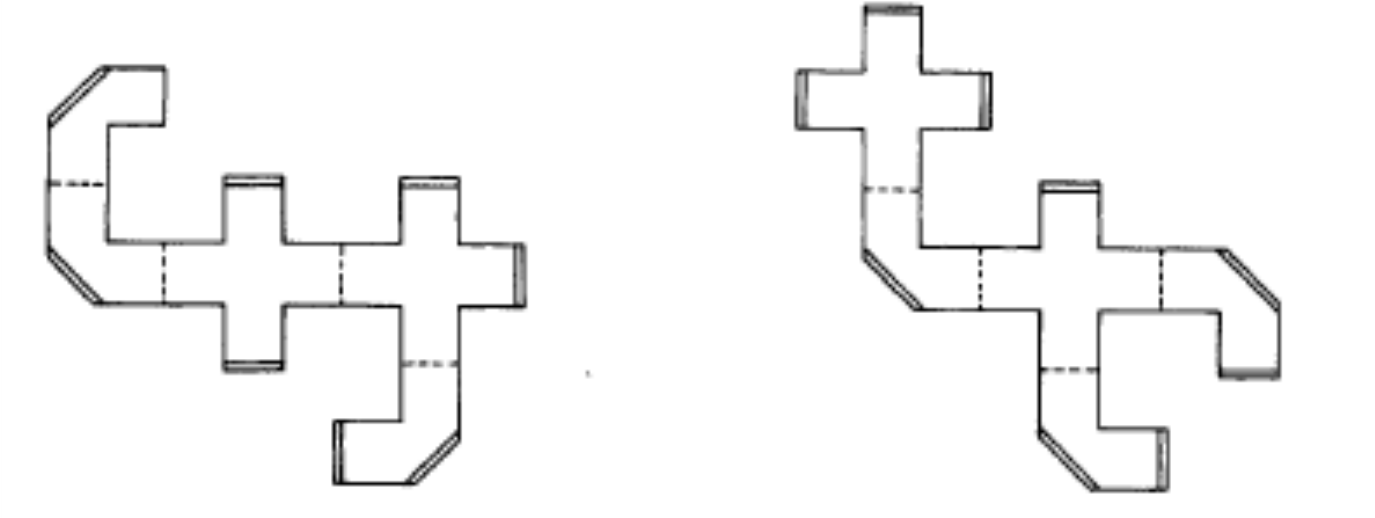}
  \caption{Isospectral orbifolds (reprinted from \cite{GoWeWo1992}).}
  \label{fig:planar}
  \end{figure}
The edges marked by double lines are reflectors. (See Subsection~\ref{subsec:orb pos} for a review of orbifolds and, in particular, Example~\ref{2dorb} for the notion of reflectors.) Let $D_1$ and $D_2$  be the underlying plane domains.   In \cite{GoWeWo1992}, Gordon, Webb, and Wolpert used Sunada's technique, along with Remark~\ref{paorb} to show that $D_1$ and $D_2$ are Neumann isospectral and also isospectral for the mixed Dirichlet-Neumann problem, with Neumann conditions on the edges that correspond to reflectors.   Dirichlet isospectrality of $D_1$ and $D_2$ was then proved by using the facts that (i) Buser's surfaces are Dirichlet isospectral, (ii) the domains are Dirichlet-Neumann isospectral, and (iii) the mixed Dirichlet-Neumann, respectively, Dirichlet, spectra of the domains correspond to the part of the Dirichlet spectra of Buser's surfaces for which the eigenfunctions lie in the subspace of reflection-invariant, respectively anti-invariant, functions.  

As noted in \cite{GoHeWe2021}, Sunada's Theorem along with Remark~\ref{paorb} does show that $D_1$ and $D_2$ are isospectral for the mixed Steklov-Neumann problem, with Neumann conditions on the edges corresponding to reflectors.  Although not specifically mentioned in \cite{GoHeWe2021}, they are also isospectral for the analogous mixed Steklov-Dirichlet problem; this follows by the same argument used to prove Dirichlet isospectrality of $D_1$ and $D_2$ described above.

 \subsection{Torus action technique}\label{subsec.torus}
 
The torus action technique yields isospectral manifolds that differ in their local as well as global geometry, thus allowing one to identify curvature properties that are not spectral invariants.  The method was first introduced for the Laplacian in a very special setting in \cite{Go1994} and then generalised in a series of papers by numerous authors.   The most general version is due to Schueth \cite{Sc2001}.   (See \cite{Sc2001} for further history and references.)  The Steklov version below is  the analog in the Steklov setting of Schueth's method.   

View tori as compact, connected, abelian Lie groups.   Let~$T$ be a torus acting effectively by isometries on a compact, connected Riemannian manifold~$\om$.  The union of those orbits on which $T$ acts freely is an open, dense submanifold of $\om$ that we will denote by $\widehat{\om}$; it carries the structure of a principal $T$-bundle.

\begin{thm}\label{thm:Torus_method}\cite[Theorem 3.2]{GoHeWe2021}
Let $T$ be a torus which acts isometrically and effectively on two compact,
connected Riemannian manifolds $(\om,g)$ and $(\om',g')$ with smooth boundary.   For each subtorus $W\subset T$ of codimension one, suppose that there exists
a $T$-equivariant diffeomorphism $F_W: \om\to \om'$ such that
\begin{enumerate}
\item\label{F_WVolPres} 
$F_W:\om\to \om'$ is volume-preserving with respect to the Riemannian volume densities;
\item $F_W|_{\partial \om}: \partial\om\to \partial \om'$ is volume preserving.  (Here the volumes are measured with respect to the Riemannian metrics induced on the boundaries by $g$ and $g'$.)

\item $F_W$ induces an isometry $\overline{F}_W: (W\bs\widehat{\om}, g_W)\to(W\bs\widehat{\om'},g'_W)$, where $g_W$ and $g'_W$ are the metrics induced by $g$ and $g'$ on the quotients.
\end{enumerate}

Then $(\om,g)$ and $(\om',g')$ are strongly Steklov isospectral.
\end{thm}

The proof uses Fourier decomposition with respect to the $T$-action to decompose $C^\infty(\om)$ and $C^\infty(\om')$ into a sum of subspaces $C^\infty(\om)^W$, respectively $C^\infty(\om')^W$, where $C^\infty(\om)^W$ and $C^\infty(\om')^W$ consist of all $W$-invariant elements of $C^\infty(\om)$, respectively $C^\infty(\om')$.   We have $F_W^* : C^\infty(\om')^W\simeq C^\infty(\om)^W$ and one shows that the map $F_W^*$ preserves the Steklov-Rayleigh quotients, i.e., 
\[\frac{\|\nabla(F_W^*u)\|^2_{L^2(\om)}}{\|F_W^*u\|^2_{L^2(\partial \om)}}=\frac{\|\nabla u\|^2_{L^2(\om')}}{\|u\|^2_{L^2(\partial \om')}}\]

Like Sunada's technique, the torus action technique is quite robust in the sense that the same method works for a wide range of operators.  For example, if $\mu$ and $\mu'$ are $T$-invariant admissible Radon measures on $\om$ and $\om'$, respectively, and if one replaces condition (2) in the theorem by the condition $F_W^*\mu'=\mu$, then one can conclude that $\mu$ and $\mu'$ have the same variational spectrum.
On the other hand, because the theorem depends on the use of Riemannian submersions, the technique is not applicable to the Hodge Laplacian on $p$-forms or to the Dirichlet-to-Neumann operators on $p$-forms defined by Raulot-Savo and Karpukhin.
  
\begin{ex}\cite{Go2001}, \cite{Sc2001}, \cite{GoHeWe2021}
Let $B^{d+1}$ be a $(d+1)$-dimensional ball.    For $d\geq 7$, there exist continuous families of non-isometric Riemannian metrics on $B^{d+1}$ that are mutually strongly Steklov isospectral.   The metrics are also isospectral for both the Neumann and Dirichlet problems and the boundary spheres have the same Laplace spectra.  For $d=5,6$, there exist pairs of such metrics satsifying the same properties.

\end{ex}

\begin{ex}\label{canthear}  For each of the following geometric properties, the torus action method yields pairs of manifolds $\om$ and $\om'$ such that $\om$ and $\om'$ are simultaneously Steklov, Neumann, and Dirichlet isospectral and their boundaries $\sig$ and $\sig'$ are Laplace isospectral, but
\begin{itemize}
\item $\om$ has constant Ricci curvature while $\om'$ has variable Ricci curvature (examples in dimension 10) ;
\item  $\sig$ has constant scalar curvature while $\sig'$ has variable scalar curvature (examples for $\om$ and $\om'$ of dimension 9) ;
\item the curvature tensor of $\om$ is parallel, the curvature tensor of $\om'$ is not (examples in dimension $4(\ell +1)$, $\ell\geq 2$).
\end{itemize}

The counterexamples appear in \cite{GoSz2002} although the Steklov isospectrality was not mentioned. there.    (The second item is stated in a different form in \cite{GoSz2002}:  a pair of (non-product) isospectral metrics $g_1$ and $g_2$ is defined on the manifold $\om:=\R^6\times T^3$, where $T^3$ is a 3-torus.   One of these metrics induces a metric of constant scalar curvature on $\sig=\partial \om$ and the other a metric of variable scalar curvature.   The torus action method implies that the two metrics on $\om$ are Neumann, Dirichlet, and Steklov isospectral and that the metrics on $\sig$ are Laplace isospectral, although only the latter fact is actually stated in \cite{GoSz2002}.)
\end{ex}

In spite of the many examples of isospectral manifolds, the gap in our knowledge between what one can glean from, say, spectral asymptotics of the Steklov spectrum as in Section~\ref{sec:inv probs pos} and what the counterexamples tell us remains huge.  A sampling of the myriad open questions:

\begin{ques}\label{quesgeq3} For manifolds of dimension $\geq 3$, can one tell from the Steklov spectrum:
\begin{itemize}
\item Whether the manifold has constant sectional curvature? 
\item Whether the induced metric on the boundary is Einstein?
\item Whether the induced metric on the boundary has constant sectional curvature?
\item Whether the mean curvature of the boundary is constant?
\item Whether the principal curvatures of the boundary are constant?  

 \end{itemize}

\end{ques}

While more can be gleaned currently from spectral asymptotics of the Laplacian than from the Steklov spectrum, the questions above are not answered for the Laplacian either, although it is known for the Laplacian that any Riemannian metric $g_0$ of constant curvature $\kappa$ on a closed Riemannian manifold $M$  is at least spectrally isolated in the sense that no sufficiently nearby Riemannian metric has the same Laplace spectrum.     This was proven by S. Tanno for $\kappa>0$, R. Kuwabara for $\kappa=0$, and V. Sharafutdinov for $\kappa<0$.

\subsection{Quotients of Euclidean balls}\label{ballquots}

Recall that a spherical space form is a quotient $\Gamma\bs\sph^d$ where $\Gamma$ is a finite subgroup of the orthogonal group of $\operatorname{O}(d+1,\R)$ acting freely on $\sph^d$.    One can also consider finite subgroups $\Gamma<\operatorname{O}(d+1,\R)$ whose action on $\sph^d$ is not free; we will refer to the quotient $\Gamma\bs\sph^d$ in this case as an ``orbifold spherical space form''.  The literature on the Laplace spectrum and also on the Hodge spectra on $p$-forms contains rich collections of examples of isospectral spherical space forms and also of isospectral orbifold spherical space forms that exhibit interesting properties.  The following elementary proposition shows that each such example yields a corresponding example of Steklov isospectral quotients of Euclidean balls.  

We remind the reader of the notation $\stek_p^{\rs}(\om)$, respectively, $\stek_p^{K}(\om)$ for the Steklov spectra on $p$-forms introduced by Raulot and Savo (see Definition~\ref{dtnp}), respectively by Karpukhin (see Definition~\ref{dtnp}). Recall that when $p=0$, $\stek_p^{\rs}(\om)$ coincides with the usual notion of the Steklov spectrum $\stek(\om)$ on functions.   

\begin{prop}\label{prop.ballquot}  Let $\Gamma_1$ and $\Gamma_2$ be finite subgroups of the orthogonal group $\operatorname{O}(d+1,\R)$.  Let $\B^{d+1}$ be the unit ball in $\R^{d+1}$ and let $\om_i=\Gamma_i\bs \B^{d+1}$, $i=1,2$.   Let $p\in \{0,\dots, d\}$.  Then the following are equivalent:
\begin{itemize}
\item[(i)]  $\stek_p^{\rs}(\om_1)=\stek_p^{\rs}(\om_2)$;
\item[(ii)]  The Hodge Laplacians of $\Gamma_1\bs \sph^{d}$ and $\Gamma_2\bs \sph^{d}$ on $p$-forms are isospectral.    
\end{itemize}
Moreover, when these equivalent conditions hold, we also have that \[\stek_p^{K}(\om_1)=\stek_p^{K}(\om_2).\]

\end{prop}

\begin{proof}
Fix $p$.
As discussed in Example~\ref{pball}, the Hodge-$p$ Laplacian on $\sph^d$ and the Dirichlet-to-Neumann operator $\DtN_p^{\rs}$ on $\B^{d+1}$ have the same eigenspaces in $\Acal^p(\sph^d)$.   We denote these eigenspaces by $E_k$, $k=1,2,\dots$.     The Hodge-$p$ eigenspace of $\Gamma_i\bs \  \sph^d$ corresponds by pullback to $E_k^{\Gamma_i}$, $k=1,2,\dots$, where $E_k^{\Gamma_i}$ denotes the $\Gamma_i$-invariant $p$-forms in $E_k$.   Let 
\[\widehat{E}_k =\{\widehat{\eta}\in \Acal^p(\B^{d+1}): \eta\in E_k\}\]
where $\widehat{\eta}$ is the tangential harmonic extension of $\eta$ as in Definition~\ref{dtnp}.  

Observe that 
\begin{equation}\label{hetaeta} \eta\in E_k^{\Gamma_i}\iff \widehat{\eta}\in \widehat{E}_k^{\Gamma_i}.\end{equation}
 Indeed the ``if'' statement is trivial.  The ``only if'' statement follows from the observation that if $\eta$ is $\Gamma_i$-invariant, then averaging $\widehat{\eta}$ by the isometric action of $\Gamma_i$ gives another tangential harmonic tangential extension of $\eta$, which must equal $\widehat{\eta}$ by uniqueness.   We can now conclude that $\DtN_p^{\rs}(\om_i)$ and the Hodge-$p$ operator of $\Gamma_i\bs \sph^d$ have the same eigenspaces $E_k^{\Gamma_i}$.  The equivalence of (i) and (ii) follows.   

The final statement follows in a similar way. Although the co-closed harmonic extension $\teta$ of $\eta$ used in Karpukhin's definition~\ref{dtnk} of $\DtN_p^K$ is not unique, the fact that the operator is well-defined independent of the choice of such extension permits an easy modification of the observation~\eqref{hetaeta}. We can't conclude the converse of the final statement since the $0$-eigenspace of $\DtN_p^K(\B^{d+1})$ is the sum of many of the $E_k$.

\end{proof}

The proposition above in the case $p=0$ was observed in \cite[Example 6.1]{ADGHRS2023}.

\begin{remark}\label{rem.orb ball quot} In Proposition~\ref{prop.ballquot}, $\om_i$, $i=1,2$, is necessarily an orbifold with a singularity at the point corresponding to the origin in $\B^{d+1}$.  If $\Gamma_i$ doesn't act freely on $\sph^d$, then $\om_i$ will also contain higher-dimensional singular strata.

\end{remark}

We give two applications of Proposition~\ref{prop.ballquot}.
The first application is motivated by the following question:

\begin{ques}\label{ques: inverse prob forms} Do the Steklov spectra on $p$-forms for different choices of $p$ contain different geometric information?  Equivalently, for $p\neq q$, do there exist compact Riemannian manifolds $\om_1$ and $\om_2$ with boundary such that $\stek_p^{\rs}(\om_1)=\stek_p^{\rs}(\om_2)$ but $\stek_q^{\rs}(\om_1)\neq \stek_q^{\rs}(\om_2)$?  (One can ask the same question for $\stek_p^{K}(\om_i)$) 
\end{ques}

Concerning the analogous question for the Hodge Laplacian on closed Riemannian manifolds, A. Ikeda \cite{Ik1989} showed that for every $p_0$, there exist pairs of Lens spaces (i.e., spherical space forms with cyclic fundamental groups) whose Hodge Laplacians on $p$-forms are isospectral for every $p\in \{0,\dots, p_0\}$ but not for $p=p_0 +1$.  The articles \cite{GoMc2006} (corrected in \cite{GoMc2019}) and \cite{La2019} contain many other examples of Lens spaces that are $p$-isospectral for some (not necessarily consecutive) but not all values of $p$. By Proposition~\ref{prop.ballquot}, each of these examples yields examples of flat orbifolds of the form $\Gamma_i\bs \B^{d+1}$ whose Steklov spectra on $p$-forms coincide for some but not all values of $p$.   To our knowledge, there are no examples currently known of pairs of Riemannian manifolds, as opposed to orbifolds, whose Steklov spectra on forms exhibit such behaviour.

Proposition~\ref{prop.ballquot} also has applications to Question~\ref{ques.sing type}, which asks the extent to which the Steklov spectrum of an orbifold $\orb$ recognises the types of singularities in $\orb$.
Indeed, Shams, Stanhope and Webb \cite{ShStWe2006} and also Rosetti, Schueth and Weilandt \cite[Example 2.9]{RoScWe2008} used the Sunada technique to construct orbifold spherical space forms $\Gamma_1\bs \sph^d$ and $\Gamma_2\bs \sph^d$ that are isospectral both for the Laplacian and for the Hodge Laplacian on $p$-forms for all $p$ but that have
singularities of different isotropy types.   By Proposition~\ref{prop.ballquot}, the orbifolds $\Gamma_1\bs \B^{d+1}$ and $\Gamma_2\bs \B^{d+1}$ are Steklov isospectral and also have the same Steklov spectra on $p$-forms for all $p$ (with respect to both the Raulot-Savo and the Karpukhin notions).  These orbifolds have different types of singularities both on their boundaries $\Gamma_1\bs \sph^d$ and $\Gamma_2\bs \sph^d$ and in their interiors.

\section{Geometry of eigenfunctions}\label{sec:eigenfunctions}

In this section we review some of the recent developments concerning the qualitative behaviour of Steklov eigenfunctions.

As can be seen from Example \ref{example: cylinder} in particular, there is a general heuristic for the behaviour of Steklov eigenfunctions:
\begin{itemize}
    \item Oscillation near the boundary, of frequency $\sigma$.
    \item Rapid decay into the interior of $\Omega$ as $\sigma\to\infty$. Equivalently, concentration of mass near the boundary $\partial\Omega$.
\end{itemize}

\subsection{Interior decay}

In 2001, Hislop and Lutzer \cite{HiLu2001} showed that when the boundary $\partial\Omega$ is smooth, Steklov eigenfunctions decay super-polynomially into the interior.
\begin{thm}\cite[Theorem 1.1]{HiLu2001} Suppose that $\Omega\subseteq\mathbb R^{d+1}$ has smooth boundary $\Sigma$. Let $K$ be a compact subset of the interior of $\Omega$. Then for any $j\in\mathbb N$,
\begin{equation}\label{eq:hlbound}
\|u_k\|_{H^1(K)} = O(k^{-j}),
\end{equation}
with the implied constants depending only on $K$ and $j$.
\end{thm}
This decay estimate may be extended uniformly up to the boundary, as shown recently by Helffer and Kachmar \cite[Theorem 1.9]{HeKa2022}.

The proof of Hislop and Lutzer uses a Green's function (layer potential) representation for the harmonic extensions of the boundary eigenfunctions. It also uses the fact, which follows from the Calder\'on-Vaillancourt theorem for pseudodifferential operators, that $\mathcal D^j$ is bounded from $H^j(\Sigma)$ to $L^2(\Sigma)$ for any $j$.  
\begin{remark}
The same result holds, with a virtually identical proof, if $H^1(K)$ is replaced by $H^m(K)$ or by $C^m(K)$ for any $m$. It also holds if $\mathbb R^{d+1}$ is replaced by any smooth $(d+1)$-dimensional ambient Riemannian manifold $M$. The proof is again nearly identical, using the smoothness of the Green's function on $M$ away from the diagonal.
\end{remark}
Hislop and Lutzer also conjectured in this paper that ``the decay is actually of order $e^{-\textrm{dist}(K,\Sigma)|m|}$ in the case of an analytic boundary." This was shown in by Polterovich, Sher, and Toth\cite{PoShTo2019} for domains $\Omega\subseteq\mathbb R^2$.

Galkowski and Toth have now proven the Hislop-Lutzer conjecture in all dimensions \cite{GaTo2019}. Their result is stated semiclassically and gives an extremely sharp description of the decay of Steklov eigenfunctions into the interior. Non-semiclassically, it translates as follows:
\begin{thm}\cite[Theorem 1]{GaTo2019} Let $\Omega$ be a real-analytic Riemannian manifold of dimension $d+1$ with real-analytic boundary. There exists a small tubular neighborhood of $\Sigma\subseteq\Omega$ in which the Steklov eigenfunctions $u_k$ satisfy the estimate
\[|u_k(x)|\leq Ck^{\frac d2-\frac 14}e^{-k(d(x,\Sigma)-cd^2(x,\Sigma))}.\]
Here $C$ depends only on the size of the tubular neighborhood, $c$ is an explicit geometric constant, and $d(x,\Sigma)$ is the distance to the boundary.

Similar bounds hold for derivatives of $u_k(x)$, with an additional factor of $k$ for each derivative.
\end{thm}
As Galkowski and Toth observe, this bound immediately implies the analogue of Hislop and Lutzer's result, by using the maximum principle: for any compact set $K$ contained in the interior of $\Omega$,
\begin{equation}\label{eq:gtbound}
|u_k(x)|\leq Ce^{-kc}\end{equation}
for some positive constant $c$.

It is not yet clear exactly which assumptions are necessary for which kind of decay. This motivates the following open questions.

\begin{ques}\label{ques:interiordecaysmooth} Does \eqref{eq:gtbound} hold if $\Sigma$ and/or $\Omega$ are assumed smooth rather than real analytic?
\end{ques}

The proof techniques of \cite{GaTo2019} are based on analytic microlocal analysis and thus use the analyticity hypotheses in extremely strong ways. So one would perhaps guess that the answer to Open Question \ref{ques:interiordecaysmooth} is ``no". However, the example of a cylinder (Example \ref{example: cylinder}) shows that even if $\Sigma$ is not assumed real analytic, \eqref{eq:gtbound} holds for $\Omega=\Sigma\times[0,1]$. Similar bounds also seem to hold for certain warped products \cite{DaHeNi2021}. This suggests that some form of partial analyticity, perhaps in the direction normal to the boundary, might suffice.

\begin{ques}\label{ques:decaynonsmooth} What kind of decay properties do the Steklov eigenfunctions possess if $\Sigma$ is not assumed to be smooth?
\end{ques}

In particular, it is natural to ask if we still have a bound of the form \eqref{eq:hlbound}, despite the fact that the pseudodifferential techniques involved in its proof do not go through when $\Sigma$ is not smooth. If not, perhaps one may replace $O(k^{-\infty})$ with $O(k^{-\alpha})$ for some $\alpha>0$ depending on the regularity of $\Sigma$. There are hints in \cite{LPPS2019} that \eqref{eq:hlbound} may \emph{not} hold in the case where $\Omega$ is a planar polygon. Curiously, this may be very sensitive to the arithmetic properties of the angles of the polygon, with decay rates depending on the angles, and super-polynomial decay if and only if all angles are certain rational multiples of $\pi$.

We should also mention the question of concentration in the case of multiple boundary components. Specifically, do Steklov eigenfunctions generically concentrate on a single boundary component, or is their mass more often split evenly between multiple boundary components? This question has been recently studied by Daud\'e, Helffer, and Nicoleau in the setting of warped product manifolds \cite{DaHeNi2021}. They observe that if the warped product is asymmetric (a generic assumption in this setting), Steklov eigenfunctions localise near one component each. It is thus a natural question to ask:

\begin{ques}\label{ques:generic}Do Steklov eigenfunctions generically concentrate on a single boundary component?
\end{ques}

This open question has been answered in dimension two by Martineau, who showed that this concentration occurs under the (generic) assumption that the ratios of the lengths of each boundary component are neither rational numbers nor Liouville numbers \cite[Theorem 3]{Ma2018}. It remains open in higher dimensions.

\subsection{Nodal sets}

There has also been recent progress on understanding the nodal sets of the Steklov eigenfunctions $u_k$ and the Dirichlet-to-Neumann eigenfunctions $\varphi_k$. 

\subsubsection{Size of nodal sets} As with eigenfunctions of the Laplacian, there is a version of Yau's conjecture for the size of the nodal set, stated as Open Problem 11 in \cite{GiPo2017}:
\begin{ques}\label{ques:yau} Do there exist constants $c$ and $C$ depending only on $\Omega$ for which
\begin{equation}\label{eq:yauint}
    c\sigma_k\leq\mathcal H_{d}(u_k)\leq C\sigma_k;
\end{equation}
\begin{equation}\label{eq:yaubound}
    c\sigma_k\leq\mathcal H_{d-1}(\varphi_k)\leq C\sigma_k?
\end{equation}
\end{ques}
\begin{remark}
By considering the example of a cylinder, the statement \eqref{eq:yaubound} would imply Yau's conjecture for the Laplace nodal sets. It is thus highly nontrivial. 
\end{remark}
Much of the progress on Question \ref{ques:yau} has been concentrated in the case where $\Omega$ is real analytic. In that setting, the study of the boundary nodal sets was largely initiated by Bellova and Lin \cite{BeLi2015}, who proved an upper bound of the form
\[\mathcal H_{d-1}(\varphi_k)\leq C\sigma_k^6\]
for the size of the nodal sets of the boundary eigenfunctions $\varphi_k$. This was then improved by Zelditch \cite{Ze2015}, who proved the optimal upper bound of $C\sigma_k$, again under the assumption that $\Omega$ is analytic. Results for the lower bounds are somewhat weaker; the state of the art is a result by Wang and Zhu \cite{WaZh2015} that
\[\mathcal H_{d-1}(\varphi_k)\geq c\sigma_k^{\frac{3-d}{2}}.\]

Concerning the study of the interior nodal sets, in the real analytic setting, the case $d=1$ of Question \ref{ques:yau} was answered in the affirmative in \cite{PoShTo2019} by using analytic microlocal analysis. Zhu then recently proved the upper bound in \eqref{eq:yauint} when $\Omega$ is a real analytic manifold of \emph{any} dimension \cite{Zh2020}. Zhu's proof uses a series of classical ideas including doubling properties and Carleman estimates. Together with Lin, Zhu has also used these ideas to obtain upper bounds on the size of the nodal sets for related eigenvalue problems \cite{LiZh2020}.

If $\Omega$ is only assumed smooth rather than real analytic, the results are naturally weaker. There is a lower bound: it was proved by Sogge, Wang, and Zhu \cite{SoWaZh2016} that
\[\mathcal H_{d}(u_k)\geq c\sigma_k^{1-\frac{d+1}{2}}.\]
The sharpest known upper bound in this sort of generality in the $d=1$ case is due to Zhu \cite{Zh2016}, who showed that
\[\mathcal H_{1}(u_k)\leq C\sigma_k^{3/2}.\]
In the case of higher dimensions, the best available result is a polynomial upper bound proved by Georgiev and Roy-Fortin \cite{GeRoFo2019}, namely that
\[\mathcal H_{d}(u_k)\leq C\sigma_k^{\alpha(d)}\]
where $\alpha(d)$ is a positive number depending on the dimension.

However, in the case where $\Omega$ is a domain in Euclidean space, there are notable recent improvements by Decio \cite{De2021, De2021_2}. In particular Decio proved the following:
\begin{thm}\cite{De2021_2} Suppose that $\Omega$ is a domain in Euclidean space with $C^2$ boundary. Then there exist positive constants $c$, $C$, and $e$ depending only on the domain such that
\[c\leq \mathcal H_{d-1}(u_k)\leq C\sigma_k\log(\sigma_k + e).\]
\end{thm}
Observe that this is an improvement over the lower bound of \cite{SoWaZh2016} and the upper bound of \cite{Zh2016}, albeit in a slightly more restrictive setting. Decio's methods use machinery developed by Logunov to study Yau's conjecture for Laplace eigenfunctions.

\subsubsection{Density of nodal sets} It is well-known that nodal sets of Laplace eigenfunctions are dense on the scale $\lambda^{-1/2}$. In \cite{GiPo2017}[Open Problem 10], Girouard and Polterovich asked the analogous question for both the Steklov and Dirichlet-to-Neumann eigenfunctions, namely whether they are dense on the scale $\sigma^{-1}$. They point out that this cannot always hold without some additional regularity assumptions, as for example Dirichlet-to-Neumann eigenfunctions on a rectangle, of arbitrarily high eigenvalue, may be constant and nonzero along the full length of a side.

There has been significant recent progress on this question. For Steklov eigenfunctions, it was answered in the negative by Bruno and Galkowski \cite{BrGa2020}. In simplified and weakened form, their result reads:
\begin{thm}\cite{BrGa2020} There exists a compact domain $\Omega\subseteq\mathbb R^2$ with analytic boundary, and a fixed value $r_1>0$, for which each Steklov eigenfunction of $\Omega$ has a nodal domain which contains a ball of radius $r_1$.
\end{thm}
Note that the results of Bruno and Galkowski are in fact significantly stronger. They show that $\Omega$ may be chosen arbitrarily close to any fixed domain $\Omega_0$ with analytic boundary. They also give upper bounds for the `oscillation' of Steklov eigenfunctions on larger subsets of $\Omega$.

Despite both of these negative results, Decio has recently proved a positive density result, involving balls in $\Omega$ which are nonetheless centered at a point in $\partial\Omega$:
\begin{thm}\cite{De2021} Suppose that $\Omega$ is a Lipschitz domain in Euclidean space. Then there is a constant $C=C(\Omega)$ for which any ball of radius $C/\sigma_k$ centered at a point of $\partial\Omega$ intersects the nodal set of $u_{\sigma_k}$ nontrivially.
\end{thm}

\subsection{Nodal count}

Another natural question about nodal sets is to count nodal domains. This is inspired by the famous Courant nodal domain theorem, which states that the $k$th eigenfunction of the Laplacian on a compact manifold has at most $k$ nodal domains. It is an old result of Kuttler and Sigillito \cite{KuSi1969} that the same is true for Steklov eigenfunctions.

In \cite{HaSh2021}, Hassannezhad and Sher have recently investigated the analogue for eigenfunctions of the Steklov problem with a potential $q$.
\begin{thm}\label{thm:mainthmhash} \cite[Theorem 1.1]{HaSh2021} Let $\Omega$ be a Lipschitz domain in a smooth manifold $M$ and let $q\in L^{\infty}(\Omega)$ be a potential. Let $N_k$ be the number of nodal domains of a $k$th eigenfunction $u_k$ of the Steklov problem for $\Delta+q$ on $\Omega$. Then
\[N_k\leq k + d,\]
where $d$ is the number of non-positive Dirichlet eigenvalues of $\Delta + q$.
\end{thm}
This theorem is sharp, as the authors show via explicit examples.

The idea of the proof is to use Steklov-Robin duality. This idea is originally due to Friedlander \cite{Fr1991}, see also \cite{Ma1991} and sequels. Steklov eigenvalues and eigenfunctions, with the spectral parameter in the boundary condition, may be viewed alternatively as Robin eigenvalues and eigenfunctions, with the spectral parameter in the interior. After suitably generalising this set of ideas, the proof of Theorem \ref{thm:mainthmhash} is a direct consequence of a Courant-type theorem for Robin eigenfunctions.

There are many open questions. First, if $N_k$ is a Laplace nodal count, it is a result of Pleijel \cite{Pl1956} that
\[\limsup\frac{N_k}{k}<\gamma<1,\]
where $\gamma$ is an explicit constant.
\begin{ques}\label{ques:pleijel1}
Is there a Pleijel-type theorem for the nodal counts of Steklov eigenfunctions?
\end{ques}

Interestingly, very little seems to be known about the nodal counts of the \emph{Dirichlet-to-Neumann} eigenfunctions $\phi_k$. The same explicit examples show that a strict bound of $k$ is impossible, at least in the same generality considered in \cite{HaSh2021}. However one may ask the following:
\begin{ques}\label{ques:pleijel2} Is there a theorem like Theorem \ref{thm:mainthmhash} for Dirichlet-to-Neumann eigenfunctions? If so, can it be improved to a Pleijel-type result?
\end{ques}

\subsection{Other properties of eigenfunctions}

In the setting of the Laplacian, Uhlenbeck showed in 1976 \cite{uhlenbeck1976} that for \emph{generic} metrics on a fixed manifold $\Omega$,
\begin{itemize}
\item All eigenvalues are simple;
\item All eigenfunctions have zero as a regular value;
\item All eigenfunctions are Morse functions.
\end{itemize}
This work has recently been extended to the Steklov setting by Wang \cite{wang2022gen}. In particular, Wang shows analogues of each of the three results above, proving that the Steklov eigenvalues are generically simple. The latter two results are demonstrated for the Dirichlet-to-Neumann eigenfunctions, that is, the restrictions to the boundary of the Steklov eigenfunctions.

\begin{appendices}

\section{Variational eigenvalues of Radon measures}
\label{Section:Radon}

Let $(\Omega,g)$ be a compact Riemannian manifold of dimension $n=d+1$ with boundary $\Sigma$. Let $\mu$ be a nonzero Radon measure on $\Omega$ and consider the induced Rayleigh--Radon quotient
\[R_{\mu}(u):=\frac{\int_\Omega|\nabla u|^2\,dV_{\om}}{\int_{\Omega}u^2\,d\mu},\]
which is initially defined for $u\in C^\infty(\Omega)$ with $\int_{\Omega}u^2\,d\mu\neq 0$. 
The variational eigenvalues $\lambda_k(\Omega,g,\mu)$ are then defined through the min-max variational formula~\eqref{def:VarEigenRadon}.  (We will often write $\lambda_k(\mu)$ for $\lambda_k(\Omega,g,\mu)$, suppressing the name of the Riemannian manifold if it is fixed.)
The goal of this appendix is to give enough background to state Theorem~\ref{thm:ContinuityRadon}, which provides continuity for these eigenvalues. See section \ref{subsection:introvareigenradon} for examples connecting variational eigenvalues of Radon measures to various classical eigenvalue problems.

It follows from $R_\mu(1)=0$ that $\lambda_0(\Omega,g,\mu)=0$. For $\Omega$ connected, it is natural to expect that $\lambda_1(\Omega,g,\mu)$ should be positive. However, for an arbitrary Radon measure, the analogy between the variational eigenvalues $\lambda_k(\Omega,g,\mu)$ and the eigenvalues of usual elliptic operators could become very weak.
\begin{ex}
 Let $p_1,\cdots,p_\ell\in\Omega$ be distinct and consider the sum of delta-measures
 $\mu=\sum_i\delta_{p_i}$.
Because the capacity of a point is 0, we see directly that
$\lambda_0(\mu)=\lambda_1(\mu)=\cdots=\lambda_{\ell-1}(\mu)=0$. 
Moreover, since $L^2(\Omega,\mu)$ is only $\ell$-dimensional, the infimum in the definition of  $\lambda_\ell(\mu)$ is over the empty set, so $\lambda_\ell(\mu)=+\infty$.
\end{ex}
To develop the theory beyond a mere observation and catalogue of examples, a sound functional setting is required.
Define $H^1(\Omega,\mu)$ to be the completion of   $C^\infty(\Omega)$ with respect to the norm given by
\begin{equation}
  \label{eq:completion}
  \|u\|^2_{H^1(\Omega,\mu)} = \int_\Omega u^2\,d\mu + \int_\Omega |\nabla u|_g^2\,dv_g =
  \|u\|^2_{L^2(\Omega,\mu)} + \|{\nabla u}\|_{L^2(\Omega,g)}^2.
\end{equation}
Through this completion, the natural map $C^\infty(\Omega)\to L^2(\Omega,\mu)$ induces a bounded map \[\tau^{\mu}:H^1(\Omega,\mu)\rightarrow L^2(\Omega,\mu),\] which we call the \emph{trace map} induced by $\mu$. 
For all measures $\mu$ that we will use, the space $H^1(\Omega,\mu)$ coincides (although the norm may differ) with the usual Sobolev space $H^1(\Omega,dV_g)$ and the map $\tau^\mu$ is explicitly identified. 
Nevertheless, for an arbitrary Radon measure, the spaces $H^1(\Omega,\mu)$ could be very different from the usual ones.
\begin{ex}
 Let $\mu=\delta_{p}$ be a delta-measure supported at $p\in\Omega$. Then $H^1(\Omega,\mu)$ is naturally identified with $\R\times H^1(\Omega,dV_g))/\R$ and $\tau^\mu(t,f)=t$.
\end{ex}
By limiting our attention to a more manageable class of Radon measures, we will recover many of the expected properties observed in classical eigenvalue problems. The following definition was proposed by Girouard, Kapukhin and Lagac\'e in~\cite[Definition 3.3]{GiKaLa2021}.
\begin{defn}\label{def:admissible}
    A Radon measure $\mu$ on $\Omega$ is \emph{admissible} if the following conditions are satisfied:
    \begin{enumerate}
        \item[A1)] For each $p\in\Omega$, $\mu(\{p\})=0$;
        \item[A2)] The map $\tau^{\mu}:H^1(\Omega,\mu)\to L^2(\Omega,\mu)$ is compact;
        \item[A3)] There exists $K > 0$ such that for all $u \in C^\infty(\Omega)$ with vanishing $\mu$-average,
  \begin{equation}\label{Ineq:Poincare}
    \int_\Omega u^2\, d\mu \le K \int_\Omega |\nabla u|^2\, dV.
  \end{equation}
    \end{enumerate}
\end{defn}
One can show that the best constant in $(A3)$ is $\lambda_1(\Omega,g,\mu)^{-1}$.
Moreover, the Poincar\'e-type inequality $(A3)$ implies a Poincar\'e--Wirtinger type inequality for all $u\in C^\infty(\Omega)$:
\[\int_\Omega u^2\, d\mu \le K'( \int_\Omega |\nabla u|^2\, dV+\int_\Omega u^2\,dV).\]
In particular, the natural map $j:C^\infty(\Omega)\to L^2(\Omega,\mu)$ extends to a bounded linear map $T_\mu:H^1(\Omega,dV_g))\to L^2(\Omega,\mu)$.  (We emphasize that $T_\mu$ and the map $\tau^\mu$ defined above have different domains as normed spaces.)
\begin{lemma}\cite[Theorem 3.4]{GiKaLa2021}\label{Lemma:AdminissibleAlternative}
Let $\mu$ be a nonzero Radon measure on $\Omega$ such that $\mu(\{p\})=0$ for each $p\in\Omega$. Then $\mu$ is admissible if and only if the natural map $j:C^\infty(\Omega)\to L^2(\Omega,\mu)$ extends to a compact linear map
$T_\mu:H^1(\Omega,dV_g)\to L^2(\Omega,\mu)$. 
\end{lemma}
To complete this picture, it is useful to know that the Sobolev--Radon spaces associated to admissible Radon measures are isomorphic to each other.
\begin{lemma}\cite[Theorem 3.5]{GiKaLa2021}
If $\mu,\xi$ are two admissible measures, then the identity map on $C^\infty(\Omega)$ extends to a bounded invertible linear map $T_{\mu,\xi}:H^1(\Omega,\mu)\to H^1(\Omega,\xi)$.
\end{lemma}
In particular, since $dV_g$ is itself an admissible measure, $H^1(\Omega,\mu)\cong H^1(\Omega,dV_g)$ for any admissible $\mu$.
\begin{remark}
The literature on Sobolev spaces is replete with compactness criteria, which are useful in deciding which measures are admissible. See the book~\cite[Chapter 11]{Ma2011} by Maz’ya. In particular, if $\mu$ is admissible then any set $E\subset\Omega$ of vanishing capacity must satisfy $\mu(E)=0$. A very interesting geometric criterion for admissibility in terms of a quantity called the \emph{lower $\infty$-dimension of $\mu$} could also be deduced from the work of Hu, Lau, and Ngai~\cite{HuLaNg2006}. See also~\cite[Section 2]{Ko2014}.
\end{remark}
The following Proposition shows that variational eigenvalues of admissible measures behave similarly to the eigenvalues of elliptic operators with compact resolvent.
\begin{prop}\label{thm:SpectralTheoremMeasure}
Let $\mu$ be an admissible measure on the compact connected Riemannian manifold $(\Omega,g)$.
Then the variational eigenvalues $\lambda_k(\mu)=\lambda_k(\Omega,g,\mu)$ form an unbounded sequence
\[
0=\lambda_0(\mu)<\lambda_1(\mu)\leqslant\lambda_2(\mu)\leqslant\ldots\nearrow\infty.
\]
A real number $\lambda$ is an eigenvalue if and only if there exists $0\neq u\in H^1(\Omega,\mu)$ such that
\begin{equation}
\label{eigenfunctions_measures:def}
\int_\Omega  \nabla u\cdot \nabla \phi \, dV = \lambda\int_\Omega \tau^{\mu}(u) \tau^{\mu}(\phi)\, d\mu,\qquad\qquad\text{for all } \phi\in H^1(\Omega,\mu).
\end{equation}
In that case we call $u$ an eigenfunction corresponding to $\lambda$.
Moreover, there exists a sequence $(u_j)\subset H^1(\Omega,\mu)$ of eigenfunctions corresponding to $\lambda_j(\mu)$ such that the functions $\tau^\mu(f_j)$ form an orthonormal basis of $L^2(\Omega,\mu)$.
\end{prop}
In~\cite{Ko2014} the proof of this result is described in terms of a classical recursive procedure for constructing the eigenvalues. However, it is interesting to observe that the eigenvalues $\lambda_k(\Omega,g,\mu)$ are the eigenvalues of a measure geometric Laplace operator $\Delta_\mu$ of Krein--Feller type. The whole theory could also be presented from that perspective. See~\cite{HuLaNg2006} and also the recent preprint~\cite{KeNi2022} of Kesseb\"ohmer and Niemann where these operators are studied for the purpose of studying spectral asymptotics on fractals. In particular, it is interesting to compare Proposition~\ref{thm:SpectralTheoremMeasure} with~\cite[Theorem 1.2]{HuLaNg2006}.

\subsection{Continuity properties of variational eigenvalues}

A sequence $(\mu_n)$ of Radon measures on $\Omega$ converges to $\mu$ in the weak-$\star$ topology if $\int_{\Omega}f\,d\mu_n\xrightarrow{n\to\infty}\int_{\Omega}f\,d\mu$ for each $f\in C^0(\Omega)$. In that case we write
$\mu_n\xrightharpoonup{\star}\mu$.
In~\cite{Ko2014}, Kokarev proved upper semicontinuity of variational eigenvalues:
\[\limsup_{n\to\infty}\lambda_k(\mu_n)\leq\lambda_k(\mu).\]
This holds for all Radon measures, whether they are admissible or not. 

\begin{ex}Let us consider a simple application of the upper semicontinuity.
Let $M$ be a closed Riemannian manifold of dimension $d$. Let $\Sigma\subset M$ be a submanifold of dimension $0<n\leq d-2$, with measure $dV_\Sigma$. The inclusion $\iota:\Sigma\to M$ allows the definition of the push-forward probability measure $\mu=|\Sigma|^{-1}\iota_{\star}dV_\Sigma$ on $M$. This measure is not admissible, since $\Sigma$ has vanishing capacity and one can show that $\lambda_k(M,g,\mu)=0$ for each $k$. 
Consider tubular neighborhoods $T_\eps:=\{x\in M\,:\,d(x,\Sigma)<\eps\}$, and let $\Omega_\eps=M\setminus T_\eps$ with corresponding inclusions still written $\iota:\partial\Omega_\eps\to M$.
The boundary probability measures $\mu_\eps=|\partial T_\eps|^{-1}\iota_{\star}dV_{\partial T_\eps}$ are admissible. In fact we have seen in example~\ref{example:transmission} that their variational eigenvalues are related to the transmission eigenvalues of $\Omega_\eps$ in $M$:
\[\lambda_k(M,g,\mu_\eps)=\tau_k(M,\Omega_\eps)|\partial\Omega_\eps|\geq\sigma_k(\Omega_\eps)|\partial\Omega_\eps|.\]
It is clear from their definition that $\mu_\eps\xrightharpoonup{\star}\mu$, and it follows from the semi-continuity property that
\[\lim_{\eps\to 0}\sigma_k(\Omega_\eps)|\partial\Omega_\eps|=\lim_{\eps\to 0}\lambda_k(M,g,\mu_\eps)=0.\]
This raises the question to understand the asymptotic behaviour of $\sigma_k(\Omega_\eps)$ as $\eps\to 0$. For connected submanifolds $\Sigma$, this was studied by Brisson in~\cite{Br2022}, where she proved for instance that for $d\geq 3$ and $0<n\leq d-2$, and for each index $k\in\N$,
\begin{gather*}
    \eps\sigma_k(\Omega_\eps)\xrightarrow{\eps\to0}d-n-2.
\end{gather*}
See Example \ref{JadeBrisson}.
\end{ex}

In~\cite[Proposition 4.8]{GiKaLa2021}, Girouard, Karpukhin and Lagac\'e gave sufficient conditions for the continuity of $\lambda_k(\mu)$ under weak-$\star$ convergence. 
\begin{thm}\label{thm:ContinuityRadon}
Let $(\Omega,g)$ be a compact Riemannian manifold (with or without boundary), and let $\Omega_n\subset\Omega$ be a sequence of domains such that $|\Omega\setminus\Omega_n|\to 0$. Let $\mu_n$ be a sequence of admissible Radon measures supported on $\overline{\Omega}_n$ and let $\mu$ be an admissible measure on $\Omega$. Suppose that the following conditions are satisfied:
\begin{itemize}
    \item[M1)] $\mu_n\xrightharpoonup{\star}\mu$;
    \item[M2)] There is a bounded family of extension maps $J_n:H^1(\Omega_n,\mu_n)\to H^1(\Omega,\mu)$.
\end{itemize}
Suppose moreover that the the measures $\mu_n,\mu$ induce elements of $W^{1,1}(\Omega,dV_g)^\star$ such that $\mu_n\to\mu$ in $W^{1,1}(\Omega,dV_g)^\star$. Then for each index $k$,
\[\lim_{n\to\infty}\lambda_k(\Omega_n,\mu_n)=\lambda_k(\Omega,\mu).\]
\end{thm}
Here $W^{1,1}(\Omega,dV_g)$ is the usual Sobolev space of $L^1$-functions with weak gradient in $L^1$.
\begin{remark}
See also the paper~\cite{FrMi2020} by Freiberg and Minorics for related results in dimension 1, presented in terms of Krein--Feller operators.
\end{remark}
\begin{remark}
Theorem \ref{thm:ContinuityRadon} provides a flexible tool to discuss eigenvalue convergence. It is the main technical tool behind the proof of Theorem~\ref{thm: GiKaLa 8pi} and Theorem~\ref{thm: GKL sharp}, via homogenisation by perforation. It also provides a convenient setting to discuss boundary homogenisation, as in the work of Bucur and Nahon~\cite{BuNa2021}. See Remark~\ref{rem:BucurNahon}. 
\end{remark}

\section{Open problems}\label{app:open ques}

We list here the open problems and questions that have been proposed throughout the paper.

\begin{itemize}

\item[\ref{ques:conformal}]
Describe the class of all smooth compact surfaces $\Omega\subset\R^3$ with boundary $\partial\Omega=\partial\D$ that admit a conformal parametrisation $\Phi:\D\to\Omega$ such that $|\Phi'|\equiv 1$ on $\partial\D$.
\item[\ref{ques:karpstern}.] If $M$ is a closed surface and $\Omega$ is a domain in that surface, does the strict inequality 
\[\sigma_k(\Omega,g)L(\partial\Omega)<\lambda_k^*(M,[g])\]
hold for $k\geq 3$? (It is known for $k=1$ and $k=2$.)
\item[\ref{conj:escobar}]
(Escobar's Conjecture, \cite{Es1999})
Let $\Omega$ be a smooth compact connected Riemannian manifold of dimension
$\ge 3$ with boundary $\Sigma=\partial \Omega$. Suppose that the Ricci curvature of $\Omega$ is non-negative
and that the second fundamental form $\rho$ of $\Sigma$ is bounded below by $c>0$.
Then $\sigma_1(\Omega) \ge c$, with equality if and only if $\Omega$ is the Euclidean ball of radius $\frac{1}{c}$. (Progress on this is discussed in subsection \ref{subsec.lowerboundgeom}).
\item[\ref{question lower}] 
Under the hypotheses (H1)-(H3) of Theorem \ref{comparison}, is it possible to find an explicit lower bound for $\sigma_1(\Omega)$ in terms of geometric invariants of $\Omega$ and of it boundary $\Sigma$ even when $\Sigma$ is not assumed to be connected?
 \item[\ref{question cheeger}] 
Can one define a different Cheeger-type constant for which $\sigma_1$ satisfies a Buser-type inequality?
\item[\ref{ques:logbase2}]
Is the coefficient $\frac{C'}{\log^2(k+2)}$ of $i_{k+1}(\Omega)$ in the following estimate (estimate (\ref{better})) sharp?
\begin{equation*}
\sigma_{2k+1}(\Omega)\ge \frac{C'}{\log^2(k+2)} i_{k+1}(\Omega).
\end{equation*}

\item[\ref{ques:higherordercheeger}]
Can one find a higher order Cheeger inequality without any dependence on $k$?
\item[\ref{question:maxdom2}]
If $(M,g)$ is a complete Riemannian manifold of dimension $\ge 3$ of infinite volume, with Ricci curvature bounded from below, can one construct domains $\Omega \subset M$ with arbitrarily large first nonzero normalised eigenvalue (for the different normalisations mentioned in subsection \ref{Upperexamples})?
\item[\ref{ques:intrinsicextrinsic}]
Let $\Omega$ be a compact Riemannian manifold of dimension $d+1\ge 3$ with (connected) boundary $\Sigma$. Is it possible to construct a family $(g_{\epsilon})_{\epsilon>0}$ of Riemannian metrics on $\Omega$ that stays constant on $\Sigma$ and satisfies $\sigma_1(\Omega,g_{\epsilon}) \vert \Sigma\vert\vert(\Omega, g_{\epsilon})\vert^{\frac{1-d}{d+1}} \to \infty$ as $\epsilon \to 0$?
\item[\ref{ques:weinstockhnsn}]Let $\Omega$ be a bounded convex domain in $\mathbb H^{d+1}$ or $\mathbb S^{d+1}$. Let $\Omega^*$ be a ball in the same space with $|\partial\Omega^*|=|\partial\Omega|$. Is it true that 
\begin{equation*}
		\sigma_1(\Omega) \le \sigma_1(\Omega^*)
	\end{equation*}
 with equality iff $\Omega$ is a ball?
 \item[\ref{ques:diamterm}]  The term $\frac{sn_{K}(\diam(\Omega))}{sn_{\kappa}(\diam(\Omega))}$ in \eqref{LiWaWu} may become very large when $\diam(\Omega)$ becomes large. Is it possible to establish a better estimate or to construct an example of a domain $\Omega$ with large diameter and $\sigma_1(\Omega)$ large?
 \item[\ref{ques:cartanhadamard}]
Can we get estimates like (\ref{LiWaWu}) for the other eigenvalues of domains in Cartan-Hadamard manifolds? The methods used in \cite{CoElGi2011} or \cite{Ha2011} do not seem to apply.
\item[\ref{ques:diameterhnsn}]
	Is it possible to get inequalities similar to those in Theorem \ref{diam1} for domains in the hyperbolic space or the sphere?
 \item[\ref{ques:specconvtozero}]
If the diameter of a domain $\Omega$ is fixed and its volume tends to $0$, can one say that all the eigenvalues of the domain tend to $0$?
\item[\ref{ques:manyquestionsdoublyconnected}] Can the results on doubly connected domains at the end of Subsection \ref{UpperDomains} be improved? (See the text for the full question).
\item[\ref{ques:rev1}] Can Xiong's results on upper and lower bounds for Steklov eigenvalues on manifolds of revolution, Theorems \ref{thm:xiong1} and \ref{thm:xiong2}, be improved by imposing a geometric constraint such as $\vert Ricci \vert \le a^2$? 
\item[\ref{ques:rev2}] Can Xiong's results (see previous question) be generalised to manifolds of revolution with two boundary components?
\item[\ref{ques:2bc1}] Find a sharp lower bound for $\sigma_{k}(\Omega)$, where $\Omega$ is a hypersurface of revolution with boundary components $\left(\mathbb \Sp^{d} \times \{0\}\right) \cup \left(\mathbb \Sp^{d} \times \{\delta\}\right)$ and $\delta<2$.
\item[\ref{ques:2bc2}] Find a sharp upper bound for $\sigma_{k}(\Omega)$, where $\Omega$ is a hypersurface of revolution with boundary components $\left(\mathbb \Sp^{d} \times \{0\}\right) \cup \left(\mathbb \Sp^{d} \times \{\delta\}\right)$.
\item[\ref{ques:2bc3}]
In the case of a hypersurface of revolution with either one or two boundary components, can we obtain results as in Theorem \ref{thm:xiong2}, that is, sharp bounds for the difference $\sigma_{(k+1)}-\sigma_{(k)}$ and the ratio $\frac{\sigma_{(k+1)}}{\sigma_{(k)}}$?
\item[\ref{open:arbitrarilylarge}]
Given a $d$-dimensional compact submanifold $\Sigma$ in $\R^m$, is it possible to construct a family of $d+1$- dimensional submanifolds $\Omega$ of $\R^m$ with boundary $\Sigma$ for which $\sigma_1(\Omega)$ becomes arbitrarily large?

\item[\ref{ques:eucl1}]
Is it possible to get a similar inequality 
 to \eqref{ineqmean} for submanifolds of hyperbolic space?

\item[\ref{ques:eucl2}] 
Is it possible to generalise Inequality (\ref{ineqmean}) to other eigenvalues? Note that a similar question for the spectrum of the Laplacian was solved only recently (and partially) by Kokarev  in~\cite[Theorem 1.6]{Ko2020}.

\item[\ref{ques:any higher maximisers}] Can one find examples of compact surfaces with boundary and integers $k>1$ for which a $\sigma_k$-maximising metric exists.
\item[\ref{ques:gapk}.] For given $k$, which compact surfaces $\om$ satisfy $\gap_k(\om)>0$? (See Definition \ref{def.gap} for a definition of $\gap_k(\om)$).

\item [\ref{ques:sig >2pi}.] Is $\sigma_1^*(\om,[g])>2\pi$ for every conformal class $[g]$ when the surface $\om $ is not diffeomorphic to a disk?

\item[\ref{ques:ksoq3}] \cite[Open Question 3]{KaSt2021}. In the setting of Theorem~\ref{thm.KS asymptotics}, if the limiting surface in $\sph^{N-1}$ realising $\lambda_1^*(M)$ is embedded, does it necessarily follow that the minimal surfaces in $\B^{N}$ realising $\sigma_1^*(M_b)$ are embedded for all sufficiently large $b$?

\item[\ref{fraser li conj}] (Fraser and Li's Conjecture, \cite[Conjecture 3.3]{FrLi2014}) If $\om$ is a properly embedded free boundary minimal hypersurface in $\B^n$, then $\sigma_1(\om)=1$, i.e., $\om$ has spectral index one.
\item[\ref{conj:mcKu}]\cite[Conjecture 5]{KuMc2022} The first Steklov eigenspace of any properly embedded free boundary minimal surface $S$ in $\B^n$ coincides with the span of the coordinate functions of the embedding.
\item[\ref{ques: conf ext metric vs pair}] 
If $g$ is a $\bsig_k$-(conformally) extremal metric, is $(g,1)$ necessarily a $\bsig_k$-(conformally) extremal pair?  
\item[\ref{ques: conf spec simple}] Given a compact manifold $\om$ with boundary and a conformal class $[g]$ of Riemannian metrics on $\om$, define
\[\bsig_k(\om, [g])=\sup_{g'\in [g]}\, \bsig(\om,g).\]
Is \[\bsig_k(\om,[g])<\bsig_{k+1}(\om, [g]) \mbox{\,\,for all}\,\,k?\]
\item[\ref{ques:discretcomp}]For all $1\le k \le \vert B\vert-1$,
is it possible to obtain an inequality of the form
\[
A_1 \le \frac{\sigma_k(\Omega)}{\sigma_k(\Gamma)}\le A_2,
\]
where $\Gamma$ is a discretisation of $\Omega$ defined as in subsection \ref{Discretisation and spectrum}?
\item[\ref{ques:cheegerdiscrete}] In the discrete setting, is it possible to have a higher order Cheeger inequality without any dependence on $k$?
\item[\ref{ques:discretequasiisom}] Let $\Gamma_1,\Gamma_2$ be two infinite roughly quasi-isometric graphs. If there exists a constant $C_1(\Gamma_1)$ such that for each finite subgraph $\Omega$ of $\Gamma_1$, $\sigma_1(\bar \Omega)\le \frac{C_1(\Gamma_1)}{\vert B\vert^{\alpha}}$, with $B$ the boundary of $\Omega$ and $\alpha>0$, is an analogous property also true for the finite subgraphs of $\Gamma_2$?

\item[\ref{ques:yangyusharp}] Is Yang-Yu's inequality \eqref{yyiso} for eigenvalues of the Raulot-Savo Dirichlet-to-Neumann operator sharp when $k=1$ and $p<\frac{d+1}{2}$?   Is it sharp when $k>1?$
    
\item[\ref{ques: karpukhin pseudo}] Is Karpukhin's Dirichlet-to-Neumann operator $\kDtN_p$ a pseudo-differential operator?

\item[\ref{ques: Karp Weyl error}] Can the error bound in the Weyl Law~\eqref{eq:karp p weyl} for the asymptotics  of $\ksig_{k,p}$ be improved to $O(\sigma^{d-1})$? (This would follow from an affirmative answer to Question~\ref{ques: karpukhin pseudo}.  See \cite[Remark 5.9]{GKLP2021} for further comments).
\item[\ref{prob:karp}] Given a closed oriented Riemannian manifold $(\sig,h)$ of dimension $d=2p+1$, let $[\sig,h]_m$ denote the collection of all orientable Riemannian manifolds $(\om,g)$ with $\partial \om =\sig$, $g_{|\sig}=h$ and $\beta_{d-p}(\om)=m$.  For fixed $m$ and $k$, investigate \[\sup\,\{\ksig_{k,p}(\om,g): (\om,g)\in [\sig,h]_m\}.\] 
\item[\ref{conjs:karpuhkin}]Some conjectures from Karpukhin \cite{Ka2019}:
\begin{itemize}
\item The inequality in Theorem~\ref{khpsp} is sharp for all $k$.
\item When $k\leq \frac{1}{2}\binom{2p+2}{p+1}$, equality holds only for the Euclidean unit ball.
\end{itemize}
\item[\ref{ques:weylrough}]
 Can it be shown that the Steklov Weyl asymptotics \eqref{eq:stekweyl} hold whenever $\partial\Omega$ is Lipschitz? It has been done when $d+1=2$ \cite{KLP2022}.

\item[\ref{ques:multconnops}] 
Are families of multiply connected, compact, Steklov isospectral planar domains necessarily compact in the $C^{\infty}$ topology?

\item[\ref{ques:polygonsbigangles}]
When obtaining Steklov spectral asymptotics for curvilinear polygons, can the condition that the angles of $\Omega$ are in $(0,\pi)$ be relaxed to $(0,2\pi)$?

\item[\ref{ques:recoverangles}]
Under the hypotheses of Theorem \ref{thm:mainklpps}, can the angles themselves, and not just $\cos(\pi^2/2\alpha)$, be recovered from the Steklov spectrum of a (possibly curvilinear) polygon?

\item[\ref{ques: orb boundary}]
Does the Steklov spectrum detect the presence of singularities on the boundary of a compact Riemannian orbifold? (Some results are known for orbisurfaces, see subsection \ref{subsec:orb pos}).

\item[\ref{ques: orb interior}]
Does the Steklov spectrum detect the presence of singularities in the interior of a compact Riemannian orbifold? (Some results are known for orbisurfaces, see subsection \ref{subsec:orb pos}).

\item[\ref{ques:mixed vs pure}] Can a mixed Steklov-Neumann problem on a compact Riemannian manifold $\om_1$  and a pure Steklov problem on a compact Riemannian manifold $\om_2$ have the same spectrum?

\item[\ref{ques.sing type}] Among Riemannian orbifolds with singularities, to what extent does the Steklov spectrum recognise the types of singularities?

\item[\ref{prob: orb heat}]  Develop Steklov heat asymptotics for compact Riemannian orbifolds with boundary.

\item[\ref{ques.NotLapIsosp}]
Do there exist Steklov isospectral Riemannian manifolds whose boundaries are not Laplace isospectral? Equivalently, does the Steklov spectrum of a Riemannian manifold determine the Laplace spectrum of its boundary?

\item[\ref{ques.NotNeumIsosp}]
Do there exist Steklov isospectral Riemannian manifolds that are not Neumann isospectral (other than $\sigma$-isometric surfaces)?

\item[\ref{ques:isospplanedomains}] 
Do there exist non-isometric Steklov isospectral plane domains?

\item[\ref{quesgeq3}] For manifolds of dimension $\geq 3$, can one tell from the Steklov spectrum:
\begin{itemize}
\item Whether the manifold has constant sectional curvature? 
\item Whether the induced metric on the boundary is Einstein?
\item Whether the induced metric on the boundary has constant sectional curvature?
\item Whether the mean curvature of the boundary is constant?
\item Whether the principal curvatures of the boundary are constant?  

 \end{itemize}

\item[\ref{ques: inverse prob forms}] Do there exist compact Riemannian manifolds $\om_1$ and $\om_2$ with boundary such that we have $p$ and $q$ with $\stek_p^{\rs}(\om_1)=\stek_p^{\rs}(\om_2)$ but $\stek_q^{\rs}(\om_1)\neq \stek_q^{\rs}(\om_2)$?  One can ask the same question for $\stek_p^{K}(\om_i)$. (Note that answers to some similar problems for orbifolds seem to be ``yes", see subsection \ref{ballquots}).

\item[\ref{ques:interiordecaysmooth}] Does exponential decay of Steklov eigenfunctions into the interior, for example \eqref{eq:gtbound}, hold if $\Sigma$ and/or $\Omega$ are merely assumed smooth rather than real analytic?

\item[\ref{ques:decaynonsmooth}] More generally, what kind of decay properties do the Steklov eigenfunctions possess if $\Sigma$ is not assumed to be smooth?

\item[\ref{ques:generic}] Do Steklov eigenfunctions generically concentrate on a single boundary component? (The case of dimension 2 has been answered in the affirmative by Martineau \cite{Ma2018}).

\item[\ref{ques:yau}]\cite[Open Problem 11]{GiPo2017} Does the Steklov analogue of the nodal volume conjecture of Yau hold? That is, do there exist constants $c$ and $C$ depending only on $\Omega$ for which
\begin{equation*}
    c\sigma_k\leq\mathcal H_{d}(u_k)\leq C\sigma_k\textrm{ and/or }
    c\sigma_k\leq\mathcal H_{d-1}(\varphi_k)\leq C\sigma_k?
\end{equation*}

\item[\ref{ques:pleijel1}]
Is there a Pleijel-type theorem for the nodal counts of Steklov eigenfunctions?

\item[\ref{ques:pleijel2}] Is there a theorem like Theorem \ref{thm:mainthmhash} for Dirichlet-to-Neumann eigenfunctions? If so, can it be improved to a Pleijel-type result?

\end{itemize}

\end{appendices}

\bibliographystyle{plain}
\bibliography{surveybib}

\begin{thebibliography}{100}

\bibitem{AbFr2002}
Miguel Abreu and Pedro Freitas.
\newblock On the invariant spectrum of {$S^1$}-invariant metrics on {$S^2$}.
\newblock {\em Proc. London Math. Soc. (3)}, 84(1):213--230, 2002.

\bibitem{Ag2006}
Mikhail~S Agranovich.
\newblock On a mixed {P}oincar{\'e}-{S}teklov type spectral problem in a
  {L}ipschitz domain.
\newblock {\em Russian Journal of Mathematical Physics}, 13(3):239--244, 2006.

\bibitem{AkElOs2017}
Eldar Akhmetgaliyev, Chiu-Yen Kao, and Braxton Osting.
\newblock Computational methods for extremal {S}teklov problems.
\newblock {\em SIAM J. Control Optim.}, 55(2):1226--1240, 2017.

\bibitem{ABHN2021}
Abdelkader Al~Sayed, Beniamin Bogosel, Antoine Henrot, and Florent Nacry.
\newblock Maximization of the {S}teklov eigenvalues with a diameter constraint.
\newblock {\em SIAM J. Math. Anal.}, 53(1):710--729, 2021.

\bibitem{An1986}
Colette Ann\'{e}.
\newblock Perturbation du spectre {$X\setminus TUB^\epsilon Y$} (conditions de
  {N}eumann).
\newblock In {\em S\'{e}minaire de {T}h\'{e}orie {S}pectrale et
  {G}\'{e}om\'{e}trie, {N}o. 4, {A}nn\'{e}e 1985--1986}, pages 17--23. Univ.
  Grenoble I, Saint-Martin-d'H\`eres, 1986.

\bibitem{ADGHRS2019}
Teresa Arias-Marco, Emily~B. Dryden, Carolyn~S. Gordon, Asma Hassannezhad,
  Allie Ray, and Elizabeth Stanhope.
\newblock Spectral geometry of the {S}teklov problem on orbifolds.
\newblock {\em Int. Math. Res. Not. IMRN}, (1):90--139, 2019.

\bibitem{ADGHRS2023}
Teresa Arias-Marco, Emily~B. Dryden, Carolyn~S. Gordon, Asma Hassannezhad,
  Allie Ray, and Elizabeth Stanhope.
\newblock Applications of possibly hidden symmetry to {S}teklov and mixed
  {S}teklov problems on surfaces, 2023.
\newblock arXiv:2301.09010.

\bibitem{Ar2014}
Sinan Ariturk.
\newblock A disc maximizes {L}aplace eigenvalues among isoperimetric surfaces
  of revolution, 2014.
\newblock arXiv:1410.2221.

\bibitem{Ar2016}
Sinan Ariturk.
\newblock Maximal spectral surfaces of revolution converge to a catenoid.
\newblock {\em Proc. A.}, 472(2194):20160239, 12, 2016.

\bibitem{Ar2018}
Sinan Ariturk.
\newblock An annulus and a half-helicoid maximize {L}aplace eigenvalues.
\newblock {\em J. Spectr. Theory}, 8(2):315--346, 2018.

\bibitem{Au2013}
Erwann Aubry.
\newblock Approximation of the spectrum of a manifold by discretization, 2013.
\newblock arXiv:1703.02587.

\bibitem{BaCu2016}
Rondinelle Batista and Antonio~W. Cunha.
\newblock Estimates of the first {S}teklov eigenvalue of properly embedded
  minimal hypersurfaces with free boundary.
\newblock {\em Bull. Braz. Math. Soc. (N.S.)}, 47(3):871--881, 2016.

\bibitem{BeSh2008}
Mikhail Belishev and Vladimir Sharafutdinov.
\newblock Dirichlet to {N}eumann operator on differential forms.
\newblock {\em Bull. Sci. Math.}, 132(2):128--145, 2008.

\bibitem{BeLi2015}
Katar\'{\i}na Bellov\'{a} and Fang-Hua Lin.
\newblock Nodal sets of {S}teklov eigenfunctions.
\newblock {\em Calc. Var. Partial Differential Equations}, 54(2):2239--2268,
  2015.

\bibitem{Be1979}
Marcel Berger.
\newblock Une in\'{e}galit\'{e} universelle pour la premi\`ere valeur propre du
  laplacien.
\newblock {\em Bull. Soc. Math. France}, 107(1):3--9, 1979.

\bibitem{BiSa2014}
Binoy and G.~Santhanam.
\newblock Sharp upper bound and a comparison theorem for the first nonzero
  {S}teklov eigenvalue.
\newblock {\em J. Ramanujan Math. Soc.}, 29(2):133--154, 2014.

\bibitem{Bl1983}
David~D. Bleecker.
\newblock The spectrum of a {R}iemannian manifold with a unit {K}illing vector
  field.
\newblock {\em Trans. Amer. Math. Soc.}, 275(1):409--416, 1983.

\bibitem{BoBuGi2017}
Beniamin Bogosel, Dorin Bucur, and Alessandro Giacomini.
\newblock Optimal shapes maximizing the {S}teklov eigenvalues.
\newblock {\em SIAM J. Math. Anal.}, 49(2):1645--1680, 2017.

\bibitem{BrDeRu2012}
Lorenzo Brasco, Guido De~Philippis, and Berardo Ruffini.
\newblock Spectral optimization for the {S}tekloff-{L}aplacian: the stability
  issue.
\newblock {\em J. Funct. Anal.}, 262(11):4675--4710, 2012.

\bibitem{BMWE2018}
Hubert~L. Bray, William~P. Minicozzi, II, Michael Eichmair, Lan-Hsuan Huang,
  Shing-Tung Yau, Karen Uhlenbeck, Rob Kusner, Fernando Cod\'{a}~Marques,
  Chikako Mese, and Ailana Fraser.
\newblock The mathematics of {R}ichard {S}choen.
\newblock {\em Notices Amer. Math. Soc.}, 65(11):1349--1376, 2018.

\bibitem{Br2012}
Simon Brendle.
\newblock A sharp bound for the area of minimal surfaces in the unit ball.
\newblock {\em Geom. Funct. Anal.}, 22(3):621--626, 2012.

\bibitem{Br2019}
Jade Brisson.
\newblock Probl\`emes isop\'erim\'etriques et isospectralit\'e pour le
  probl\`eme de {S}teklov.
\newblock 2019.
\newblock Master's thesis, Universit\'e Laval.

\bibitem{Br2022}
Jade Brisson.
\newblock Tubular excision and {S}teklov eigenvalues.
\newblock {\em J. Geom. Anal.}, 32(5):Paper No. 166, 24, 2022.

\bibitem{Br2001}
Friedemann Brock.
\newblock An isoperimetric inequality for eigenvalues of the {S}tekloff
  problem.
\newblock {\em ZAMM Z. Angew. Math. Mech.}, 81(1):69--71, 2001.

\bibitem{Br1986}
Robert Brooks.
\newblock The spectral geometry of a tower of coverings.
\newblock {\em J. Differential Geom.}, 23(1):97--107, 1986.

\bibitem{BrMa2001}
Robert Brooks and Eran Makover.
\newblock Riemann surfaces with large first eigenvalue.
\newblock {\em J. Anal. Math.}, 83:243--258, 2001.

\bibitem{BrHa2012}
Andries~E. Brouwer and Willem~H. Haemers.
\newblock {\em Spectra of graphs}.
\newblock Universitext. Springer, New York, 2012.

\bibitem{BrGa2020}
Oscar~P. Bruno and Jeffrey Galkowski.
\newblock Domains without dense {S}teklov nodal sets.
\newblock {\em J. Fourier Anal. Appl.}, 26(3):Paper No. 45, 29, 2020.

\bibitem{BFNT2021}
Dorin Bucur, Vincenzo Ferone, Carlo Nitsch, and Cristina Trombetti.
\newblock Weinstock inequality in higher dimensions.
\newblock {\em J. Differential Geom.}, 118(1):1--21, 2021.

\bibitem{BuGiTr2020}
Dorin Bucur, Alessandro Giacomini, and Paola Trebeschi.
\newblock {$L^\infty$} bounds of {S}teklov eigenfunctions and spectrum
  stability under domain variation.
\newblock {\em J. Differential Equations}, 269(12):11461--11491, 2020.

\bibitem{BuHeMi2021}
Dorin Bucur, Antoine Henrot, and Marco Michetti.
\newblock Asymptotic behaviour of the {S}teklov spectrum on dumbbell domains.
\newblock {\em Comm. Partial Differential Equations}, 46(2):362--393, 2021.

\bibitem{BuNa2021}
Dorin Bucur and Micka\"{e}l Nahon.
\newblock Stability and instability issues of the {W}einstock inequality.
\newblock {\em Trans. Amer. Math. Soc.}, 374(3):2201--2223, 2021.

\bibitem{BuIvKu2015}
Dmitri Burago, Sergei Ivanov, and Yaroslav Kurylev.
\newblock A graph discretization of the {L}aplace-{B}eltrami operator.
\newblock {\em J. Spectr. Theory}, 4(4):675--714, 2014.

\bibitem{Bu1980}
Peter Buser.
\newblock On {C}heeger's inequality {$\lambda _{1}\geq h^{2}/4$}.
\newblock In {\em Geometry of the {L}aplace operator ({P}roc. {S}ympos. {P}ure
  {M}ath., {U}niv. {H}awaii, {H}onolulu, {H}awaii, 1979)}, Proc. Sympos. Pure
  Math., XXXVI, pages 29--77. Amer. Math. Soc., Providence, R.I., 1980.

\bibitem{Bu1982}
Peter Buser.
\newblock A note on the isoperimetric constant.
\newblock {\em Ann. Sci. \'Ecole Norm. Sup. (4)}, 15(2):213--230, 1982.

\bibitem{Bu1986}
Peter Buser.
\newblock Isospectral {R}iemann surfaces.
\newblock {\em Ann. Inst. Fourier (Grenoble)}, 36(2):167--192, 1986.

\bibitem{Bu1992}
Peter Buser.
\newblock {\em Geometry and spectra of compact {R}iemann surfaces}, volume 106
  of {\em Progress in Mathematics}.
\newblock Birkh\"{a}user Boston, Inc., Boston, MA, 1992.

\bibitem{BuBuDo1988}
Peter Buser, Marc Burger, and Jozef Dodziuk.
\newblock Riemann surfaces of large genus and large {$\lambda_1$}.
\newblock In {\em Geometry and analysis on manifolds ({K}atata/{K}yoto, 1987)},
  volume 1339 of {\em Lecture Notes in Math.}, pages 54--63. Springer, Berlin,
  1988.

\bibitem{CaFrSc2020}
Alessandro Carlotto, Giada Franz, and Mario~B. Schulz.
\newblock Free boundary minimal surfaces with connected boundary and arbitrary
  genus.
\newblock {\em Camb. J. Math.}, 10(4):835--857, 2022.

\bibitem{Ca2002}
Gilles Carron.
\newblock D\'{e}terminant relatif et la fonction {X}i.
\newblock {\em Amer. J. Math.}, 124(2):307--352, 2002.

\bibitem{CaRu2019}
Philippe Castillon and Berardo Ruffini.
\newblock A spectral characterization of geodesic balls in non-compact rank one
  symmetric spaces.
\newblock {\em Ann. Sc. Norm. Super. Pisa Cl. Sci. (5)}, 19(4):1359--1388,
  2019.

\bibitem{Ch1984}
Isaac Chavel.
\newblock {\em Eigenvalues in {R}iemannian geometry}, volume 115 of {\em Pure
  and Applied Mathematics}.
\newblock Academic Press, Inc., Orlando, FL, 1984.
\newblock Including a chapter by Burton Randol, With an appendix by Jozef
  Dodziuk.

\bibitem{Ch1993}
Isaac Chavel.
\newblock {\em Riemannian geometry---a modern introduction}, volume 108 of {\em
  Cambridge Tracts in Mathematics}.
\newblock Cambridge University Press, Cambridge, 1993.

\bibitem{Ch2001}
Isaac Chavel.
\newblock {\em Isoperimetric inequalities}, volume 145 of {\em Cambridge Tracts
  in Mathematics}.
\newblock Cambridge University Press, Cambridge, 2001.
\newblock Differential geometric and analytic perspectives.

\bibitem{Ch1969}
Jeff Cheeger.
\newblock A lower bound for the smallest eigenvalue of the {L}aplacian.
\newblock In {\em Problems in analysis ({P}apers dedicated to {S}alomon
  {B}ochner, 1969)}, pages 195--199. Princeton Univ. Press, Princeton, N. J.,
  1970.

\bibitem{ChSh2018}
Qun Chen and Jianghai Shi.
\newblock A {R}eilly inequality for the first non-zero eigenvalue of a class of
  operators on {R}iemannian manifold.
\newblock {\em Bull. Braz. Math. Soc. (N.S.)}, 49(3):481--493, 2018.

\bibitem{ChSh2022}
Qun Chen and Jianghai Shi.
\newblock Upper bounds for the first non-zero {S}teklov eigenvalue via
  anisotropic mean curvatures.
\newblock {\em Results Math.}, 77(1):Paper No. 6, 24, 2022.

\bibitem{Ch1990}
Roger Chen.
\newblock Neumann eigenvalue estimate on a compact {R}iemannian manifold.
\newblock {\em Proc. Amer. Math. Soc.}, 108(4):961--970, 1990.

\bibitem{CiGi2018}
Donato Cianci and Alexandre Girouard.
\newblock Large spectral gaps for {S}teklov eigenvalues under volume
  constraints and under localized conformal deformations.
\newblock {\em Ann. Global Anal. Geom.}, 54(4):529--539, 2018.

\bibitem{CiKaMe2019}
Donato Cianci, Mikhail Karpukhin, and Vladimir Medvedev.
\newblock On branched minimal immersions of surfaces by first eigenfunctions.
\newblock {\em Ann. Global Anal. Geom.}, 56(4):667--690, 2019.

\bibitem{CiMu1982}
Do\"{\i}na Cioranescu and Fran\c{c}ois Murat.
\newblock Un terme \'{e}trange venu d'ailleurs. {II}.
\newblock In {\em Nonlinear partial differential equations and their
  applications. {C}oll\`ege de {F}rance {S}eminar, {V}ol. {III} ({P}aris,
  1980/1981)}, volume~70 of {\em Res. Notes in Math.}, pages 154--178,
  425--426. Pitman, Boston, Mass.-London, 1982.

\bibitem{Co2017}
Bruno Colbois.
\newblock The spectrum of the {L}aplacian: a geometric approach.
\newblock In {\em Geometric and computational spectral theory}, volume 700 of
  {\em Contemp. Math.}, pages 1--40. Amer. Math. Soc., Providence, RI, 2017.

\bibitem{CoDo1994}
Bruno Colbois and Jozef Dodziuk.
\newblock Riemannian metrics with large {$\lambda_1$}.
\newblock {\em Proc. Amer. Math. Soc.}, 122(3):905--906, 1994.

\bibitem{CoDrEl2008}
Bruno Colbois, Emily~B. Dryden, and Ahmad El~Soufi.
\newblock Extremal {$G$}-invariant eigenvalues of the {L}aplacian of
  {$G$}-invariant metrics.
\newblock {\em Math. Z.}, 258(1):29--41, 2008.

\bibitem{CoDrEl2010}
Bruno Colbois, Emily~B. Dryden, and Ahmad El~Soufi.
\newblock Bounding the eigenvalues of the {L}aplace-{B}eltrami operator on
  compact submanifolds.
\newblock {\em Bull. Lond. Math. Soc.}, 42(1):96--108, 2010.

\bibitem{CoEl2003}
Bruno Colbois and Ahmad El~Soufi.
\newblock Extremal eigenvalues of the {L}aplacian in a conformal class of
  metrics: the `conformal spectrum'.
\newblock {\em Ann. Global Anal. Geom.}, 24(4):337--349, 2003.

\bibitem{CoElGi2011}
Bruno Colbois, Ahmad El~Soufi, and Alexandre Girouard.
\newblock Isoperimetric control of the {S}teklov spectrum.
\newblock {\em J. Funct. Anal.}, 261(5):1384--1399, 2011.

\bibitem{CoElGi2019}
Bruno Colbois, Ahmad El~Soufi, and Alexandre Girouard.
\newblock Compact manifolds with fixed boundary and large {S}teklov
  eigenvalues.
\newblock {\em Proc. Amer. Math. Soc.}, 147(9):3813--3827, 2019.

\bibitem{CoGi2014}
Bruno Colbois and Alexandre Girouard.
\newblock The spectral gap of graphs and {S}teklov eigenvalues on surfaces.
\newblock {\em Electron. Res. Announc. Math. Sci.}, 21:19--27, 2014.

\bibitem{CoGi2022}
Bruno Colbois and Alexandre Girouard.
\newblock Metric upper bounds for {S}teklov and {L}aplace eigenvalues, 2022.
\newblock arXiv:2108.03101.

\bibitem{CoGiGi2019}
Bruno Colbois, Alexandre Girouard, and Katie Gittins.
\newblock Steklov eigenvalues of submanifolds with prescribed boundary in
  {E}uclidean space.
\newblock {\em J. Geom. Anal.}, 29(2):1811--1834, 2019.

\bibitem{CoGiHa2020}
Bruno Colbois, Alexandre Girouard, and Asma Hassannezhad.
\newblock The {S}teklov and {L}aplacian spectra of {R}iemannian manifolds with
  boundary.
\newblock {\em J. Funct. Anal.}, 278(6):108409, 38, 2020.

\bibitem{CoGiMe2020}
Bruno Colbois, Alexandre Girouard, and Antoine M\'{e}tras.
\newblock Hypersurfaces with prescribed boundary and small {S}teklov
  eigenvalues.
\newblock {\em Canad. Math. Bull.}, 63(1):46--57, 2020.

\bibitem{CoGiRa2018}
Bruno Colbois, Alexandre Girouard, and Binoy Raveendran.
\newblock The {S}teklov spectrum and coarse discretizations of manifolds with
  boundary.
\newblock {\em Pure Appl. Math. Q.}, 14(2):357--392, 2018.

\bibitem{CoGi2021}
Bruno Colbois and Katie Gittins.
\newblock Upper bounds for {S}teklov eigenvalues of submanifolds in {E}uclidean
  space via the intersection index.
\newblock {\em Differential Geom. Appl.}, 78:Paper No. 101777, 21, 2021.

\bibitem{CoMa2008}
Bruno Colbois and Daniel Maerten.
\newblock Eigenvalues estimate for the {N}eumann problem of a bounded domain.
\newblock {\em J. Geom. Anal.}, 18(4):1022--1032, 2008.

\bibitem{CoVe2021}
Bruno Colbois and Sheela Verma.
\newblock Sharp {S}teklov upper bound for submanifolds of revolution.
\newblock {\em J. Geom. Anal.}, 31(11):11214--11225, 2021.

\bibitem{Cr1980}
Christopher~B. Croke.
\newblock Some isoperimetric inequalities and eigenvalue estimates.
\newblock {\em Ann. Sci. \'{E}cole Norm. Sup. (4)}, 13(4):419--435, 1980.

\bibitem{DaHeNi2021}
Thierry Daud\'e, Bernard Helffer, and Fran\c{c}ois Nicoleau.
\newblock Exponential localization of {S}teklov eigenfunctions on warped
  product manifolds: the flea on the elephant phenomenon, 2021.
\newblock Ann. Math. Québec.

\bibitem{DaKaNi2021}
Thierry Daud\'{e}, Niky Kamran, and Fran\c{c}ois Nicoleau.
\newblock Stability in the inverse {S}teklov problem on warped product
  {R}iemannian manifolds.
\newblock {\em J. Geom. Anal.}, 31(2):1821--1854, 2021.

\bibitem{De2021_2}
Stefano Decio.
\newblock Hausdorff measure bounds for nodal sets of {S}teklov eigenfunctions.
\newblock To appear at Anal. PDE; arXiv:2104.10275.

\bibitem{De2021}
Stefano Decio.
\newblock Nodal {S}ets of {S}teklov {E}igenfunctions near the {B}oundary:
  {I}nner {R}adius {E}stimates.
\newblock {\em Int. Math. Res. Not. IMRN}, (21):16709--16729, 2022.

\bibitem{Di2004}
Bodo Dittmar.
\newblock Sums of reciprocal {S}tekloff eigenvalues.
\newblock {\em Math. Nachr.}, 268:44--49, 2004.

\bibitem{DoPa1976}
Jozef Dodziuk and Vijay~Kumar Patodi.
\newblock Riemannian structures and triangulations of manifolds.
\newblock {\em J. Indian Math. Soc. (N.S.)}, 40(1-4):1--52 (1977), 1976.

\bibitem{Do1976}
Harold Donnelly.
\newblock Spectrum and the fixed point sets of isometries. {I}.
\newblock {\em Math. Ann.}, 224(2):161--170, 1976.

\bibitem{DGGW2008}
Emily~B. Dryden, Carolyn~S. Gordon, Sarah~J. Greenwald, and David~L. Webb.
\newblock Asymptotic expansion of the heat kernel for orbifolds.
\newblock {\em Michigan Math. J.}, 56(1):205--238, 2008.

\bibitem{DGGW2017}
Emily~B. Dryden, Carolyn~S. Gordon, Sarah~J. Greenwald, and David~L. Webb.
\newblock Erratum to ``{A}symptotic expansion of the heat kernel for
  orbifolds'' [{MR}2433665].
\newblock {\em Michigan Math. J.}, 66(1):221--222, 2017.

\bibitem{DuMaWaXi2021}
Feng Du, Jing Mao, Qiaoling Wang, and Changyu Xia.
\newblock Estimates for eigenvalues of weighted {L}aplacian and weighted
  {$p$}-{L}aplacian.
\newblock {\em Hiroshima Math. J.}, 51(3):335--353, 2021.

\bibitem{DuGu1975}
Johannes~Jisse Duistermaat and Victor Guillemin.
\newblock Spectrum of positive elliptic operators and periodic
  bicharacteristics.
\newblock {\em Invent. Math.}, 29:39--79, 1975.

\bibitem{Ed1993}
Julian Edward.
\newblock An inverse spectral result for the {N}eumann operator on planar
  domains.
\newblock {\em Journal of functional analysis}, 111(2):312--322, 1993.

\bibitem{Ed1993a}
Julian Edward.
\newblock Pre-compactness of isospectral sets for the {N}eumann operator on
  planar domains.
\newblock {\em Communications in partial differential equations},
  18(7-8):1249--1270, 1993.

\bibitem{EeSa1964}
James Eells, Jr. and J.~H. Sampson.
\newblock Harmonic mappings of {R}iemannian manifolds.
\newblock {\em Amer. J. Math.}, 86:109--160, 1964.

\bibitem{Ej1979}
Norio Ejiri.
\newblock A construction of nonflat, compact irreducible {R}iemannian manifolds
  which are isospectral but not isometric.
\newblock {\em Math. Z.}, 168(3):207--212, 1979.

\bibitem{ElGiJa2006}
Ahmad El~Soufi, Hector Giacomini, and Mustapha Jazar.
\newblock A unique extremal metric for the least eigenvalue of the {L}aplacian
  on the {K}lein bottle.
\newblock {\em Duke Math. J.}, 135(1):181--202, 2006.

\bibitem{ElIl2000}
Ahmad El~Soufi and Sa\"{\i}d Ilias.
\newblock Riemannian manifolds admitting isometric immersions by their first
  eigenfunctions.
\newblock {\em Pacific J. Math.}, 195(1):91--99, 2000.

\bibitem{ElIl2008}
Ahmad El~Soufi and Sa\"{\i}d Ilias.
\newblock Laplacian eigenvalue functionals and metric deformations on compact
  manifolds.
\newblock {\em J. Geom. Phys.}, 58(1):89--104, 2008.

\bibitem{Es1997}
Jos\'{e}~F. Escobar.
\newblock The geometry of the first non-zero {S}tekloff eigenvalue.
\newblock {\em J. Funct. Anal.}, 150(2):544--556, 1997.

\bibitem{Es1999}
Jos\'{e}~F. Escobar.
\newblock An isoperimetric inequality and the first {S}teklov eigenvalue.
\newblock {\em J. Funct. Anal.}, 165(1):101--116, 1999.

\bibitem{FaTaYu2015}
Xu-Qian Fan, Luen-Fai Tam, and Chengjie Yu.
\newblock Extremal problems for {S}teklov eigenvalues on annuli.
\newblock {\em Calc. Var. Partial Differential Equations}, 54(1):1043--1059,
  2015.

\bibitem{FoPaZo2017}
Abigail Folha, Frank Pacard, and Tatiana Zolotareva.
\newblock Free boundary minimal surfaces in the unit 3-ball.
\newblock {\em Manuscripta Math.}, 154(3-4):359--409, 2017.

\bibitem{FoKu1983}
David~W Fox and James~R Kuttler.
\newblock Sloshing frequencies.
\newblock {\em Zeitschrift f{\"u}r angewandte Mathematik und Physik ZAMP},
  34(5):668--696, 1983.

\bibitem{Fr2020}
Ailana Fraser.
\newblock Extremal eigenvalue problems and free boundary minimal surfaces in
  the ball.
\newblock In {\em Geometric analysis}, volume 2263 of {\em Lecture Notes in
  Math.}, pages 1--40. Springer, Cham, [2020] \copyright 2020.

\bibitem{FrLi2014}
Ailana Fraser and Martin Man-chun Li.
\newblock Compactness of the space of embedded minimal surfaces with free
  boundary in three-manifolds with nonnegative {R}icci curvature and convex
  boundary.
\newblock {\em J. Differential Geom.}, 96(2):183--200, 2014.

\bibitem{FrSa2021}
Ailana Fraser and Pam Sargent.
\newblock Existence and classification of {$\mathbb S^1$}-invariant free
  boundary minimal annuli and {M}\"{o}bius bands in {$\mathbb B^n$}.
\newblock {\em J. Geom. Anal.}, 31(3):2703--2725, 2021.

\bibitem{FrSc2011}
Ailana Fraser and Richard Schoen.
\newblock The first {S}teklov eigenvalue, conformal geometry, and minimal
  surfaces.
\newblock {\em Adv. Math.}, 226(5):4011--4030, 2011.

\bibitem{FrSc2013}
Ailana Fraser and Richard Schoen.
\newblock Minimal surfaces and eigenvalue problems.
\newblock In {\em Geometric analysis, mathematical relativity, and nonlinear
  partial differential equations}, volume 599 of {\em Contemp. Math.}, pages
  105--121. Amer. Math. Soc., Providence, RI, 2013.

\bibitem{FrSc2015}
Ailana Fraser and Richard Schoen.
\newblock Uniqueness theorems for free boundary minimal disks in space forms.
\newblock {\em Int. Math. Res. Not. IMRN}, (17):8268--8274, 2015.

\bibitem{FrSc2016}
Ailana Fraser and Richard Schoen.
\newblock Sharp eigenvalue bounds and minimal surfaces in the ball.
\newblock {\em Invent. Math.}, 203(3):823--890, 2016.

\bibitem{FrSc2019}
Ailana Fraser and Richard Schoen.
\newblock Shape optimization for the {S}teklov problem in higher dimensions.
\newblock {\em Adv. Math.}, 348:146--162, 2019.

\bibitem{FrSc2020}
Ailana Fraser and Richard Schoen.
\newblock Some results on higher eigenvalue optimization.
\newblock {\em Calc. Var. Partial Differential Equations}, 59(5):Paper No. 151,
  22, 2020.

\bibitem{FrMi2020}
Uta Freiberg and Lenon Minorics.
\newblock Eigenvalue approximation for {K}rein-{F}eller-operators.
\newblock In {\em Analysis, probability and mathematical physics on fractals},
  volume~5 of {\em Fractals Dyn. Math. Sci. Arts Theory Appl.}, pages 363--384.
  World Sci. Publ., Hackensack, NJ, [2020] \copyright 2020.

\bibitem{Fr1991}
Leonid Friedlander.
\newblock Some inequalities between {D}irichlet and {N}eumann eigenvalues.
\newblock {\em Arch. Rat. Mech. Anal.}, 116:153–160, 1991.

\bibitem{Ft2022}
Ilias Ftouhi.
\newblock Where to place a spherical obstacle so as to maximize the first
  nonzero {S}teklov eigenvalue.
\newblock {\em ESAIM Control Optim. Calc. Var.}, 28:Paper No. 6, 21, 2022.

\bibitem{GaTo2019}
Jeffrey Galkowski and John~A. Toth.
\newblock Pointwise bounds for {S}teklov eigenfunctions.
\newblock {\em J. Geom. Anal.}, 29(1):142--193, 2019.

\bibitem{GLPT2020}
Nunzia Gavitone, Domenico~Angelo La~Manna, Gloria Paoli, and Leonardo Trani.
\newblock A quantitative {W}einstock inequality for convex sets.
\newblock {\em Calc. Var. Partial Differential Equations}, 59(1):Paper No. 2,
  20, 2020.

\bibitem{GPPS2021}
Nunzia Gavitone, Gloria Paoli, Gianpaolo Piscitelli, and Rossano Sannipoli.
\newblock An isoperimetric inequality for the first {S}teklov--{D}irichlet
  {L}aplacian eigenvalue of convex sets with a spherical hole.
\newblock {\em Pacific J. Math.}, 320(2):241--259, 2022.

\bibitem{GaPi2021}
Nunzia Gavitone and Gianpaolo Piscitelli.
\newblock A monotonicity result for the first {S}teklov-{D}irichlet {L}aplacian
  eigenvalue, 2021.
\newblock To appear at Rev. Mat. Compl.; arXiv:2111.03385.

\bibitem{Gendron2020}
Germain Gendron.
\newblock Uniqueness results in the inverse spectral {S}teklov problem.
\newblock {\em Inverse Problems and Imaging}, 14(4):631--664, 2020.

\bibitem{GeRoFo2019}
Bogdan Georgiev and Guillaume Roy-Fortin.
\newblock Polynomial upper bound on interior steklov nodal sets.
\newblock {\em Journal of Spectral Theory}, 9(3):897--919, 2019.

\bibitem{Gi2003}
Peter Gilkey.
\newblock {\em Asymptotic formulae in spectral geometry}.
\newblock Studies in Advanced Mathematics. Chapman and Hall/CRC, Boca Raton,
  FL, 2004.

\bibitem{Gi2009}
Alexandre Girouard.
\newblock Fundamental tone, concentration of density, and conformal
  degeneration on surfaces.
\newblock {\em Canad. J. Math.}, 61(3):548--565, 2009.

\bibitem{GiHeLa2021}
Alexandre Girouard, Antoine Henrot, and Jean Lagac\'{e}.
\newblock From {S}teklov to {N}eumann via homogenisation.
\newblock {\em Arch. Ration. Mech. Anal.}, 239(2):981--1023, 2021.

\bibitem{GiKaLa2021}
Alexandre Girouard, Mikhail Karpukhin, and Jean Lagac\'{e}.
\newblock Continuity of eigenvalues and shape optimisation for {L}aplace and
  {S}teklov problems.
\newblock {\em Geom. Funct. Anal.}, 31(3):513--561, 2021.

\bibitem{GKLP2021}
Alexandre Girouard, Mikhail Karpukhin, Michael Levitin, and Iosif Polterovich.
\newblock The {D}irichlet-to-{N}eumann map, the boundary {L}aplacian, and
  {H}\"{o}rmander's rediscovered manuscript.
\newblock {\em J. Spectr. Theory}, 12(1):195--225, 2022.

\bibitem{GiLa2021}
Alexandre Girouard and Jean Lagac\'{e}.
\newblock Large {S}teklov eigenvalues via homogenisation on manifolds.
\newblock {\em Invent. Math.}, 226(3):1011--1056, 2021.

\bibitem{GPPS2014}
Alexandre Girouard, Leonid Parnovski, Iosif Polterovich, and David~A. Sher.
\newblock The {S}teklov spectrum of surfaces: asymptotics and invariants.
\newblock {\em Math. Proc. Cambridge Philos. Soc.}, 157(3):379--389, 2014.

\bibitem{GiPo2010}
Alexandre Girouard and Iosif Polterovich.
\newblock On the {H}ersch-{P}ayne-{S}chiffer estimates for the eigenvalues of
  the {S}teklov problem.
\newblock {\em Funktsional. Anal. i Prilozhen.}, 44(2):33--47, 2010.

\bibitem{GiPo2012}
Alexandre Girouard and Iosif Polterovich.
\newblock Upper bounds for {S}teklov eigenvalues on surfaces.
\newblock {\em Electron. Res. Announc. Math. Sci.}, 19:77--85, 2012.

\bibitem{GiPo2017}
Alexandre Girouard and Iosif Polterovich.
\newblock Spectral geometry of the {S}teklov problem (survey article).
\newblock {\em J. Spectr. Theory}, 7(2):321--359, 2017.

\bibitem{Go1994}
Carolyn Gordon.
\newblock Isospectral closed {R}iemannian manifolds which are not locally
  isometric. {II}.
\newblock In {\em Geometry of the spectrum ({S}eattle, {WA}, 1993)}, volume 173
  of {\em Contemp. Math.}, pages 121--131. Amer. Math. Soc., Providence, RI,
  1994.

\bibitem{Go2001}
Carolyn Gordon.
\newblock Isospectral deformations of metrics on spheres.
\newblock {\em Invent. Math.}, 145(2):317--331, 2001.

\bibitem{GoHeWe2021}
Carolyn Gordon, Peter Herbrich, and David Webb.
\newblock Steklov and {R}obin isospectral manifolds.
\newblock {\em J. Spectr. Theory}, 11(1):39--61, 2021.

\bibitem{GoSz2002}
Carolyn Gordon and Zoltan~I. Szabo.
\newblock Isospectral deformations of negatively curved {R}iemannian manifolds
  with boundary which are not locally isometric.
\newblock {\em Duke Math. J.}, 113(2):355--383, 2002.

\bibitem{GoWeWo1992}
Carolyn Gordon, David Webb, and Scott Wolpert.
\newblock Isospectral plane domains and surfaces via {R}iemannian orbifolds.
\newblock {\em Invent. Math.}, 110(1):1--22, 1992.

\bibitem{GoMc2006}
Ruth Gornet and Jeffrey McGowan.
\newblock Lens spaces, isospectral on forms but not on functions.
\newblock {\em LMS J. Comput. Math.}, 9:270--286, 2006.

\bibitem{GoMc2019}
Ruth Gornet and Jeffrey McGowan.
\newblock Lens spaces isospectral on forms but not on functions, 2019.
\newblock arXiv:1906.07787.

\bibitem{GrNaSi2016}
Alexander Grigor'yan, Nikolai Nadirashvili, and Yannick Sire.
\newblock A lower bound for the number of negative eigenvalues of
  {S}chr\"{o}dinger operators.
\newblock {\em J. Differential Geom.}, 102(3):395--408, 2016.

\bibitem{GrNeYa2004}
Alexander Grigor'yan, Yuri Netrusov, and Shing-Tung Yau.
\newblock Eigenvalues of elliptic operators and geometric applications.
\newblock In {\em Surveys in differential geometry. {V}ol. {IX}}, volume~9 of
  {\em Surv. Differ. Geom.}, pages 147--217. Int. Press, Somerville, MA, 2004.

\bibitem{HaHu2020}
Wen Han and Bobo Hua.
\newblock Steklov eigenvalue problem on subgraphs of integer lattices, 2020.
\newblock To appear in Comm. Anal. Geom.;arXiv:1902.05831.

\bibitem{Ha2011}
Asma Hassannezhad.
\newblock Conformal upper bounds for the eigenvalues of the {L}aplacian and
  {S}teklov problem.
\newblock {\em J. Funct. Anal.}, 261(12):3419--3436, 2011.

\bibitem{HaMi2020}
Asma Hassannezhad and Laurent Miclo.
\newblock Higher order {C}heeger inequalities for {S}teklov eigenvalues.
\newblock {\em Ann. Sci. \'{E}c. Norm. Sup\'{e}r. (4)}, 53(1):43--88, 2020.

\bibitem{HaSh2021}
Asma Hassannezhad and David Sher.
\newblock Nodal count for {D}irichlet-to-{N}eumann operators with potential,
  2021.
\newblock arXiv:2107.03370.

\bibitem{HaSi2020}
Asma Hassannezhad and Anna Siffert.
\newblock A note on {K}uttler-{S}igillito's inequalities.
\newblock {\em Ann. Math. Qu\'{e}.}, 44(1):125--147, 2020.

\bibitem{HeHu2022}
Zunwu He and Bobo Hua.
\newblock Upper bounds for the {S}teklov eigenvalues on trees.
\newblock {\em Calc. Var. Partial Differential Equations}, 61(3):Paper No. 101,
  15, 2022.

\bibitem{HeKa2022}
Bernard Helffer and Ayman Kachmar.
\newblock Semi-classical edge states for the robin laplacian.
\newblock {\em Mathematika}, 68(2):454--485, 2022.

\bibitem{He1970}
Joseph Hersch.
\newblock Quatre propri\'et\'es isop\'erim\'etriques de membranes sph\'eriques
  homog\`enes.
\newblock {\em C. R. Acad. Sci. Paris S\'er. A-B}, 270:A1645--A1648, 1970.

\bibitem{HePaSc1975}
Joseph Hersch, Lawrence~E. Payne, and Menahem~M. Schiffer.
\newblock Some inequalities for {S}tekloff eigenvalues.
\newblock {\em Arch. Rational Mech. Anal.}, 57:99--114, 1975.

\bibitem{HiLu2001}
Peter Hislop and Carl Lutzer.
\newblock Spectral asymptotics of the {D}irichlet-to-{N}eumann map on multiply
  connected domains in {$\mathbb R^d$}.

\bibitem{Ho2021}
Han Hong.
\newblock Higher dimensional surgery and {S}teklov eigenvalues.
\newblock {\em J. Geom. Anal.}, 31(12):11931--11951, 2021.

\bibitem{HoLiSe2022}
Jiho Hong, Mikyoung Lim, and Dong-Hwi Seo.
\newblock On the first {S}teklov-{D}irichlet eigenvalue for eccentric annuli.
\newblock {\em Ann. Mat. Pura Appl. (4)}, 201(2):769--799, 2022.

\bibitem{HuLaNg2006}
Jiaxin Hu, Ka-Sing Lau, and Sze-Man Ngai.
\newblock Laplace operators related to self-similar measures on {$\R^d$}.
\newblock {\em J. Funct. Anal.}, 239(2):542--565, 2006.

\bibitem{HuHuWa2017}
Bobo Hua, Yan Huang, and Zuoqin Wang.
\newblock First eigenvalue estimates of {D}irichlet-to-{N}eumann operators on
  graphs.
\newblock {\em Calc. Var. Partial Differential Equations}, 56(6):Paper No. 178,
  21, 2017.

\bibitem{HuHuWa2018}
Bobo Hua, Yan Huang, and Zuoqin Wang.
\newblock Cheeger estimates of {D}irichlet-to-{N}eumann operators on infinite
  subgraphs of graphs.
\newblock {\em J. Spectr. Theory}, 12(3):1079--1108, 2022.

\bibitem{Ik1989}
Akira Ikeda.
\newblock Riemannian manifolds {$p$}-isospectral but not {$(p+1)$}-isospectral.
\newblock In {\em Geometry of manifolds ({M}atsumoto, 1988)}, volume~8 of {\em
  Perspect. Math.}, pages 383--417. Academic Press, Boston, MA, 1989.

\bibitem{IlMa2011}
Sa\"{\i}d Ilias and Ola Makhoul.
\newblock A {R}eilly inequality for the first {S}teklov eigenvalue.
\newblock {\em Differential Geom. Appl.}, 29(5):699--708, 2011.

\bibitem{Iv1980}
Victor Ivrii.
\newblock Second term of the spectral asymptotic expansion of the
  {L}aplace-{B}eltrami operator on manifolds with boundary.
\newblock {\em Functional Analysis and Its Applications}, 14(2):98--106, 1980.

\bibitem{JLNNP2005}
Dmitry Jakobson, Michael Levitin, Nikolai Nadirashvili, Nilima Nigam, and Iosif
  Polterovich.
\newblock How large can the first eigenvalue be on a surface of genus two?
\newblock {\em Int. Math. Res. Not.}, (63):3967--3985, 2005.

\bibitem{JaNaPo2006}
Dmitry Jakobson, Nikolai Nadirashvili, and Iosif Polterovich.
\newblock Extremal metric for the first eigenvalue on a {K}lein bottle.
\newblock {\em Canad. J. Math.}, 58(2):381--400, 2006.

\bibitem{Ja2014}
Pierre Jammes.
\newblock Prescription du spectre de {S}teklov dans une classe conforme.
\newblock {\em Anal. PDE}, 7(3):529--549, 2014.

\bibitem{Ja2015}
Pierre Jammes.
\newblock Une in\'egalit\'e de {C}heeger pour le spectre de {S}teklov.
\newblock {\em Ann. Inst. Fourier (Grenoble)}, 65(3):1381--1385, 2015.

\bibitem{JoSh2018}
Alexandre Jollivet and Vladimir Sharafutdinov.
\newblock An inequality for the {S}teklov spectral zeta function of a planar
  domain.
\newblock {\em J. Spectr. Theory}, 8(1):271--296, 2018.

\bibitem{JoSh2018_2}
Alexandre Jollivet and Vladimir Sharafutdinov.
\newblock Steklov zeta-invariants and a compactness theorem for isospectral
  families of planar domains.
\newblock {\em J. Funct. Anal.}, 275(7):1712--1755, 2018.

\bibitem{JoLi2005}
Mark~S. Joshi and William R.~B. Lionheart.
\newblock An inverse boundary value problem for harmonic differential forms.
\newblock {\em Asymptot. Anal.}, 41(2):93--106, 2005.

\bibitem{Ka1985}
Masahiko Kanai.
\newblock Rough isometries, and combinatorial approximations of geometries of
  noncompact {R}iemannian manifolds.
\newblock {\em J. Math. Soc. Japan}, 37(3):391--413, 1985.

\bibitem{KaLi2021}
Nikolaos Kapouleas and Martin Man-chun Li.
\newblock Free boundary minimal surfaces in the unit three-ball via
  desingularization of the critical catenoid and the equatorial disc.
\newblock {\em J. Reine Angew. Math.}, 776:201--254, 2021.

\bibitem{Ka2017}
Mikhail Karpukhin.
\newblock Bounds between {L}aplace and {S}teklov eigenvalues on nonnegatively
  curved manifolds.
\newblock {\em Electron. Res. Announc. Math. Sci.}, 24:100--109, 2017.

\bibitem{Ka2019}
Mikhail Karpukhin.
\newblock The {S}teklov problem on differential forms.
\newblock {\em Canad. J. Math.}, 71(2):417--435, 2019.

\bibitem{KaKoPo2014}
Mikhail Karpukhin, Gerasim Kokarev, and Iosif Polterovich.
\newblock Multiplicity bounds for {S}teklov eigenvalues on {R}iemannian
  surfaces.
\newblock {\em Ann. Inst. Fourier (Grenoble)}, 64(6):2481--2502, 2014.

\bibitem{KaLa2022}
Mikhail Karpukhin and Jean Lagac\'e.
\newblock Flexibility of {S}teklov eigenvalues via boundary homogenisation,
  2022.
\newblock Ann. Math. Québec, published online.

\bibitem{KLP2022}
Mikhail Karpukhin, Jean Lagac{\'e}, and Iosif Polterovich.
\newblock Weyl's law for the {S}teklov problem on surfaces with rough boundary.
\newblock {\em To appear at ARMA; arXiv:2204.05294}, 2022.

\bibitem{KaMe2021}
Mikhail Karpukhin and Antoine M\'{e}tras.
\newblock Laplace and {S}teklov extremal metrics via {$n$}-harmonic maps.
\newblock {\em J. Geom. Anal.}, 32(5):Paper No. 154, 36, 2022.

\bibitem{KNPP2021}
Mikhail Karpukhin, Nikolai Nadirashvili, Alexei~V. Penskoi, and Iosif
  Polterovich.
\newblock An isoperimetric inequality for {L}aplace eigenvalues on the sphere.
\newblock {\em J. Differential Geom.}, 118(2):313--333, 2021.

\bibitem{KNPP2020}
Mikhail Karpukhin, Nikolai Nadirashvili, Alexei~V. Penskoi, and Iosif
  Polterovich.
\newblock Conformally maximal metrics for {L}aplace eigenvalues on surfaces.
\newblock In {\em Surveys in differential geometry 2019. {D}ifferential
  geometry, {C}alabi-{Y}au theory, and general relativity. {P}art 2}, volume~24
  of {\em Surv. Differ. Geom.}, pages 205--256. Int. Press, Boston, MA, [2022]
  \copyright 2022.

\bibitem{KNPS2021}
Mikhail Karpukhin, Micka\"{e}l Nahon, Iosif Polterovich, and Daniel Stern.
\newblock Stability of isoperimetric inequalities for {L}aplace eigenvalues on
  surfaces, 2021.
\newblock arXiv:2106.15043.

\bibitem{KaSt2021}
Mikhail Karpukhin and Daniel Stern.
\newblock From {S}teklov to {L}aplace: free boundary minimal surfaces with many
  boundary components, 2021.
\newblock To appear at Duke Math. J.; arXiv:2109.11029.

\bibitem{KaSt2022}
Mikhail Karpukhin and Daniel Stern.
\newblock Existence of harmonic maps and eigenvalue optimization in higher
  dimensions, 2022.
\newblock arXiv:2207.13635.

\bibitem{KaSt2020}
Mikhail Karpukhin and Daniel~L. Stern.
\newblock Min-max harmonic maps and a new characterization of conformal
  eigenvalues, 2020.
\newblock arXiv:2004.04086.

\bibitem{KeNi2022}
Marc Kesseböhmer and Aljoscha Niemann.
\newblock Spectral dimensions of {K}rein--{F}eller operators in higher
  dimensions, 2022.
\newblock arXiv.2202.05247.

\bibitem{Ke2016}
Daniel Ketover.
\newblock Free boundary minimal surfaces of unbounded genus, 2016.
\newblock arXiv:1612.08691.

\bibitem{Ko2014}
Gerasim Kokarev.
\newblock Variational aspects of {L}aplace eigenvalues on {R}iemannian
  surfaces.
\newblock {\em Adv. Math.}, 258:191--239, 2014.

\bibitem{Ko2020}
Gerasim Kokarev.
\newblock Conformal volume and eigenvalue problems.
\newblock {\em Indiana Univ. Math. J.}, 69(6):1975--2003, 2020.

\bibitem{Ko1993}
Nicholas Korevaar.
\newblock Upper bounds for eigenvalues of conformal metrics.
\newblock {\em J. Differential Geom.}, 37(1):73--93, 1993.

\bibitem{KLPPS2021}
Stanislav Krymski, Michael Levitin, Leonid Parnovski, Iosif Polterovich, and
  David~A. Sher.
\newblock Inverse {S}teklov spectral problem for curvilinear polygons.
\newblock {\em Int. Math. Res. Not. IMRN}, (1):1--37, 2021.

\bibitem{KuMc2022}
Robert Kusner and Peter McGrath.
\newblock On {S}teklov eigenspaces for free boundary minimal surfaces in the
  unit ball, 2022.
\newblock To appear at Amer. J. Math.;arXiv:2011.06884.

\bibitem{KuSi1969}
J.~R. Kuttler and V.~G. Sigillito.
\newblock An inequality of a stekloff eigenvalue by the method of defect.
\newblock {\em Proc. Amer. Math. Soc.}, 20:357--360, 1969.

\bibitem{KuKu2014}
Nikolay Kuznetsov, Tadeusz Kulczycki, Mateusz Kwa\'{s}nicki, Alexander Nazarov,
  Sergey Poborchi, Iosif Polterovich, and Bart\l~omiej Siudeja.
\newblock The legacy of {V}ladimir {A}ndreevich {S}teklov.
\newblock {\em Notices Amer. Math. Soc.}, 61(1):9--22, 2014.

\bibitem{Kw2016}
Kwok-Kun Kwong.
\newblock Some sharp {H}odge {L}aplacian and {S}teklov eigenvalue estimates for
  differential forms.
\newblock {\em Calc. Var. Partial Differential Equations}, 55(2):Art. 38, 14,
  2016.

\bibitem{La2019}
Emilio~A. Lauret.
\newblock A computational study on lens spaces isospectral on forms.
\newblock {\em Exp. Math.}, 30(2):268--282, 2021.

\bibitem{LeYe2021}
Jaehoon Lee and Eungbeom Yeon.
\newblock A new approach to the {F}raser-{L}i conjecture with the {W}eierstrass
  representation formula.
\newblock {\em Proc. Amer. Math. Soc.}, 149(12):5331--5345, 2021.

\bibitem{LGOT2014}
James~R. Lee, Shayan~Oveis Gharan, and Luca Trevisan.
\newblock Multiway spectral partitioning and higher-order {C}heeger
  inequalities.
\newblock {\em J. ACM}, 61(6):Art. 37, 30, 2014.

\bibitem{LeUh1989}
John~M. Lee and Gunther Uhlmann.
\newblock Determining anisotropic real-analytic conductivities by boundary
  measurements.
\newblock {\em Comm. Pure Appl. Math.}, 42(8):1097--1112, 1989.

\bibitem{LPPS2017}
Michael Levitin, Leonid Parnovski, Iosif Polterovich, and David~A. Sher.
\newblock Sloshing, {S}teklov and corners: asymptotics of sloshing eigenvalues.
\newblock {\em J. Anal. Math.}, 146(1):65--125, 2022.

\bibitem{LPPS2019}
Michael Levitin, Leonid Parnovski, Iosif Polterovich, and David~A. Sher.
\newblock Sloshing, {S}teklov and corners: asymptotics of {S}teklov eigenvalues
  for curvilinear polygons.
\newblock {\em Proc. Lond. Math. Soc. (3)}, 125(3):359--487, 2022.

\bibitem{Li2020}
Martin Li.
\newblock Free boundary minimal surfaces in the unit ball : recent advances and
  open questions.
\newblock In {\em First International Congress of Chinese Mathematicians},
  AMS/IP Studies in Advanced Mathematics. Int. Press, Boston, MA, 2019.

\bibitem{LiYa1979}
Peter Li and Shing-Tung Yau.
\newblock Estimates of eigenvalues of a compact {R}iemannian manifold.
\newblock In {\em Geometry of the {L}aplace operator ({P}roc. {S}ympos. {P}ure
  {M}ath., {U}niv. {H}awaii, {H}onolulu, {H}awaii, 1979)}, Proc. Sympos. Pure
  Math., XXXVI, pages 205--239. Amer. Math. Soc., Providence, R.I., 1980.

\bibitem{LiYa1982}
Peter Li and Shing-Tung Yau.
\newblock A new conformal invariant and its applications to the {W}illmore
  conjecture and the first eigenvalue of compact surfaces.
\newblock {\em Invent. Math.}, 69(2):269--291, 1982.

\bibitem{LiWaWu2020}
Xiaolong Li, Kui Wang, and Haotian Wu.
\newblock An upper bound for the first nonzero {S}teklov eigenvalue, 2020.
\newblock arXiv:2003.03093.

\bibitem{LiZh2020}
Fanghua Lin and Jiuyi Zhu.
\newblock Upper bounds of nodal sets for eigenfunctions of eigenvalue problems.
\newblock {\em Math. Ann.}, 382(3-4):1957--1984, 2022.

\bibitem{Lo1996}
Joachim Lohkamp.
\newblock Discontinuity of geometric expansions.
\newblock {\em Comment. Math. Helv.}, 71(2):213--228, 1996.

\bibitem{MaSh2015}
Eugene~Gennad'evich Mal’kovich and Vladimir~Al'tafovich Sharafutdinov.
\newblock Zeta-invariants of the {S}teklov spectrum of a planar domain.
\newblock {\em Siberian Mathematical Journal}, 56(4):678--698, 2015.

\bibitem{Ma2005}
Tatiana Mantuano.
\newblock Discretization of compact {R}iemannian manifolds applied to the
  spectrum of {L}aplacian.
\newblock {\em Ann. Global Anal. Geom.}, 27(1):33--46, 2005.

\bibitem{Ma2018}
Joanie Martineau.
\newblock Concentration des fonctions propres de {S}teklov sur les composantes
  connexes de la fronti{\`e}re.
\newblock 2018.
\newblock Master's thesis, Universit\'e de Montr\'eal.

\bibitem{Ma2019}
Henrik Matthiesen.
\newblock Extremal metrics for {L}aplace eigenvalues in perturbed conformal
  classes on products.
\newblock {\em J. Geom. Anal.}, 29(3):2456--2468, 2019.

\bibitem{MaPe2020}
Henrik Matthiesen and Romain Petrides.
\newblock Free boundary minimal surfaces of any topological type in euclidean
  balls via shape optimization, 2020.
\newblock arXiv:2004.06051.

\bibitem{MaSi2019}
Henrik Matthiesen and Anna Siffert.
\newblock Handle attachment and the normalized first eigenvalue, 2019.
\newblock arXiv:1909.03105.

\bibitem{MaSi2021}
Henrik Matthiesen and Anna Siffert.
\newblock Existence of metrics maximizing the first eigenvalue on
  non-orientable surfaces.
\newblock {\em J. Spectr. Theory}, 11(3):1279--1296, 2021.

\bibitem{Ma2011}
Vladimir Maz'ya.
\newblock {\em Sobolev spaces with applications to elliptic partial
  differential equations}, volume 342 of {\em Grundlehren der mathematischen
  Wissenschaften}.
\newblock Springer, Heidelberg, augmented edition, 2011.

\bibitem{Ma1991}
Rafe Mazzeo.
\newblock Remarks on a paper of {L}. {F}riedlander concerning inequalities
  between {N}eumann and {D}irichlet eigenvalues.
\newblock {\em Int. Math. Res. Not.}, 4:41–48, 1991.

\bibitem{Mc2018}
Peter McGrath.
\newblock A characterization of the critical catenoid.
\newblock {\em Indiana Univ. Math. J.}, 67(2):889--897, 2018.

\bibitem{McSi1967}
Henry~P. McKean and Isadore Singer.
\newblock Curvature and the eigenvalues of the {L}aplacian.
\newblock {\em J. Diff. Geom.}, 1:43--69, 1967.

\bibitem{Me1983}
Richard Melrose.
\newblock Isospectral sets of drumheads are compact in $c^{\infty}$, 1983.
\newblock Preprint available at
  \url{https://math.mit.edu/~rbm/papers/isospectral/isospectral.pdf}.

\bibitem{Mi2021}
Deborah Michel.
\newblock Eigenvalue and gap estimates of isometric immersions for the
  {D}irichlet-to-{N}eumann operator acting on {$p$}-forms.
\newblock {\em C. R. Math. Acad. Sci. Paris}, 357(2):180--187, 2019.

\bibitem{Mi20222}
Marco Michetti.
\newblock {S}teklov and {N}eumann eigenvalues: inequalities, asymptotic and
  mixed problems, 2022.
\newblock Thesis.

\bibitem{Mi2015}
Laurent Miclo.
\newblock On hyperboundedness and spectrum of {M}arkov operators.
\newblock {\em Invent. Math.}, 200(1):311--343, 2015.

\bibitem{MiTa1999}
Marius Mitrea and Michael Taylor.
\newblock Boundary layer methods for {L}ipschitz domains in {R}iemannian
  manifolds.
\newblock {\em J. Funct. Anal.}, 163(2):181--251, 1999.

\bibitem{MoZh2022}
Peter Monk and Yangwen Zhang.
\newblock An {HDG} method for the {S}teklov eigenvalue problem.
\newblock {\em IMA J. Numer. Anal.}, 42(3):1929--1962, 2022.

\bibitem{Na1996}
Nicolai Nadirashvili.
\newblock Berger's isoperimetric problem and minimal immersions of surfaces.
\newblock {\em Geom. Funct. Anal.}, 6(5):877--897, 1996.

\bibitem{Na2002}
Nikolai Nadirashvili.
\newblock Isoperimetric inequality for the second eigenvalue of a sphere.
\newblock {\em J. Differential Geom.}, 61(2):335--340, 2002.

\bibitem{NaSi2015}
Nikolai Nadirashvili and Yannick Sire.
\newblock Conformal spectrum and harmonic maps.
\newblock {\em Mosc. Math. J.}, 15(1):123--140, 182, 2015.

\bibitem{NaSi2015_2}
Nikolai Nadirashvili and Yannick Sire.
\newblock Maximization of higher order eigenvalues and applications.
\newblock {\em Mosc. Math. J.}, 15(4):767--775, 2015.

\bibitem{Na1967}
Mark~Aronovich Naimark.
\newblock {\em Linear differential operators}.
\newblock Frederick Ungar Publishing Co., New York, 1967-1968.

\bibitem{NaSh2019}
Shin Nayatani and Toshihiro Shoda.
\newblock Metrics on a closed surface of genus two which maximize the first
  eigenvalue of the {L}aplacian.
\newblock {\em C. R. Math. Acad. Sci. Paris}, 357(1):84--98, 2019.

\bibitem{OPS1988a}
Brad Osgood, Ralph Phillips, and Peter Sarnak.
\newblock Compact isospectral sets of surfaces.
\newblock {\em Journal of functional analysis}, 80(1):212--234, 1988.

\bibitem{OPS1988}
Brad Osgood, Ralph Phillips, and Peter Sarnak.
\newblock Extremals of determinants of {L}aplacians.
\newblock {\em Journal of functional analysis}, 80(1):148--211, 1988.

\bibitem{OPS1989}
Brad Osgood, Ralph Phillips, and Peter Sarnak.
\newblock Moduli space, heights and isospectral sets of plane domains.
\newblock {\em Annals of Mathematics}, pages 293--362, 1989.

\bibitem{OuKaOs2021}
\'{E}douard Oudet, Chiu-Yen Kao, and Braxton Osting.
\newblock Computation of free boundary minimal surfaces {\it via} extremal
  {S}teklov eigenvalue problems.
\newblock {\em ESAIM Control Optim. Calc. Var.}, 27:Paper No. 34, 30, 2021.

\bibitem{PaPiSa2021}
Gloria Paoli, Gianpaolo Piscitelli, and Rossanno Sannipoli.
\newblock A stability result for the {S}teklov {L}aplacian eigenvalue problem
  with a spherical obstacle.
\newblock {\em Commun. Pure Appl. Anal.}, 20(1):145--158, 2021.

\bibitem{Per2019}
H\'{e}l\`ene Perrin.
\newblock Lower bounds for the first eigenvalue of the {S}teklov problem on
  graphs.
\newblock {\em Calc. Var. Partial Differential Equations}, 58(2):Paper No. 67,
  12, 2019.

\bibitem{Per2021}
H\'{e}l\`ene Perrin.
\newblock Isoperimetric upper bound for the first eigenvalue of discrete
  {S}teklov problems.
\newblock {\em J. Geom. Anal.}, 31(8):8144--8155, 2021.

\bibitem{Pe1950}
Arthur~S Peters.
\newblock The effect of a floating mat on water waves.
\newblock {\em Communications on Pure and Applied Mathematics}, 3(4):319--354,
  1950.

\bibitem{Pe2014}
Romain Petrides.
\newblock Existence and regularity of maximal metrics for the first {L}aplace
  eigenvalue on surfaces.
\newblock {\em Geom. Funct. Anal.}, 24(4):1336--1376, 2014.

\bibitem{Pe2015}
Romain Petrides.
\newblock Bornes sur des valeurs propres et m\'etriques extr\'emales,, 2015.
\newblock Ph.D. thesis,
  \url{https://webusers.imj-prg.fr/~romain.petrides/These.pdf}.

\bibitem{Pe2018}
Romain Petrides.
\newblock On the existence of metrics which maximize {L}aplace eigenvalues on
  surfaces.
\newblock {\em Int. Math. Res. Not. IMRN}, (14):4261--4355, 2018.

\bibitem{Pe2019}
Romain Petrides.
\newblock Maximizing {S}teklov eigenvalues on surfaces.
\newblock {\em J. Differential Geom.}, 113(1):95--188, 2019.

\bibitem{Pe2022}
Romain Petrides.
\newblock Maximizing one {L}aplace eigenvalue on n-dimensional manifolds, 2022.
\newblock arXiv:2211.15636.

\bibitem{PiVe2020}
Stefano Pigola and Giona Veronelli.
\newblock The smooth {R}iemannian extension problem.
\newblock {\em Ann. Sc. Norm. Super. Pisa Cl. Sci. (5)}, 20(4):1507--1551,
  2020.

\bibitem{Pl1956}
Ake Pleijel.
\newblock Remarks on {C}ourant’s nodal line theorem.
\newblock {\em Comm. Pure Appl. Math.}, 9:543–--550, 1956.

\bibitem{PoSh2015}
Iosif Polterovich and David~A Sher.
\newblock Heat invariants of the {S}teklov problem.
\newblock {\em The Journal of Geometric Analysis}, 25(2):924--950, 2015.

\bibitem{PoShTo2019}
Iosif Polterovich, David~A. Sher, and John~A. Toth.
\newblock Nodal length of {S}teklov eigenfunctions on real-analytic
  {R}iemannian surfaces.
\newblock {\em J. Reine Angew. Math.}, 754:17--47, 2019.

\bibitem{PrSt2019}
Luigi Provenzano and Joachim Stubbe.
\newblock Weyl-type bounds for {S}teklov eigenvalues.
\newblock {\em J. Spectr. Theory}, 9(1):349--377, 2019.

\bibitem{RaTa1975}
Jeffrey Rauch and Michael Taylor.
\newblock Potential and scattering theory on wildly perturbed domains.
\newblock {\em J. Functional Analysis}, 18:27--59, 1975.

\bibitem{RaSa2012}
Simon Raulot and Alessandro Savo.
\newblock On the first eigenvalue of the {D}irichlet-to-{N}eumann operator on
  forms.
\newblock {\em J. Funct. Anal.}, 262(3):889--914, 2012.

\bibitem{RaSa2014}
Simon Raulot and Alessandro Savo.
\newblock On the spectrum of the {D}irichlet-to-{N}eumann operator acting on
  forms of a {E}uclidean domain.
\newblock {\em J. Geom. Phys.}, 77:1--12, 2014.

\bibitem{RoScWe2008}
Juan~Pablo Rossetti, Dorothee Schueth, and Martin Weilandt.
\newblock Isospectral orbifolds with different maximal isotropy orders.
\newblock {\em Ann. Global Anal. Geom.}, 34(4):351--366, 2008.

\bibitem{Ro1979}
Grigori Rozenbljum.
\newblock Asymptotic behavior of the eigenvalues for some two-dimensional
  spectral problems.
\newblock {\em Probl. Mat. Anal.}, 7:188--203, 245, 1979.

\bibitem{Sc2001}
Dorothee Schueth.
\newblock Isospectral metrics on five-dimensional spheres.
\newblock {\em J. Differential Geom.}, 58(1):87--111, 2001.

\bibitem{Sc1995}
G\"{u}nter Schwarz.
\newblock {\em Hodge decomposition---a method for solving boundary value
  problems}, volume 1607 of {\em Lecture Notes in Mathematics}.
\newblock Springer-Verlag, Berlin, 1995.

\bibitem{Se2021}
Dong-Hwi Seo.
\newblock A shape optimization problem for the first mixed
  {S}teklov-{D}irichlet eigenvalue.
\newblock {\em Ann. Global Anal. Geom.}, 59(3):345--365, 2021.

\bibitem{Sh1971}
Shawky~E. Shamma.
\newblock Asymptotic behavior of {S}tekloff eigenvalues and eigenfunctions.
\newblock {\em SIAM J. Appl. Math.}, 20:482--490, 1971.

\bibitem{ShStWe2006}
Naveed Shams, Elizabeth Stanhope, and David~L. Webb.
\newblock One cannot hear orbifold isotropy type.
\newblock {\em Arch. Math. (Basel)}, 87(4):375--384, 2006.

\bibitem{ShSh2013}
Vladimir Sharafutdinov and Clayton Shonkwiler.
\newblock The complete {D}irichlet-to-{N}eumann map for differential forms.
\newblock {\em J. Geom. Anal.}, 23(4):2063--2080, 2013.

\bibitem{ShYu2016}
Yongjie Shi and Chengjie Yu.
\newblock Trace and inverse trace of {S}teklov eigenvalues.
\newblock {\em J. Differential Equations}, 261(3):2026--2040, 2016.

\bibitem{ShYu2017}
Yongjie Shi and Chengjie Yu.
\newblock Trace and inverse trace of {S}teklov eigenvalues {II}.
\newblock {\em J. Differential Equations}, 262(3):2592--2607, 2017.

\bibitem{ShiYu20222}
Yongjie Shi and Chengjie Yu.
\newblock Dirichlet-to-{N}eumann maps for differential forms on graphs and
  their eigenvalues.
\newblock {\em J. Math. Anal. Appl.}, 515(2):Paper No. 126451, 26, 2022.

\bibitem{ShiYu2022}
Yongjie Shi and Chengjie Yu.
\newblock A {L}ichnerowicz-type estimate for {S}teklov eigenvalues on graphs
  and its rigidity.
\newblock {\em Calc. Var. Partial Differential Equations}, 61(3):Paper No. 98,
  22, 2022.

\bibitem{SoWaZh2016}
Christopher~D. Sogge, Xing Wang, and Jiuyi Zhu.
\newblock Lower bounds for interior nodal sets of {S}teklov eigenfunctions.
\newblock {\em Proc. Amer. Math. Soc.}, 144(11):4715--4722, 2016.

\bibitem{Su1985}
Toshikazu Sunada.
\newblock Riemannian coverings and isospectral manifolds.
\newblock {\em Ann. of Math. (2)}, 121(1):169--186, 1985.

\bibitem{Ta1966}
Tsunero Takahashi.
\newblock Minimal immersions of {R}iemannian manifolds.
\newblock {\em J. Math. Soc. Japan}, 18:380--385, 1966.

\bibitem{Ta1994}
Hiroshi Takeuchi.
\newblock Some conformal properties of {$p$}-harmonic maps and a regularity for
  sphere-valued {$p$}-harmonic maps.
\newblock {\em J. Math. Soc. Japan}, 46(2):217--234, 1994.

\bibitem{Ts2022}
L\'{e}onard Tschanz.
\newblock Upper bounds for {S}teklov eigenvalues of subgraphs of polynomial
  growth {C}ayley graphs.
\newblock {\em Ann. Global Anal. Geom.}, 61(1):37--55, 2022.

\bibitem{Ts2023}
L\'{e}onard Tschanz.
\newblock The {S}teklov {P}roblem on {T}riangle-{T}iling {G}raphs in the
  {H}yperbolic {P}lane.
\newblock {\em J. Geom. Anal.}, 33(5):161, 2023.

\bibitem{uhlenbeck1976}
Karen Uhlenbeck.
\newblock Generic properties of eigenfunctions.
\newblock {\em American Journal of Mathematics}, 98(4):1059--1078, 1976.

\bibitem{Ve1984}
Gregory Verchota.
\newblock Layer potentials and regularity for the {D}irichlet problem for
  {L}aplace's equation in {L}ipschitz domains.
\newblock {\em J. Funct. Anal.}, 59(3):572--611, 1984.

\bibitem{Ve2018}
Sheela Verma.
\newblock Bounds for the {S}teklov eigenvalues.
\newblock {\em Arch. Math. (Basel)}, 111(6):657--668, 2018.

\bibitem{SaVe2020}
Sheela Verma and G.~Santhanam.
\newblock On eigenvalue problems related to the {L}aplacian in a class of
  doubly connected domains.
\newblock {\em Monatsh. Math.}, 193(4):879--899, 2020.

\bibitem{BeFr2005}
Joachim von Below and Gilles Fran\c{c}ois.
\newblock Spectral asymptotics for the {L}aplacian under an eigenvalue
  dependent boundary condition.
\newblock {\em Bull. Belg. Math. Soc. Simon Stevin}, 12(4):505--519, 2005.

\bibitem{wang2022gen}
Lihan Wang.
\newblock Generic properties of {S}teklov eigenfunctions.
\newblock {\em Transactions of the American Mathematical Society},
  375(11):8241--8255, 2022.

\bibitem{WaXi2009}
Qiaoling Wang and Changyu Xia.
\newblock Sharp bounds for the first non-zero {S}tekloff eigenvalues.
\newblock {\em J. Funct. Anal.}, 257(8):2635--2644, 2009.

\bibitem{WaZh2015}
Xing Wang and Jiuyi Zhu.
\newblock A lower bound for nodal sets of {S}teklov eigenfunctions.
\newblock {\em Math. Res. Lett.}, 22:1243--1253, 2015.

\bibitem{We1954}
Robert Weinstock.
\newblock Inequalities for a classical eigenvalue problem.
\newblock {\em J. Rational Mech. Anal.}, 3:745--753, 1954.

\bibitem{We1911}
Hermann Weyl.
\newblock {\"U}ber die asymptotische verteilung der eigenwerte.
\newblock {\em Nachrichten von der Gesellschaft der Wissenschaften zu
  G{\"o}ttingen, Mathematisch-Physikalische Klasse}, 1911:110--117, 1911.

\bibitem{XiXi2019}
Chao Xia and Changwei Xiong.
\newblock Escobar's conjecture on a sharp lower bound for the first nonzero
  {S}teklov eigenvalue.
\newblock To appear in Peking Math. J.;arXiv:1907.07340.

\bibitem{Xi2018}
Changwei Xiong.
\newblock Comparison of {S}teklov eigenvalues on a domain and {L}aplacian
  eigenvalues on its boundary in {R}iemannian manifolds.
\newblock {\em J. Funct. Anal.}, 275(12):3245--3258, 2018.

\bibitem{Xi2021}
Changwei Xiong.
\newblock Optimal estimates for {S}teklov eigenvalue gaps and ratios on warped
  product manifolds.
\newblock {\em Int. Math. Res. Not. IMRN}, (22):16938--16962, 2021.

\bibitem{Xi2022}
Changwei Xiong.
\newblock On the spectra of three {S}teklov eigenvalue problems on warped
  product manifolds.
\newblock {\em J. Geom. Anal.}, 32(5):Paper No. 153, 35, 2022.

\bibitem{YaYu2017_2}
Liangwei Yang and Chengjie Yu.
\newblock Estimates for higher {S}teklov eigenvalues.
\newblock {\em J. Math. Phys.}, 58(2):021504, 9, 2017.

\bibitem{YaYu2017}
Liangwei Yang and Chengjie Yu.
\newblock A higher dimensional generalization of {H}ersch-{P}ayne-{S}chiffer
  inequality for {S}teklov eigenvalues.
\newblock {\em J. Funct. Anal.}, 272(10):4122--4130, 2017.

\bibitem{Ya1982}
Shing-Tung Yau.
\newblock Problem section.
\newblock In {\em Seminar on {D}ifferential {G}eometry}, volume 102 of {\em
  Ann. of Math. Stud.}, pages 669--706. Princeton Univ. Press, Princeton, N.J.,
  1982.

\bibitem{YoXiLi2019}
Chun'guang You, Hehu Xie, and Xuefeng Liu.
\newblock Guaranteed eigenvalue bounds for the {S}teklov eigenvalue problem.
\newblock {\em SIAM J. Numer. Anal.}, 57(3):1395--1410, 2019.

\bibitem{Ze2015}
Steve Zelditch.
\newblock Hausdorff measure of nodal sets of analytic {S}teklov eigenfunctions.
\newblock {\em Math. Res. Lett.}, 22(6):1821--1842, 2015.

\bibitem{Zh2016}
Jiuyi Zhu.
\newblock Interior nodal sets of {S}teklov eigenfunctions on surfaces.
\newblock {\em Anal. PDE}, 9(4):859--880, 2016.

\bibitem{Zh2020}
Jiuyi Zhu.
\newblock Geometry and interior nodal sets of {S}teklov eigenfunctions.
\newblock {\em Calc. Var. Partial Differential Equations}, 59(5):Paper No. 150,
  23, 2020.

\end{thebibliography}

\end{document}